\documentclass[11pt]{article}
\usepackage{mathrsfs}
\usepackage{amssymb}
\usepackage{color}
\usepackage{amsmath}
\allowdisplaybreaks[4]

\usepackage{amsthm}
\usepackage{amsmath}
\usepackage{graphicx}

 \newtheorem{thm}{Theorem}[section]
 
 \newtheorem{lem}[thm]{Lemma}
 
 \theoremstyle{definition}
 
 \newtheorem{rem}[thm]{Remark}
 \numberwithin{equation}{section}
 \def\bR{\mathbb{R}}

\def\bN{\mathbb{N}}
\def\bZ{\mathbb{Z}}

\newtheorem{theorem}{Theorem}[section]

\theoremstyle{definition}

\newcommand{\newcom}{\newcommand}

\newcom{\al}{\alpha}
\newcom{\be}{\beta}
\newcom{\eps}{\epsilon}
\newcom{\veps}{\varepsilon}
\newcom{\e}{\varepsilon}
\newcom{\ga}{\gamma}
\newcom{\Ga}{\Gamma}
\newcom{\ka}{\kappa}
\newcom{\La}{\Lambda}
\newcom{\la}{\lambda}
\newcom{\Om}{\Omega}
\newcom{\Th}{\Theta}
\newcom{\Si}{\Sigma}
\newcom{\si}{\sigma}
\newcom{\tht}{\theta}
\newcom{\dtri}{\nabla}
\newcom{\tri}{\triangle}
\newcom{\oo}{\infty}
\newcom{\vphi}{\varphi}
\newcom{\cB}{{\mathcal B}}
\newcom{\cC}{{\mathcal C}}
\newcom{\cD}{{\mathcal D}}
\newcom{\cF}{{\mathcal F}}
\newcom{\cH}{{\mathcal H}}
\newcom{\cL}{{\mathcal L}}
\newcom{\cM}{{\mathcal M}}
\newcom{\cN}{{\mathcal N}}
\newcom{\cP}{{\mathcal P}}
\newcom{\cS}{{\mathcal S}}
\newcom{\cQ}{{\mathcal Q}}
\newcom{\cT}{{\mathcal T}}
\newcom{\cY}{{\mathcal Y}}
\newcom{\cZ}{{\mathcal Z}}
\newcom{\R}{\mathbb R}
\newcom{\T}{\mathbb T}
\newcom{\BR}{{\mathbb{R}}}
\newcom{\Z}{\mathbb Z}
\newcom{\C}{\mathbb C}
\newcom{\E}{\mathbb E}
\newcom{\wt}{\widetilde}
\newcom{\wu}{\underline{w}}

\newcommand{\vc}[1]{{\bf #1}}

\newcom{\ve}{\vc{e}}
\newcom{\vN}{\vc{N}}
\newcom{\vn}{\vc{n}}
\newcom{\vG}{\vc{G}}
\newcom{\vF}{\vc{F}}
\newcom{\vZ}{\vc{Z}}
\newcom{\vf}{\vc{f}}
\newcom{\vg}{\vc{g}}
\newcom{\vq}{\vc{q}}
\newcom{\vu}{\vc{u}}
\newcom{\vv}{\vc{v}}
\newcom{\vw}{\vc{w}}
\newcom{\vb}{\vc{b}}
\newcom{\vh}{\vc{h}}
\newcom{\vz}{\vc{z}}
\newcom{\vup}{\vu^{+}}
\newcom{\vum}{\vu^{-}}
\newcom{\vvp}{\vv^{+}}
\newcom{\vvm}{\vv^{-}}
\newcom{\vbp}{\vb^{+}}
\newcom{\vbm}{\vb^{-}}
\newcom{\vhp}{\vh^{+}}
\newcom{\vhm}{\vh^{-}}
\newcom{\Omp}{{\Om^+}}
\newcom{\Omm}{{\Om^-}}
\newcom{\vupm}{{\vu^{\pm}}}
\newcom{\vvpm}{{\vv^{\pm}}}
\newcom{\vbpm}{{\vb^{\pm}}}
\newcom{\vhpm}{{\vh^{\pm}}}
\newcom{\vwp}{{\vc{w}^+}}
\newcom{\vwm}{{\vc{w}^-}}
\newcom{\vwpm}{{\vc{w}^{\pm}}}
\newcom{\Ompm}{{\Omega^{\pm}}}
\newcom{\vom}{\boldsymbol{\omega}}
\newcom{\vvap}{\boldsymbol{\varpi}}
\newcom{\vop}{\vom^{+}}
\newcom{\vnu}{\boldsymbol{\nu}}
\newcom{\vopm}{\vom^{\pm}}
\newcom{\vjp}{\vj^+}
\newcom{\vjm}{\vj^-}
\newcom{\vjpm}{\vj^{\pm}}
\newcom{\vj}{\boldsymbol{\xi}}
\newcom{\Ds} {\langle\nabla\rangle^{s-\f12}}
\newcom{\vcY}{\ud{\mathcal{Y}}}
\newcom{\vcZ}{\ud{\mathcal{Z}}}
\newcom{\p}{\partial}
\newcom{\f}{\frac}
\newcom{\ud}{\underline}

\def\div{\mathop{\rm div}\nolimits}
\def\curl{\mathop{\rm curl}\nolimits}

\def\d{\mathrm{d}}
\newcommand{\dx}{{\rm d} {x}}

\newcommand{\dt}{{\rm d} t }

\newcom{\ds}{{\rm d} s}

\theoremstyle{remark}
\newtheorem{remark}[theorem]{Remark}

\begin{document}
\title{Global current-vortex sheets in the two-dimensional ideal incompressible MHD }

\author{Yuan Cai\footnote{School of Mathematical Sciences, Fudan University, Shanghai 200433, P. R. China. Email: caiy@fudan.edu.cn
}
\and
Zhen Lei\footnote{School of Mathematical Sciences, Fudan University; Shanghai Center for Mathematical Sciences, Shanghai 200433, P. R. China. Email: zlei@fudan.edu.cn}
}
\date{}
\maketitle

\begin{abstract}
The magnetohydrodynamic current-vortex sheet is a free boundary problem involving
 a moving free surface separating two plasma regions.
We prove the global nonlinear stability of current-vortex sheet in the two dimensional ideal incompressible magnetohydrodynamics under the strong horizontal background magnetic field. This appears to be the first result
on the global solutions of the free boundary problems for the ideal (inviscid and non-resistive) incompressible rotational fluids.
The strong magnetic field plays a crucial role in the global in time stabilization effect.
The proof relies on the understanding of the interplay between the dynamics of the fluids inside the domain and on the free interface, a design of multiple-level energy estimates with different weights, and the inherent structures of the problem.

\end{abstract}

\maketitle


\tableofcontents




\section{Introduction}
Magnetohydrodynamics (MHD) is the fundamental equation in magneto-fluid mechanics. It describes the dynamics of electrically conducting fluids
arising from plasma or some other physical phenomenons.
Magnetohydrodynamics is the study of the dynamics of electrically conducting fluids.
Examples of such fluids include plasmas, liquid metals, and salt water or electrolytes.
The velocity field obeys the Navier-Stokes equations with Lorentz force.
 The magnetic field satisfies the non-resistive Maxwell-Faraday equations which
describes the Faraday's law of induction.
When the fluid is assumed to be ideal and a perfect conductor,  both the Reynolds number and the magnetic Reynolds number are very high, and hence the viscosity and the resistivity can be ignored. The ideal incompressible MHD reads
%
\begin{equation} \label{A1}
\begin{cases}
\p_t u + u\cdot\nabla u + \nabla  p  = h\cdot \nabla h, 
\\
\p_t h + u\cdot\nabla h = h\cdot\nabla u,\\
\div u = 0,\,\,
\div h=0.
\end{cases}
\end{equation}
Here $u$ represents the velocity field, $h$ the magnetic field and $ p $ the scalar total pressure.

Between two plasmas with different properties, a plasma-plasma interface naturally arises, referring to the boundary or transition region. This concept is crucial in various contexts, including space physics, where the Earth's magnetosphere interacts with the solar wind, and in laboratory settings like fusion experiments. The interface moves freely with the fluids. The total pressure is balanced on the interface, while the normal part of the magnetic field vanishes, and the tangent part may jump, thus forming a tangential discontinuity.

Mathematically, plasma-plasma interface is formulated as a free boundary problem for the MHD
system, see, for example \cite{CMST, SWZ18}. It is referred to as the current-vortex sheet problem for the incompressible MHD.
 Denote
\[
\Om =\bR \times [-1,1] \subseteq \BR^2.
\]
$\Ga_f$ is a graph of the free surface:
\[
\Ga_f=\left\{x\in \Om | x_2=f(t,x_1) \in (-1, 1),  x_1 \in \BR \right\},
\]
such that
\[
\Om_f=\big\{ x \in \Om| -1<x_2 < f(t,x_1),  x_1 \in \BR\big\},
\quad\hat{\Om}_f=\big\{ x \in \Om| f(t,x_1)<x_2<1,  x_1 \in \BR \big\}
.
\]
The unknowns are denoted by
\[
u := u |_{\Om_f}, \quad h := h|_{\Om_f}, \quad p := p|_{\Om_f},
\]
\[
\hat{u} := u |_{\hat{\Om}_f}, \quad \hat{h} := h|_{\hat{\Om}_f}, \quad \hat{p} := p|_{\hat{\Om}_f}.
\]
They satisfy
\begin{equation}\label{A5}
	\begin{cases}
		\p_t u + u \cdot\nabla u - h \cdot\nabla h + \nabla p = 0\quad \text{ in }\, \Om_f,  \\
		\p_t h + u\cdot\nabla h - h \cdot\nabla u = 0\quad\text{ in }\, \Om_f,  \\
		\div u = 0,\,\div h = 0\quad\text{ in }\, \Om_f,  \\
		u\cdot n_f =V(t,x), \quad h \cdot n_f = 0\quad\text{ on }\,\Ga_f,
	\end{cases}
\end{equation}
and
\begin{equation}\label{A6}
	\begin{cases}
		\p_t \hat{u} + \hat{u} \cdot\nabla \hat{u} - \hat{h} \cdot\nabla \hat{h} + \nabla \hat{p} = 0\quad \text{ in }\, \hat{\Om}_f,  \\
		\p_t \hat{h} + \hat{u}\cdot\nabla \hat{h} - \hat{h} \cdot\nabla \hat{u} = 0\quad\text{ in }\, \hat{\Om}_f,  \\
		\div \hat{u} = 0,\,\div \hat{h} = 0\quad\text{ in }\, \hat{\Om}_f,  \\
		\hat{u}\cdot n_f =V(t,x), \quad \hat{h} \cdot n_f = 0\quad\text{ on }\,\Ga_f.
	\end{cases}
\end{equation}
with a contact discontinuity condition for the pressure
\begin{align}\label{A7}
	[p]= p -\hat{p}=0\quad \text{on}\,\, \Ga_f.
\end{align}
Here $n_f$ is the outward unit normal to $\Ga_f$,
$V(t,x)$ is the normal velocity of $\Ga_f$.
On the fixed boundaries $\Ga=\bR \times\{x_2=-1\} $, $\hat{\Ga}=\bR \times\{x_2=1\} $,
we impose the following slip boundary conditions for velocity fields and frozen-in boundary conditions for magnetic fields:
\begin{equation}\label{A8}
	u\cdot n =0,\,\, h\cdot n = 0\,\,\text{on}\,\,  \Ga ,\quad
	\hat{u}\cdot n =0,\,\, \hat{h}\cdot n = 0\,\,\text{on}\,\, \hat{\Ga},
\end{equation}
where $n=e_2=(0,1)$.


Under this setting, the boundary conditions on $\Ga_f$ in \eqref{A5}-\eqref{A7} are transformed into
\begin{equation}\label{A10}
	[p]=0, \quad u\cdot N_f =\hat{u}\cdot N_f = \p_t f,
	\quad h\cdot N_f =\hat{h}\cdot N_f = 0\quad\text{on}\,\,\Ga_f,
\end{equation}
where
\[
N_f=e_2-e_1\p_1f, \quad n_f=\f{N_f}{|N_f|},\quad e_1=(1,0).
\]
The system \eqref{A5}-\eqref{A6}, \eqref{A10} is  supplemented with initial data
\begin{align}\label{A12}
	\begin{cases}
		u(0,x) = u^{0}(x), \quad h(0,x)=h^{0}(x)\quad \text{in}\,\,\Om_{f_0},\\
		\hat{u}(0,x) = \hat{u}^{0}(x), \quad \hat{h}(0,x)=\hat{h}^{0}(x)\quad \text{in}\,\,\hat{\Om}_{f_0},\\
		f(0,x_1)=f^0(x_1),\quad \p_tf(0,x_1)=v^0(x_1) ,\,\, x_1\in\bR
	\end{cases}
\end{align}
where the initial data satisfies
\begin{equation}\label{A13}
	\begin{cases}
		\div u^{0}=0, \quad\div h^{0}=0\quad\text{ in}\,\, \Om_{f_0},\\
		\div \hat{u}^{0}=0, \quad\div \hat{h}^{0}=0\quad\text{ in}\,\, \hat{\Om}_{f_0},\\
		u^0\cdot n_{f_0} = \hat{u}^0\cdot n_{f_0},
		\quad h^0 \cdot n_{f_0} =\hat{h}^0 \cdot n_{f_0}= 0\quad\text{on}\,\,\Gamma_{f_0},\\
		u^0 \cdot e_2=h^0 \cdot e_2=0\,\, \text{ on } \Ga ,
		\quad\hat{u}^0 \cdot e_2=\hat{h}^0 \cdot e_2=0\,\, \text{ on }\hat{\Ga} .
	\end{cases}
\end{equation}
Here $u^{0}, h^{0}$ and $\hat{u}^{0}, \hat{h}^{0}$ denote the initial velocity field and magnetic field in the lower and upper region, respectively.
$f^0, v^0$ denote the initial amplitude and velocity field  of the free surface, respectively.
The initial data $u^{0}, h^{0}, \hat{u}^{0}, \hat{h}^{0}, f^0, v^0$ are assumed to be smooth and decay fast
enough at spatial infinity.

\begin{remark}
Let us remark that the divergence free restriction on $h$ and $\hat{h}$ is automatically satisfied
if $\div h^0=0,\,\div \hat{h}^0=0$ since
\begin{equation*}
\p_t (\div h)+u\cdot\nabla  (\div h)=0,\quad \p_t (\div \hat{h})+\hat{u}\cdot\nabla  (\div \hat{h})=0.
\end{equation*}
Similar argument can be also applied to yield $h\cdot N_f=0,\,\hat{h}\cdot N_f=0$ if $h^0\cdot N_{f^0}=0,\,\hat{h}^0\cdot N_{f^0}=0$.
\end{remark}

 Coulombel, Morando,  Secchi and Trebeschi \cite{CMST} proved an {\it a priori estimate} without loss of derivatives for the nonlinear current-vortex sheet problem under a strong stability condition. 
In an important paper \cite{SWZ18}, Sun, Wang and Zhang proved the nonlinear stability of the system \eqref{A5}-\eqref{A6}, \eqref{A10} locally in time in the three-dimensional setting under the
stability condition
\begin{align}\label{stab-cond}
| [u]\times h |^2+| [u]\times \hat{h} |^2<2|h \times \hat{h}|^2.
\end{align}
The most important observation of their work is that the system \eqref{A5}-\eqref{A6}, \eqref{A10} is strictly
hyperbolic under
the stability condition \eqref{stab-cond} in the sense that the interface graph satisfies
a second-order hyperbolic equation.

In this article, we prove the global current-vortex sheet for ideal incompressible MHD equations with strong magnetic fields in two dimensions. It appears to be the first result on the global nonlinear stability of free boundary problems for the ideal (inviscid and non-resistive) incompressible rotational fluids. The equation of the interface 
derived in \cite{SWZ18} and the stability effect of strong magnetic field \cite{BSS, CL} played a crucial role in our current work. Apart from those important existing observations,
our proof highly relies on a new approach to understanding  the interplay between the dynamics of the fluids inside the domain and on the free interface, a design of multiple-level energy estimates with different weights, and the inherent structures hidden within the nonlinearities. We will explain them in more detail in subsection \ref{MaI}.


The goal of this paper is to study the global well-posedness of this system under strong background magnetic field along the horizontal direction, which is set to be $e_1=(1,0)$. We show that the steady solution $(u, h,\hat{u},\hat{h}, f) = (\textbf{0}, e_1,\textbf{0}, e_1,0)$ is nonlinearly stable under small initial perturbations. Such kinds of solutions will produce Alfv\'en  waves in literatures (see, for instance, \cite{Alfven42}).

\subsection{Review of related works}
There is a large body of mathematical literature on free boundary problems in fluid mechanics.
Here we are not going to give a complete review on all the related important references, but just
mention some of them which are closely related to our work.

For the free boundary problem in incompressible Euler equations,
early results on the local well-posedness of the system include those by Nalimov \cite{N}, 
Yosihara \cite{Yo},
Kano-Nishida \cite{KN}
 and Craig \cite{Craig}.
These results deal with small
perturbations of initial interfaces.
The  local well-posedness for general initial data was achieved by the breakthrough work of Wu \cite{Wu1,Wu2} (see also Lannes \cite{Lannes}).
She showed that in dimensions 2 and 3,
strong Rayleigh-Taylor sign condition
\begin{align*}
-\f{\p p}{\p n}\geq c_0 >0
\end{align*}
always holds for the infinite depth water wave problem, and the problem is locally well-posed in Sobolev spaces $H^s$
, $s \geq 4$, for given $H^s$ data with arbitrary size.
Since then, well-posedness
for water waves and free-boundary Euler equations with additional effects such as the surface tension, bottom and non-zero vorticity
has been addressed by several authors.
Beyer-Gunther \cite{BG} took into account the effects of surface tension,
and Lannes \cite{Lannes} treated the case of non-trivial bottom topography.
For the general free boundary problem in incompressible Euler equations, the first local well-posedness in 3D was obtained by Lindblad \cite{Lindblad05} for the case without surface tension (see Christodoulou and Lindblad \cite{CL_00} for the a priori estimates) and
by Coutand and Shkoller \cite{CS07} for the case with (and without) surface tension. We also refer to the results of Shatah and Zeng \cite{SZ1, SZ2, SZ3}
 and Zhang and Zhang
\cite{ZZ}.
For the development of singularities
in finite time of free boundary problems,
we refer to Castro, C\'{o}rdoba, Fefferman, Gancedo and G\'{o}mez-Serrano
\cite{CCFG1, CCFG2}, Fefferman, Ionescu and Lie \cite{Feff}, Coutand \cite{Coutand}, etc.
We also refer to Alazard, Burq and Zuily \cite{ABZ-D, ABZ}, Ambrose and Masmoudi \cite{Am05, Am07, Am09}, Wu \cite{Wu5}, Christianson-Hur-Staffilani \cite{CHS},
Coutand and Shkoller \cite{DS_10}, Gu and Lei \cite{Gu_2011, Gu_2016},
Masmoudi and Rousset \cite{MasRou},  Wang and Xin \cite{Wang_15}, etc. for more works on the free boundary problem of inviscid fluids.


For the irrotational inviscid fluids under gravity or with surface capillary forms, certain dispersive effects can be used to establish the long time existence for the small initial data. The global solution for the three dimensional gravity water wave was obtained by Germain, Masmoudi and Shatah \cite{GMS1}, Wu \cite{Wu4}.
 The three dimensional capillary water  waves was latter on obtained also by  Germain, Masmoudi and Shatah \cite{GMS2}.
For the more difficult two dimensional case, we refer to
Wu \cite{Wu3} on the almost global solutions and
Ionescu and Pusateri \cite{IP} and Alazard and Delort \cite{AD} for the global solutions (see also  Hunter-Ifrim-Tataru for another proof of the almost global solution \cite{HIT} and global solutions \cite{IT}).
For the two dimensional global solutions for capillary wave, we refer to  Ionescu and Pusateri \cite{IP2}.
We also refer to  Deng-Ionescu-Pausader-Pusateri \cite{DIPP}
on the global solutions of the gravity-capillary water-wave system in three dimensions. For the global solutions in three dimensional finite depth case, we refer to Wang \cite{W, W1}.

The vortex sheet for the incompressible Euler equations has vorticity concentrated as a measure (delta function)
in a set of codimension one, a curve or a surface for two dimensional flow or three dimensional flow, respectively.
We refer to \cite{MB} for more introductions.
The MHD current-vortex sheet problem describes a velocity and magnetic field tangential discontinuity in two different ideal MHD flows. The total pressure is balanced on the interface, while the normal part of the magnetic field vanishes.
Syrovatskii \cite{Sy} and Axford \cite{Ax} found  the
 necessary and sufficient condition for the planar (constant coefficients) incompressible
current-vortex sheet a long time ago.
The linear stability condition reads as follows
\begin{align}\label{stab-cond0}
&\big|[u]\big|^2\le 2\big(|h|^2+|\hat{h}|^2\big),\\\nonumber
&\big|[u]\times h\big|^2+\big|[u]\times \hat{h}\big|^2\le 2\big|h\times \hat{h}\big|^2.
\end{align}
Trakhinin \cite{Tra-in1} proved an {\it a priori} estimate without loss of derivative from data for the linearized system with variable coefficients.
Under the stability condition \eqref{stab-cond},
 Morando, Trakhinin and Trebeschi \cite{Mo1} proved an {\it a priori} estimate with a loss of three derivatives for the linearized system. Under a strong stability condition,
Trakhinin \cite{Tra-in1} proved an {\it a priori} estimate without loss of derivative from data for the linearized system with variable coefficients. Coulombel, Morando,  Secchi and Trebeschi \cite{CMST} proved an {\it a priori estimate} without loss of derivatives for the nonlinear current-vortex sheet problem under the following stability condition
\begin{align*} 
\textrm{max}\big( |[u]\times h|, |[u]\times \hat{h}| \big)<|h \times \hat{h}|.
\end{align*}
The nonlinear stability was solved by Sun, Wang and Zhang \cite{SWZ18} under the stability condition \eqref{stab-cond}.
It should be noted that the above mentioned results on the vortex sheet problem are all local in time.
Recently,  Wang and Xin \cite{Wang_20}   proved the global solutions plasma-vacuum interface problem for the incompressible inviscid resistive MHD
in a horizontally periodic slab impressed by a uniform non-horizontal magnetic field
and with surface tension on the free surface.

For more related studies on the (local in time)
stability of the compressible current-vortex sheets, we refer to Trakinin \cite{Tra1, Trak_09}, Chen and Wang \cite{Chen_08}, Wang and Yu \cite{WY}.
For the studies on the on the 
 stability of the plasma-vacuum interface problem,
we refer to Secchi and Trakhinin \cite{Secchi_14, Trak_10,Secchi_13}, Morando, Trakhinin and Trebeschi \cite{Mo_14}, Sun, Wang and Zhang \cite{SWZ19}, Gu \cite{Gu_17}, Hao and Luo \cite{Hao_13},
 Gu and Wang \cite{GW_16}.

Despite these progress, it is still unknown whether the free boundary problem of the ideal incompressible fluids for the general small initial data admits global solutions or not.
In this paper, we are going to address the global solutions for
the current-vortex sheet free boundary problem of the
two dimensional ideal incompressible MHD for small initial data with possibly large graph amplitude.

\subsection{Els\"{a}sser variables }
To study the global nonlinear stability of the MHD current-vortex sheet, we need to explore the propagation of Alfv\'{e}n waves. In terms of the  Els\"{a}sser variables
\begin{equation*}
	Z_{\pm}=u \pm  h .
\end{equation*}
The ideal incompressible MHD equations \eqref{A1} can be rewritten by
\begin{equation} \label{A2}
	\begin{cases}
		\p_t Z_{+} +Z_{-}\cdot\nabla Z_{+} + \nabla p =0, \\
		\p_t Z_{-} +Z_{+}\cdot\nabla Z_{-} + \nabla p =0, \\
		\div Z_{\pm}=0.
	\end{cases}
\end{equation}
When \eqref{A2} are linearized around the static solution with a constant magnetic field $e_1$, one obtains that the fluctuations $\La_\pm = Z_\pm \mp e_1$ propagate
along the $e_1$ magnetic field in opposite directions. This suggests that in the original
nonlinear problem, a strong enough magnetic field will reduce the nonlinear
interactions \cite{Kr} and inhibit formation of strong gradients. This effect was observed
in direct numerical simulations of equations \eqref{A2} with periodic boundary conditions
\cite{FPSM}.
These calculations showed that in the presence of a strong enough magnetic
field, solutions remain analytic in a strip whose width is bounded from below.
The rigorous mathematical verification of this effect is first proved by Bardos-Sulem-Sulem \cite{BSS} in $\bR^2$ and $\bR^3$, see also  \cite{HeXuYu, CL, WZ} in $\bR^3$ for the viscous case. The result \cite{CL} is also valid in in $\bR^2$.

We use the following  Els\"{a}sser variables in the upper and lower domains
\begin{equation*}
\begin{cases}
Z_{\pm}=u \pm  h \quad \textrm{in}\,\,\Om_f, \\
\hat{Z}_{\pm}=\hat{u} \pm  \hat{h} \quad \textrm{in}\,\,\hat{\Om}_f.
\end{cases}
\end{equation*}
In terms of the Els\"{a}sser variables, \eqref{A5}-\eqref{A6} can be rewritten by
\begin{equation} \label{A15}
\begin{cases}
\p_t Z_{+} +Z_{-}\cdot\nabla Z_{+} + \nabla p =0, \\
\p_t Z_{-} +Z_{+}\cdot\nabla Z_{-} + \nabla p =0, \\
\div Z_{\pm}=0,
\end{cases}
\textrm{in}\,\, \Omega_f,
\end{equation}
respectively,
\begin{equation} \label{A16}
\begin{cases}
\p_t \hat{Z}_{+} +\hat{Z}_{-}\cdot\nabla \hat{Z}_{+} + \nabla \hat{p} =0, \\
\p_t \hat{Z}_{-} +\hat{Z}_{+}\cdot\nabla \hat{Z}_{-} + \nabla \hat{p} =0, \\
\div \hat{Z}_{\pm}=0,
\end{cases}
\textrm{in}\,\, \hat{\Omega}_f.
\end{equation}
By \eqref{A10}, the moving equation of the free surface can be written as follows
\begin{equation}\label{A17}
\p_t f = Z_{\pm} \cdot N_f=\hat{Z}_{\pm} \cdot N_f\quad \text{on}\,\, \Ga_f .
\end{equation}

The global stability of the free surface will be achieved by imposing strong background magnetic fields.
Inspired by the stability conditions  \eqref{stab-cond}, \eqref{stab-cond0}, we consider the case where the background magnetic fields in the domains $\Om_f$ and $\hat{\Om}_f$  are both set to be
$$e_1=(1,0).$$
Introducing the perturbation variables:
\begin{equation*}
\La_{+}=Z_{+}-e_1,\, \La_{-}=Z_{-} +e_1\,\, \textrm{in}\,\,\Om_f,
\end{equation*}
and
\begin{equation*}
\hat{\La}_{+}=\hat{Z}_{+}-e_1,\, \hat{\La}_{-}=\hat{Z}_{-}+e_1\,\, \textrm{in}\,\,\hat{\Om}_f.
\end{equation*}
Then we write down the evolution equations for  $\La_{+}$, $\La_{-}$ and
$\hat{\La}_{+}$, $\hat{\La}_{-}$:
\begin{equation} \label{A18}
\begin{cases}
\p_t \La_{+} - \p_1 \La_{+}
+\La_{-}\cdot\nabla \La_{+} + \nabla p =0,\\
\p_t \La_{-} + \p_1 \La_{-}
+\La_{+}\cdot\nabla \La_{-} + \nabla p =0, \\
\div \La_{\pm}=0,
\end{cases}
 \textrm{in}\,\,\Om_f,
\end{equation}
respectively,
\begin{equation} \label{A19}
\begin{cases}
\p_t \hat{\La}_{+} - \p_1 \hat{\La}_{+}
+\hat{\La}_{-}\cdot\nabla \hat{\La}_{+} + \nabla \hat{p} =0,\\
\p_t \hat{\La}_{-} + \p_1 \hat{\La}_{-}
+\hat{\La}_{+}\cdot\nabla \hat{\La}_{-} + \nabla \hat{p} =0,\\
\div \hat{\La}_{\pm}=0.
\end{cases}
 \textrm{in}\,\,\hat{\Om}_f,
\end{equation}
The boundary conditions \eqref{A8} and \eqref{A10} become
\begin{align}\label{A20}
\begin{cases}
\La_{\pm}\cdot N_f =\hat{\La}_{\pm}\cdot N_f = (\p_t\pm\p_1) f\quad\text{on}\,\,\Ga_f, \\
\hat{\La}_{\pm}\cdot e_2=0
\,\, \text{on}\,\, \hat{\Ga},
\quad  \La_{\pm}\cdot e_2=0
\,\, \text{on}\,\, \Ga, \\
p=\hat{p} \quad \text{on}\,\, \Ga_f.
\end{cases}
\end{align}
The initial data \eqref{A12} becomes
\begin{align*}
\begin{cases}
\La_{\pm}(0,x)= \La_{\pm}^0(x) \,\,\,\, \text{in}\,\, \Om_{f_0},\\
\hat{\La}_{\pm}(0,x)= \hat{\La}_{\pm}^0(x) \,\,\,\, \text{in}\,\, \hat{\Om}_{f_0},\\
f(0,x_1)=f^0(x_1),\quad
\p_tf(0,x_1)=v^0(x_1),\,\, x_1\in\bR,
\end{cases}
\end{align*}
where $\La_{\pm}^0=u^{0}\pm (h^{0}-e_1)$, $\hat{\La}_{\pm}^0=\hat{u}^{0}\pm (\hat{h}^{0}-e_1)$.
It is clear from \eqref{A13} that the initial data satisfies
\begin{equation}\label{com}
\left\{\begin{aligned}
&\div \La_{\pm}^0 = 0  \,\, \text{ in } \Om_{f_0},\,\,  \\
&\div \hat{\La}_{\pm}^0 = 0  \,\, \text{ in } \hat{\Om}_{f_0},\,\,  \\
& v^0=(\La_{\pm}^0 \pm e_1)\cdot N_{f_0}=(\hat{\La}_{\pm}^0 \pm e_1)\cdot N_{f_0} \,\, \text{ on } \Ga_{f_0},\\
&\La_{\pm}^0\cdot e_2=0\,\, \text{ on }\Ga  ,\quad
\hat{\La}_{\pm}^0\cdot e_2=0\,\, \text{ on }\hat{\Ga}  .
\end{aligned}\right.
\end{equation}

We will prove that for sufficiently small initial disturbance,
the weighted Sobolev norm of $\Lambda_\pm$ and $f$
will be uniformly bounded for all the time.
The combination of these uniform estimates with the local existence theory will yield the global solutions of the free boundary problem for the MHD current-vortex sheet in two dimensions.

\subsection{Main result}

The weighted energy $E_s(t)$, the ghost weight energy $G_s(t)$ and other notations appeared in the following theorem will be explained in Section 2. Note that
\begin{align*}
	E_s(0) &= \sum_{+,-}\sum_{|\alpha|\leq 3}
	\big(\|\langle x_1\rangle^{5\mu} \nabla^\alpha\La_{\pm}^0\|^2_{L^2(\Om_{f_0})}
	+\|\langle x_1\rangle^{5\mu} \nabla^\alpha\hat{\La}_{\pm}^0\|^2_{L^2(\hat{\Om}_{f_0})}\big)\\
	&+\sum_{+,-}\sum_{|\alpha|\leq s}
	\big(\|\langle x_1\rangle^{2\mu} \nabla^\alpha\La_{\pm}^0\|^2_{L^2(\Om_{f_0})}
	+\|\langle x_1\rangle^{2\mu} \nabla^\alpha\hat{\La}_{\pm}^0\|^2_{L^2(\hat{\Om}_{f_0})}\big)\\
	&+\sum_{|a|\leq s-1} \big(\| \langle x_1 \rangle^{2\mu} \langle \p_1 \rangle^{\f12 }\p_1^a v^0\|^2_{L^2(\bR)}
	+\|\langle x_1 \rangle^{2\mu} \langle \p_1 \rangle^{\f12 } \p_1^{a+1}f^0\|^2_{L^2(\bR)}\big).
\end{align*}
We state the main result of this paper  as follows.
\begin{thm}\label{thm}
Let $s\geq 4$ be an integer, $1/2< \mu \leq 3/5$ and $0 < c_0 < \frac{1}{2}$. Assume $\La_{\pm}^0\in H^s(\Om_{f_0}),\, \hat{\La}_{\pm}^0\in H^s(\hat{\Om}_{f_0}) $,
$f^0 \in  \dot{H}^1(\BR) \cap \dot{H}^{s+\frac12}(\BR)\cap L^{\infty}(\BR)$,
$v^0\in H^{s-\f12}(\BR)$. Suppose that $\La_{\pm}^0$, $\hat{\La}_{\pm}^0$, $f^0$
 and $v^0$ satisfy the  conditions \eqref{com}
and
\begin{equation*}
	-1 + c_0 \leq f^0 \leq 1 - c_0.
\end{equation*}
There exists a positive constant $\epsilon_0$ which
depends only on $s,\ \mu$ and $c_0$ such that, if $E_s(0)
\leq \epsilon_0$, then the free boundary problem for the current-vortex sheet of the two dimensional incompressible MHD  \eqref{A18}-\eqref{A20}  with  the following initial data
\begin{align*}
&\La_{\pm}(0,x)= \La_{\pm}^0(x),\quad \hat{\La}_{\pm}(0,x)= \hat{\La}_{\pm}^0(x),\\
& f(0,x_1)=f^0(x_1),\quad
\p_tf(0,x_1)=v^0(x_1)
\end{align*}
has a unique global solution which satisfy
\begin{align*}
&E_s(t) + \int_0^t G_s(\tau)\d\tau\leq C_0E_s(0),\quad
\| f(t,\cdot)\|_{L^\infty}\leq \| f^0\|_{L^\infty} + C_1E_s^{\f12}(0)\leq 1-\f{c_0}{2},
\end{align*}
for some $C_0,C_1 > 1$ and all $t\geq 0$.
\end{thm}

\begin{remark}
The local well-posedness of \eqref{A5}-\eqref{A10} can be carried out as in \cite{SWZ18}.
For the convenience of readers,  we will sketch its proof in the last section.
\end{remark}

\begin{remark}
If $h=0$, the incompressible MHD current-vortex sheet
reduces to classical vortex sheet for Euler equations.
In this case, the evolution of the two dimensional vortex sheet can be described by Birkhoff-Rott (BR) equation.
The linear stability theory, by considering the perturbation of a flat constant solution, showed that the Fourier mode has a exponential growth component referred to as the Kelvin-Helmholtz instability.
For more details, please refer to Chapter 9 in the book \cite{MB}.
For the two dimensional incompressible MHD current-vortex sheet problem, the nonzero background magnetic field is crucial for the global existence of solution.
The nonzero magnetic field acts to damp the exponential
growth in the high modes,  transforming the problem into one that is globally nonlinearly stable.
This is achieved through the propagation of one-dimensional characteristics where the null conditions are satisfied.
\end{remark}

\begin{remark}
The weighted energy $E_s$ is small for all time, implying that the derivatives of $f$ are also small.
The amplitude of the free surface $\| f \|_{L^\infty}$ does not necessarily need to be small.
For example, $f$ can be bounded and grow slowly from $-1+\frac{c_0}{2}$ to $1-\frac{c_0}{2}$ in the $x_1$ direction.
The condition $\|f^0\|_{L^\infty}<1-c_0$ is imposed to ensure that the free surface will not intersect with the top and bottom fixed boundaries.
\end{remark}

\subsection{Main ideas and sketch of the proof} \label{MaI}
For the remarkable works in proving the long time existence of solutions in water wave problems,
 two main methods are prevalent in the literature:
the vector field method and the method of space-time resonances.
We refer to Wu \cite{Wu3, Wu4}, Alazard, Delort \cite{AD}, Hunter, Ifrim, Tataru \cite{HIT, IT}, and others for the (variant) vector field method, and to
Germain, Shatah, Masmoudi \cite{GMS1, GMS2},  Ionescu, Pusateri \cite{IP, IP2}, and others for the method of space-time resonances.
In many of these works,  the irrotational condition is crucial as one can fully describe the evolution of the free boundary in terms of the free boundary and the trace of the velocity potential on the free boundary.
By the expansion of the Dirichlet-Neumann operator, it reduces to the study of the nonlocal quasilinear dispersive equations.
However, when the vorticity is nontrivial, one has to understand the long-time behavior of the flows within the moving domain, which seems to require 
entirely new strategies.
To obtain the global solutions for the two dimensional MHD current-vortex sheet free boundary problem,
one will inevitably encounter the nontrivial nature of rotationality.
It relies on the strong magnetic field,  which leads to the propagation of Alfv\'{e}n waves \cite{Alfven42}.
We also refer to Bardos-Sulem-Sulem \cite{BSS}, Cai-Lei \cite{CL}, He-Xu-Yu \cite{HeXuYu}, Wei-Zhang \cite{WZ} on global small solutions of ideal incompressible MHD in the whole space $\BR^n$ or in a slab.

\subsubsection{Alfv\'en waves and the null conditions}
The second order evolution equation for the free surface derived by Sun-Wang-Zhang \cite{SWZ18} plays a crucial role in our current work.
In terms of the Els\"{a}sser variables, the evolution equation of the free boundary reads as follows (see Lemma \ref{Lem-sf-eq}):
\begin{align*}
\p^2_t f=
&-\f12 (\ud{Z_{+}^1}+\ud{Z_{-}^1}+\ud{\hat{Z}_{+}^1}+\ud{\hat{Z}_{-}^1})
\p_1\p_t f
-\f12(\ud{Z_{+}^1} \ud{Z_{-}^1}
+\ud{\hat{Z}_{+}^1}\ud{\hat{Z}_{-}^1})\p_1^2 f \nonumber\\
&-\f12 N_f\cdot\ud{\nabla \hat{p}}
-\f12 N_f\cdot\ud{\nabla p} \,.
\end{align*}
Here $\ud{g}$ stands for the trace of $g$ on $\Gamma_f.$
$Z_{\pm}^1$, $\hat{Z}_{\pm}^1$ represent the first components of $Z_{\pm}=(Z_{\pm}^1,Z_{\pm}^2),$ $\hat{Z}_{\pm}=(\hat{Z}_{\pm}^1,\hat{Z}_{\pm}^2),$ respectively.
Under the strong background magnetic field assumption along
$e_1$ direction, the equation of the free surface 
 turns out to be one dimensional nonlocal quasilinear wave equation
(see Section 2.3):
\begin{align} \label{A26}
\p_t^2 f-\p_1^2 f
=&-\f12\big(\ud{\La_{+}^1}+\ud{\hat{\La}_{+}^1} \big) \p_1 (\p_t-  \p_1) f
 -\f12\big(\ud{\La_{-}^1}+\ud{\hat{\La}_{-}^1} \big) \p_1 (\p_t+  \p_1) f\\\nonumber
&-\f12\big(\ud{\La_{+}^1} \ud{\La_{-}^1} +\ud{\hat{\La}_{+}^1} \ud{\hat{\La}_{-}^1} \big)\p_1^2 f
-\f12(N_f\cdot\ud{\nabla p} +N_f\cdot\ud{\nabla \hat{p}}).
\end{align}
From the linearized level, $ (\p_t+\p_1) f$ and $(\p_t-\p_1) f$  propagate along the background magnetic field 
 in opposite directions. Moreover, due to the kinetic boundary condition $\eqref{A20}_1$ and $\eqref{A20}_2$, roughly, $ (\p_t+\p_1) f\sim \Lambda_+^2\sim \hat{\La}_+^2$,
$ (\p_t-\p_1) f\sim \La_-^2\sim\hat{\La}_-^2$.
On the other hand, one observes from the linearized equations \eqref{A18}, \eqref{A19} that $\La_{+}$ ($\hat{\La}_{+}$) and $\La_{-}$ ($\hat{\La}_-$)  propagate
along the background magnetic field in opposite directions.
For the nonlinearities of the system both inside the domain and on the free surface (see \eqref{A18}, \eqref{A19}, \eqref{A26}), they  are always the combination of the interaction of
the left Alfv\'en waves $\Lambda_+, (\p_t+\p_1)f$ and the right Alfv\'en waves $\Lambda_-, (\p_t-\p_1)f$.
Moreover,   the pressure (see \eqref{B22}) holds similar structure: it is some combination of the nonlocal nonlinear interaction of left and right Alfv\'en waves modulo the pseudo-differential operator.
In other words, the systems satisfy the null conditions.
These intrinsic structures make it reasonable to expect the nonlinearities will finally be negligible in the perturbation regime.

\subsubsection{Weighted energy estimates}

In order to capture the inherent structure of the system,
we introduce suitable weight
functions to obtain the directional decay of solutions and apply the ghost weight energy
method by Alinhac \cite{Alinhac00} to perform energy estimates with weights.
This will yield the fact that energies of good unknowns with appropriate weights are always integrable in time.

Due to the nontrivial nature of fluid vorticity,
the previous strategy where the problem reduced to the free boundary is not available in the current setting.
We need to explore the dynamics inside the domain, and examine the interplay between inside the domain and the free boundary.
The weighted energy estimate is consisted of three parts: the weighted energy of the free surface,
the weighted energy of the tangential derivative part, and the vorticity part.
With the estimate for the free surface, the tangential derivative part, and the vorticity and current,
we can retrieve the estimate for the velocity and magnetic field by solving the weighted div-curl system.

Precisely, the weighted energy estimate consists of the following five parts:
\begin{itemize}
\item The highest order weighted energy estimate of the free surface
\item The estimate of the amplitude of the free surface graph
\item The highest order weighted energy estimate of the vorticity and current
\item The lower order weighted energy estimate of the tangential derivative part with increased power of weights
\item The lower order weighted energy estimate of the vorticity and current with increased power of weights
\end{itemize}
We refer to \eqref{AA8}-\eqref{AA12} in Section 2.4 for the 
expression of the five parts energy estimate.
The weighted energy estimate of the vorticity part is standard.
The estimate of the amplitude of the free surface is carried out by using the characteristic integral method.
In the sequel, we sketch the weighted energy estimate of the free surface, as well as the weighted tangential derivative estimate of the plasmas inside the domain.
We will explain main ideas and issues associated with these estimates.

\textbf{Weighted energy estimate of the free surface: strong null structure and the variant weights}

The weighted energy energy estimate of the free surface 
involves the derivatives of the highest order $(s\geq 4)$.
We take the weighted energy estimate for $(\p_t-\p_1)f$ as an example.
Define the weight functions $w^\pm(t,x_1)=x_1\pm t$.
Let $s\geq 4$ be an integer, $0\leq a\leq s-1$, $a\in \bZ$.
Applying $\langle \p_1\rangle^{\f12}\p_1^a$ to \eqref{A26}.
Then taking the $L^2$ inner product of the consequent equations with
 $2\langle w^-\rangle^{4\mu}(\p_t-\p_1)\langle \p_1\rangle^{\f12}\p_1^af\cdot e^{q^-}$ where
$e^{q^-}=e^{q(-w^+)}$, $q(\theta)=\int_0^\theta \langle\tau\rangle^{-2\mu} \d\tau$ ($\f12<\mu\leq \f35$),
one has (see Section 6.2)
\begin{align}\label{A27}
& \f{\d}{\dt}
\int_{\BR}\big| \langle w^-\rangle^{2\mu}(\p_t-\p_1) \langle \p_1\rangle^{\f12}\p_1^a f
\big|^2 e^{q^-}\d x_1
+2\int_{\BR}\f{\big| \langle w^-\rangle^{2\mu}(\p_t-\p_1) \langle \p_1\rangle^{\f12}\p_1^a f
\big|^2}{\langle w^+\rangle^{2\mu}} e^{q^-}\d x_1\\\nonumber
&=2\int_{\BR}(\p_t+\p_1) \langle w^-\rangle^{2\mu} (\p_t-\p_1) \langle \p_1\rangle^{\f12}\p_1^a f \cdot \langle w^-\rangle^{2\mu}(\p_t-\p_1) \langle \p_1\rangle^{\f12}\p_1^a f e^{q^-}\d x_1\\\nonumber
&=-\int_{\BR} \langle w^-\rangle^{4\mu} (\p_t-\p_1) \langle \p_1\rangle^{\f12}\p_1^a f\cdot
\Big((\ud{\La^1_{+}}\ud{\La^1_{-}}
+\ud{\hat{\La}^1_{+}}\ud{\hat{\La}^1_{-}})\p_1^2 \langle \p_1\rangle^{\f12}\p_1^a f \\\nonumber
&\qquad\quad+  (\ud{\La^1_{+}}+\ud{\hat{\La}^1_{+}}) \p_1(\p_t-\p_1)
 \langle \p_1\rangle^{\f12}\p_1^a f
+(\ud{\La^1_{-}}+\ud{\hat{\La}^1_{-}})\p_1 (\p_t+\p_1)
\langle \p_1\rangle^{\f12}\p_1^a f \Big) e^{q^-} \d x_1 \nonumber\\
&\quad-\int_{\BR} \langle w^-\rangle^{4\mu} (\p_t-\p_1)\langle \p_1\rangle^{\f12}\p_1^a f\cdot
 \Big( [\langle \p_1\rangle^{\f12}\p_1^a,\ud{\La^1_{+}}\ud{\La^1_{-}}
+\ud{\hat{\La}^1_{+}}\ud{\hat{\La}^1_{-}}] \p_1^2f  \nonumber\\
&\qquad\quad+  [\langle \p_1\rangle^{\f12}\p_1^a,\ud{\La^1_{+}}+\ud{\hat{\La}^1_{+}}]
  \p_1(\p_t-\p_1) f  +[\langle \p_1\rangle^{\f12}\p_1^a,\ud{\La^1_{-}}+\ud{\hat{\La}^1_{-}}]
\p_1(\p_t+\p_1) f \Big) e^{q^-} \d x_1 \nonumber\\\nonumber
&\quad-\int_{\BR} \langle w^-\rangle^{4\mu} (\p_t-\p_1)\langle \p_1\rangle^{\f12}\p_1^a f\cdot
\langle \p_1\rangle^{\f12}\p_1^a\Big( N_f\cdot(\ud{\nabla p} +\ud{\nabla \hat{p} })\Big) e^{q^-} \d x_1.
\end{align}
For the last three lines of \eqref{A27} corresponding to the weighted commutator terms and the pressure terms, respectively, we will explain them in Section 1.4.3. For the third and forth lines of \eqref{A27} containing the highest-order derivatives,
 they may lead to derivative loss. We need to apply integration by parts to save one derivative.
Here, we pick the tricky term (the integration procedure of other terms is standard)
\begin{align*}
&-\int_{\BR} \langle w^-\rangle^{4\mu} (\p_t-\p_1)\langle \p_1\rangle^{\f12}\p_1^a f\cdot
  (\ud{\La^1_{-}}+\ud{\hat{\La}^1_{-}})\cdot \p_1(\p_t+\p_1)
 \langle \p_1\rangle^{\f12}\p_1^a f e^{q^-} \d x_1.
\end{align*}
For this term, by using integration by parts, it is equal to
\begin{align}\label{A29}
&-\f{\d}{\dt}\int_{\BR} \langle w^-\rangle^{4\mu}  (\p_t-\p_1)\langle \p_1\rangle^{\f12}\p_1^a  f \big(\ud{\La_{-}^1}
+\ud{\hat{\La}_{-}^1}) \cdot \p_1  \langle \p_1 \rangle^{\f12}\p_1^a  f e^{q^-} \d x_1 \\
&+\int_{\BR} \langle w^-\rangle^{4\mu}  (\p_t-\p_1)\langle \p_1\rangle^{\f12}\p_1^a  f\cdot (\p_t+ \p_1)\big(\ud{\La_{-}^1}+\ud{\hat{\La}_{-}^1})
 \p_1  \langle \p_1\rangle^{\f12}\p_1^a  f e^{q^-} \d x_1  \nonumber\\
&+\int_{\BR} \langle w^-\rangle^{4\mu} (\p_t-  \p_1)\langle \p_1\rangle^{\f12}\p_1^a  f \big(\ud{\La_{-}^1}
+\ud{\hat{\La}_{-}^1}) \cdot   \p_1  \langle \p_1\rangle^{\f12}\p_1^af(\p_t+\p_1) e^{q^-}\d x_1 \nonumber\\\nonumber
&+\int_{\BR} \langle w^-\rangle^{4\mu} (\p_t-  \p_1)(\p_t+\p_1)\langle \p_1\rangle^{\f12}\p_1^a  f \big(\ud{\La_{-}^1}
+\ud{\hat{\La}_{-}^1}) \cdot   \p_1  \langle \p_1\rangle^{\f12}\p_1^af e^{q^-} \d x_1 .
\end{align}
For the last line of \eqref{A29}, we need to plug in  $(\p_t^2-\p_1^2)\langle \p_1\rangle^{\f12}\p_1^a  f$ by using
 the equation of the free surface \eqref{A26} to identify the symmetry to reduce one derivative.
 Here it reflects the ``\textbf{strong null structure}'' which has an effect similar to the normal form.
 We mention that the effect that some of the quadratic terms becoming cubic terms will not help to achieve the global solutions since the unknowns of the solutions do not decay in time.

Finally, the third and forth lines of of \eqref{A27} will become:
\begin{align*}
& \f14 \f{\d}{\dt}\int_{\BR}
|\p_1 \langle \p_1\rangle^{\f12}\p_1^a f|^2
e^{q^-}\cdot\langle w^-\rangle^{4\mu} (\ud{\La^1_{+}}-\ud{\hat{\La}^1_{+}}) (\ud{\La^1_{-}}-\ud{\hat{\La}^1_{-}}) \d x_1\\
&-\f14\int_{\BR} \big| \p_1 \langle \p_1\rangle^{\f12}\p_1^a f  \big|^2 e^{q^-}\cdot (\p_t-  \p_1)
\Big(\langle w^-\rangle^{4\mu} \big(\ud{\La_{+}^1}-\ud{\hat{\La}_{+}^1} \big)\big(\ud{\La_{-}^1}-\ud{\hat{\La}_{-}^1}) \Big)
  \d x_1\\
  & +\textrm{nonlinear terms  can be estimated}.
\end{align*}
Here, in order to employ the ghost weight energy, the weight applied on $\p^{\leq 1} \La_-$ and on $\p^{\leq 1} \hat{\La}_-$ should be $\langle w^-\rangle^{5\mu}$. However, $\p^{\leq 1} \La_-$ and  $\p^{\leq 1} \hat{\La}_-$ can only carry the weight  $\langle w^-\rangle^{2\mu}$.

\textbf{Variant weight choices for the free surfaces and the lower order derivative energy estimate}

In order to  estimate
\begin{align*}
\sum_{+,-}\sum_{|a|\leq 1}
\| \langle w^\pm\rangle^{5\mu} \big(|\p_1^a \ud{\La_{\pm}^1}|+|\p_1^a\ud{\hat{\La}_{\pm}^1}|\big) \|_{L^\infty}.
\end{align*}
We introduce the lower order weighted energy with increased power of weights:
\begin{align*}
E_3^b (t)=
 \sum_{+,-}\sum_{|\alpha|\leq 3}
\big(\| \langle w^\pm \rangle^{5\mu}\nabla^\alpha \La_{\pm} \|^2_{L^2(\Om_f)}
+\| \langle w^\pm \rangle^{5\mu}\nabla^\alpha \hat{\La}_{\pm}\|^2_{L^2(\hat{\Om}_f)}\big).
\end{align*}

\textbf{The lower order weighted energy estimate of the tangential derivative part with increased power of weights inside the domain}

In the weighted energy estimate of $\La_\pm$ and $\hat{\La}_\pm$ inside the domain
$\Om_f$ and $\hat{\Om}_f$, to use the boundary conditions, we need to apply the tangential derivative estimate and  vorticity and current estimate. This is different from the case of whole space or in a strip as in \cite{BSS, HeXuYu, CL, WZ} where the traditional derivatives can be applied.

For the lower-order $(3-\textrm{order})$ tangential weighted energy estimate of $\La_\pm$ and $\hat{\La}_\pm$ inside the domain $\Om_f$ and $\hat{\Om}_f$,
the estimates are conducted by applying increased power of weights $\langle w^\pm\rangle^{5\mu}$.
At the same time, we also need to use the ghost weight method.
Due to the non-vanishing boundary conditions, some of the estimates will reduce to the free boundary.
In the sequel, we will sketch the weighted energy estimate for $\La^+$ and $\hat{\La}^+$ as an example.

We denote by $\p_\tau$ the tangential derivative where the definition is given in \eqref{tangD-}.
There holds $\ud{\p_\tau g}= \ud{\p_1 g}+\p_1f \ud{\p_2 g}=\p_1 \ud{g}.$
For $0\leq m\leq 3$, applying derivative $\p_\tau^m$ to $\eqref{A18}_1$, $\eqref{A19}_1$. 
Denote $e^{q^+}=e^{q(w^-)}$ where $q(\theta)=\int_0^\theta \langle\tau\rangle^{-2\mu} d\tau$.
Then by using the weights  $\langle w^+ \rangle^{5\mu} $ and the boundary conditions, we derive that
(see Section 7)
\begin{align}\label{A30}
&\f12\f{\d}{\dt} \Big( \int_{\Om_f}
| \langle w^+ \rangle^{5\mu} \p_\tau^m \La_{+}|^2 e^{q^+} \dx
+\int_{\hat{\Om}_f}
| \langle w^+ \rangle^{5\mu} \p_\tau^m \hat{\La}_{+}|^2 e^{q^+} \dx \Big)\\\nonumber
&+\int_{\Om_f}
 \f{|\langle w^+ \rangle^{5\mu}\p_\tau^m \La_{+} |^2}{ \langle w^- \rangle^{2\mu} }
 e^{q^+}\dx
+\int_{\hat{\Om}_f} \f{|\langle w^+ \rangle^{5\mu}\p_\tau^m \hat{\La}_{+} |^2}{ \langle w^- \rangle^{2\mu} }
 e^{q^+}   \dx \\\nonumber
&=-\int_{\Om_f}
\langle w^+ \rangle^{10\mu}\p_\tau^m \La_{+} e^{q^+} \cdot
 \nabla \p_\tau^m p\dx
 -\int_{\hat{\Om}_f}
\langle w^+ \rangle^{10\mu}\p_\tau^m \hat{\La}_{+} e^{q^+} \cdot
 \nabla \p_\tau^m \hat{p}\dx \\\nonumber
 &\quad+\textrm{the other nonlinear terms}\\\nonumber
&= -\int_{\Ga_f} e^{q^+}\langle w^+\rangle^{10\mu} \p_\tau^m p n_f\cdot \p_\tau^m \La_{+}\d\sigma(x)
+\int_{\Ga_f} e^{q^+}\langle w^+\rangle^{10\mu} \p_\tau^m \hat{p} n_f\cdot \p_\tau^m \hat{\Lambda}_{+}\d\sigma(x)+...
\end{align}
Here we omit the other nonlinear terms since the estimate is standard. We will focus on the pressure estimate on the free boundary.
For the first two terms in the last line of \eqref{A30}, due to the boundary condition $\eqref{A20}_1$,
we derive
\begin{align*}
& -\int_{\Ga_f} e^{q^+}\langle w^+\rangle^{10\mu} \p_\tau^m p n_f\cdot \p_\tau^m \La_{+}\d\sigma(x)
+\int_{\Ga_f} e^{q^+}\langle w^+\rangle^{10\mu} \p_\tau^m \hat{p} n_f\cdot \p_\tau^m \hat{\Lambda}_{+}\d\sigma(x)
\\
&= \int_{\BR} e^{q^+}\langle w^+\rangle^{10\mu} \p_1^m \ud{\hat{p}} N_f \cdot
\p_1^m\big(\ud{  \hat{\Lambda}_{+}}-\ud{  \La_{+}}\big) \d x_1  \\
&= \int_{\BR} e^{q^+}\langle w^+\rangle^{10\mu}  \p_1^m \ud{\hat{p}} \sum_{0<a\leq m}C_m^a\p_1^{a+1} f \cdot \p_1^{m-a}
 \big(\ud{ \hat{\La}_{+}^1}- \ud{\La_{+}^1}\big) \d x_1.
\end{align*}
Note that the term $\p_1^{a+1} f$ contains higher order derivatives
when $a=m=3$. Fortunately, we can distribute the weight $\langle w^+\rangle^{10\mu}$ onto the pressure and onto $\La_{+}$, $\hat{\La}_{+}$
but not on $f$. While $\p_1^{a+1} f$ will be bounded by the higher order energy of the free surface.

\subsubsection{Further technical issues}

\begin{itemize}
\item Low frequency issue for the weighted estimate of the pressure
	
In the weighted gradient estimate of the pressure, we will encounter the low-frequency issue: the deficiency of the estimate of the pressure itself with weights.
More precisely, we collect the elliptic equation and the boundary conditions for the pressure  as follows (see Section 2.3):
\begin{equation*}
\begin{cases}
\Delta p =-\nabla \cdot (\La_{-} \cdot \nabla \La_{+}) \quad \textrm{in}\,\, \Om_f,\\
\Delta \hat{p} =-\nabla \cdot (\hat{\La}_{-} \cdot \nabla \hat{\La}_{+})\quad \textrm{in}\,\, \hat{\Om}_f,\\
p =\hat{p}\quad \textrm{on} \,\, \Gamma_f,\\
N_f\cdot\nabla p-N_f\cdot\nabla \hat{p}= -\big(\ud{\La_{+}^1} \ud{\La_{-}^1}- \ud{\hat{\La}_{+}^1} \ud{\hat{\La}_{-}^1} \big) \p_1^2 f \\
\qquad-\big(\ud{\La_{+}^1}-\ud{\hat{\La}_{+}^1} \big) \p_1 (\p_t-  \p_1) f
  -\big(\ud{\La_{-}^1}-\ud{\hat{\La}_{-}^1}\big) \p_1 (\p_t+  \p_1) f
\quad\text{on}\quad\Ga_f ,\\
\p_2 p=0\,\,   \textrm{on} \,\,  \Ga,
\quad \p_2 \hat{p}=0\,\,   \textrm{on} \,\,  \hat{\Ga}.
\end{cases}
\end{equation*}
The weighted energy estimate of the pressure is sketched as follows
\begin{align}\label{A34}
&\| \langle w^+\rangle^\mu \langle w^-\rangle^\nu \nabla p\|^2_{L^2(\Om_f)}+
\| \langle w^+\rangle^\mu \langle w^-\rangle^\nu \nabla \hat{p}\|^2_{L^2(\hat{\Om}_f)} \\\nonumber
&=2\int_{\Om_f} \nabla(\langle w^+\rangle^\mu \langle w^-\rangle^\nu)p \cdot  \langle w^+\rangle^\mu \langle w^-\rangle^\nu \nabla p \dx \\\nonumber
&\quad+2\int_{\hat{\Om}_f} \nabla(\langle w^+\rangle^\mu \langle w^-\rangle^\nu)\hat{p} \cdot  \langle w^+\rangle^\mu \langle w^-\rangle^\nu \nabla \hat{p} +...,\,\,\mu,\nu\in\bR.
\end{align}
To estimate the first two terms on the right hand side of \eqref{A34} which correspond to the commutators,
it is imperative to have control over the pressure with weights: $\| \langle w^+\rangle^\mu \langle w^-\rangle^\nu p\|_{L^2(\Om_f)}$ and
$\| \langle w^+\rangle^\mu \langle w^-\rangle^\nu\hat{p}\|_{L^2(\hat{\Om}_f)}$.

Fortunately, we derive
an identity for the pressure (see Section 4.1):
\begin{align*}
	&\f{\d}{\dx_1} \Big(\int_{-1}^f p\, \dx_2 +\int_f^1 \hat{p}\, \dx_2
	+ \int_{-1}^f \La_{-}^1 \La_{+}^1\, \dx_2+\int_f^1 \hat{\La}_{-}^1 \hat{\La}_{+}^1\, \dx_2
	\Big)=0\,.
\end{align*}
Due to the assumption that $p,\hat{p}, \La_\pm, \hat{\La}_\pm$ have appropriate decay as $x_1$ tends to infinity, thus
\begin{align}\label{A32}
	&  \int_{-1}^f p\, \dx_2 +\int_f^1 \hat{p}\, \dx_2
	+ \int_{-1}^f \La_{-}^1 \La_{+}^1\, \dx_2+\int_f^1 \hat{\La}_{-}^1 \hat{\La}_{+}^1\, \dx_2=0\,.
\end{align}
Similar to the solvability condition for the Neumann problem for Poisson equation in a bounded domain,
the above \eqref{A32} is exactly the solvability condition for the boundary value problem for the pressure.
With \eqref{A32}, we can apply the Poincar\'e inequality to estimate the pressure itself with weights.

\item  The weighted Sobolev inequalities, the weighted trace theorem, the weighted chain rule, the weighted commutator estimate and the weighted div-curl decomposition

The order of the derivatives on the free surface is a half order higher than in the plasma region.
Consequently, the fractional derivatives are extensively used in the current paper.  In the application, we also need apply the weight functions. Therefore, we require the commutator estimate of the fractional derivatives and weight functions.

A key point to achieve global solutions is the $L^1$ integrability in time,
 which forces $\mu>\f12$ with the weights $\langle w^\pm \rangle^{2\mu}$.
 Unfortunately, the weighted commutator estimate for Muckenhoupt weights in literature
 (for instance, see \cite{CN}) requires the weights belong to the $\mathcal{A}_p$ class.
This only permits $0\leq \mu< 1/4 \,(p=2)$ in the current setting. This is far from enough for the purpose of achieving global solutions.
Hence we present a whole set of the weighted estimate: the weighted Sobolev inequalities, the weighted trace theorem, the weighted chain rule, the weighted commutator estimate and so on.
We use the inhomogeneous fractional derivatives which resort to the singular integral with Bessel potentials.
A delicate analysis is needed to avoid the singularity of the integral kernel.
At the same time, the structure of the weight functions is kept in order to make use of the null structure.
We refer to Section 3 for the details.

%

\item Weight estimate of $N_f\cdot(\ud{\nabla p} +\ud{\nabla \hat{p} })$ in \eqref{A27}

Due to the dependence of regularity of the free surface,
we would obtain $s-$order weighted energy estimate of the pressure inside the domain (see Lemma \ref{lemPre}).
In the weighted estimate of the free surface, as seen in \eqref{A27},
this is not enough to obtain the control of the $s-\f12$ order weighted estimate of $N_f\cdot(\ud{\nabla p} +\ud{\nabla \hat{p} })$.
To address this issue, we resort to the classical decomposition for Poisson equation: $p=p^1+p^2$, $\hat{p}=\hat{p}^1+\hat{p}^2$, where $p^1$ and $\hat{p}^1$ are harmonic functions with inhomogeneous boundary conditions, $p^2$ and $\hat{p}^2$ satisfy the Poisson equations with homogeneous boundary conditions.
Thus by using the known results for the Dirichlet-Neumann (DN) operator, the commutator estimate and the inherent structure,
 we could save one order derivative and attain $s-\f12$ order weighted estimate of  $N_f\cdot(\ud{\nabla p} +\ud{\nabla \hat{p} })$.
\end{itemize}

\subsection{Further discussions}
Now we discuss some other related interesting problems.
\begin{itemize}

\item
The background magnetic field in $\Om_f$ and $\hat{\Om}_f$ are the same and both taken to be $e_1$.
The argument also works for the case of reversed background magnetic field:
taking background magnetic field  in $\Om_f$ to be $e_1$,
and in $\hat{\Omega}_f$ to be $-e_1$, and vice versa.
 The proof is exactly the same to the argument in this paper.
For the problem where the
background magnetic field in $\Om_f$ and $\hat{\Om}_f$ are taken to be $ke_1$ and $le_1$
 with non-identical strengths:
$|k|\neq |l|$ ($|k|>0$, $|l|>0$),
it is interesting to investigate whether a global solution exists or not.

\item
The global solutions obtained in this paper work in the two dimensional setting.
Whether the global nonlinear stability result holds for the general three dimensional case is not clear.
Under the collinear background magnetic field setting in the current article, the local in time stability result is also not clear
for the corresponding three dimensional problem.
Actually, this setting seems not to satisfy the stability condition \eqref{stab-cond} in three dimensional case.
For the corresponding long time nonlinear stability problem under the non-collinearity condition,
the method is the current paper does not work since the system may not be one dimensional wave type equations.

\item
For the two dimensional plasma-vacuum free boundary problem,
we will study the global small solutions in a forthcoming paper.
\end{itemize}

The remaining part of this paper is organized as follows: In
Section 2, we will derive the second order evolution equation of the free surface.
After that,  the Alfv\'en waves and the null condition will be intuitively shown.
Then we will introduce the various energy notations.
Section 3 is devoted to the weighted Sobolev inequalities, weighted trace theorem, weighted commutator estimate and so on.
Section 4 and Section 5 are devoted to the weighted estimate of the pressure.
In Section 6, we treat the weighted energy estimate of the free surface.
The energy estimate inside the domain will be shown in Section 7 and Section 8.
In Section 9, we will sketch the local well-posedness to
the two dimensional ideal incompressible MHD current-vortex sheet.
Section 10 is an appendix, in which we include the estimate for the Dirichlet-Neumann operator.

\section{Equations for the free surface, null condition and various energies}

In this section, we present the second-order evolution equation for the free surface in Lemma \ref{Lem-sf-eq}. This equation was derived by Sun, Wang, and Zhang \cite{SWZ18}. The key point is that this equation turns out to be a strictly hyperbolic system under some stability condition. To prove the global existence of the global solutions, we formulate the equation in terms of the $Z_\pm$ variables instead of the velocity field and the magnetic field. After collecting the evolution equations for $\Lambda_\pm$ and $\hat{\Lambda}_\pm$ in the plasma region, the evolution equation for the free surface, the elliptic equation for the pressure, and their boundary conditions, we briefly look into the structure of the systems and explain that they satisfy the null condition. Then we introduce various energy notations and briefly explain the five parts energy estimates.

We denote $\p=\nabla=(\p_1,\p_2)$, $\p^\alpha=\p_1^{\alpha_1}\p_2^{\alpha_2}$ for $\alpha\in\mathbb{N}^2$. For $1\leq p\leq +\infty$, we use $\|g \|_{L^p(\Omega_f)}$ and $\|g\|_{L^p(\hat{\Omega}_f)}$ to denote the $L^p$ norm in $\Omega_f$ and in $\hat{\Omega}_f$ respectively, $\|g \|_{L^p(\mathbb{R})}$ is abbreviated as $\|g \|_{L^p}$. We use Japanese brackets to represent $\langle a\rangle=(1+a^2)^{\frac{1}{2}}$. For $m\in\mathbb{N}$, denote $|\p^m h|=\sum_{|\alpha|=m} |\p^\alpha h|$, $|\p^{\leq m} h|=\sum_{|\alpha|\leq m} |\p^\alpha h|$. Similarly, $\|W\p^{\leq m}g\|_{L^p}=\sum_{|\alpha|\leq m}\|W\p^\alpha g\|_{L^p}$. Throughout the whole paper, the summation convention over repeated indices is always used.

\subsection{Second order equations of the free surface}
In this subsection, we will derive the second order evolution equation
for the moving surface.

Throughout the whole paper, for function $g:\Om_f\rightarrow\BR$  or $g:\hat{\Om}_f\rightarrow\BR$, we use $\ud{g}$ to stand for the trace of $g$ on $\Gamma_f.$
The free surface moves along the normal component of the velocity field.
Combined with the boundary condition, we have
\begin{align}\label{B1}
	\p_t f
	=u\cdot N_f|_{\Ga_f}=\hat{u}\cdot N_f|_{\Ga_f}
	= Z_{\pm}\cdot N_f|_{\Ga_f}=\hat{Z}_{\pm}\cdot N_f|_{\Ga_f} ,
\end{align}
where $N_f$ is the outer normal vector field of the graph:
\begin{equation*}
	N_f=-e_1 \p_1 f+e_2.
\end{equation*}
Note that $f$ depends only on $(t,x_1)$. Hence
\begin{equation*}
	\nabla f=(\p_1,\p_2) f=
	(\p_1f,0).
\end{equation*}

Now we present the second-order evolution equation for the free surface.
\begin{lem}\label{Lem-sf-eq}
For the free surface, there holds
\begin{align}\label{B10}
\p^2_t f=
&-\f12 (\ud{Z_{+}^1}+\ud{Z_{-}^1}+\ud{\hat{Z}_{+}^1}+\ud{\hat{Z}_{-}^1})
\p_1\p_t f
-\f12(\ud{Z_{+}^1} \ud{Z_{-}^1}
+\ud{\hat{Z}_{+}^1}\ud{\hat{Z}_{-}^1})\p_1^2 f \nonumber\\
&-\f12 N_f\cdot\ud{\nabla \hat{p}}
-\f12 N_f\cdot\ud{\nabla p} \,.
\end{align}
\end{lem}
\begin{proof}
We will show
\begin{align}\label{B8}
\p_t^2 f&=- (\ud{Z_{+}^1}+\ud{Z_{-}^1})
\p_1\p_t f-\ud{Z_{+}^1} \ud{Z_{-}^1} \p_1^2 f- N_f\cdot\ud{\nabla p}\,
\end{align}
and
\begin{align}\label{B9}
\p_t^2 f=- (\ud{\hat{Z}_{+}^1}+\ud{\hat{Z}_{-}^1})
\p_1\p_t f-\ud{\hat{Z}_{+}^1} \ud{\hat{Z}_{-}^1} \p_1^2 f- N_f\cdot\ud{\nabla \hat{p}}.
\end{align}
Then taking the average of \eqref{B8} and \eqref{B9} yields \eqref{B10}.

Recalling that the MHD system in $\Om_f$ can be written as follows:
\begin{equation} \label{B3}
\begin{cases}
\p_t Z_{+} +Z_{-}^1\p_1 Z_{+}
+Z_{-}^2\p_2 Z_{+} + \nabla p =0, \\
\p_t Z_{-} +Z_{+}^1\p_1 Z_{-}
+Z_{+}^2\p_2 Z_{-} + \nabla p =0,
\end{cases}
\textrm{in}\,\,\Omega_f.
\end{equation}
Employing the expression \eqref{B1} for $\p_tf$, $N_f$ and $\eqref{B3}_2$, we have
\begin{align*}
&(\p_t Z_{+} +\p_tf\p_2 Z_{+})\big|_{x_2=f(t,x_1)}\\
&=\big(-Z_{-}\cdot\nabla Z_{+} -\nabla p + Z_{-}\cdot(-\nabla f+e_2) \p_2 Z_{+}\big)\big|_{x_2=f(t,x_1)}\nonumber\\\nonumber
&=\big(-Z_{-}^1 \p_1 Z_{+}
-Z_{-}\cdot \nabla f \p_2 Z_{+}
-\nabla p\big)\big|_{x_2=f(t,x_1)}\nonumber\\\nonumber
&=\big(-Z_{-}^1 \p_1 \ud{Z_{+}} -\nabla p \big)\big|_{x_2=f(t,x_1)} \\
&=\big(-Z_{-} \cdot\nabla \ud{Z_{+}} -\nabla p \big)\big|_{x_2=f(t,x_1)} .
\end{align*}
Note
\begin{align*}
&\nabla N_f=\nabla (e_2-e_1\p_1f)= -e_1\otimes \nabla\p_1f.
\end{align*}
Consequently,
we deduce by integration by parts to obtain that
\begin{align*} 
\p_t^2 f=&\big( (\p_t Z_{+} +\p_2 Z_{+}\p_tf)\cdot N_f+Z_{+}\cdot\p_t N_f \big)\big|_{x_2=f(t,x_1)}\nonumber \\
=&\big(-(Z_{-} \cdot\nabla \ud{Z_{+}}) \cdot N_f -\nabla p\cdot N_f   +Z_{+}\cdot\p_t N_f \big)\big|_{x_2=f(t,x_1)}\nonumber\\
=&-\ud{Z_{-}}\cdot\nabla \p_t f-(\ud{Z_{-}}\otimes \ud{Z_{+}}):(\nabla\otimes \nabla) f
-N_f\cdot\ud{\nabla p}-\ud{Z_{+}}\cdot\nabla \p_t f \nonumber\\
=&-\ud{Z_{-}^1}\p_1\p_t f  -\ud{Z_{+}^1}\p_1\p_t f -N_f\cdot\ud{\nabla p}
-\ud{Z_{-}^1}\ud{Z_{+}^1}\p_1^2 f,
\end{align*}
where $(\ud{Z_{-}}\otimes \ud{Z_{+}}):(\nabla\otimes \nabla) f
=\ud{Z_{-}^1} \ud{Z_{+}^1}\p_1^2f $.
This yields \eqref{B8}.
 \eqref{B9} can be derived in similar way.
This ends the proof of the lemma.
\end{proof}

\subsection{Equations restricted to the free surface}
In this subsection, we derive the evolution equations for $\ud{Z_{\pm}}$ and $\ud{\hat{Z}_{\pm}}$
for the purpose of the future uses.
Note that the equation $\p_t f=\ud{Z_{\pm}}\cdot N_f$ can be organized as follows:
\begin{align*}
(\p_t+\ud{Z_{\pm}^1}\p_1)f=\ud{Z_{\pm}^2}.
\end{align*}
Hence restricting to the free surface in the trace sense, we deduce by \eqref{A15} and \eqref{A16} that
\begin{align*}
(\p_t +\ud{Z_{\pm}^1}\p_1)\ud{Z_{\mp}^1}
&=\ud{(\p_t +{Z_{\pm}^1}\p_1)Z_{\mp}^1}
+\ud{\p_2 Z_{\mp}^1}  (\p_t +\ud{Z_{\pm}^1}\p_1)f \\
&=\ud{(\p_t +{Z_{\pm}^1}\p_1 +{Z_{\pm}^2}\p_2 )Z_{\mp}^1}
=  -\ud{\p_1 p}.
\end{align*}
Similar calculation yields
\begin{equation*}
(\p_t +\ud{Z_{\pm}^1}\p_1)\ud{Z_{\mp}}
=-\ud{\nabla p},\\
\end{equation*}
and
\begin{equation*}
(\p_t +\ud{\hat{Z}_{\pm}^1}\p_1)\ud{\hat{Z}_{\mp}}
=-\ud{\nabla \hat{p}}.
\end{equation*}
Employing the perturbation variables $\La_\pm$, the above equations become
\begin{equation}\label{B20}
\begin{cases}
(\p_t \pm\p_1)  \ud{\La_{\mp}}    = -\ud{\La_{\pm}^1}\p_1\ud{\La_{\mp}}-\ud{\nabla p},\\
(\p_t \pm\p_1)\ud{\hat{\La}_{\mp}}= -\ud{\hat{\La}_{\pm}^1}\p_1\ud{\hat{\La}_{\mp}}-\ud{\nabla\hat{p}}.
\end{cases}
\end{equation}

%
%
%

\subsection{Perturbation variables, Alfv\'en waves, equations for the pressure, null conditions}
Recalling the equations for $\La_\pm$, $\hat{\La}_\pm$ and their boundary conditions as follows: 
\begin{equation} \label{BB3}
\begin{cases}
\p_t \La_{+} - \p_1  \La_{+}
+\La_{-}\cdot\nabla \La_{+} + \nabla p =0,\\
\p_t \La_{-} + \p_1 \La_{-}
+\La_{+}\cdot\nabla \La_{-} + \nabla p =0, \\
\div \La_{\pm}=0,
\end{cases}
 \textrm{in}\,\,\Om_f,
\end{equation}
and
\begin{equation} \label{BB4}
\begin{cases}
\p_t \hat{\La}_{+} - \p_1 \hat{\La}_{+}
+\hat{\La}_{-}\cdot\nabla \hat{\La}_{+} + \nabla \hat{p} =0,\\
\p_t \hat{\La}_{-} + \p_1 \hat{\La}_{-}
+\hat{\La}_{+}\cdot\nabla \hat{\La}_{-} + \nabla \hat{p} =0,\\
\div \hat{\La}_{\pm}=0,
\end{cases}
 \textrm{in}\,\,\hat{\Om}_f,
\end{equation}
and
\begin{equation}\label{BB5}
\begin{cases}
\p_tf\pm\p_1 f= \ud{\La_{\pm}} \cdot N_f= \ud{\hat{\La}_{\pm}} \cdot N_f\quad \textrm{on}\,\, \Ga_f,\\
p=\hat{p} \quad \textrm{on}\,\, \Ga_f,\\
\La_{\pm}\cdot n=0 \quad \textrm{on}\,\, \Ga,\\
\hat{\La}_{\pm}\cdot n=0 \quad \textrm{on}\,\, \hat{\Ga}.
\end{cases}
\end{equation}
In terms of the perturbation variables $\La_\pm$, $\hat{\La}_\pm$,
 the equation of the free surface \eqref{B10} can be written as follows
\begin{align}\label{BB9}
\p_t^2 f-\p_1^2 f=
&-\f12(\ud{\La_{+}^1}+\ud{\La_{-}^1} +\ud{\hat{\La}_{+}^1}+\ud{\hat{\La}_{-}^1})\p_1 \p_t f +\f12(-  \ud{\La_{-}^1}+\ud{\La_{+}^1}
-\ud{\hat{\La}_{-}^1}+\ud{\hat{\La}_{+}^1}  )\p_1^2 f\\\nonumber
&-\f12(\ud{\La_{+}^1} \ud{\La_{-}^1}+\ud{ \hat{\La}_{+}^1} \ud{ \hat{\La}_{-}^1})\p_1^2 f
 -\f12(N_f\cdot\ud{\nabla p} +N_f\cdot\ud{\nabla \hat{p}})\,.
\end{align}
To show the null structure, the above \eqref{BB9} can be organized as follows:
\begin{align}\label{BB10}
\p_t^2 f-\p_1^2 f
=&-\f12\big(\ud{\La_{+}^1}+\ud{\hat{\La}_{+}^1} \big) \p_1 (\p_t-  \p_1) f
 -\f12\big(\ud{\La_{-}^1}+\ud{\hat{\La}_{-}^1} \big) \p_1 (\p_t+  \p_1) f\\\nonumber
&-\f12\big(\ud{\La_{+}^1} \ud{\La_{-}^1} +\ud{\hat{\La}_{+}^1} \ud{\hat{\La}_{-}^1} \big)\p_1^2 f
-\f12(N_f\cdot\ud{\nabla p} +N_f\cdot\ud{\nabla \hat{p}}).
\end{align}

Now we collect the equation and the boundary condition for the pressure.
The estimate of the pressure plays an important role in this paper.
It follows from \eqref{B8} and \eqref{B9} that
\begin{align} \label{BB11}
N_f\cdot\ud{\nabla p}-N_f\cdot\ud{\nabla \hat{p}}
&= -\big(\ud{\La_{+}^1}-\ud{\hat{\La}_{+}^1} \big) \p_1 (\p_t-  \p_1) f
  -\big(\ud{\La_{-}^1}-\ud{\hat{\La}_{-}^1}\big) \p_1 (\p_t+  \p_1) f \\\nonumber
&\quad-\big(\ud{\La_{+}^1} \ud{\La_{-}^1}- \ud{\hat{\La}_{+}^1} \ud{\hat{\La}_{-}^1} \big) \p_1^2 f
\,\quad \text{on}\,\, \Ga_f.
\end{align}
For the boundary condition of the pressure on $\Ga$ and $\hat{\Ga}$, restricting $\eqref{A18}_1$ on $\Ga$ and  restricting $\eqref{A19}_1$ on $\hat{\Ga}$ in the trace sense, respectively, there holds the following Neumann boundary condition for the pressure
\begin{align*}
\p_2 p=0\,\, \textrm{on} \,\,  \Ga,
\quad
\p_2 \hat{p}=0 \,\,  \textrm{on} \,\,  \hat{\Ga}.
\end{align*}
On the other hand,
taking the divergence to the first equation of (\ref{BB3}) and (\ref{BB4}) respectively, we obtain
\begin{align*}
&\Delta p=-\nabla \cdot (\La_{-} \cdot \nabla\La_{+})
\quad \textrm{in}\,\,\Om_f,\\
&\Delta \hat{p}=-\nabla \cdot (\hat{\La}_{-} \cdot \nabla \hat{\La}_{+})
\quad \textrm{in}\,\,\hat{\Om}_f.
\end{align*}
Note $\eqref{A20}_3$, hence we collect the above elliptic equations and the boundary conditions for the pressure as follows
\begin{equation}\label{B22}
\begin{cases}
&\Delta p =-\nabla \cdot (\La_{-} \cdot \nabla \La_{+}) \quad \textrm{in}\,\, \Om_f,\\
&\Delta \hat{p} =-\nabla \cdot (\hat{\La}_{-} \cdot \nabla \hat{\La}_{+})\quad \textrm{in}\,\, \hat{\Om}_f,\\
&p =\hat{p}\quad \textrm{on} \,\, \Gamma_f,\\
&N_f\cdot\nabla p-N_f\cdot\nabla \hat{p}= -\big(\ud{\La_{+}^1} \ud{\La_{-}^1}- \ud{\hat{\La}_{+}^1} \ud{\hat{\La}_{-}^1} \big) \p_1^2 f \\
&\qquad-\big(\ud{\La_{+}^1}-\ud{\hat{\La}_{+}^1} \big) \p_1 (\p_t-  \p_1) f
  -\big(\ud{\La_{-}^1}-\ud{\hat{\La}_{-}^1}\big) \p_1 (\p_t+  \p_1) f
\quad\text{on}\,\,\Ga_f ,\\
&\p_2 p=0 \,\, \textrm{on} \,\,  \Ga,\quad \p_2 \hat{p}=0 \,\, \textrm{on} \,\,  \hat{\Ga}.
\end{cases}
\end{equation}
It can be roughly seen that the pressure consists of three parts: inside $\Om_f$, inside $\hat{\Om}_f$ and on $\Ga_f$. The estimate of the pressure is complicated. We refer to Section 4 and 5.

From the linearized equations of \eqref{BB3}, \eqref{BB4}, and \eqref{BB10}, it can be seen that $\Lambda_{+}$, $\hat{\Lambda}_{+}$, $(\partial_t+\partial_1) f$ and $\Lambda_{-}$, $\hat{\Lambda}_{-}$, $(\partial_t-\partial_1) f$ propagate along the background magnetic field in opposite directions. For the nonlinear terms of \eqref{BB3}, \eqref{BB4}, and \eqref{BB10}, they are always the combination of the interactions of the left Alfv\'en waves variables $\Lambda_{+}$, $\hat{\Lambda}_{+}$, $(\partial_t+\partial_1) f$ and the right Alfv\'en waves $\Lambda_{-}$, $\hat{\Lambda}_{-}$, $(\partial_t-\partial_1) f$. On the other hand, we see from \eqref{B22} that the pressure terms have the similar structure modulo the pseudo-differential operator. In other words, the equations \eqref{BB3}, \eqref{BB4} inside the domain and \eqref{BB10} of the free surface satisfy the same type of null conditions.

\subsection{Energy notations and five parts energy estimate}

Define the weight functions $w^\pm(t,x_1)=x_1\pm t$.
Employing the weight functions, for positive integer $0\leq s\leq 3$ and $1/2<\mu \leq 3/5$, we define the lower order weighted energy of the plasma as follows:
\begin{align*}
E^b_s (t)=&\sum_{+,-}\sum_{ |\alpha|\leq s}
 \big(\|\langle w^\pm \rangle^{5\mu}\nabla^\alpha \La_{\pm}(t,\cdot)\|^2_{L^2(\Om_f)}
+ \| \langle w^\pm \rangle^{5\mu}\nabla^\alpha \hat{\La}_{\pm}(t,\cdot)\|^2_{L^2(\hat{\Om}_f)} \big)\,.
\end{align*}
The lower order ghost weight energy is defined as follows:
\begin{align*}
G_s^b (t)=&\sum_{+,-} \sum_{|\alpha|\leq s}
\Big(\big\|
\frac{\langle w^\pm \rangle^{5\mu}\nabla^\alpha \La_{\pm}}{ \langle w^\mp \rangle^{\mu} } (t,\cdot) \big\|^2_{L^2(\Om_f)}
+\big\|\frac{\langle w^\pm \rangle^{5\mu}\nabla^\alpha \hat{\La}_{\pm}}{ \langle w^\mp \rangle^{\mu} }(t,\cdot) \big\|^2_{L^2(\hat{\Om}_f)}\Big).
\end{align*}
For positive integer $s\geq 4$ and $1/2<\mu \leq 3/5$, we define the weighted energy of the plasma as follows:
\begin{align*}
E^b_s (t)
=E^b_3 (t) +\sum_{+,-}\sum_{ |\alpha|\leq s} \big( \| \langle w^\pm \rangle^{2\mu}\nabla^\alpha \La_{\pm}(t,\cdot)\|^2_{L^2(\Om_f)}
 +\| \langle w^\pm \rangle^{2\mu}\nabla^\alpha \hat{\La}_{\pm}(t,\cdot)\|^2_{L^2(\hat{\Om}_f)} \big)\,.
\end{align*}
The ghost weight energy is defined as follows:
\begin{align*}
G_s^b (t)
=G_3^b (t) +\sum_{+,-} \sum_{|\alpha|\leq s}
\Big(\big\|
\frac{\langle w^\pm \rangle^{2\mu}\nabla^\alpha \La_{\pm}}{ \langle w^\mp \rangle^{\mu} } (t,\cdot)\big\|^2_{L^2(\Om_f)}
+\big\|\frac{\langle w^\pm \rangle^{2\mu}\nabla^\alpha \hat{\La}_{\pm}}{ \langle w^\mp \rangle^{\mu} } (t,\cdot) \big\|^2_{L^2(\hat{\Om}_f)}\Big) \,.
\end{align*}
For positive integer $s\geq 1$,
the weighted energy of the free surface is denoted by:
\begin{align*}
&E^f_{s+\f12} (t)=  \sum_{+,-}\sum_{|a|\leq s-1} \|\langle w^\pm\rangle^{2\mu}(\p_t \pm\p_1)\langle\p_1\rangle^{\f12}\p_1^a f (t,\cdot)\|^2_{L^2(\bR)} \,.
\end{align*}
The ghost weight energy is defined as follows:
\begin{align*}
&G^f_{s+\f12} (t)= \sum_{+,-} \sum_{|a|\leq s-1} \big\|
 \f{\langle w^\pm\rangle^{2\mu}(\p_t \pm \p_1)\langle\p_1\rangle^{\f12}\p_1^a f }
{\langle w^\mp\rangle^{\mu}} (t,\cdot)\big\|^2_{L^2(\bR)}.
\end{align*}
Thus the total weighted energy and the weighted ghost weight energy are defined as follows. For integer $s\geq 3$, we define
\begin{align*}
E_s(t)= E_s^b(t)+E_{s+\f12}^f(t),\\
G_s(t)= G_s^b(t)+G_{s+\f12}^f(t).
\end{align*}

The local well-posedness of the system will be presented in the last section.
The local existence of solutions is obtained by the method of successive approximations.
The unknowns will be in Sobolev spaces without weights.
The estimate is treated for the successive of the unknowns, and then we prove the contraction of the iteration map.
The limiting unknowns $(f,\La_\pm,\hat{\La}_\pm,p,\hat{p})$ will satisfy the current-vortex sheet problem \eqref{A18}-\eqref{A20}.
Precisely, for $s\geq 4$, there exists $T>0$ such that the system  admits a unique solution $(f,\La_\pm,\hat{\La}_\pm)$
satisfying
\begin{align*}
&\p_t f,\p_1 f\in L^{\infty}([0,T),H^{s-\f12}(\bR)),\quad f\in L^{\infty}([0,T),L^\infty(\bR)),\\
& \La_\pm\in L^{\infty}([0,T),H^{s}(\Om_f)),\quad \hat{\La}_\pm\in L^{\infty}([0,T),H^{s}(\hat{\Om}_f)).
\end{align*}
Once the local existence of the solutions is obtained.
We only need treat the current-vortex sheet problem \eqref{A18}-\eqref{A20} instead of the sequence of the unknowns.

By the local in time stability result in the last section, the maximal possible time that the solutions exist  depends on the size
of the Sobolev norm of the initial data.
To extend the local solution to be a global one, we will show \textit{a priori} estimate for \eqref{A18}-\eqref{A20} uniformly in time.
The \textit{a priori} estimate will be conducted in Sobolev spaces with weights. Precisely,
we will show that there hold the following weighted energy estimates uniformly in time for $s\geq 4$:
\begin{align}
\label{AA1}
& E_s(t) +\int_0^t G_s (\tau) \d\tau \leq C E_s(0) +C\int_0^t E_s^{\f 12}(\tau) G_s(\tau)\d\tau,\\
\label{AA2}
&\| f(t,\cdot)\|_{L^\infty} \leq \| f(0,\cdot)\|_{L^\infty}
+  \f12 M \sup_{0\leq \tau \leq t} (1+C(E_s^f(\tau))^{\f12}) (E_s^b(\tau))^{\f12}
,
\end{align}
where $M=\int_{\bR} \langle z  \rangle^{-2\mu} \d z$, the constant $C$ in \eqref{AA1} polynomially depends on $\sup_{0\leq \tau\leq t}E_s(\tau)$ and $\sup_{0\leq \tau\leq t}\| f(\tau,\cdot)\|_{L^\infty}$.
By the weighted div-curl Lemma \ref{lemDCw1} and Lemma \ref{lemDCw2},
the weighted energy will be controlled by the weighted energy of the tangential derivative part, the vorticity part, and the divergence part,
or it can be bounded by the weighted energy of the free surface, the vorticity part, and the divergence part.
Hence there hold
\begin{align*}
&     E_3^b(t)\lesssim E_3^{b,\tau}(t)+E_3^{b,\omega}(t),
\quad G_3^b(t)\lesssim G_3^{b,\tau}(t)+G_3^{b,\omega}(t), \\
&     E_s (t) \lesssim E_s^{b,\omega}(t)+E_{s+\f12}^f (t),
\quad G_s (t) \lesssim G_s^{b,\omega}(t)+G_{s+\f12}^f (t),\quad s\geq 4.
\end{align*}
The definition for  $E_3^{b,\tau},\,G_3^{b,\tau},\,E_3^{b,\omega},\,G_3^{b,\omega}$
and $E_s^{b,\omega},\, G_s^{b,\omega}$ are given in \eqref{F01}, \eqref{F02},
\eqref{F113}, \eqref{F114},  \eqref{F121}, \eqref{F122}.
Thus in order to show \eqref{AA1}, \eqref{AA2}, we will show
\begin{align}
\label{AA8}
\sup_{0\leq \tau\leq t}E_{s+{\f12}}^f(\tau) + \int_0^t G_{s+{\f12}}^f(\tau)\d\tau
&\lesssim E_{s+{\f12}}^f(0)+\int_0^t (1+E_s^{\f12}+E_s )E_s^{\f12} G_s (\tau)\d\tau,\\
\label{AA9}
\sup_{0\leq \bar{\tau}\leq t} E_3^{b,\tau}(\bar{\tau}) +\int_0^t G_3^{b,\tau} (\bar{\tau}) \d \bar{\tau}
&\lesssim E_3^{b,\tau}(0) +\int_0^t E_4^{\f 12}(\bar{\tau}) G_3^b(\bar{\tau})\d \bar{\tau},\\
\label{AA10}
\sup_{0\leq \bar{\tau}\leq t} E_3^{b,\omega}(\bar{\tau}) +\int_0^t G_3^{b,\omega} (\bar{\tau}) \d \bar{\tau}
&\lesssim E_3^{b,\omega}(0) +\int_0^t E_3^{\f 12}(\bar{\tau}) G_3^b(\bar{\tau})\d \bar{\tau},\\
\label{AA11}
\sup_{0\leq \bar{\tau}\leq t} E_s^{b,\omega}(\bar{\tau}) +\int_0^t G_s^{b,\omega} (\bar{\tau}) \d \bar{\tau}
&\lesssim  E_s^{b,\omega}(0) +\int_0^t E_s^{\f 12}(\bar{\tau}) G_s^b(\bar{\tau})\d \bar{\tau},\\
\label{AA12}
\| f(t,\cdot)\|_{L^\infty} \leq \| f(0,\cdot)\|_{L^\infty}
+  &\f12 M \sup_{0\leq \tau \leq t} (1+C(E_s^f(\tau))^{\f12}) (E_s^b(\tau))^{\f12}
.
\end{align}
The proof of \eqref{AA8} and \eqref{AA12} is given in Section \ref{Ef}.
The proof of \eqref{AA9} is given in Section \ref{Et}.
The proof of \eqref{AA10} and \eqref{AA11} is given in Section \ref{Ev}.
Once \eqref{AA8}-\eqref{AA12} are obtained, we see that Theorem \ref{thm} will be obtained by the continuity argument.

Throughout this paper, we use $A\lesssim B$ to denote $A \leq C B$ for some absolute positive constant $C$, whose meaning may change from line to line. The constant $C$ never depend on time. The dependence on other parameters is usually clear from the context and we usually suppress this dependence. By taking an appropriately small $\epsilon_0$ in Theorem \ref{thm},
we can assume that $E_s\leq C_0\epsilon\leq C_0\epsilon_0\ll1$,
which is always assumed in this paper.


\section{Preliminary weighted estimates}

In this section, we present 
the weighted Sobolev inequalities, the weighted trace theorem, the weighted commutator estimate,
the weighted tangential derivative estimate, the weighted chain rules and the weighted div-curl theorem.
These estimates all adapts to the weights $\langle w^+ \rangle^{\mu}\langle w^- \rangle^{\nu}$ for any $\mu\in \bR,\, \nu\in \bR$.
Here we recall the weights $w^\pm=x_1\pm t$.
These weighted estimates will be extensively used in the energy estimate for the unknowns of the free surface and of the plasma.

In some of these estimates, we need to use the fractional derivatives.
There are many methods to explain the fractional derivative, we refer to \cite{Adams, Caf, Di}
and the references therein.
In this paper, we use the inhomogeneous fractional
derivatives $\langle \p\rangle^{\mu}$ ($\mu\in\bR$). They will be explained as singular integrals with Bessel potentials.
We will only use the fractional derivatives defined in the whole space $\bR$.
Different from  the bounded domain, this definition will not cause any confusion.
Several definitions, for instance the difference quotient method, the frequency method, and so on, are equivalent.
On the other hand, compared with the Riesz potentials for homogeneous fractional derivatives,
the advantage is that the Bessel potentials have exponential decay at infinity.

Now we explain the main ideas of the proof of these weighted estimates.
For the weighted trace theorem and the weighted commutator estimate, we transform them into the commutator of the fractional derivatives and the weights $\langle w^+ \rangle^{\mu}\langle w^- \rangle^{\nu}$. These commutators mainly consist of a singular integral with weights, see, for instance $\eqref{D14}$ and Lemma \ref{lemD0}:
$$\textrm{p.v.} \int_{\BR}  (1-\p_1^2)B_{2-s}(x_1-x_1')
\big(  \langle w^+(x_1)\rangle^{\mu}\langle w^-(x_1)\rangle^{\nu} -\langle w^+(x_1')\rangle^{\mu}\langle w^-(x_1')\rangle^{\nu}\big) h(x_1')\d x_1'.$$
These commutators can be estimated. Actually, for the integral domain near the origin:
$\{x_1':|x_1-x_1'|\leq \f12\}$, there holds $\langle w^\pm(x_1)\rangle
\sim \langle w^\pm(x_1')\rangle$. Thus we can freely move the weights $\langle w^\pm(x_1)\rangle$ and $\langle w^\pm(x_1')\rangle$ and keep the structure of the integral kernel near the origin.
For the integral domain away from the origin:
$\{x_1':|x_1-x_1'|\geq \f12\}$, there holds $\langle w^\pm(x_1)\rangle^{\mu}
\leq 3^{|\mu|} |x_1-x_1'|^{|\mu|}\langle w^\pm(x_1')\rangle^{\mu}$ for $\mu\in\bR$. The growth factor $|x_1-x_1'|^{|\mu|}$ or $|x_1-x_1'|^{|\nu|}$
will be absorbed into the exponential decay of the Bessel potentials. Thus the integral kernel is always summable away from the origin.
As for the weighted Sobolev inequalities, the weighted tangential derivative estimate, the weighted chain rules, and the weighted div-curl theorem, they are the direct consequence of the corresponding classical estimates without weights $\langle w^+ \rangle^{\mu}\langle w^- \rangle^{\nu}$.

We mention that the weighted Kato-Ponce commutator estimate in the literature
(for instance, see \cite{CN}) requires that the weights belong to the $\mathcal{A}_p$ class,
which is not applicable in our setting since it only permits the weights $\langle w^\pm \rangle^{\mu}$ with $0\leq \mu< 1/4\, (p=2)$.
Hence we present a self-contained proof of the weighted estimate for the inhomogeneous fractional derivatives
which will allows for a large range of $\mu$ an $\nu$ for the weights $\langle w^+ \rangle^{\mu}\langle w^- \rangle^{\nu}$.


\subsection{Weighted Sobolev inequalities}
In the sequel, we present several weighted Sobolev inequalities.
 They are the direct consequences of
the  Sobolev imbedding and the properties of the weight functions $\langle w^\pm \rangle$.
\begin{lem}\label{Sobo1}
Let $\Omega\subseteq \bR^2$ with $\p\Omega\in C^1$, $\f{1}{q}\geq \f{1}{p}-\f{k}{2}$, $1\geq \f{1}{p}>\f{k}{2}$.
For any $\lambda\in \BR$, $\mu\in \BR$, there holds
\begin{align}
\label{SobWeit1}
\big\| \langle w^+ \rangle^{\lambda} \langle w^- \rangle^{\mu} g (t,\cdot)\big\|_{L^{q}(\Om)}
&\lesssim \sum_{|\alpha|\leq k}
\big\|\langle w^+ \rangle^{\lambda}\langle w^- \rangle^{\mu} \nabla^\alpha g (t,\cdot)\big\|_{L^p(\Om)},
\end{align}
provided the right hand side is finite.
\end{lem}
\begin{proof}
For $\f{1}{q}\geq \f{1}{p}-\f{k}{2}$, $1\geq \f{1}{p}>\f{k}{2}$, by $W^{k,p}(\Omega)
\hookrightarrow L^{q}(\Omega)$,
we have
\begin{align*}
\big\|\langle w^+ \rangle^{\lambda} \langle w^- \rangle^{\mu} g (t,\cdot)\big\|_{L^{q}(\Omega)}
&\lesssim \sum_{|\alpha|\leq k}
\big\|\nabla^\alpha \big(\langle w^+ \rangle^{\lambda}\langle w^- \rangle^{\mu} g (t,\cdot)\big)\big\|_{ L^p(\Omega)} \\
&\lesssim \sum_{|\alpha|\leq k}\big\|\langle w^+ \rangle^{\lambda}\langle w^- \rangle^{\mu} \nabla^\alpha g (t,\cdot)\big\|_{L^p(\Omega)},
\end{align*}
which yields \eqref{SobWeit1}.
\end{proof}

Similar weighted Sobolev inequalities also hold on the free surface.
Here we directly state the weighted Sobolev inequalities on the real line.
The details of the proof are omitted since it is the same to Lemma \ref{Sobo1}.
\begin{lem}\label{Sobo2}
Let $h(t,\cdot )\in W^{1,p}(\BR )$, $1<p<q\leq \infty$.
For any $\lambda\in \BR$, $\mu\in \BR$, there holds
\begin{align*}
&\big\|\langle w^+ \rangle^{\lambda}\langle w^-  \rangle^{\mu}
 h (t,\cdot) \big\|_{L^q(\BR )}
\lesssim \sum_{|a|\leq 1}
\big\|\langle w^+ \rangle^{\lambda}\langle w^-  \rangle^{\mu}
\nabla^a h(t,\cdot)\big\|_{L^p(\BR )} ,
\end{align*}
provided the right hand side is finite.
\end{lem}

\subsection{Weighted trace theorem}

Denote $w(x_1)$ by some (weight) function. We remind that $w(x_1)$ is just some function and it may not be weights $w^\pm$
 where $w^\pm=x_1\pm t$.

We begin by writing down the definition of inhomogeneous fractional derivatives.
This definition is well known to the experts.
Let $h=h(x_1)$ be a $C^2_c(\BR)$ function. Then
for $0<s<1$,  there holds
\begin{align}\label{D1}
&\langle \p_1\rangle^{s} h(x_1)
= (1-\p_1^2)^{\f s2 -1} (1-\p_1^2)h(x_1)\\\nonumber
&= \int_{\BR} B_{2-s}(x_1-x_1') (1-\p_{x_1'}^2)h(x_1') \d x_1' \\\nonumber
&=  \int_{\BR} B_{2-s}(x_1-x_1') h(x_1') \d x_1'
+ \text{p.v.} \int_{\BR} \p_1^2B_{2-s}(x_1-x_1') (h(x_1)-h(x_1')) \d x_1'.
\end{align}
Here p.v. denotes the principal integral.
 $B_s(x_1)$ are Bessel potentials smoothing on $\BR \backslash \{0\}$ satisfying $B_s(x_1)>0$:
\begin{align*}
B_s(x_1)=\mathcal{F}^{-1}\big( (1+|\xi|^2)^{-\f{s}{2}} \big) (x_1),
\end{align*}
where $\mathcal{F}^{-1}$ stands for the inverse Fourier transform.
For the Bessel potentials, we have the following asymptotic behavior near $0$ and at infinity.
\begin{lem}\label{lemBe}
If $0<s<1$, there hold
\begin{align*}
&|B_{2-s}(x_1)|+|\p_1 B_{2-s}(x_1)|+|\p_1^2B_{2-s}(x_1)| \leq C e^{-\f{|x_1|}{2}},\,\,
\text{for}\,\,  |x_1|\geq \f12,
\end{align*}
and
\begin{align*}
&\begin{cases}
B_{2-s}(x_1) \sim 1,\\
\p_1   B_{2-s}(x_1) \sim  x_1|x_1|^{-s-1},\\
\p_1^2 B_{2-s}(x_1) \sim |x_1|^{-s-1},\\
\p_1^3 B_{2-s}(x_1) \sim  x_1|x_1|^{-s-3},\\
\end{cases}
\,\, \text{for}\,\,  0<|x_1|\leq 2.
\end{align*}
\end{lem}
\begin{proof}
We refer to Grafakos's book in chapter 6.1.2 \cite{G_book}, Stein's book in Chapter V \cite{Stein} or Aronszajn-Smith \cite{AS} in section 4 Chapter II
for the proof of the properties of the Bessel potentials.
\end{proof}
\begin{remark}
\eqref{D1} holds for $h\in C^2_c(\BR)$. It is standard to extend the definition
as long as the right hand side of \eqref{D1} is summable.
Hence in the following argument, we will not require $h\in C^2_c(\BR)$ but only require the integral in \eqref{D1} is finite.
\end{remark}
Now we study the fractional derivative with some weight functions.
Multiplying \eqref{D1} by some function $w(x_1)$, we write
\begin{align}\label{D13}
w(x_1) \langle \p_1\rangle^{s} h(x_1)
&=w(x_1) \int_{\BR} B_{2-s}(x_1-x_1') h(x_1') \d x_1' \\\nonumber
&\quad + w(x_1)\ \text{p.v.} \int_{\BR} \p_1^2B_{2-s}(x_1-x_1') (h(x_1)-h(x_1')) \d x_1' .
\end{align}
For the second term on the right hand side of \eqref{D13}, we write
\begin{align*}
&w(x_1)\ \text{p.v.} \int_{\BR} \p_1^2 B_{2-s}(x_1-x_1')(h(x_1)-h(x_1')) \d x_1'\\
&= \langle \p_1\rangle^{s} \big( w(x_1) h(x_1)\big)
- \int_{\BR}  B_{2-s}(x_1-x_1')
w(x_1')h(x_1') \d x_1'\\
&\quad - \ \text{p.v.} \int_{\BR} \p_1^2 B_{2-s}(x_1-x_1')
\big[w(x_1)-w(x_1')\big]h(x_1') \d x_1'.
\end{align*}
Consequently, \eqref{D13} becomes
\begin{align}\label{D14}
[\langle \p_1\rangle^{s}, w(x_1)] h(x_1)
=-\text{p.v.} \int_{\BR}  (1-\p_1^2)B_{2-s}(x_1-x_1')
\big(  w(x_1)-w(x_1')\big) h(x_1')\d x_1',
\end{align}
where we have used the following commutator notation:
\begin{align*}
[\langle\nabla\rangle^{s},g]h
=\langle\nabla\rangle^{s}(g h)
-g\langle\nabla\rangle^{s}h.
\end{align*}
Here \eqref{D14} holds for any functions $w(x_1)$ and $h(x_1)$ such that the singular integral make sense.

Before the weighted trace theorem and the weighted commutator estimate, we first estimate \eqref{D14} for some appropriate weight functions.
We recall that the weight functions  are given by $w^\pm=x_1\pm t$.
\begin{lem}\label{lemD0}
Let $0<s<1$, $1\leq q\leq +\infty$, $\mu\in\BR,\ \nu\in\BR$.
There holds
\begin{align*} 
&\big\| \textrm{p.v.} \int_{\BR}  (1-\p_1^2)B_{2-s}(x_1-x_1')
\big(  \langle w^+(x_1)\rangle^{\mu}\langle w^-(x_1)\rangle^{\nu} -\langle w^+(x_1')\rangle^{\mu}\langle w^-(x_1')\rangle^{\nu}\big) h(x_1')\d x_1' \big\|_{L^q(\BR)}\\
&\lesssim \| \langle w^+\rangle^{\mu}\langle w^-\rangle^{\nu} h \|_{L^q(\BR )},
\end{align*}
provided the right hand side is finite.
\end{lem}
\begin{proof}
We divide the
estimate into two cases:
$\{x_1':|x_1-x_1'|\leq \f12\}$ and $\{x_1':|x_1-x_1'|\geq \f12\}$.
Firstly, for $|x_1-x_1'|\leq \f12$, note
\begin{align*}
\langle w^\pm(x_1) \rangle-\langle w^\pm(x_1')\rangle
=(x_1-x_1')\cdot
\big( \langle w^\pm \rangle^{-1} w^\pm \p_1 w^\pm \big)(x_\beta),
\end{align*}
where $x_\beta$ is between $x_1$ and $x_1'$.
Then there holds
\begin{align}\label{D15}
&\langle w^\pm(x_1)\rangle
\sim \langle w^\pm(x_1')\rangle
\sim \langle w^\pm(x_\beta)\rangle .
\end{align}
Consequently, we obtain
\begin{align*}
&|\langle w^+(x_1)\rangle^{\mu}\langle w^-(x_1)\rangle^{\nu}-\langle w^+(x_1')\rangle^{\mu}\langle w^-(x_1')\rangle^{\nu}|\\
&\lesssim |x_1-x_1'|\cdot \big( \langle w^+(x_1')\rangle^{\mu-1}\langle w^-(x_1')\rangle^{\nu}+\langle w^+(x_1')\rangle^{\mu}\langle w^-(x_1')\rangle^{\nu-1}\big) .
\end{align*}
Thus we deduce that
\begin{align*}
& \big|\ \text{p.v.} \int_{|x_1-x_1'|\leq \f12} (1-\p_1^2) B_{2-s}(x_1-x_1')
\big[\langle w^+(x_1)\rangle^{\mu}\langle w^-(x_1)\rangle^{\nu} -\langle w^+(x_1')\rangle^{\mu}\langle w^-(x_1')\rangle^{\nu}\big] h(x_1') \d x_1'\big|\\
&\lesssim \int_{|x_1-x_1'|\leq \f12} | (x_1-x_1')(1-\p_1^2)B_{2-s}(x_1-x_1')|\cdot
\langle w^+(x_1')\rangle^{\mu}\langle w^-(x_1')\rangle^{\nu} |h(x_1')| \d x_1' .
\end{align*}
Taking the $L^q$ norm of the above expression,
by the Young inequality, it can be  bounded by
\begin{align*}
\int_{|x_1'|\leq \f12}  |x_1'(1-\p_1^2) B_{2-s}(x_1')| \d x_1'
\cdot \| \langle w^+\rangle^{\mu}\langle w^-\rangle^{\nu}
h \|_{L^q(\BR )}
 \lesssim \| \langle w^+\rangle^{\mu}\langle w^-\rangle^{\nu}
h \|_{L^q(\BR )}.
\end{align*}

Next, note
\begin{align*} 
\langle w^\pm(x_1)\rangle
= \langle w^\pm(x_1')\rangle +\int_{x_1'}^{x_1} \f{w^\pm(\tilde{x}_1)\p_{\tilde{x}_1} w^\pm(\tilde{x}_1)}
 {\langle w^\pm(\tilde{x}_1)\rangle} \d\tilde{x}_1.
\end{align*}
Thus  if $|x_1-x_1'|\geq \f12$,
there holds
\begin{align} \label{D19}
\langle w^\pm(x_1)\rangle
\leq 3 |x_1-x_1'|\langle w^\pm(x_1')\rangle.
\end{align}
Then
\begin{align*}
&\big|\text{p.v.}\int_{|x_1-x_1'|\geq \f12}  (1-\p_1^2)B_{2-s}(x_1-x_1')
\big(  \langle w^+(x_1)\rangle^{\mu}\langle w^-(x_1)\rangle^{\nu} -\langle w^+(x_1')\rangle^{\mu}\langle w^-(x_1')\rangle^{\nu} \big) h(x_1')\d x_1'\big| \\\nonumber
&\lesssim \int_{|x_1-x_1'|\geq \f12}  |(1-\p_1^2)B_{2-s}(x_1-x_1')|\cdot |x_1-x_1'|^{|\mu|+|\nu|} \langle w^+(x_1')\rangle^{\mu}\langle w^-(x_1')\rangle^{\nu} |h(x_1')| \d x_1' .
\end{align*}
Consequently, by the Young inequality and the properties of Bessel potentials of Lemma \ref{lemBe},
the $L^q$ norm of
the above expression is further bounded by
\begin{align*}
 \int_{|x_1'|\geq \f12} |x_1'|^{|\mu|+|\nu|}\cdot |(1-\p_1^2)B_{2-s}(x_1')| \d x_1'
\cdot\|  \langle w^+\rangle^{\mu}\langle w^-\rangle^{\nu} h \|_{L^q(\BR )}
\lesssim  \| \langle w^+\rangle^{\mu}\langle w^-\rangle^{\nu} h\|_{L^q(\BR )}.
\end{align*}
This finishes the proof of the lemma.
\end{proof}

Now we study the weighted trace theorem.
Let  $g=g(x)$  be a two dimensional function defined in a two dimensional domain $\Om\subseteq \bR^2$ with boundary  parameterized by $\p\Om=\{x_2=f(x_1)| x_1\in\bR \}$.
Let $\ud{g}(x_1)$ be the trace of $g(x)$ on the boundary $\p\Om$. We start by recalling the  classical trace theorem in the two dimensional setting.
\begin{lem}[Classical trace theorem in 2D]
\label{lemD1}
Let $s> \f12$. There holds
\begin{align*}
&\|  \langle \p_1 \rangle^{s-\f12} \ud{g} \|_{L^2(\BR )}
\lesssim  \| \langle \nabla \rangle^{s} g\|_{L^2(\Omega)}.
\end{align*}
Moreover, there holds
\begin{align*}
&\|  \ud{g} \|_{L^2(\BR )} \lesssim  \|   g\|_{H^1(\Omega)}.
\end{align*}
\end{lem}
A direct consequence of the second inequality of Lemma \ref{lemD1} is the following weighted trace theorem.
\begin{lem}[Relaxed weighted trace theorem] \label{lemD2}
For any $\mu\in \BR,\ \nu\in\BR$, there holds
\begin{align*}
\| \langle w^+\rangle^{\mu}\langle w^-\rangle^{\nu}  \ud{g} \|_{L^2(\BR )}
&\lesssim  \sum_{|\alpha|\leq 1} \| \langle w^+\rangle^{\mu}\langle w^-\rangle^{\nu}  \p^{\alpha} g \|_{L^2(\Omega)}.
\end{align*}
\end{lem}
 Next, we study the weighted trace theorem for the fractional derivatives.
 \begin{lem}[Weighted trace theorem]\label{lemD3}
For any $\mu\in \BR,\ \nu\in \BR$, there holds
\begin{align*}
\| \langle w^+\rangle^{\mu}\langle w^-\rangle^{\nu} \langle \p_1 \rangle^{\f12} \ud{g} \|_{L^2(\BR )}
&\lesssim  \sum_{|\alpha|\leq 1} \|\langle w^+\rangle^{\mu}\langle w^-\rangle^{\nu} \p^{\alpha} g \|_{L^2(\Omega)},
\end{align*}
provided the right hand side is finite.
\end{lem}
\begin{proof}
We start by rewriting \eqref{D14} as follows:
\begin{align}\label{D24}
&\langle w^+(x_1)\rangle^{\mu}\langle w^-(x_1)\rangle^{\nu} \langle \p_1\rangle^{\f 12} \ud{g}(x_1)
=\langle \p_1\rangle^{\f 12} \big( \langle w^+(x_1)\rangle^{\mu}\langle w^-(x_1)\rangle^{\nu} \ud{g}(x_1)\big)
\\\nonumber
&\qquad+\text{p.v.} \int_{\BR}  (1-\p_1^2)B_{\f 32}(x_1-x_1')
\big(\langle w^+(x_1)\rangle^{\mu}\langle w^-(x_1)\rangle^{\nu}-\langle w^+(x_1')\rangle^{\mu}\langle w^-(x_1')\rangle^{\nu}\big) \ud{g}(x_1')\d x_1'.
\end{align}
The first term on the right hand side \eqref{D24} can be controlled from the classical trace theorem of Lemma \ref{lemD1}.
For the second term, it is bounded from Lemma \ref{lemD0}
and the relaxed trace theorem Lemma \ref{lemD2}.
This finishes the proof of the lemma.
\end{proof}

\subsection{Weighted commutator estimate: fractional derivative}
Different from the notations in the above subsection,
$g$ and $h$ in this subsection denote functions defined in $\BR$.

Before the weighted commutator estimate,
we first present several lemmas.
\begin{lem}\label{lemD21}
Let  $0< s< 1$, $1\leq q\leq  \infty$, there holds
\begin{align*}
\| \langle \p_1\rangle^{s} h\|_{L^q(\BR )}
\lesssim \|   h  \|_{L^q(\BR )}
+\|  \p_1  h  \|_{L^q(\BR )}.
\end{align*}
\end{lem}
\begin{proof}
Let us copy \eqref{D1} and write
\begin{align*}
&\langle \p_1\rangle^{s} h(x_1)=  \int_{\BR} B_{2-s}(x_1-x_1') h(x_1') \d x_1'
+ \text{p.v.} \int_{\BR} \p_1^2B_{2-s}(x_1-x_1') (h(x_1)-h(x_1')) \d x_1'.
\end{align*}
By Young's inequality, it is easy to see that
\begin{align*}
\big\| \int_{\BR} B_{2-s}(x_1-x_1') h(x_1') \d x_1' \big\|_{L^q(\BR )}
\lesssim \| h \|_{L^q(\BR )} .
\end{align*}
On the other hand, by Young's inequality, we have
\begin{align*}
\big\|\text{p.v.} \int_{|x_1-x_1'|\geq \f12} \p_1^2B_{2-s}(x_1-x_1') (h(x_1)-h(x_1')) \d x_1'\big\|_{L^q(\BR )}
\lesssim \| h \|_{L^q(\BR )}.
\end{align*}
For  $|x_1-x_1'|\leq \f12$, we calculate
\begin{align*}
&\text{p.v.} \int_{|x_1-x_1'|\leq \f12} \p_1^2B_{2-s}(x_1-x_1') (h(x_1)-h(x_1')) \d x_1' \\\nonumber
&=\text{p.v.} \int_{|x_1-x_1'|\leq \f12}\int_0^1 (x_1'-x_1) \p_1^2B_{2-s}(x_1-x_1') \p_1h (x_1+\tau(x_1'-x_1)) \d \tau  \d x_1'\\\nonumber
&=\text{p.v.} \int_{|\bar{x}_1|\leq \f12}\int_0^1 \bar{x}_1 \p_1^2 B_{2-s}(\bar{x}_1) \p_1h(x_1+\tau \bar{x}_1) \d\tau  \d \bar{x}_1.
\end{align*}
Taking the $L^q(\BR )$ norm of the above expression, by
the Minkowski's inequality,
it is bounded by
\begin{align*}
\|  \p_1  h  \|_{L^q(\BR )}.
\end{align*}
Thus the lemma is proved.
\end{proof}
\begin{lem}\label{lemD23}
Let $s>0$, $\mu\in\BR,\, \nu\in\BR $, $1\leq q\leq\infty$. There holds
\begin{align*}
\|  \langle w^+\rangle^{\mu}\langle w^-\rangle^{\nu} \langle \p_1\rangle^{-s} g \|_{L^q(\BR )}
&\lesssim \| \langle w^+\rangle^{\mu}\langle w^-\rangle^{\nu} g\|_{L^q(\BR )} ,\\
\|  \langle w^+\rangle^{\mu}\langle w^-\rangle^{\nu} g\|_{L^q(\BR )}
&\lesssim \| \langle w^+\rangle^{\mu}\langle w^-\rangle^{\nu}\langle \p_1\rangle^{s} g\|_{L^q(\BR )} ,
\end{align*}
provided that the right hand sides are finite.
\end{lem}
\begin{proof}
The second inequality is a direct consequence of the first one.
Thus it suffices to prove the first inequality.

For $s>0$, by \eqref{D15} and \eqref{D19}, we write
\begin{align}\label{D25}
&\langle w^+\rangle^{\mu}\langle w^-\rangle^{\nu} |\langle \p_1\rangle^{-s} g(x_1)|  \\\nonumber
&=\langle w^+(x_1)\rangle^{\mu}\langle w^-(x_1)\rangle^{\nu}\Big|\int_{\BR}  B_s(x_1-x_1') g(x_1')\d x_1'\Big| \\\nonumber
&\lesssim
\int_{|x_1-x_1'|\leq \f12} B_s(x_1-x_1') \langle w^+(x_1')\rangle^{\mu}\langle w^-(x_1')\rangle^{\nu} |g(x_1')|\d x_1'  \\\nonumber
&\quad+\int_{|x_1-x_1'|\geq \f12} |x_1-x_1'|^{|\mu|+|\nu|} B_s(x_1-x_1') \langle w^+(x_1')\rangle^{\mu}\langle w^-(x_1')\rangle^{\nu} |g(x_1')| \d x_1'.
\end{align}
Here Bessel potentials $B_s(x_1)$ satisfy
\begin{align*}
&B_s(x_1)\lesssim e^{-\f{|x_1|}{2}}\quad \textrm{for}\,\,|x_1|\geq \f12,
\end{align*}
and
\begin{align*}
&B_s(x_1)\sim
\begin{cases}
|x_1|^{s-1}, \quad \textrm{for} \quad 0<s<1 ,\\
\ln|x_1|^{-1}, \quad \textrm{for} \quad s=1,\\
1, \quad \textrm{for} \quad s>1,
\end{cases}
\quad\textrm{for}\,\, 0<|x_1|\leq \f12.
\end{align*}
Taking the $L^q$ norm of \eqref{D25} for $1\leq q\leq +\infty$, by Young's inequality,
we obtain
\begin{align*}
\| \langle w^+\rangle^{\mu}\langle w^-\rangle^{\nu}\langle \p_1\rangle^{-s} g   \|_{L^q(\BR )}
\lesssim \| \langle w^+\rangle^{\mu}\langle w^-\rangle^{\nu} g \|_{L^q(\BR )} .
\end{align*}
This finishes the proof of the lemma.
\end{proof}

\begin{lem}\label{lemW}
Let $\mu,\ \nu\in\BR$.
There hold
\begin{align*}
&\| \langle w^+\rangle^{\mu}\langle w^-\rangle^{\nu} \langle \p_1\rangle^{-\f12}\p_1 h \|_{L^2(\BR)}
\lesssim  \| \langle w^+\rangle^{\mu}\langle w^-\rangle^{\nu} \langle \p_1\rangle^{\f12}  h \|_{L^2(\BR)},\\
&\| \langle w^+\rangle^{\mu}\langle w^-\rangle^{\nu} \langle \p_1\rangle^{-1}\p_1 h \|_{L^2(\BR)}
\lesssim  \| \langle w^+\rangle^{\mu}\langle w^-\rangle^{\nu}  h \|_{L^2(\BR)}.
\end{align*}
\end{lem}
\begin{proof}
The first and the second inequality are equivalent. It suffices to show the second one.

We write
\begin{align*}
\langle \p_1\rangle^{-1}\p_1 h(x_1)
&=\int_{\BR} B_1(x_1-x_1') \p_{x_1'} h(x_1')\d x_1'\\
&= \text{p.v.}\int_{\BR} \p_1B_1(x_1-x_1')   h(x_1')\d x_1'\, .
\end{align*}
Here
$B_1(x_1)$ is the Bessel potential.
If $|x_1-x_1'|\geq 1$, we estimate
\begin{align*}
&\langle w^+(x_1)\rangle^{\mu}\langle w^-(x_1)\rangle^{\nu}\big| \text{p.v.}\int_{|x_1-x_1'|\geq 1} \p_1B_1(x_1-x_1')   h(x_1')\d x_1'\big|\\
&\lesssim   \int_{|x_1-x_1'|\geq 1} |x_1-x_1'|^{|\mu|+|\nu|}\cdot |\p_1B_1(x_1-x_1') |
\langle w^+(x_1')\rangle^{\mu}\langle w^-(x_1')\rangle^{\nu}|h(x_1')|\d x_1'\,.
\end{align*}
Consequently,
\begin{align*}
\big\| \int_{|x_1-x_1'|\geq 1} |x_1-x_1'|\cdot |\p_1B_1(x_1-x_1') |
\langle w^+(x_1')\rangle^{\mu}\langle w^-(x_1')\rangle^{\nu}|h(x_1')|\d x_1' \big\|_{L^2}
\lesssim \|\langle w^+\rangle^{\mu}\langle w^-\rangle^{\nu} h  \big\|_{L^2}.
\end{align*}
Next, if $|x_1-x_1'|\leq 1$, we have
\begin{align*}
&\langle w^+(x_1)\rangle^{\mu}\langle w^-(x_1)\rangle^{\nu} \text{p.v.}\int_{|x_1-x_1'|\leq 1} \p_1B_1(x_1-x_1')   h(x_1')\d x_1'\\
&= \text{p.v.}\int_{|x_1-x_1'|\leq 1} \p_1B_1(x_1-x_1') \langle w^+(x_1')\rangle^{\mu}\langle w^-(x_1')\rangle^{\nu} h(x_1')\d x_1'\\
&\quad + \text{p.v.}\int_{|x_1-x_1'|\leq 1} \p_1B_1(x_1-x_1')  \big(\langle w^+(x_1)\rangle^{\mu}\langle w^-(x_1)\rangle^{\nu}-\langle w^+(x_1')\rangle^{\mu}\langle w^-(x_1')\rangle^{\nu}
\big) h(x_1')\d x_1'.
\end{align*}
The first term is Calderon-Zygmund type singular integral. Hence
\begin{align*}
\big\| \text{p.v.}\int_{|x_1-x_1'|\leq 1} \p_1B_1(x_1-x_1')  \langle w^+(x_1')\rangle^{\mu}\langle w^-(x_1')\rangle^{\nu} h(x_1')\d x_1' \big\|_{L^2}
\leq \|\langle w^+\rangle^{\mu}\langle w^-\rangle^{\nu} h  \big\|_{L^2}.
\end{align*}
For the second term,
note
\begin{align*}
&\big| \langle w^+(x_1)\rangle^{\mu}\langle w^-(x_1)\rangle^{\nu}-\langle w^+(x_1')\rangle^{\mu}\langle w^-(x_1')\rangle^{\nu} \big|\\
&\lesssim  |x_1-x_1'| \big(\langle w^+\rangle^{\mu-1}\langle w^-\rangle^{\nu}+\langle w^+\rangle^{\mu}\langle w^-\rangle^{\nu-1}\big)(x_1')\, .
\end{align*}
Consequently,
\begin{align*}
&\big\| \text{p.v.}\int_{|x_1-x_1'|\leq 1} \p_1B_1(x_1-x_1')  \big(\langle w^+(x_1)\rangle^{\mu}\langle w^-(x_1)\rangle^{\nu}-\langle w^+(x_1')\rangle^{\mu}\langle w^-(x_1')\rangle^{\nu}
\big) h(x_1')\d x_1' \big\|_{L^2}\\
&\leq \|\langle w^+\rangle^{\mu}\langle w^-\rangle^{\nu} h  \big\|_{L^2}.
\end{align*}
\end{proof}
Applying Lemma \ref{lemD0} to \eqref{D14} with $w=\langle w^+\rangle^{\mu}\langle w^-\rangle^{\nu}$,  by Lemma \ref{lemD23}, we have the following
commutator estimate.
\begin{lem}\label{lemD5}
For  $0< s<1$, $\mu\in \BR,\, \nu\in\BR$, $1\leq q\leq\infty$, there hold
\begin{align*}
& \| [\langle \p_1\rangle^{s}, \langle w^+\rangle^{\mu}\langle w^-\rangle^{\nu}] g \|_{L^q(\BR )}
\lesssim\| \langle w^+\rangle^{\mu}\langle w^-\rangle^{\nu} g\|_{L^q(\BR )} ,\\\nonumber
& \| \langle \p_1\rangle^{s}\big( \langle w^+\rangle^{\mu}\langle w^-\rangle^{\nu}  g\big)  \|_{L^q(\BR )}
\lesssim\| \langle w^+\rangle^{\mu}\langle w^-\rangle^{\nu} \langle \p_1\rangle^{s} g\|_{L^q(\BR )} ,
\end{align*}
provided the right hand sides are finite.
\end{lem}

Next we recall a version of the  commutator estimate \cite{KP, Li}.
\begin{lem}\label{lemD6}
Let  $0<s,s_1<1$, $\f12 =\f{1}{q_1}+\f{1}{q_2}=\f{1}{q_3}+\f{1}{q_4}$, $2\leq q_1, q_2, q_3, q_4 \leq\infty$, there hold
\begin{align*}
 \big\|[\langle\p_1\rangle^s,g] h \big\|_{L^2(\BR)}
&\lesssim
\| \langle\p_1\rangle^s  g \|_{L^{q_1}(\BR)}
\| h \|_{L^{q_2}(\BR)},\\
\big\|[\langle\p_1\rangle^{s_1}\p_1,g] h \big\|_{L^2(\BR )}
&\lesssim
\| \p_1  g \|_{L^{q_1}(\BR)}
\| \langle\p_1\rangle^{s_1-2}\p_1 h \|_{L^{q_2}(\BR )}
+\| \langle\p_1\rangle^{s_1}\p_1 g \|_{L^{q_3}(\BR)}
\|  h \|_{L^{q_4}(\BR)} .
\end{align*}
\end{lem}
Now we present the weighted commutator estimate
for fractional derivative ($0<s<1$) in one space dimension.
Later we will consider higher order weighted commutator estimate.
Similar estimates also hold  in general space dimension and the proof is the same.
\begin{lem}[Weighted commutator estimate, first version] \label{lemD8}
Let  $0<s<1$, $\f12 =\f{1}{q_1}+\f{1}{q_2}$, $2\leq q_1, q_2 \leq\infty$, $\mu,\mu_1,\mu_2,\nu,\nu_1,\nu_2\in\BR$,
 $\mu=\mu_1+\mu_2$, $\nu=\nu_1+\nu_2$. There holds
\begin{align*}
 \big\| \langle w^+\rangle^{\mu}\langle w^-\rangle^{\nu}
[\langle\p_1\rangle^{s},g] h \big\|_{L^2(\BR )}
&\lesssim  \| \langle w^+\rangle^{\mu_1}\langle w^-\rangle^{\nu_1} \langle \p_1\rangle^{s} g \|_{L^{q_1}(\BR)} \| \langle w^+\rangle^{\mu_2}\langle w^-\rangle^{\nu_2}  h  \|_{L^{q_2}(\BR )} ,
\end{align*}
provided the right hand sides are finite.
\end{lem}
\begin{proof}
By commutator notation, we write
\begin{align}\label{D30}
\langle w^+\rangle^{\mu}\langle w^-\rangle^{\nu}   [\langle\p_1\rangle^{s},g] h
&=-[\langle \p_1\rangle^{s},\langle w^+\rangle^{\mu}\langle w^-\rangle^{\nu}] (gh) \\[-4mm]\nonumber\\\nonumber
&\quad+\big[\langle \p_1\rangle^{s},
\langle w^+\rangle^{\mu_1}\langle w^-\rangle^{\nu_1} g\big] (\langle w^+\rangle^{\mu_2}\langle w^-\rangle^{\nu_2}h)  \\[-4mm]\nonumber\\\nonumber
&\quad+\langle w^+\rangle^{\mu_1}\langle w^-\rangle^{\nu_1} g
\cdot
[\langle \p_1\rangle^{s}, \langle w^+\rangle^{\mu_2}\langle w^-\rangle^{\nu_2} ] h.
\end{align}
For the second term on the right hand side of \eqref{D30}, by the classical commutator estimate Lemma \ref{lemD6}, it is bounded by
\begin{align*}
\| \langle w^+\rangle^{\mu_1}\langle w^-\rangle^{\nu_1} \langle \p_1\rangle^{s} g \|_{L^{q_1}(\BR)} \| \langle w^+\rangle^{\mu_2}\langle w^-\rangle^{\nu_2}  h  \|_{L^{q_2}(\BR )}.
\end{align*}
For the first and the third term, by Lemma \ref{lemD5} and the H\"{o}lder inequality, it is bounded by
\begin{align*}
\| \langle w^+\rangle^{\mu_1}\langle w^-\rangle^{\nu_1} g \|_{L^{q_1}(\BR)} \| \langle w^+\rangle^{\mu_2}\langle w^-\rangle^{\nu_2}  h  \|_{L^{q_2}(\BR )}.
\end{align*}
Thus the lemma is proved.
\end{proof}

\subsection{Commutator estimate: higher order derivative}

In this subsection, we study the weighted commutator estimate with higher order fractional derivatives
\begin{align*}
& \big\| w [\langle\p_1\rangle^{s}\p_1,g] h \big\|_{L^2(\BR )}
\end{align*}
for $0<s<1$. Here $w$ is a function. We will study the commutator for $w=\langle w^+\rangle^{\mu}\langle w^-\rangle^{\nu}$.

We first study the commutator
\begin{align}\label{D31}
&\big\| [\langle \p_1\rangle^{s}\p_1, w] h \big\|_{L^2(\BR )}
=\big\| \langle \p_1\rangle^{s}\p_1 \big( w  h\big)
-w\langle \p_1\rangle^{s}\p_1 h \big\|_{L^2(\BR )}.
\end{align}
Define the symmetric commutator
\begin{align*}
[\langle\p_1\rangle^{s}\p_1,w,h]
=\langle\p_1\rangle^{s}\p_1(wh)-w \langle\p_1\rangle^{s}\p_1 h - h \langle\p_1\rangle^{s}\p_1 w.
\end{align*}
Let $h=h(x_1)$ be a function defined in $\BR$.
 By \eqref{D1}, we have
\begin{align*}
&\langle \p_1\rangle^{s} \p_1 h(x_1)
= \int_{\BR} B_{2-s}(x_1-x_1') (1-\p_{x_1'}^2) \p_{x_1'}h(x_1') \d x_1' \\\nonumber
&= \int_{\BR} \p_1 B_{2-s}(x_1-x_1') h(x_1') \d x_1'
 + \text{p.v.} \int_{\BR} \p_1^2B_{2-s}(x_1-x_1') (\p_1 h(x_1)-\p_{x_1'} h(x_1')) \d x_1'.
\end{align*}
To estimate \eqref{D31}, we calculate the symmetric commutator \cite{P}
\begin{align}\label{D32}
&[\langle\p_1\rangle^{s}\p_1,w,h](x_1)=\langle \p_1\rangle^{s}\p_1 \big( w(x_1)  h(x_1)\big)
-w(x_1) \langle \p_1\rangle^{s}\p_1 h(x_1)
-h(x_1) \langle \p_1\rangle^{s}\p_1 w(x_1)
\\\nonumber
&= \int_{\BR} \p_1 B_{2-s}(x_1-x_1') \big(w(x_1') h(x_1')- w(x_1) h(x_1')-w(x_1') h(x_1) \big)\d x_1' \\\nonumber
&\quad+ \ \text{p.v.} \int_{\BR} \p_1^2B_{2-s}(x_1-x_1') \big(w(x_1)-w(x_1')\big)\p_{x_1'}\big(h(x_1)-h(x_1')\big) \d x_1' \\\nonumber
&=\int_{\BR} \p_1 B_{2-s}(x_1-x_1') \big(w(x_1') h(x_1')- w(x_1) h(x_1')-w(x_1') h(x_1) \big)\d x_1' \\\nonumber
&\quad +\text{p.v.} \int_{\BR} \p_1^2B_{2-s}(x_1-x_1')
\p_{x_1'}  w(x_1') \big(h(x_1)-h(x_1')\big) \d x_1' \\\nonumber
&\quad + \text{p.v.} \int_{\BR} \p_1^3B_{2-s}(x_1-x_1') \big(w(x_1)-w(x_1')\big)\big(h(x_1)-h(x_1')\big) \d x_1' .
\end{align}
For $w=\langle w^+\rangle^{\mu}\langle w^-\rangle^{\nu}$, in order to estimate \eqref{D31},
we need estimate $h(x_1) \langle \p_1\rangle^{s}\p_1 w(x_1)$
and the right hand side of \eqref{D32}.

We first calculate $\langle \p_1\rangle^{s}\p_1 w(x_1)$.
\begin{lem}\label{lemD24}
For $0<s<1$, $\mu \in \BR,\ \nu\in\BR$.
There holds
\begin{align*}
| \langle \p_1\rangle^{s}\p_1
\big(\langle w^+(x_1)\rangle^{\mu}\langle w^-(x_1)\rangle^{\nu} \big)|
\lesssim \langle w^+(x_1)\rangle^{\mu-1}\langle w^-(x_1)\rangle^{\nu} +\langle w^+(x_1)\rangle^{\mu}\langle w^-(x_1)\rangle^{\nu-1}.
\end{align*}
\end{lem}
\begin{proof}
Denote $v(x_1)=\p_1\langle w^+(x_1)\rangle^{\mu}\langle w^-(x_1)\rangle^{\nu}$.
Obviously one has
\begin{align*}
v(x_1)\lesssim \langle w^+(x_1)\rangle^{\mu-1}\langle w^-(x_1)\rangle^{\nu} +\langle w^+(x_1)\rangle^{\mu}\langle w^-(x_1)\rangle^{\nu-1}.
\end{align*}
By \eqref{D1}, we write
\begin{align}\label{D41}
&\langle \p_1\rangle^{s}\p_1 (\langle w^+(x_1)\rangle^{\mu}\langle w^-(x_1)\rangle^{\nu})
=\langle \p_1\rangle^{s} v(x_1) \\\nonumber
&=\int_{\BR} B_{2-s}(x_1-x_1')v(x_1')  \d x_1'
 +\text{p.v.}\int_{\BR} \p_1^2B_{2-s}(x_1-x_1') \big(v(x_1)-v(x_1')\big) \d x_1'.
\end{align}
Recalling that for the weight functions $w^\pm$, there hold
\begin{align*}
&\langle w^\pm(x_1)\rangle
\sim \langle w^\pm(x_1')\rangle,\quad \textrm{if}\ \ |x_1-x_1'|\leq \f12, \\
&\langle w^\pm(x_1')\rangle
\leq 3  |x_1-x_1'|
\langle w^\pm(x_1)\rangle ,\quad \textrm{if}\ \ |x_1-x_1'|\geq \f12.
\end{align*}
Thus for the first term on the right hand side of \eqref{D41},
we estimate similar to the proof of Lemma \ref{lemD0}:
\begin{align*}
& |\int_{\BR} B_{2-s}(x_1-x_1') v(x_1) \d x_1'| \\\nonumber
&\lesssim
\big( \langle w^+(x_1)\rangle^{\mu-1}\langle w^-(x_1)\rangle^{\nu} +\langle w^+(x_1)\rangle^{\mu}\langle w^-(x_1)\rangle^{\nu-1} \big) \\[-5mm]\nonumber\\
&\quad \Big(\int_{|x_1-x_1'|\leq \f12} |B_{2-s}(x_1-x_1')|  \d x_1'
+ \int_{|x_1-x_1'|\geq \f12} |x_1-x_1'|^{|\mu|+|\nu| }\cdot |B_{2-s}(x_1-x_1')|\d x_1'\Big)\\[-4mm]\nonumber\\
&\lesssim \langle w^+(x_1)\rangle^{\mu-1}\langle w^-(x_1)\rangle^{\nu} +\langle w^+(x_1)\rangle^{\mu}\langle w^-(x_1)\rangle^{\nu-1} .
\end{align*}
For the last term of \eqref{D41},
if $|x_1-x_1'|\geq \f12$, by the exponential decay
of the Bessel potentials at infinity, we obtain
\begin{align*}
& \big|\text{p.v.}
\int_{|x_1-x_1'|\geq \f12} \p_1^2B_{2-s}(x_1-x_1') \big(
v(x_1)-v(x_1') \big) \d x_1'\big| \\\nonumber
&\lesssim \langle w^+(x_1)\rangle^{\mu-1}\langle w^-(x_1)\rangle^{\nu} +\langle w^+(x_1)\rangle^{\mu}\langle w^-(x_1)\rangle^{\nu-1} .
\end{align*}
If $|x_1-x_1'|\leq \f12$, we have
\begin{align*}
|v(x_1)-v(x_1')|
\lesssim |x_1-x_1'|
\big( \langle w^+(x_1)\rangle^{\mu-1}\langle w^-(x_1)\rangle^{\nu} +\langle w^+(x_1)\rangle^{\mu}\langle w^-(x_1)\rangle^{\nu-1} \big) .
\end{align*}
Consequently,   we obtain
\begin{align*}
& \big|\text{p.v.}
\int_{|x_1-x_1'|\leq \f12} \p_1^2B_{2-s}(x_1-x_1') \big(v(x_1)-v(x_1')\big) \d x_1'\big|\\\nonumber
&\lesssim \langle w^+(x_1)\rangle^{\mu-1}\langle w^-(x_1)\rangle^{\nu} +\langle w^+(x_1)\rangle^{\mu}\langle w^-(x_1)\rangle^{\nu-1}.
\end{align*}
This finishes the proof of the lemma.
\end{proof}

For the first line on the right hand side of \eqref{D32},
by the similar argument in the proof of Lemma \ref{lemD0}, we have
\begin{lem}\label{lemD25}
Let $0<s<1$, $\mu \in\BR,\, \nu\in\BR $. There holds
\begin{align*}
&\big\| \int_{\BR} \p_1 B_{2-s}(x_1-x_1') \big(\langle w^+(x_1')\rangle^{\mu}\langle w^-(x_1')\rangle^{\nu} h(x_1') - \langle w^+(x_1)\rangle^{\mu}\langle w^-(x_1)\rangle^{\nu} h(x_1')\\\nonumber
&\qquad\quad -\langle w^+(x_1')\rangle^{\mu}\langle w^-(x_1')\rangle^{\nu} h(x_1) \big)\d x_1'  \big\|_{L^2(\BR)}
\lesssim \| \langle w^+\rangle^{\mu}\langle w^-\rangle^{\nu}  h\|_{L^2(\BR)},
\end{align*}
provided the right hand side is finite.
\end{lem}
For the last line of \eqref{D32}, we have the following estimate.
\begin{lem}\label{lemD26}
Let $0<s<1$. For any $\mu\in\BR,\, \nu\in\BR $, there holds
\begin{align*} 
&\big\| \text{p.v.} \int_{\BR} \p_1^3B_{2-s}(x_1-x_1')
\big(\langle w^+(x_1)\rangle^{\mu}\langle w^-(x_1)\rangle^{\nu}
-\langle w^+(x_1')\rangle^{\mu}\langle w^-(x_1')\rangle^{\nu} \big) \\\nonumber
&\qquad\quad \cdot\big(h(x_1)-h(x_1')\big) \d x_1' \big\|_{L^2(\BR)}
 \lesssim \big\| \langle w^+\rangle^{\mu}\langle w^-\rangle^{\nu} \langle \p_1\rangle^s h  \big\|_{L^2(\BR)},
\end{align*}
provided the right hand side is finite.
\end{lem}
\begin{proof}
We similarly divide the integral domain $\BR$ into two domains
$\{x_1':|x_1-x_1'|\leq \f12\}$ and $\{x_1':|x_1-x_1'|\geq \f12\}$.
Recalling that if $|x_1-x_1'|\geq \f12$, there hold
\begin{align*}
\langle w^\pm(x_1)\rangle
\leq 3 |x_1-x_1'|\langle w^\pm(x_1')\rangle,
\quad  \langle w^\pm(x_1')\rangle
\leq 3 |x_1-x_1'|\langle w^\pm(x_1)\rangle.
\end{align*}
Then
\begin{align*}
\langle w^\pm(x_1)\rangle^{\mu}
\leq 3^{|\mu|} |x_1-x_1'|^{|\mu|}\langle w^\pm(x_1')\rangle^{\mu},
\quad  \langle w^\pm(x_1')\rangle^{\mu}
\leq 3 |x_1-x_1'|^{|\mu|}\langle w^\pm(x_1)\rangle^{\mu}.
\end{align*}
Thus we can estimate similar to Lemma \ref{lemD0} by the exponential decay
of the Bessel potentials. Precisely, by Young's inequality, we deduce
\begin{align*} 
&\big\| \int_{|x_1-x_1'|\geq \f12} \p_1^3B_{2-s}(x_1-x_1')
\big(\langle w^+(x_1)\rangle^{\mu}\langle w^-(x_1)\rangle^{\nu}-\langle w^+(x_1')\rangle^{\mu}\langle w^-(x_1')\rangle^{\nu} \big)\big(h(x_1)-h(x_1')\big) \d x_1' \big\|_{L^2(\BR)} \\[-5mm]\nonumber\\\nonumber
&\lesssim \big\| \int_{|x_1-x_1'|\geq \f12} |\p_1^3B_{2-s}(x_1-x_1')|\cdot
|x_1-x_1'|^{|\mu|+|\nu|}\langle w^+(x_1)\rangle^{\mu}\langle w^-(x_1)\rangle^{\nu}
|h(x_1)|   \d x_1' \big\|_{L^2(\BR)} \\\nonumber
&\quad+\big\| \int_{|x_1-x_1'|\geq \f12} |\p_1^3 B_{2-s}(x_1-x_1')|
\cdot |x_1-x_1'|^{|\mu|+|\nu|}\langle w^+(x_1')\rangle^{\mu}\langle w^-(x_1')\rangle^{\nu}
|h(x_1')|  \d x_1' \big\|_{L^2(\BR)} \\\nonumber
&\lesssim \|  \langle w^+\rangle^{\mu}\langle w^-\rangle^{\nu} h \|_{L^2(\BR)}.
\end{align*}

Next we estimate the case of $|x_1-x_1'|\leq \f 12$.
To simplify the presentation, denote $v(x_1)=\langle w^+(x_1)\rangle^{\mu}\langle w^-(x_1)\rangle^{\nu} $.
We write
\begin{align*}
v(x_1)=v(x_1')+(x_1-x_1')\p_1v(x_1')+\f12 (x_1-x_1')^2\p_1^2v(x_\beta),
\end{align*}
where $x_\beta$ is between $x_1$ and $x_1'$.
There hold
\begin{align*}
&|\p_1v(x_1)| \lesssim \langle w^+(x_1)\rangle^{\mu-1}\langle w^-(x_1)\rangle^{\nu}+\langle w^+(x_1)\rangle^{\mu}\langle w^-(x_1)\rangle^{\nu-1},\\
&|\p_1^2v(x_1)| \lesssim \langle w^+(x_1)\rangle^{\mu-2}\langle w^-(x_1)\rangle^{\nu}+\langle w^+(x_1)\rangle^{\mu-1}\langle w^-(x_1)\rangle^{\nu-1}
+\langle w^+(x_1)\rangle^{\mu-1}\langle w^-(x_1)\rangle^{\nu-2} .
\end{align*}
Thus
\begin{align}\label{D54}
&\text{p.v.} \int_{|x_1-x_1'|\leq \f12} \p_1^3B_{2-s}(x_1-x_1')
\big(v(x_1) -v(x_1') \big)\big(h(x_1)-h(x_1')\big) \d x_1' \\\nonumber
&=\f12\text{p.v.} \int_{|x_1-x_1'|\leq \f12} (x_1-x_1')^2 \p_1^3B_{2-s}(x_1-x_1')
 \, \p_1^2v(x_\beta) \big(h(x_1)-h(x_1')\big) \d x_1' \\\nonumber
&\quad+\text{p.v.} \int_{|x_1-x_1'|\leq \f12} (x_1-x_1')\p_1^3 B_{2-s}(x_1-x_1')
 \p_1v(x_1') \big(h(x_1)-h(x_1')\big) \d x_1' .
\end{align}
For the second line of \eqref{D54}, we have the following bound:
\begin{align*}
&\f12\big| \text{p.v.} \int_{|x_1-x_1'|\leq \f12} (x_1-x_1')^2 \p_1^3B_{2-s}(x_1-x_1')
 \, \p_1^2v(x_\beta) \big(h(x_1)-h(x_1')\big) \d x_1'\big| \\\nonumber
&\lesssim  \int_{|x_1-x_1'|\leq \f12} |x_1-x_1'|^{-s }
\langle w^+(x_1')\rangle^{\mu}\langle w^-(x_1')\rangle^{\nu}  |h(x_1)-h(x_1')| \d x_1' .
\end{align*}
Taking the $L^2$ norm of the above expression, by the Young inequality,
it is further controlled by
\begin{align*}
\big\| \langle w^+\rangle^{\mu}\langle w^-\rangle^{\nu} h\big\|_{L^2(\BR)} .
\end{align*}
For the last line of \eqref{D54}, the integral kernel is still too singular. We further organize that
\begin{align*}
 \p_1v(x_1') \big(h(x_1)-h(x_1')\big)
&=\big(\p_1v(x_1)\cdot h(x_1)-\p_1v(x_1')\cdot h(x_1')\big) \\\nonumber
&\quad-\big( \p_1v(x_1)-\p_1v(x_1')\big)  h(x_1)  .
\end{align*}
Consequently, we organize the last line of \eqref{D54} as follows
\begin{align}\label{D56}
&\text{p.v.} \int_{|x_1-x_1'|\leq \f12} (x_1-x_1')\p_1^3 B_{2-s}(x_1-x_1')
 \p_1v(x_1') \big(h(x_1)-h(x_1')\big) \d x_1' \\\nonumber
&=\text{p.v.} \int_{|x_1-x_1'|\leq \f12} (x_1-x_1')\p_1^3 B_{2-s}(x_1-x_1')
 \big(\p_1v(x_1)\cdot h(x_1)-\p_1v(x_1')\cdot h(x_1')\big)  \d x_1'\\\nonumber
&\quad-\text{p.v.} \int_{|x_1-x_1'|\leq \f12} (x_1-x_1')\p_1^3 B_{2-s}(x_1-x_1')
\big(\p_1v(x_1)-\p_1v(x_1')\big)  h(x_1) \d x_1' .
\end{align}
To estimate the last line of \eqref{D56}, note
\begin{align*}
&|\p_1v(x_1)-\p_1v(x_1')| \\
&\lesssim |x_1-x_1'| \big(\langle w^+(x_1)\rangle^{\mu-2}\langle w^-(x_1)\rangle^{\nu}+\langle w^+(x_1)\rangle^{\mu-1}\langle w^-(x_1)\rangle^{\nu-1}
+\langle w^+(x_1)\rangle^{\mu-1}\langle w^-(x_1)\rangle^{\nu-2}\big) .
\end{align*}
Hence the last line of \eqref{D56} have the following bound
\begin{align*}
&\big|\text{p.v.} \int_{|x_1-x_1'|\leq \f12} (x_1-x_1')\p_1^3 B_{2-s}(x_1-x_1')
\big(\p_1v(x_1)-\p_1v(x_1')\big)  h(x_1) \d x_1' \big|\\\nonumber
&\lesssim  \int_{|x_1-x_1'|\leq \f12} (x_1-x_1')^2|\p_1^3 B_{2-s}(x_1-x_1')|
\langle w^+(x_1)\rangle^{\mu}\langle w^-(x_1)\rangle^{\nu} |h(x_1)| \d x_1'.
\end{align*}
Note that the integral kernel $x_1^2\p_1^3B_{2-s}(x_1)$ belongs to $L^1$.
Thus taking the $L^2$ norm of the above expression, 
it is controlled by
\begin{align*}
\| \langle w^+\rangle^{\mu}\langle w^-\rangle^{\nu} h \|_{L^2(\BR)} .
\end{align*}
For the second line of \eqref{D56}, we write
\begin{align} \label{D58}
& \text{p.v.} \int_{|x_1-x_1'|\leq \f12} (x_1-x_1')\p_1^3 B_{2-s}(x_1-x_1')
\cdot \big(\p_1v(x_1)\cdot h(x_1)-\p_1v(x_1')\cdot h(x_1')\big)  \d x_1'\\\nonumber
&\sim \text{p.v.} \int_{|x_1-x_1'|\leq \f12} |x_1-x_1'|^{-1-s}
 \big(\p_1v(x_1)\cdot h(x_1)-\p_1v(x_1')\cdot h(x_1')\big)  \d x_1'  \\\nonumber
&= |\p_1|^s \big(\p_1v(x_1)\cdot h(x_1)\big)
- \text{p.v.} \int_{|x_1-x_1'|\geq \f12} |x_1-x_1'|^{-1-s}
 \big(\p_1v(x_1)\cdot h(x_1)-\p_1v(x_1')\cdot h(x_1')\big)  \d x_1'  .
\end{align}
Thus taking the $L^2$ norm of \eqref{D58}, for the last second term, repeating the argument of the proof of Lemma \ref{lemD5}, it is further bounded by
\begin{align*}
\big\| |\p_1|^s \big(\p_1v \cdot h \big)\big\|_{L^2(\BR)}
\leq \| \langle \p_1\rangle^s \big(\p_1v \cdot h \big)\|_{L^2(\BR)}
\lesssim \| \langle w^+\rangle^{\mu}\langle w^-\rangle^{\nu} \langle \p_1\rangle^s h \|_{L^2(\BR)}.
\end{align*}
For the last term of \eqref{D58}, by the Young inequality, we have
\begin{align*}
&\big\| \text{p.v.} \int_{|x_1-x_1'|\geq \f12} |x_1-x_1'|^{-1-s}
 \big(\p_1v(x_1)\cdot h(x_1)-\p_1v(x_1')\cdot h(x_1')\big)  \d x_1'  \big\|_{L^2(\BR)} \\
&\lesssim \|  \p_1v \cdot h  \|_{L^2(\BR)}
\lesssim  \| \langle w^+\rangle^{\mu}\langle w^-\rangle^{\nu}   h \|_{L^2(\BR)} .
\end{align*}
This finishes the proof of the lemma.
\end{proof}

For the second line on the right hand side of \eqref{D32},
if $|x_1-x_1'|\geq \f12$, we can use the exponential decay of the Bessel potentials.
If $|x_1-x_1'|\leq \f12$, we estimate similar to \eqref{D56} and obtain the same bound.
Thus we have
\begin{lem}\label{lemD27}
Let $0<s<1$, $\mu \in \BR,\, \nu\in\BR$.
There holds
\begin{align*}
&\big\| \text{p.v.} \int_{\BR} \p_1^2B_{2-s}(x_1-x_1')
\p_{x_1'}\langle w^+(x_1')\rangle^{\mu}\langle w^-(x_1')\rangle^{\nu} \\\nonumber
&\qquad \cdot\big(h(x_1)-h(x_1')\big) \d x_1' \big\|_{L^2(\BR)}
\lesssim \| \langle w^+\rangle^{\mu}\langle w^-\rangle^{\nu}  \langle \p_1\rangle^s h  \|_{L^2(\BR)},
\end{align*}
provided the right hand side is finite.
\end{lem}
Now we present the estimate for \eqref{D31}.
Combining \eqref{D32}, Lemma \ref{lemD23}, Lemma \ref{lemD24}, Lemma \ref{lemD25}, Lemma \ref{lemD26}, Lemma \ref{lemD27}, we have the following commutator estimate.
\begin{lem}\label{lemD9}
For $0< s<1$, $\mu \in \BR,\, \nu\in\BR$.   there holds
\begin{align*}
& \big\| [\langle \p_1\rangle^{s} \p_1,\langle w^+\rangle^{\mu}\langle w^-\rangle^{\nu}] h \big\|_{L^2(\BR )}
\lesssim \| \langle w^+\rangle^{\mu}\langle w^-\rangle^{\nu}\langle \p_1 \rangle^s h\|_{L^2(\BR )} ,
\end{align*}
provided the right hand side is finite.
\end{lem}

Next, we study the  weighted commutator estimate for higher order fractional derivative.
The estimate also holds for the general dimensional case and the proof is similar.
\begin{lem}[Weighted commutator estimate, second version] \label{lemD10}
Let  $0<s<1$, $\f12 =\f{1}{q_1}+\f{1}{q_2}=\f{1}{q_3}+\f{1}{q_4}$, $2\leq q_1, q_2, q_3, q_4 \leq\infty$, $\mu,\mu_1,\mu_2,\nu,\nu_1,\nu_2\in\BR$,
 $\mu=\mu_1+\mu_2$, $\nu=\nu_1+\nu_2$. There holds
\begin{align*}
& \big\| \langle w^+\rangle^{\mu}\langle w^-\rangle^{\nu} [\langle\p_1\rangle^{s}\p_1,g] h \big\|_{L^2(\BR )} \\
&\lesssim \| \langle w^+\rangle^{\mu_1}\langle w^-\rangle^{\nu_1} \langle \p_1\rangle^s\p_1^{\leq 1} g \|_{L^{q_1}}
 \| \langle w^+\rangle^{\mu_2}\langle w^-\rangle^{\nu_2}  h \|_{L^{q_2}} \\
&\quad+ \| \langle w^+\rangle^{\mu_1}\langle w^-\rangle^{\nu_1}\p_1^{\leq 1} g \|_{L^{q_3}}
\| \langle w^+\rangle^{\mu_2}\langle w^-\rangle^{\nu_2} \langle \p_1\rangle^s h \|_{L^{q_4}} ,
\end{align*}
provided the right hand side is finite.
\end{lem}
\begin{proof}
Similar to \eqref{D30}, we have
\begin{align}\label{D61}
\langle w^+\rangle^{\mu}\langle w^-\rangle^{\nu}   [\langle\p_1\rangle^{s}\p_1,g] h
&=-[\langle \p_1\rangle^{s}\p_1,\langle w^+\rangle^{\mu}\langle w^-\rangle^{\nu}] (gh) \\[-4mm]\nonumber\\\nonumber
&\quad+\big[\langle \p_1\rangle^{s}\p_1,\langle w^+\rangle^{\mu_1}\langle w^-\rangle^{\nu_1} g\big] (\langle w^+\rangle^{\mu_2}\langle w^-\rangle^{\nu_2}h)  \\[-4mm]\nonumber\\\nonumber
&\quad+\langle w^+\rangle^{\mu_1}\langle w^-\rangle^{\nu_1} g\cdot [\langle \p_1\rangle^{s}\p_1, \langle w^+\rangle^{\mu_2}\langle w^-\rangle^{\nu_2} ] h.
\end{align}
For the first term  on the right hand side of \eqref{D61}, by Lemma \ref{lemD9}, Lemma \ref{lemD8} and the H\"{o}lder inequality, it is bounded by
\begin{align*}
& \| \langle w^+\rangle^{\mu}\langle w^-\rangle^{\nu}   \langle\p_1\rangle^{s} (gh) \|_{L^2(\BR )} \\
&\lesssim \| \langle w^+\rangle^{\mu_1}\langle w^-\rangle^{\nu_1} \langle \p_1\rangle^{s} g \|_{L^{q_1}(\BR)} \| \langle w^+\rangle^{\mu_2}\langle w^-\rangle^{\nu_2}  h  \|_{L^{q_2}(\BR )}\\
&\quad+\| \langle w^+\rangle^{\mu_1}\langle w^-\rangle^{\nu_1}  g \|_{L^{q_3}(\BR)} \| \langle w^+\rangle^{\mu_2}\langle w^-\rangle^{\nu_2} \langle \p_1\rangle^{s} h  \|_{L^{q_4}(\BR )}.
\end{align*}
For the third term on the right hand side of \eqref{D61}, by Lemma \ref{lemD9} and the H\"{o}lder inequality, it is bounded by
\begin{align*}
\| \langle w^+\rangle^{\mu_1}\langle w^-\rangle^{\nu_1}  g \|_{L^{q_3}(\BR)} \| \langle w^+\rangle^{\mu_2}\langle w^-\rangle^{\nu_2} \langle \p_1\rangle^{s} h  \|_{L^{q_4}(\BR )}.
\end{align*}
For the second term on the right hand side of \eqref{D61},
 by applying Lemma \ref{lemD21}, Lemma \ref{lemD23}, Lemma \ref{lemD5} and  Lemma \ref{lemD6}, we have
\begin{align*}
&\big\| \big[\langle \p_1\rangle^{s}\p_1,\langle w^+\rangle^{\mu_1}\langle w^-\rangle^{\nu_1} g\big] (\langle w^+\rangle^{\mu_2}\langle w^-\rangle^{\nu_2}h) \big\|_{L^2} \\[-4mm]\nonumber\\\nonumber
&\lesssim  \| \langle \p_1\rangle^s\p_1\big( \langle w^+\rangle^{\mu_1}\langle w^-\rangle^{\nu_1} g\big)
\|_{L^{q_1}} \| \langle w^+\rangle^{\mu_2}\langle w^-\rangle^{\nu_2}  h \|_{L^{q_2}}\\\nonumber
&\quad+ \| \p_1 \big( \langle w^+\rangle^{\mu_1}\langle w^-\rangle^{\nu_1} g\big)  \|_{L^{q_3}}
\| \langle \p_1\rangle^s\big( \langle w^+\rangle^{\mu_2}\langle w^-\rangle^{\nu_2} h \big) \|_{L^{q_4}}\\\nonumber
&\lesssim  \| \langle w^+\rangle^{\mu_1}\langle w^-\rangle^{\nu_1} \langle \p_1\rangle^s\p_1^{\leq 1} g \|_{L^{q_1}}
 \| \langle w^+\rangle^{\mu_2}\langle w^-\rangle^{\nu_2}  h \|_{L^{q_2}}\\\nonumber
&\quad+ \| \langle w^+\rangle^{\mu_1}\langle w^-\rangle^{\nu_1}\p_1^{\leq 1} g \|_{L^{q_3}}
\| \langle w^+\rangle^{\mu_2}\langle w^-\rangle^{\nu_2} \langle \p_1\rangle^s h \|_{L^{q_4}} .
\end{align*}
This finishes the proof of the lemma.
\end{proof}

Now we show the more general commutator estimate for higher order derivatives.
\begin{lem}[Weighted commutator estimate, higher order version] \label{lemD11}
Let $a\geq 2$ be an integer,  $0<s<1$, $\mu,\mu_1,\mu_2,\nu,\nu_1,\nu_2\in\BR$.
 $\mu=\mu_1+\mu_2$, $\nu=\nu_1+\nu_2$. There holds
\begin{align*}
&\| \langle w^+\rangle^{\mu}\langle w^-\rangle^{\nu} \big[\langle\p_1\rangle^s\p_1^a,g \big] \p_1h\|_{L^2(\BR)} \\\nonumber
&\lesssim \| \langle w^+\rangle^{\mu_1}\langle w^-\rangle^{\nu_1} \langle\p_1\rangle^s\p_1^{\leq a} g\|_{L^2}
          \| \langle w^+\rangle^{\mu_2}\langle w^-\rangle^{\nu_2} \langle\p_1\rangle^s\p_1^{\leq a} h\|_{L^2},
\end{align*}
provided the right hand side is finite.
\end{lem}
\begin{remark}
The assumption $a\geq 2$ is not sharp. Actually it holds only if $s+a>\f{d}{2}+1$ where $d$ is the space dimension.
To save the length of this paper, we will not pursue along this direction.
\end{remark}
\begin{proof}
By the Leibniz rule, we write
\begin{align}\label{D72}
\big[ \langle\p_1\rangle^s\p_1^a,g \big] \p_1h
=\langle\p_1\rangle^s \sum_{b+c=a,\ c\neq a}C_a^b\p_1^bg  \p_1^{c+1}h
-\langle\p_1\rangle^s  (\p_1 g  \p_1^ah)
+\big[ \langle\p_1\rangle^s \p_1,g \big] \p_1^a h,
\end{align}
where $ C_a^b=\f{a!}{b!(a-b)!}$ is the binomial coefficient.
For the second and the last term of \eqref{D72}, by the weighted commutator estimate
 Lemma \ref{lemD6}, Lemma \ref{lemD8}, Lemma \ref{lemD10} and
the weighted Sobolev inequalities Lemma \ref{Sobo2},
 we have
\begin{align*}
&\big\|\langle w^+\rangle^{\mu}\langle w^-\rangle^{\nu} \langle\p_1\rangle^s  (\p_1 g  \p_1^ah)\big\|_{L^2(\BR)}
+\big\|\langle w^+\rangle^{\mu}\langle w^-\rangle^{\nu} \big[ \langle\p_1\rangle^s\p_1,g \big] \p_1^{a} h
\big\|_{L^2(\BR)} \\\nonumber
&\lesssim \| \langle w^+\rangle^{\mu_1}\langle w^-\rangle^{\nu_1} \langle \p_1\rangle^s\p_1^{\leq 1} g \|_{L^{\infty}}
 \| \langle w^+\rangle^{\mu_2}\langle w^-\rangle^{\nu_2}  \p_1^{a} h \|_{L^2}\\[-5mm]\nonumber\\\nonumber
&\quad + \| \langle w^+\rangle^{\mu_1}\langle w^-\rangle^{\nu_1}\p_1^{\leq 1} g \|_{L^{\infty}}
\| \langle w^+\rangle^{\mu_2}\langle w^-\rangle^{\nu_2} \langle \p_1\rangle^s \p_1^{a} h \|_{L^2} \\[-5mm]\nonumber\\\nonumber
&\lesssim \| \langle w^+\rangle^{\mu_1}\langle w^-\rangle^{\nu_1} \langle\p_1\rangle^s\p_1^{\leq a} g\|_{L^2}
          \| \langle w^+\rangle^{\mu_2}\langle w^-\rangle^{\nu_2} \langle\p_1\rangle^s\p_1^{\leq a} h\|_{L^2}.
\end{align*}
It remains to estimate the first term of \eqref{D72}. By standard calculation similar to the commutator estimate and by using the weighted commutator estimate  Lemma \ref{lemD6}, Lemma \ref{lemD8} and
the weighted Sobolev inequalities Lemma \ref{Sobo2}, there holds
\begin{align*}
&\big\|\langle w^+\rangle^{\mu}\langle w^-\rangle^{\nu} \langle\p_1\rangle^s \sum_{b+c=a,\ c\neq a}C_a^b\p_1^bg  \p_1^{c+1}h \big\|_{L^2(\BR)}\\
&\lesssim \| \langle w^+\rangle^{\mu_1}\langle w^-\rangle^{\nu_1} \langle \p_1\rangle^s\p_1^{\leq 1} g \|_{L^{\infty}}
 \| \langle w^+\rangle^{\mu_2}\langle w^-\rangle^{\nu_2}  \p_1^{a} h \|_{L^2} \\[-5mm]\nonumber\\\nonumber
&\quad+ \| \langle w^+\rangle^{\mu_1}\langle w^-\rangle^{\nu_1}\p_1^{\leq 1} g \|_{L^{\infty}}
\| \langle w^+\rangle^{\mu_2}\langle w^-\rangle^{\nu_2} \langle \p_1\rangle^s \p_1^{a} h \|_{L^2} \\[-5mm]\nonumber\\\nonumber
&\quad+ \| \langle w^+\rangle^{\mu_1}\langle w^-\rangle^{\nu_1} \p_1^{a} g \|_{L^2}
 \| \langle w^+\rangle^{\mu_2}\langle w^-\rangle^{\nu_2}\langle \p_1\rangle^s\p_1^{\leq 1} h \|_{L^{\infty}}\\[-5mm]\nonumber\\\nonumber
&\quad+ \| \langle w^+\rangle^{\mu_1}\langle w^-\rangle^{\nu_1} \langle \p_1\rangle^s \p_1^{a} g \|_{L^2}
\| \langle w^+\rangle^{\mu_2}\langle w^-\rangle^{\nu_2} \p_1^{\leq 1} h \|_{L^\infty} \\[-5mm]\nonumber\\\nonumber
&\quad +\| \langle w^+\rangle^{\mu_1}\langle w^-\rangle^{\nu_1} \langle \p_1\rangle^s\p_1^2 g \|_{L^2}
 \| \langle w^+\rangle^{\mu_2}\langle w^-\rangle^{\nu_2}  \p_1^{a-1} h \|_{L^\infty} +... \\[-5mm]\nonumber\\\nonumber
&\lesssim \| \langle w^+\rangle^{\mu_1}\langle w^-\rangle^{\nu_1} \langle\p_1\rangle^s\p_1^{\leq a} g\|_{L^2}
          \| \langle w^+\rangle^{\mu_2}\langle w^-\rangle^{\nu_2} \langle\p_1\rangle^s\p_1^{\leq a} h\|_{L^2}.
\end{align*}
This finishes the proof of the lemma.
\end{proof}

\subsection{Tangential derivative}
In this subsection, we introduce the tangential derivative $\p_\tau$ and
then present its elementary estimates.
For function $g$ defined in $\Om_f$ or $\hat{\Om}_f$, recalling that $\ud{g}$ denotes
the trace of $g$ on $\Ga_f.$ We expect
\begin{align*}
&\p_1 \ud{g}
= \ud{\p_1 g}+\p_1f \ud{\p_2 g}
= \ud{\p_\tau g}.
\end{align*}
Thus we introduce
\begin{equation*}
\begin{cases}
\Psi(t,x)= \zeta(x_2)\p_1f \quad \textrm{in}\,\, \Om_f,\\
\hat{\Psi}(t,x)= \zeta(x_2)\p_1f \quad \textrm{in}\,\, \hat{\Om}_f,
\end{cases}
\end{equation*}
where $\zeta(x_2)\in C^\infty(\BR)$ and
\begin{equation*}
\zeta(x_2)=
\begin{cases}
1 , &\text{ if } |x_2|\leq 1-\f{c_0}{2} ,\\
0 , &\text{ if } |x_2|\geq 1-\f{c_0}{2},
\end{cases}
\end{equation*}
where  $c_0$ is the constant in Theorem \ref{thm}.
The tangential derivative $\p_\tau$ in $\Om_f$ and $\hat{\Om}_f$ is
defined as follows:
\begin{equation}\label{tangD-}
\begin{cases}
\p_\tau=\Psi\p_2+\p_1 \quad \textrm{in}\,\, \Om_f,\\
\p_\tau=\hat{\Psi}\p_2+\p_1 \quad \textrm{in}\,\, \hat{\Om}_f.
\end{cases}
\end{equation}

\begin{remark}
The  tangential derivative defined by \eqref{tangD-}  will 
lose half a derivative compared with the  more natural one by \eqref{tangD}.
However, when applying the weight functions, \eqref{tangD}  will involve the commutator of the nonlocal operator and the weight functions, which will lead to technical complications.
 In the lower order energy estimate, the derivative loss is acceptable.
 Here, the simplified form \eqref{tangD-} is sufficient. In the highest order estimate, we need to use \eqref{tangD} to avoid the derivative loss problem to close the energy estimate.
\end{remark}

\begin{remark}
The cut-off function $\zeta$ is introduced so that the support of $\Psi$ and $\hat{\Psi}$ is concentrated near the free surface.
When $x$ is far from the free surface, particularly near the fixed boundaries $x_2=\pm1$, we have $\p_\tau=\p_1$.
\end{remark}


It is easy to verify that there holds the following estimate for $\Psi$ and $\hat{\Psi}$.
 \begin{lem}\label{lemPsi-s}
For any $k\in\bN,$ $\nu\in\bR$, there hold
\begin{align*}
\| \p^k\Psi\|_{L^2(\Om_f)}+\| \p^k\hat{\Psi}\|_{L^2(\hat{\Om}_f)}
&\lesssim \|\p_1f(t,\cdot) \|_{H^k(\BR)}, \\
\| \langle w^\pm\rangle^{\nu} (\p_t\pm\p_1)\p^k\Psi\|_{L^2(\Om_f)}
&\lesssim \|\langle w^\pm\rangle^{\nu} (\p_t\pm\p_1)\p_1f(t,\cdot) \|_{H^{k}(\BR)},\\
\| \langle w^\pm\rangle^{\nu} (\p_t\pm\p_1)\p^k\hat{\Psi}\|_{L^2(\hat{\Om}_f)}
&\lesssim \|\langle w^\pm\rangle^{\nu} (\p_t\pm\p_1)\p_1f(t,\cdot) \|_{H^{k}(\BR)}\, .
\end{align*}
\end{lem}
Next, we present a simple tangential derivative estimate.
\begin{lem}\label{lemTang}
For any $k\in\bZ^+,$ $\mu,\,\nu\in\bR$, there holds
\begin{align*}
&\|\langle w^+\rangle^{\mu}\langle w^-\rangle^{\nu}\p_\tau^k g\|_{L^2(\Om_f)}
\lesssim \sum_{1\leq a\leq k}\|\langle w^+\rangle^{\mu}\langle w^-\rangle^{\nu} \p^a g\|_{L^2(\Om_f)}
\mathcal{P}(1+\| \p_1 f\|_{H^2}),\quad 1\leq k\leq 3,\\
&\|\langle w^+\rangle^{\mu}\langle w^-\rangle^{\nu}\p_\tau^k g\|_{L^2(\Om_f)}
\lesssim \sum_{1\leq a\leq k}\|\langle w^+\rangle^{\mu}\langle w^-\rangle^{\nu} \p^a g\|_{L^2(\Om_f)}
\mathcal{P}(1+\| \p_1 f\|_{H^{k-1}}),\quad k\geq 3,\\
&\|\langle w^+\rangle^{\mu}\langle w^-\rangle^{\nu}\p_\tau^k g\|_{L^2(\hat{\Om}_f)}
\lesssim \sum_{1\leq a\leq k}\|\langle w^+\rangle^{\mu}\langle w^-\rangle^{\nu} \p^a g\|_{L^2(\hat{\Om}_f)}
\mathcal{P}(1+\| \p_1 f\|_{H^2}),\quad 1\leq k\leq 3,\\
&\|\langle w^+\rangle^{\mu}\langle w^-\rangle^{\nu}\p_\tau^k g\|_{L^2(\hat{\Om}_f)}
\lesssim \sum_{1\leq a\leq k}\|\langle w^+\rangle^{\mu}\langle w^-\rangle^{\nu} \p^a g\|_{L^2(\hat{\Om}_f)}
\mathcal{P}(1+\| \p_1 f\|_{H^{k-1}}),\quad k\geq 3,
\end{align*}
where $\mathcal{P}: \BR^+\rightarrow\BR^+$ is a polynomial.
\end{lem}
\begin{proof}
It suffices to show the above first two inequalities in the lemma. The proof of the remaining two inequalities is the same.

Through a simple calculation, we deduce that
\begin{align*}
\p_\tau&=\p_1 + \Psi\p_2,\\
\p_\tau^2&=(\p_1 + \Psi\p_2)(\p_1 + \Psi\p_2)=\p_1^2 +\p_1\Psi\p_2+2\Psi\p_{12}+\Psi\p_2\Psi\p_2
+\Psi^2\p_2^2,\\
\p_\tau^3&=(\p_1 + \Psi\p_2)(\p_1^2 +\p_1\Psi\p_2+2\Psi\p_{12}+\Psi\p_2\Psi\p_2
+\Psi^2\p_2^2\big).
\end{align*}
Hence
\begin{align*}
|\p_\tau g|&\leq(1+|\Psi|)|\p g|,\\
|\p_\tau^2 g|&\lesssim (1+|\Psi|)^2\big( |\p^2 g| + |\p \Psi \p g|\big),\\
|\p_\tau^3 g|&\lesssim (1+|\Psi|)^3\Big( |\p^3 g| + |\p \Psi \p^2 g|
+\big( |\p^{1\leq a\leq 2} \Psi|+|\p \Psi|^2\big) |\p g|\Big).
\end{align*}
Consequently,
\begin{align*}
\|\langle w^+\rangle^{\mu}\langle w^-\rangle^{\nu}\p_\tau g\|_{L^2(\Om_f)}&\leq(1+\|\Psi\|_{L^\infty} ) \|\langle w^+\rangle^{\mu}\langle w^-\rangle^{\nu}\p g\|_{L^2(\Om_f)},\\
\|\langle w^+\rangle^{\mu}\langle w^-\rangle^{\nu}\p_\tau^2 g\|_{L^2(\Om_f)}&\lesssim (1+\|\Psi\|_{L^\infty} )^2 \big(1+\|\p \Psi\|_{H^1}\big) \|\langle w^+\rangle^{\mu}\langle w^-\rangle^{\nu}\p g\|_{H^1(\Om_f)} ,\\
\|\langle w^+\rangle^{\mu}\langle w^-\rangle^{\nu}\p_\tau^3 g\|_{L^2(\Om_f)}&\lesssim (1+\|\Psi\|_{L^\infty} )^3 \big(1+\|\p \Psi\|_{H^1}\big)^2 \|\langle w^+\rangle^{\mu}\langle w^-\rangle^{\nu}\p g\|_{H^2(\Om_f)}.
\end{align*}
Then the result for $1\leq k\leq 3$ follows from Lemma \ref{Sobo1} and Lemma \ref{lemPsi-s}.
For $k\geq 4$, we use induction argument.
We assume the result holds for all positive integer no more than $k$.
Note
\begin{align*}
\p_\tau^{k+1} g=\p_\tau^k \p_1g +\p_\tau^k (\Psi\p_2g).
\end{align*}
Then we have
\begin{align*}
&\|\langle w^+\rangle^{\mu}\langle w^-\rangle^{\nu}\p_\tau^{k+1} g\|_{L^2(\Om_f)}\leq
\|\langle w^+\rangle^{\mu}\langle w^-\rangle^{\nu}\p_\tau^k \p_1g\|_{L^2(\Om_f)} +\|\langle w^+\rangle^{\mu}\langle w^-\rangle^{\nu}\p_\tau^k (\Psi\p_2g)\|_{L^2(\Om_f)}\\
&\lesssim \sum_{1\leq a\leq k}\big( \|\langle w^+\rangle^{\mu}\langle w^-\rangle^{\nu} \p^a \p_1g\|_{L^2(\Om_f)}+\|\langle w^+\rangle^{\mu}\langle w^-\rangle^{\nu} \p^a (\Psi\p_2g)\|_{L^2(\Om_f)}\big)
\mathcal{P}(1+\| \p_1 f\|_{H^{k-1}})\\
&\lesssim \sum_{1\leq a\leq k+1}  \|\langle w^+\rangle^{\mu}\langle w^-\rangle^{\nu} \p^a g\|_{L^2(\Om_f)} (1+ \| \p_1 f\|_{H^{k}})
\mathcal{P}(1+\| \p_1 f\|_{H^{k-1}}).
\end{align*}
This finishes the proof of the lemma.
\end{proof}

\subsection{Weighted chain rules}
Before the energy estimate, we present a weighted chain rule estimate.
\begin{lem}\label{lemG3}
Let  $g=g(x)$  be a function defined in a two dimensional domain $\Om\subseteq \bR^2$ with boundary  parameterized by $\p\Om=\{x_2=f(x_1)| x_1\in\bR \}$.
Let $\ud{g}(x_1)$ be the trace of $g(x)$ on the boundary $\p\Om$,
$k\geq 1$ be a positive integer, $0<s<1$, $\mu\in\BR,\, \nu\in\BR $.
For all $|a|\leq k$, there hold
\begin{align*}
&\big\| \langle w^+\rangle^{\mu}\langle w^-\rangle^{\nu} \p_1^a \ud{g}
\big\|_{L^2(\BR)}   \\
&\lesssim
\begin{cases}
\sum_{|\alpha|\leq k}\big\| \langle w^+\rangle^{\mu}\langle w^-\rangle^{\nu}  \ud{\p^\alpha g}\big\|_{L^2(\BR)}
\cdot  \big(1+ \|\p_1f\|_{L^\infty(\BR)} \big)\quad k=1,\\
\sum_{|\alpha|\leq k}\big\| \langle w^+\rangle^{\mu}\langle w^-\rangle^{\nu}  \ud{\p^\alpha g}\big\|_{L^2(\BR)}
\cdot  \mathcal{P}\big(1+ \|\p_1f\|_{H^{k-1}(\BR)} \big)\quad k\geq 2,\\
\end{cases}
\\
&\big\| \langle w^+\rangle^{\mu}\langle w^-\rangle^{\nu} \langle\p_1\rangle^{s} \p_1^a \ud{g}\big\|_{L^2(\BR)}   \\
&\lesssim
\begin{cases}
\sum_{|\alpha|\leq k}\big\| \langle w^+\rangle^{\mu}\langle w^-\rangle^{\nu} \langle\p_1\rangle^{s}  \ud{\p^\alpha g}\big\|_{L^2(\BR)}
\cdot \mathcal{P}\big(1+ \|\p_1f\|_{H^s(\BR)}+\|\p_1f\|_{L^\infty(\BR)} \big)\quad k=1,\\
\sum_{|\alpha|\leq k}\big\| \langle w^+\rangle^{\mu}\langle w^-\rangle^{\nu} \langle\p_1\rangle^{s}  \ud{\p^\alpha g}\big\|_{L^2(\BR)}
\cdot \mathcal{P}\big(1+ \|\p_1f\|_{ H^{k+s-1}(\BR)} \big)\quad k\geq 2,
\end{cases}
\end{align*}
where $\mathcal{P}: \BR^+\rightarrow\BR^+$ is a polynomial.
\end{lem}
\begin{proof}
The proof of first and the second inequality is the same.
In the sequel, we only present the details for the first one.

The proof is conducted  by induction argument.
The case of $k=1,\, 2$ is an easy consequence of chain rules.
Assuming for $k\geq 2$, there holds
\begin{align}\label{G0}
&\sum_{|a|\leq k}\big\| \langle w^+\rangle^{\mu}\langle w^-\rangle^{\nu} \p_1^a \ud{g}
\big\|_{L^2(\BR)} \\\nonumber
&\lesssim \sum_{|\alpha|\leq k}
\big\| \langle w^+\rangle^{\mu}\langle w^-\rangle^{\nu} \ud{\p^\alpha g}\big\|_{L^2(\BR)}
\cdot \mathcal{P}\big(1+ \|\p_1 f\|_{H^{k-1}(\BR)}\big)  .
\end{align}
Next, we calculate
\begin{align*}
&\p_1^{k+1} \ud{g}=\p_1^k
  \big( \ud{\p_1 g}+ \ud{\p_2 g}\p_1 f\big) ,\\\nonumber
&=\p_1^k\ud{\p_1 g}+
\sum_{a+b=k}C_k^a \p_1^a \ud{\p_2 g}\p_1^{b+1} f ,
\end{align*}
where $ C_k^a=\f{k!}{a!(k-a)!}$ is the binomial coefficient.
Then by the H\"older inequality, the Sobolev inequality and \eqref{G0}, we have
\begin{align*}
& \big\| \langle w^+\rangle^{\mu}\langle w^-\rangle^{\nu} \p_1^{k+1} \ud{g} \big\|_{L^2(\BR)} \\
&\lesssim \big\| \langle w^+\rangle^{\mu}\langle w^-\rangle^{\nu} \p_1^k\ud{\p_1 g} \big\|_{L^2(\BR)}
+\sum_{a+b=k} \big\| \langle w^+\rangle^{\mu}\langle w^-\rangle^{\nu} \p_1^a \ud{\p_2 g}\p_1^{b+1} f \big\|_{L^2(\BR)} \\
&\lesssim \big\| \langle w^+\rangle^{\mu}\langle w^-\rangle^{\nu} \p_1^k\ud{\p_1 g} \big\|_{L^2(\BR)}
+\sum_{|a|\leq k}\big\| \langle w^+\rangle^{\mu}\langle w^-\rangle^{\nu} \p_1^a \ud{\p_2 g}\|_{L^2(\BR)}
\cdot   \|\p_1 f\|_{H^{k-1}(\BR)} \\
&\lesssim
\sum_{|\alpha|\leq k+1}\big\| \langle w^+\rangle^{\mu}\langle w^-\rangle^{\nu}  \ud{\p^\alpha g}\big\|_{L^2(\BR)}
(1+\|f\|_{\dot{H}^1(\BR)}+\|f\|_{\dot{H}^{k+1}(\BR)}) \\
&\quad \cdot \mathcal{P}\big(1+\|f\|_{\dot{H}^1(\BR)}+\|f\|_{\dot{H}^{k+1}(\BR)}\big) .
\end{align*}
Thus the case for $k+1$ is proved.
This finishes the proof of the lemma.
\end{proof}

\subsection{Weighted div-curl theorem}

We will use weighted div-curl theorem to treat the energy estimate.

We first recall a div-curl estimate.
\label{lemDC0}\begin{lem}
Let $s\geq 3$ be an integer, $\Om\subseteq\bR^2$
with $\p\Om\in H^{s+\f12}$.
Assume that $v\in L^2(\Om)$, $\div v\in H^{s-1}(\Om)$, $\curl v\in H^{s-1}(\Om)$,
$\ud{v}\cdot n\in H^{s-\f12}(\bR)$.
Then for integer $\sigma\in[1,s]$,
 there holds
\begin{align*}
&\| \nabla^\sigma v \|_{L^2(\Om)} \leq C\Big( \| \div \nabla^{\leq \sigma-1}v \|_{L^2(\Om)} + \| \curl \nabla^{\leq \sigma-1} v \|_{L^2(\Om)} +\| v \|_{L^2(\Om)}
+  \| v\cdot n\|_{H^{\sigma-\f12}(\p\Om)}  \Big),
\end{align*}
where the constant $C$ depends on $\|\p\Om\|_{H^{s+\f12}}$.
\end{lem}
The statement of the above theorem differs slightly from the div-curl theorem (see \cite{Lannes, CS} for example).
Based on the assumption and the div-curl theorem, we have $v\in H^s(\Om)$.
Note that we have sufficient regularity assumption on $v$, and therefore, the theorem follows from the
Lax-Milgram theorem and regularity theory for elliptic equation. For detailed information, we refer to
\cite{Lannes, CS}, for example.

Now we show the weighted version of the div-curl estimate.
\begin{lem}\label{lemDCw1}
Let $s\geq 3$ be an integer, $\mu,\nu\in\BR$,
 $\Om\subseteq\bR^2$
  with $\p\Om\in H^{s+\f12}$.
  Assume that $v\in L^2(\Om)$, $\div v\in H^{s-1}(\Om)$, $\curl v\in H^{s-1}(\Om)$,
$\ud{v}\cdot n\in H^{s-\f12}(\bR)$.
Then for integer $\sigma\in[1,s]$, there holds
\begin{align*}
&\sum_{|\alpha|\leq \sigma}\| \langle w^+\rangle^{\mu}\langle w^-\rangle^{\nu} \nabla^\alpha v \|_{L^2(\Om)}\\
&\leq C\sum_{|\beta|\leq \sigma-1}\big( \| \langle w^+\rangle^{\mu}\langle w^-\rangle^{\nu} \div \p^\beta v \|_{L^2(\Om)} + \| \langle w^+\rangle^{\mu}\langle w^-\rangle^{\nu} \curl \p^\beta v \|_{L^2(\Om)}\big) \\
&\quad +C
\big(\| \langle w^+\rangle^{\mu}\langle w^-\rangle^{\nu}  v \|_{L^2(\Om)}
+\|\langle w^+\rangle^{\mu}\langle w^-\rangle^{\nu} v \cdot n \|_{H^{\sigma-\f12}(\p\Om)} \big),
\end{align*}
where the constant $C$ depends on $\|\p\Om\|_{H^{s+\f12}}$.
\end{lem}
\begin{proof}
Firstly, we have
\begin{align*}
\| \langle w^+\rangle^{\mu}\langle w^-\rangle^{\nu}\nabla^\sigma v \|_{L^2(\Om)}
&\leq \| \nabla (\langle w^+\rangle^{\mu}\langle w^-\rangle^{\nu}\nabla^{\sigma-1} v) \|_{L^2(\Om)} +
\| \nabla (\langle w^+\rangle^{\mu}\langle w^-\rangle^{\nu}) \nabla^{\sigma-1} v \|_{L^2(\Om)}\\
&\leq \| \nabla^\sigma (\langle w^+\rangle^{\mu}\langle w^-\rangle^{\nu} v) \|_{L^2(\Om)} +
C \| \langle w^+\rangle^{\mu}\langle w^-\rangle^{\nu}\nabla^{\leq \sigma-1} v \|_{L^2(\Om)}.
\end{align*}
Consequently, by Lemma \ref{lemDC0}, we have
\begin{align*}
&\| \nabla^\sigma (\langle w^+\rangle^{\mu}\langle w^-\rangle^{\nu} v) \|_{L^2(\Om)} \\
&\lesssim   \| \div \nabla^{\leq \sigma-1}(\langle w^+\rangle^{\mu}\langle w^-\rangle^{\nu} v) \|_{L^2(\Om)}
+ \| \curl \nabla^{\leq \sigma-1} (\langle w^+\rangle^{\mu}\langle w^-\rangle^{\nu} v) \|_{L^2(\Om)} \\
&\quad+\|  \langle w^+\rangle^{\mu}\langle w^-\rangle^{\nu} v  \|_{L^2(\Om)}
+\| \langle w^+\rangle^{\mu}\langle w^-\rangle^{\nu} v \cdot n \|_{H^{\sigma-\f12}(\p\Om)}\\
&\lesssim \| \langle w^+\rangle^{\mu}\langle w^-\rangle^{\nu} \div \nabla^{\leq \sigma-1} v \|_{L^2(\Om)}
+ \| \langle w^+\rangle^{\mu}\langle w^-\rangle^{\nu} \curl \nabla^{\leq \sigma-1}  v \|_{L^2(\Om)}\\
&\quad+ \| \langle w^+\rangle^{\mu}\langle w^-\rangle^{\nu} \nabla^{\leq \sigma-1} v \|_{L^2(\Om)}
+\|\langle w^+\rangle^{\mu}\langle w^-\rangle^{\nu} v \cdot n \|_{H^{\sigma-\f12}(\p\Om)}.
\end{align*}
Then the lemma follows from an induction argument.
\end{proof}


To treat the lower order weighted energy estimate, we need another version of the div-curl lemma
which involves the tangential derivative estimate of the unknowns inside the domain.

\begin{lem}\label{lemDC1}
Let $s\geq 3$ be an integer, $\Om\subseteq\bR^2$
with $\p\Om\in H^{s+\f12}$.
Assume that $v\in L^2(\Om)$, $\div v\in H^{s-1}(\Om)$, $\curl v\in H^{s-1}(\Om)$,
$\p_\tau^{\leq s} v\in L^2(\Om)$.
Then for integer $\sigma\in[1,s]$,
 there holds
\begin{align*}
&\| \nabla^\sigma v \|_{L^2(\Om)} \leq C(\|\p\Om\|_{H^{s+\f12}(\p\Om)})\big(\| \nabla^{\leq \sigma-1} \div v \|_{L^2(\Om)} +
 \| \nabla^{\leq \sigma-1} \curl v \|_{L^2(\Om)} +\|\p_\tau^{\leq \sigma} v\|_{L^2(\Om)} \big).
\end{align*}
\end{lem}
\begin{proof}
The lemma follows the following two inequalities. See cf. \cite{CS, GW-23}.
\begin{align*}
&\| v \|_{H^\sigma(\Om)} \leq \| \div v \|_{H^{\sigma-1}(\Om)} + \| \curl v \|_{H^{\sigma-1}(\Om)}
+\|\p_\tau^\sigma v\cdot n\|_{H^{-\f12}(\p\Om)},\\
&\|\p_\tau v\cdot n\|_{H^{-\f12}(\p\Om)} \leq \|\p_\tau v \|_{L^2(\Om)} + \| \div v \|_{L^2(\Om)}.
\end{align*}
\end{proof}
Now we show another weighted div-curl estimate.
\begin{lem}\label{lemDCw2}
Let $s\geq 3$ be an integer, $\mu,\,\nu\in\BR$, $\Om\subseteq\bR^2$
with $\p\Om\in H^{s+\f12}$.
Assume that $v\in L^2(\Om)$, $\div v\in H^{s-1}(\Om)$, $\curl v\in H^{s-1}(\Om)$,
$\p_\tau^{\leq s} v\in L^2(\Om)$.
Then for integer $\sigma\in[1,s]$, there holds
\begin{align*}
\| \langle w^+\rangle^{\mu}\langle w^-\rangle^{\nu} \nabla^\sigma v \|_{L^2(\Om)}
\leq &C\big( \| \langle w^+\rangle^{\mu}\langle w^-\rangle^{\nu} \div \nabla^{\leq \sigma-1} v \|_{L^2(\Om)} +
 \| \langle w^+\rangle^{\mu}\langle w^-\rangle^{\nu} \curl \nabla^{\leq \sigma-1} v \|_{L^2(\Om)} \\
&\quad+\| \langle w^+\rangle^{\mu}\langle w^-\rangle^{\nu} \p_\tau^{\leq \sigma} v \|_{L^2(\Om)}   \big),
\end{align*}
where the constant $C$ depends on $\|\p\Om\|_{H^{s+\f12}}$.
\end{lem}
\begin{proof}
Firstly, we have
\begin{align*}
\| \langle w^+\rangle^{\mu}\langle w^-\rangle^{\nu}\nabla^\sigma v \|_{L^2(\Om)}
&\leq \| \nabla (\langle w^+\rangle^{\mu}\langle w^-\rangle^{\nu}\nabla^{\sigma-1} v) \|_{L^2(\Om)}
 +\| \nabla \langle w^+\rangle^{\mu}\langle w^-\rangle^{\nu}\nabla^{\sigma-1} v \|_{L^2(\Om)}\\
&\leq \| \nabla^\sigma (\langle w^+\rangle^{\mu}\langle w^-\rangle^{\nu} v) \|_{L^2(\Om)} +
C \| \langle w^+\rangle^{\mu}\langle w^-\rangle^{\nu}\nabla^{\leq \sigma-1} v \|_{L^2(\Om)}.
\end{align*}
By using Lemma \ref{lemDC1}, we have
\begin{align*}
&\| \nabla^\sigma (\langle w^+\rangle^{\mu}\langle w^-\rangle^{\nu} v) \|_{L^2(\Om)} \\
&\lesssim   \| \div \nabla^{\leq \sigma-1}(\langle w^+\rangle^{\mu}\langle w^-\rangle^{\nu} v) \|_{L^2(\Om)}
+ \| \curl \nabla^{\leq \sigma-1} (\langle w^+\rangle^{\mu}\langle w^-\rangle^{\nu} v) \|_{L^2(\Om)}\\
&\quad+\|\p_\tau^{\leq \sigma} (\langle w^+\rangle^{\mu}\langle w^-\rangle^{\nu} v) \|_{L^2(\Om)}\\
& \leq   \| \langle w^+\rangle^{\mu}\langle w^-\rangle^{\nu} \div \nabla^{\leq \sigma-1} v \|_{L^2(\Om)}
+ \| \langle w^+\rangle^{\mu}\langle w^-\rangle^{\nu} \curl \nabla^{\leq \sigma-1} v \|_{L^2(\Om)}
+\|\langle w^+\rangle^{\mu}\langle w^-\rangle^{\nu} \p_\tau^{\leq \sigma} v \|_{L^2(\Om)}\\
&\quad +C \| \langle w^+\rangle^{\mu}\langle w^-\rangle^{\nu}\nabla^{\leq \sigma-1} v \|_{L^2(\Om)}.
\end{align*}
Then the lemma follows from an induction argument.
\end{proof}
%

\section{Weighted energy estimate of the pressure}




In this section, we will present the weighted estimate of the pressure.
In the weighted gradient estimate of the pressure,
a direct energy method
leads to the estimate of
pressure itself with weights.
Fortunately, we discover an integral identity for the pressure, which
enables the application of the Poincar\'e inequality
to estimate the pressure itself with weights.

As for the weighted higher order $(2\leq k \leq s)$ estimate of the pressure, they consist of the weighted tangential derivative estimate
and the other normal and mixed derivative estimate. The former can be carried out by the weighted energy method, while the latter can be estimated by the elliptic equation of the pressure. The commutator induced by the weight functions can be treated by using the achieved weighted lower order estimate of the pressure.
In these estimate, especially in the highest order $k=s$, we need to balance the structure of the derivatives in order not to lose derivatives.

The elliptic equation and the boundary conditions for the pressure \eqref{B22} are copied  as follows
\begin{equation}\label{H0}
\begin{cases}
\Delta p =-\nabla \cdot (\La_{-} \cdot \nabla \La_{+}) \quad \textrm{in}\,\, \Om_f,\\
\Delta \hat{p} =-\nabla \cdot (\hat{\La}_{-} \cdot \nabla \hat{\La}_{+})\quad \textrm{in}\,\, \hat{\Om}_f,\\
p =\hat{p}\quad \textrm{on} \,\, \Gamma_f,\\
N_f\cdot\nabla p-N_f\cdot\nabla \hat{p}= -\big(\ud{\La_{+}^1} \ud{\La_{-}^1}- \ud{\hat{\La}_{+}^1} \ud{\hat{\La}_{-}^1} \big) \p_1^2 f \\
\qquad-\big(\ud{\La_{+}^1}-\ud{\hat{\La}_{+}^1} \big) \p_1 (\p_t-  \p_1) f
  -\big(\ud{\La_{-}^1}-\ud{\hat{\La}_{-}^1}\big) \p_1 (\p_t+  \p_1) f
\quad\text{on}\quad\Ga_f ,\\
\p_2 p=0\,\,   \textrm{on} \,\,  \Ga,
\quad \p_2 \hat{p}=0\,\,   \textrm{on} \,\,  \hat{\Ga}.
\end{cases}
\end{equation}
\begin{rem}
For the solvability of the elliptic system \eqref{H0} for the pressure.
The readers may refer to the method in \cite{SZ2}. Firstly, we can determine the value of the pressure on the free boundary (see Section two in \cite{SZ2}). Then we can solve the elliptic boundary value for $p$ in $\Om_f$ and $\hat{p}$ in $\hat{\Om}_f$ in appropriate spaces, respectively.
Having secured the existence of solutions for the pressure,
our primary focus in this paper is
to prove an a priori estimate for the pressure uniformly in time.
\end{rem}
We will prove the following weighted estimate for the pressure.
\begin{lem}\label{lemPre}
Let $s\geq 4$ be an integer, $1/2< \mu \leq 3/5$,
$\delta>0$ be a sufficiently small number, $\| \p_1f \|_{H^{s-\f12}(\BR)} \leq \delta.$
For the pressure satisfying \eqref{H0}, there hold
\begin{align*}
&\sum_{|\alpha|\leq 3 }\| \langle w^\pm\rangle^{5\mu}
\langle w^\mp\rangle^{\mu}  \p^\alpha p\|_{L^2(\Om_f)}^2
+\sum_{|\alpha|\leq 3 }\| \langle w^\pm\rangle^{5\mu}
\langle w^\mp\rangle^{\mu}  \p^\alpha \hat{p}\|_{L^2(\hat{\Om}_f)}^2\\
&\lesssim
E_3 G_3^b
\cdot \mathcal{P}(1+ \|\p_1f\|_{H^{\f52}(\BR)}),\\
&\sum_{|\alpha|\leq s }\| \langle w^\pm\rangle^{2\mu}
\langle w^\mp\rangle^{\mu}  \p^\alpha p\|_{L^2(\Om_f)}^2
+\sum_{|\alpha|\leq s }\| \langle w^\pm\rangle^{2\mu}
\langle w^\mp\rangle^{\mu}  \p^\alpha \hat{p}\|_{L^2(\hat{\Om}_f)}^2\\
&\lesssim
E_s G_s
\cdot \mathcal{P}(1+ \| \p_1f \|_{H^{s-\f12}(\BR)}),
\end{align*}
where $\mathcal{P}(\cdot)$ is a polynomial.
\end{lem}

\subsection{An integral identity for the pressure}
%
In this subsection, we show an integral identity for the pressure.
\begin{lem}
Assume $f,$ $\La_\pm$, $\hat{\La}_\pm$, $p$, $\hat{p}$ satisfies \eqref{A18}-\eqref{A20},  $\lim_{x_1\rightarrow +\infty}(\nabla^{\leq 1}p,\nabla^{\leq 1}\hat{p},\nabla^{\leq 1}\La_{\pm},\nabla^{\leq 1}\hat{\La}_{\pm})=0.$
Then there holds
\begin{align}\label{Pre-Ide}
	\int_{-1}^f p\, \dx_2 +\int_f^1 \hat{p}\, \dx_2
 + \int_{-1}^f \La_{-}^1 \La_{+}^1\, \dx_2+\int_f^1 \hat{\La}_{-}^1 \hat{\La}_{+}^1\, \dx_2=0\,.
\end{align}
\end{lem}
\begin{proof}
Let $-\infty<x_1<\bar{x}_1<+\infty$. For $\Delta p$, taking the integral in $\Om_f$ from $x_1$ to $\bar{x}_1$. Then let $\bar{x}_1$
go to $+\infty$, due to the vanishing of $(\nabla^{\leq 1}p,\nabla^{\leq 1}\La_{\pm})$ at positive infinity and the boundary condition
$\p_2 p\big|_{\Ga}=0$, we obtain
\begin{align*}
\int_{x_1}^{+\infty}\int_{-1}^f\Delta p \dx_2\dx_1
=\int_{x_1}^{+\infty} N_f\cdot\ud{\nabla p} \,\dx_1
- \int_{-1}^f \p_1 p\, \dx_2.
\end{align*}
On the other hand, by using $\Delta p = -\nabla \cdot (\La_- \cdot \nabla \La_+) $, we deduce that
\begin{align*}
&\int_{x_1}^{+\infty}\int_{-1}^f\Delta p \dx_2\dx_1
=-\int_{x_1}^{+\infty}\int_{-1}^f\nabla \cdot (\La_{-} \cdot \nabla \La_{+}) \dx_2\dx_1  \\
&=\int_{x_1}^{+\infty} \La_{-} \cdot \nabla \La_{+}^2 \big|_{\Ga} \,\dx_1
-\int_{x_1}^{+\infty} N_f\cdot(\ud{\La_{-} \cdot \nabla \La_{+}}) \,\dx_1  + \int_{-1}^f (\La_{-} \cdot \nabla \La_{+}^1)\, \dx_2\,.
\end{align*}
By the boundary condition  $\La_\pm^2\big|_{\Ga}=0$, this further yields
\begin{align}\label{PP3}
&\int_{x_1}^{+\infty} \ud{N_f\cdot\nabla p} \,\dx_1
- \int_{-1}^f \p_1 p \, \dx_2\\\nonumber
&=-\int_{x_1}^{+\infty} N_f\cdot(\ud{\La_{-} \cdot \nabla \La_{+}}) \,\dx_1
+ \int_{-1}^f (\La_{-} \cdot \nabla \La_{+}^1)\, \dx_2 \, .
\end{align}

Now we calculate the quantities in \eqref{PP3}.
Note that
\begin{align*}
&N_f\cdot\ud{\p_2\La_{+}}
=-\p_1f\ud{\p_2\La_{+}^1}+\ud{\p_2\La_{+}^2}
=-\p_1f\ud{\p_2\La_{+}^1}-\ud{\p_1\La_{+}^1}
=-\p_1\ud{\La_{+}^1}.
\end{align*}
Thus by using $\eqref{A18}_1$, we deduce that
\begin{align}\label{PP5}
& N_f\cdot(\ud{\La_{-} \cdot \nabla \La_{+}}+\ud{\nabla p})
=- N_f\cdot\ud{(\p_t-\p_1) \La_+}  \\\nonumber
&=- N_f\cdot \big( (\p_t-\p_1) \ud{\La_+}- \ud{\p_2 \La_+} (\p_t-\p_1)f \big)\\\nonumber
&=- (\p_t-\p_1) (N_f\cdot  \ud{\La_+})- \ud{\La_+^1} (\p_t-\p_1)\p_1f
-\p_1\ud{\La_{+}^1} (\p_t-\p_1)f\\\nonumber
&=- (\p_t-\p_1) (N_f\cdot  \ud{\La_+})- \p_1\big( \ud{\La_+^1} (\p_t-\p_1)f \big)\,.
\end{align}
Next, let us calculate
\begin{align}
\label{PP7}
&\int_{-1}^f \p_1 p \, \dx_2=\f{\d}{\dx_1} \int_{-1}^f p\, \dx_2 -\p_1f\ud{p},\\
\label{PP8}
&\int_{-1}^f (\La_{-} \cdot \nabla \La_{+}^1)\, \dx_2
=\int_{-1}^f \div (\La_{-} \La_{+}^1)\, \dx_2\\\nonumber
&=\int_{-1}^f \p_1 (\La_{-}^1 \La_{+}^1)\, \dx_2
+ \ud{\La_{+}^1}\ud{\La_{-}^2}\\\nonumber
&=\f{\d}{\dx_1}\int_{-1}^f (\La_{-}^1 \La_{+}^1)\, \dx_2
-\p_1f\ud{\La_{-}^1 \La_{+}^1}+ \ud{\La_{+}^1}\ud{\La_{-}^2}\\\nonumber
&=\f{\d}{\dx_1}\int_{-1}^f (\La_{-}^1 \La_{+}^1)\, \dx_2
+\ud{\La_{+}^1}N_f\cdot \ud{\La_{-}}.
\end{align}

Plugging \eqref{PP5}-\eqref{PP8} into \eqref{PP3}, by using the boundary condition $(\p_t-\p_1)f=N_f\cdot \ud{\La_{-}}$, we derive that
\begin{align}\label{PP9}
&\int_{x_1}^{+\infty} N_f\cdot(\ud{\La_{-} \cdot \nabla \La_{+}}+\ud{\nabla p}) \,\dx_1
- \int_{-1}^f (\La_{-} \cdot \nabla \La_{+}^1)\, \dx_2 \\\nonumber
&=-\int_{x_1}^{+\infty} (\p_t-\p_1) (N_f\cdot  \ud{\La_+}) \,\dx_1
+\ud{\La_+^1} (\p_t-\p_1)f -\f{\d}{\dx_1}\int_{-1}^f (\La_{-}^1 \La_{+}^1)\, \dx_2
-\ud{\La_{+}^1}N_f\cdot \ud{\La_{-}} \\\nonumber
&=-\int_{x_1}^{+\infty} (\p_t-\p_1) (N_f\cdot  \ud{\La_+}) \,\dx_1
 -\f{\d}{\dx_1}\int_{-1}^f (\La_{-}^1 \La_{+}^1)\, \dx_2 \\\nonumber
 &=\int_{-1}^f \p_1 p \, \dx_2=\f{\d}{\dx_1} \int_{-1}^f p\, \dx_2 -\p_1f\ud{p}.
\end{align}
The above \eqref{PP9} is further simplified as
\begin{align}\label{PP10}
&\f{\d}{\dx_1} \int_{-1}^f p\, \dx_2 -\p_1f\ud{p}+\f{\d}{\dx_1}\int_{-1}^f \La_{-}^1 \La_{+}^1\, \dx_2
+\int_{x_1}^{+\infty}  (\p_t-\p_1) (N_f\cdot  \ud{\La_+})  \dx_1=0\,.
\end{align}

\textbf{Similar calculation in $\hat{\Om}_f$}

Now we conduct the similar calculation in $\hat{\Om}_f$:
\begin{align}\label{PP11}
&-\int_{x_1}^{+\infty} N_f\cdot\ud{\nabla \hat{p}}   \,\dx_1
- \int_f^1 \p_1 \hat{p}\, \dx_2\\\nonumber
&=\int_{x_1}^{+\infty} N_f\cdot(\ud{\hat{\La}_{-} \cdot \nabla \hat{\La}_{+}})  \,\dx_1
+ \int_f^1(\hat{\La}_{-} \cdot \nabla \hat{\La}_{+}^1)\, \dx_2 \, ,
\end{align}
Similarly,
\begin{align}
&N_f\cdot\ud{\p_2\hat{\La}_{+}}
=-\p_1\ud{\hat{\La}_{+}^1},\nonumber\\
\label{PP12}
& N_f\cdot(\ud{\hat{\La}_{-} \cdot \nabla \hat{\La}_{+}}+\ud{\nabla \hat{p}})
=- (\p_t-\p_1) (N_f\cdot  \ud{\hat{\La}_+})- \p_1\big( \ud{\hat{\La}_+^1} (\p_t-\p_1)f \big),\\
\label{PP13}
&\int_f^1 \p_1 \hat{p} \, \dx_2=\f{\d}{\dx_1} \int_f^1  \hat{p} \, \dx_2 +\p_1f\ud{\hat{p}},\\
\label{PP14}
&\int_f^1 (\hat{\La}_{-} \cdot \nabla \hat{\La}_{+}^1)\, \dx_2
=\f{\d}{\dx_1}\int_f^1 (\hat{\La}_{-}^1 \hat{\La}_{+}^1)\, \dx_2
-\ud{\hat{\La}_{+}^1}N_f\cdot \ud{\hat{\La}_{-}}.
\end{align}
Similar to \eqref{PP9}, plugging \eqref{PP12}-\eqref{PP14} into \eqref{PP11}, by using the boundary condition $(\p_t-\p_1)f=N_f\cdot \ud{\hat{\La}_{-}}$, we derive that
\begin{align}\label{PP15}
&\int_{x_1}^{+\infty} N_f\cdot(\ud{\hat{\La}_{-} \cdot \nabla \hat{\La}_{+}}+\ud{\nabla \hat{p}}) \dx_1
+ \int_f^1(\hat{\La}_{-} \cdot \nabla \hat{\La}_{+}^1)\, \dx_2\\\nonumber
 &=-\int_{x_1}^{+\infty} (\p_t-\p_1) (N_f\cdot  \ud{\hat{\La}_+}) \,\dx_1
 +\f{\d}{\dx_1}\int_{f}^1 (\hat{\La}_{-}^1 \hat{\La}_{+}^1)\, \dx_2 \\\nonumber
&=-\int_f^1 \p_1 \hat{p} \, \dx_2=-\f{\d}{\dx_1} \int_f^1 \hat{p}\, \dx_2 -\p_1f\ud{\hat{p}}.
\end{align}
The above \eqref{PP15} is further simplified as
\begin{align}\label{PP16}
&\f{\d}{\dx_1} \int_f^1 \hat{p}\, \dx_2 +\p_1f\ud{p}
+\f{\d}{\dx_1}\int_f^1 \hat{\La}_{-}^1 \hat{\La}_{+}^1\, \dx_2
-\int_{x_1}^{+\infty} (\p_t-\p_1) (N_f\cdot  \ud{\hat{\La}_+}) \,\dx_1=0\,.
\end{align}
Taking the sum of \eqref{PP10} and \eqref{PP16}, note that $(\p_t+\p_1)f=N_f\cdot  \ud{\La_+}=N_f\cdot  \ud{\hat{\La}_+}$, this further yields
\begin{align*}
&\f{\d}{\dx_1} \Big(\int_{-1}^f p\, \dx_2 +\int_f^1 \hat{p}\, \dx_2
 + \int_{-1}^f \La_{-}^1 \La_{+}^1\, \dx_2+\int_f^1 \hat{\La}_{-}^1 \hat{\La}_{+}^1\, \dx_2
\Big) =0\,.
\end{align*}
Due to the vanishing of $(p,\hat{p}, \La_{\pm},\hat{\La}_{\pm})$ at positive infinity, there holds
\begin{align*}
\lim_{x_1\rightarrow+\infty}\Big(\int_{-1}^f p\, \dx_2 +\int_f^1 \hat{p}\, \dx_2
 + \int_{-1}^f \La_{-}^1 \La_{+}^1\, \dx_2+\int_f^1 \hat{\La}_{-}^1 \hat{\La}_{+}^1\, \dx_2
\Big)=0\,.
\end{align*}
Hence for all $x_1\in\bR$, there holds
\begin{align*}
\int_{-1}^f p\, \dx_2 +\int_f^1 \hat{p}\, \dx_2
 + \int_{-1}^f \La_{-}^1 \La_{+}^1\, \dx_2+\int_f^1 \hat{\La}_{-}^1 \hat{\La}_{+}^1\, \dx_2
=0\,.
\end{align*}
\end{proof}

With the integral identity for the pressure \eqref{Pre-Ide}, we can show the weighted Poincar\'e inequality.
\begin{lem}\label{lemPI}
Let $p\in C(\bar{\Om}_f)$, $\hat{p}\in C(\bar{\hat{\Om}}_f)$, $\ud{p}=\ud{\hat{p}}$, and it satisfies
\begin{align*}
&  \int_{-1}^f p\, \dx_2 +\int_f^1 \hat{p}\, \dx_2
 + \int_{-1}^f \La_{-}^1 \La_{+}^1\, \dx_2+\int_f^1 \hat{\La}_{-}^1 \hat{\La}_{+}^1\, \dx_2=0\,.
\end{align*}
Then there hold
\begin{align}\label{G2}
\| \langle w^+\rangle^{\mu}\langle w^-\rangle^{\nu}\ud{p} \|_{L^2(\bR)}
\lesssim& \| \langle w^+\rangle^{\mu}\langle w^-\rangle^{\nu}\p_2 p \|_{L^2(\Om_f)}
+\| \langle w^+\rangle^{\mu}\langle w^-\rangle^{\nu}\p_2 \hat{p} \|_{L^2(\hat{\Om}_f)}\\\nonumber
& +\|\langle w^+\rangle^{\mu_1}\langle w^-\rangle^{\nu_1}\La_{-}^1\|_{L^2(\Om_f)}
\|\langle w^+\rangle^{\mu_2}\langle w^-\rangle^{\nu_2}\La_{+}^1\|_{L^2(\Om_f)}\\\nonumber
&+\| \langle w^+\rangle^{\mu_1}\langle w^-\rangle^{\nu_1}\hat{\La}_{-}^1\|_{L^2{(\hat{\Om}_f)}}
\|\langle w^+\rangle^{\mu_2}\langle w^-\rangle^{\nu_2}\hat{\La}_{+}^1\|_{L^2{(\hat{\Om}_f)}}
\end{align}
and
\begin{align}\label{G3}
&\| \langle w^+\rangle^{\mu}\langle w^-\rangle^{\nu} p \|_{L^2(\Om_f)}
+\| \langle w^+\rangle^{\mu}\langle w^-\rangle^{\nu} \hat{p} \|_{L^2(\hat{\Om}_f)} \\\nonumber
&\lesssim \| \langle w^+\rangle^{\mu}\langle w^-\rangle^{\nu}\p_2 p \|_{L^2(\Om_f)}
+\| \langle w^+\rangle^{\mu}\langle w^-\rangle^{\nu}\p_2 \hat{p} \|_{L^2(\hat{\Om}_f)}\\\nonumber
&\quad +\|\langle w^+\rangle^{\mu_1}\langle w^-\rangle^{\nu_1}\La_{-}^1\|_{L^2(\Om_f)}
\|\langle w^+\rangle^{\mu_2}\langle w^-\rangle^{\nu_2}\La_{+}^1\|_{L^2(\Om_f)}\\\nonumber
&\quad+\| \langle w^+\rangle^{\mu_1}\langle w^-\rangle^{\nu_1}\hat{\La}_{-}^1\|_{L^2{(\hat{\Om}_f)}}
\|\langle w^+\rangle^{\mu_2}\langle w^-\rangle^{\nu_2}\hat{\La}_{+}^1\|_{L^2{(\hat{\Om}_f)}}
\end{align}
provided the right hand sides are finite.
\end{lem}
\begin{proof}
The pressure is continuous in $\bR\times[-1,1]$.
Due to the integral identity, there exists $\bar{x}=(\bar{x}_1,\bar{x}_2)\in \bar{\Om}_f\backslash\Ga_f$ such that
$p(\bar{x})=\int_{-1}^f \La_{-}^1 \La_{+}^1\, \dx_2+\int_{f}^1 \hat{\La}_{-}^1 \hat{\La}_{+}^1\, \dx_2,$ or $\bar{x}\in \bar{\hat{\Om}}_f\backslash\Ga_f$ such that  $\hat{p}(\bar{x})=\int_{-1}^f \La_{-}^1 \La_{+}^1\, \dx_2+\int_{f}^1 \hat{\La}_{-}^1 \hat{\La}_{+}^1\, \dx_2,$
or $\bar{x}\in \Ga_f$ such that  $p(\bar{x})=\int_{-1}^f \La_{-}^1 \La_{+}^1\, \dx_2+\int_{f}^1 \hat{\La}_{-}^1 \hat{\La}_{+}^1\, \dx_2.$

For the last case, \eqref{G2} follows directly. For the first two cases,
without loss of generality, we assume $\bar{x}\in \bar{\Om}_f\backslash\Ga_f$.
Then there holds
\begin{align*}
\langle w^+\rangle^{\mu}\langle w^-\rangle^{\nu}\big(\ud{p}(\bar{x}_1)-p(\bar{x})\big)=-\int_{f}^{\bar{x}_2}\langle w^+\rangle^{\mu}\langle w^-\rangle^{\nu} \p_2 p\, \dx_2.
\end{align*}
Then taking $L^2(\bR)$ norm yields \eqref{G2}.

For \eqref{G3}, we only estimate $p$ in $\Om_f$. The estimate of  $\hat{p}$ is similarly thus the details is omitted.
For $x=(x_1,x_2)\in \Om_f$, we write
\begin{align*}
\langle w^+\rangle^{\mu}\langle w^-\rangle^{\nu}\big(\ud{p}(x_1)-p(x)\big)=-\int_{f}^{x_2}\langle w^+\rangle^{\mu}\langle w^-\rangle^{\nu} \p_2 p\, \dx_2.
\end{align*}
Then taking $L^2(\Om_f)$ norm, combined with \eqref{G2} yields \eqref{G3}.
\end{proof}


\subsection{Some preliminary estimates}
The normal of the pressure on the free boundary  plays an important role in the estimate.
Recalling the difference of the normal of the pressure on the free boundary \eqref{BB11} that
\begin{align}\label{G-7}
N_f\cdot\ud{\nabla p}-N_f\cdot\ud{\nabla \hat{p}}
&= -\big(\ud{\La_{+}^1}-\ud{\hat{\La}_{+}^1} \big) \p_1 (\p_t-  \p_1) f
  -\big(\ud{\La_{-}^1}-\ud{\hat{\La}_{-}^1}\big) \p_1 (\p_t+  \p_1) f \\\nonumber
&\quad-\big(\ud{\La_{+}^1} \ud{\La_{-}^1}- \ud{\hat{\La}_{+}^1} \ud{\hat{\La}_{-}^1} \big) \p_1^2 f\quad \textrm{on} \,\,  \Ga_f \,.
\end{align}
In the sequel, we estimate $N_f\cdot\ud{\nabla p}-N_f\cdot\ud{\nabla \hat{p}}$ on $\Ga_f$.
\begin{lem}\label{lemPFL}
Let $\f 12<\mu\leq \f 35$.
Assuming $ \|\p_1f\|_{H^{s-\f 12}}\lesssim 1$.
Then there hold
\begin{align}\label{G-8}
&\| \langle w^\pm\rangle^{5\mu}\langle w^\mp\rangle^{\mu}
 \big( N_f\cdot\ud{\nabla p}-N_f\cdot\ud{\nabla \hat{p}} \big)\|_{L^2(\BR)}
 \leq (E_2G_2^b)^{\f 12} \cdot \mathcal{P}(1+\|\p_1f\|_{H^1(\bR)}),\\\label{G-9}
&\sum_{0\leq a\leq 1}\| \langle w^\pm\rangle^{5\mu}\langle w^\mp\rangle^{\mu} \langle\p_1\rangle^{\f 12}\p_1^a(N_f\cdot\ud{\nabla p}-N_f\cdot\ud{\nabla \hat{p}})\|_{L^2(\BR)}
\leq (E_3 G_3^b)^{\f 12} \cdot \mathcal{P}(1+\|\p_1f\|_{H^{\f52}(\bR)}),
\end{align}
where $\mathcal{P}(\cdot)$ is a polynomial.
\end{lem}
\begin{proof}
We only show \eqref{G-9}.
The first inequality \eqref{G-8} can be proved in a similar and simpler way.
For the second inequality, we only show
\begin{align*}
 \| \langle w^+\rangle^{5\mu}\langle w^-\rangle^{\mu} \langle\p_1\rangle^{\f 12}\p_1^a(N_f\cdot\ud{\nabla p}-N_f\cdot\ud{\nabla \hat{p}})\|_{L^2(\BR)}
  \leq (E_3 G_3^b)^{\f 12}\cdot \mathcal{P}(1+ \|\p_1f\|_{H^{\f52}(\bR)})
\end{align*}
for $a=1$.
The other cases can be estimated similarly.

Note that the index on the weight functions is $5\mu$ plus $\mu$, the left and right Alfv\'en waves need be written in terms of $\La_{\pm}$, $\hat{\La}_{\pm}$ instead of $(\p_t\pm \p_1)f$.
Recalling $\eqref{BB5}_1$ that
\begin{align*}
(\p_t\pm\p_1) f=\La_{\pm}\cdot N_f=\hat{\La}_{\pm}\cdot N_f.
\end{align*}
Plugging the above expression into \eqref{G-7}, thus
\begin{align}\label{G-10}
&N_f\cdot\ud{\nabla p}-N_f\cdot\ud{\nabla \hat{p}}\\\nonumber
&= - \ud{\La_{+}^1} \big( \p_1\ud{\La_{-}^2}-\p_1\ud{\La_{-}^1}\p_1f-\ud{\La_{-}^1}\p_1^2f \big)
+\ud{\hat{\La}_{+}^1}
\big( \p_1\ud{\hat{\La}_{-}^2}-\p_1\ud{\hat{\La}_{-}^1}\p_1f-\ud{\hat{\La}_{-}^1}\p_1^2f \big)\\\nonumber
&\quad  - \ud{\La_{-}^1}
\big( \p_1\ud{\La_{+}^2}-\p_1\ud{\La_{+}^1}\p_1f-\ud{\La_{+}^1}\p_1^2f \big)
+\ud{\hat{\La}_{-}^1}
\big( \p_1\ud{\hat{\La}_{+}^2}-\p_1\ud{\hat{\La}_{+}^1}\p_1f-\ud{\hat{\La}_{+}^1}\p_1^2f \big)
\\\nonumber
&\quad-\big(\ud{\La_{+}^1} \ud{\La_{-}^1}- \ud{\hat{\La}_{+}^1} \ud{\hat{\La}_{-}^1} \big) \p_1^2 f\,.
\end{align}
We only show the details for first term in the second line of \eqref{G-10}. The estimate of the other terms is similar.

By  weighted Sobolev inequalities Lemma \ref{Sobo1}, weighted commutator Lemma \ref{lemD10}, weighted chain rule Lemma \ref{lemG3}
and the weighted trace Lemma \ref{lemD3}, there holds
\begin{align}\label{G-11}
 &\| \langle w^+\rangle^{5\mu}\langle w^-\rangle^{\mu} \langle\p_1\rangle^{\f 12}\p_1 \big( \ud{\La_{+}^1}\p_1\ud{\La_{-}^2} \big)\|_{L^2(\bR)} \\\nonumber
&\lesssim \| \langle w^+\rangle^{5\mu}\langle w^-\rangle^{-\mu} \ud{\La_{+}^1}\|_{L^\infty(\bR)}
\| \langle w^-\rangle^{2\mu} \langle\p_1\rangle^{\f 12}\p_1^2\ud{\La_{-}^2} \|_{L^2(\bR)} \\\nonumber
&\quad+\| \langle w^+\rangle^{5\mu}\langle w^-\rangle^{-\mu} \p_1\ud{\La_{+}^1}\|_{L^\infty(\bR)}
\| \langle w^-\rangle^{2\mu} \langle\p_1\rangle^{\f 12} \p_1\ud{\La_{-}^2} \|_{L^2(\bR)} \\\nonumber
&\quad+ \| \langle w^+\rangle^{5\mu}\langle w^-\rangle^{-\mu} \langle\p_1\rangle^{\f 12}\p_1 \ud{\La_{+}^1}\|_{L^2(\bR)}
\| \langle w^-\rangle^{2\mu} \p_1\ud{\La_{-}^2} \|_{L^\infty(\bR)} \\\nonumber
&\lesssim (E_3 G_3^b)^{\f 12} \cdot \mathcal{P}(1+\|\p_1f\|_{ H^{\f52}(\bR)})  \, .
\end{align}

Similarly by weighted Sobolev inequalities Lemma \ref{Sobo1}, weighted commutator Lemma \ref{lemD10}, weighted chain rule Lemma \ref{lemG3}
and the weighted trace Lemma \ref{lemD3}, we estimate
\begin{align} \label{G-13}
&\| \langle w^+\rangle^{5\mu}\langle w^-\rangle^{\mu} \langle\p_1\rangle^{\f 12}\p_1 ( \ud{\La_{+}^1}\p_1\ud{\La_{-}^1}\p_1f )\|_{L^2(\bR)} \\\nonumber
&\lesssim \| \langle w^+\rangle^{5\mu}\langle w^-\rangle^{-\mu} \ud{\La_{+}^1} \p_1f \|_{L^\infty(\bR)}
\| \langle w^-\rangle^{2\mu} \langle\p_1\rangle^{\f 12}\p_1^2\ud{\La_{-}^1} \|_{L^2(\bR)} \\\nonumber
&\quad+\| \langle w^+\rangle^{5\mu}\langle w^-\rangle^{-\mu} \p_1(\ud{\La_{+}^1}\p_1f )\|_{L^\infty(\bR)}
\| \langle w^-\rangle^{2\mu} \langle\p_1\rangle^{\f 12} \p_1\ud{\La_{-}^1} \|_{L^2(\bR)} \\\nonumber
&\quad+ \| \langle w^+\rangle^{5\mu}\langle w^-\rangle^{-\mu} \langle\p_1\rangle^{\f 12}\p_1 (\ud{\La_{+}^1}\p_1f )\|_{L^2(\bR)}
\| \langle w^-\rangle^{2\mu} \p_1\ud{\La_{-}^1} \|_{L^\infty(\bR)} \\\nonumber
&\lesssim (E_3 G_3^b)^{\f 12} \cdot \mathcal{P}(1+\|\p_1f\|_{ H^{\f52}(\bR)})
+ \| \langle w^+\rangle^{5\mu}\langle w^-\rangle^{-\mu} \langle\p_1\rangle^{\f 12}\p_1 (\ud{\La_{+}^1}\p_1f )\|_{L^2(\bR)} E_3^{\f12} \\\nonumber
&\lesssim (E_3 G_3^b)^{\f 12} \cdot \mathcal{P}(1+\|\p_1f\|_{ H^{\f52}(\bR)}).
\end{align}

Similar to \eqref{G-11} and \eqref{G-13},
by weighted Sobolev inequalities Lemma \ref{Sobo1}, weighted commutator Lemma \ref{lemD10}, weighted chain rule Lemma \ref{lemG3}
and the weighted trace Lemma \ref{lemD3}, we deduce that
\begin{align*}
&\| \langle w^+\rangle^{5\mu}\langle w^-\rangle^{\mu} \langle\p_1\rangle^{\f 12}\p_1 ( \ud{\La_{+}^1} \ud{\La_{-}^1} \p_1^2 f )\|_{L^2(\bR)} \\
&\lesssim (E_3 G_3^b)^{\f 12}  \cdot \mathcal{P}(1+\|\p_1f\|_{H^{\f52}(\bR)}) \, .
\end{align*}
This ends the proof of the lemma.
\end{proof}

\begin{lem}\label{lemPFH}
Let $\f 12<\mu\leq \f 35$, $s\geq 4$ be an integer. Assuming $ \|\p_1f\|_{H^{s-\f 12}}\lesssim 1$.
For $1\leq a\leq s-2$, there holds
\begin{align*}
\| \langle w^\pm\rangle^{2\mu}\langle w^\mp\rangle^{\mu} \langle\p_1\rangle^{\f 12}\p_1^a(N_f\cdot\ud{\nabla p}-N_f\cdot\ud{\nabla \hat{p}})\|_{L^2(\BR)}
\leq E_s^{\f 12}G_s^{\f 12}  \cdot \mathcal{P}(1 +\|\p_1f\|_{H^{s-\f 12}(\bR)}),
\end{align*}
where $\mathcal{P}(\cdot)$ is a polynomial.
\end{lem}
\begin{proof}
For $1\leq a\leq s-2$, it suffices to show
\begin{align*}
 \| \langle w^+\rangle^{2\mu}\langle w^-\rangle^{\mu} \langle\p_1\rangle^{\f 12}\p_1^a(N_f\cdot\ud{\nabla p}-N_f\cdot\ud{\nabla \hat{p}})\|_{L^2(\BR)}
 \leq E_s^{\f 12}G_s^{\f 12}  \cdot \mathcal{P}(1+\|\p_1f\|_{H^{s-\f 12}(\bR)}).
\end{align*}
The other case can be estimated similarly.

Recalling \eqref{BB11} as follows:
\begin{align} \label{G-18}
N_f\cdot\ud{\nabla p}-N_f\cdot\ud{\nabla \hat{p}}
&= -\big(\ud{\La_{+}^1}-\ud{\hat{\La}_{+}^1} \big) \p_1 (\p_t-  \p_1) f
  -\big(\ud{\La_{-}^1}-\ud{\hat{\La}_{-}^1}\big) \p_1 (\p_t+  \p_1) f \\\nonumber
&\quad-\big(\ud{\La_{+}^1} \ud{\La_{-}^1}- \ud{\hat{\La}_{+}^1} \ud{\hat{\La}_{-}^1} \big) \p_1^2 f\,.
\end{align}
We first  show
\begin{align*}
 &\| \langle w^+\rangle^{2\mu}\langle w^-\rangle^{\mu} \langle\p_1\rangle^{\f 12}\p_1^a\big( (\ud{\La_{+}^1}-\ud{\hat{\La}_{+}^1} )
\p_1(\p_t-\p_1)f \big)\|_{L^2(\BR)}
\leq G_s^{\f12} E_s^{\f12} \cdot \mathcal{P}(1 +\|\p_1f\|_{H^{s-\f 12}(\bR)})\, .
\end{align*}
Note that
\begin{align}\label{G-20}
 & \| \langle w^+\rangle^{2\mu}\langle w^-\rangle^{\mu} \langle\p_1\rangle^{\f 12}\p_1^a\big( (\ud{\La_{+}^1}-\ud{\hat{\La}_{+}^1} )
\p_1(\p_t-\p_1)f \big)\|_{L^2(\BR)}\\\nonumber
&\leq\| \langle w^+\rangle^{2\mu}\langle w^-\rangle^{\mu} (\ud{\La_{+}^1}-\ud{\hat{\La}_{+}^1} )
\langle\p_1\rangle^{\f 12}\p_1^{a+1}(\p_t-\p_1)f  \|_{L^2(\BR)}\\\nonumber
&\quad+\| \langle w^+\rangle^{2\mu}\langle w^-\rangle^{\mu} [\langle\p_1\rangle^{\f 12}\p_1^a, (\ud{\La_{+}^1}-\ud{\hat{\La}_{+}^1} )]
\p_1(\p_t-\p_1)f \|_{L^2(\BR)}.
\end{align}
For the second line of \eqref{G-20}, by Lemma \ref{Sobo2} and Lemma \ref{lemD2}, one has
\begin{align*}
&\| \langle w^+\rangle^{2\mu}\langle w^-\rangle^{\mu} (\ud{\La_{+}^1}-\ud{\hat{\La}_{+}^1} )
\langle\p_1\rangle^{\f 12}\p_1^{a+1}(\p_t-\p_1)f  \|_{L^2(\BR)}\\
&\leq \Big\| \f{\langle w^+\rangle^{2\mu} (\ud{\La_{+}^1}-\ud{\hat{\La}_{+}^1} )}{\langle w^-\rangle^{\mu} }\Big\|_{L^\infty(\BR)}
\|  \langle w^-\rangle^{2\mu} \langle\p_1\rangle^{\f 12}\p_1^{a+1}(\p_t-\p_1)f  \|_{L^2(\BR)}\\
&\lesssim G_s^{\f12} E_s^{\f12} (1+\| \p_1 f\|_{L^2(\bR)}).
\end{align*}
For the third line of \eqref{G-20},
%
%
if $a=1$, by the weighted commutator estimate Lemma \ref{lemD10},
the weighted Sobolev inequalities Lemma \ref{Sobo2},
the weighted chain rules Lemma \ref{lemG3} and the weighted trace Lemma \ref{lemD3}, we deduce that
\begin{align*}
	&\| \langle w^+\rangle^{2\mu} \langle w^-\rangle^{\mu}
\big[\langle\p_1\rangle^{\f12}\p_1,\ud{\La_{+}^1 }-\ud{\hat{\La}_{+}^1 }\ \big] \p_1(\p_t-\p_1) f\|_{L^2(\BR)} \\
	&\lesssim \Big\| \f{\langle w^+\rangle^{2\mu}}{\langle w^-\rangle^{\mu}}  \langle \p_1\rangle^{\f12}\p_1
	\big( \ud{\La_{+}^1 }- \ud{\hat{\La}_{+}^1} \big)
	\Big\|_{L^2(\BR)}
	\big\| \langle w^-\rangle^{2\mu} \p_1 (\p_t-\p_1) f \big\|_{L^\infty(\BR)}  \\\nonumber
	&\quad+ \Big\| \f{\langle w^+\rangle^{2\mu}}{\langle w^-\rangle^{\mu}}  \langle \p_1\rangle^{\f12}
	\big( \ud{\La_{+}^1 }- \ud{\hat{\La}_{+}^1} \big)
	\Big\|_{L^\infty(\BR)}
	\big\| \langle w^-\rangle^{2\mu} \p_1^2 (\p_t-\p_1) f \big\|_{L^2(\BR)}  \\\nonumber
	&\lesssim E_3^{\f12} G_3^{\f12} .
\end{align*}
If $2\leq a\leq s-2$, by weighted commutator estimate Lemma \ref{lemD11}, weighted chain rule Lemma \ref{lemG3} and the weighted trace Lemma \ref{lemD3}, there holds
\begin{align*}
	&\| \langle w^+\rangle^{2\mu} \langle w^-\rangle^{\mu} \big[\langle\p_1\rangle^{\f12}\p_1^a,\ud{\La_{+}^1 }-\ud{\hat{\La}_{+}^1 }\ \big] \p_1(\p_t-\p_1) f\|_{L^2(\BR)} \\
	&\lesssim \Big\|  \f{\langle w^+\rangle^{2\mu}}{\langle w^-\rangle^{\mu}}\langle \p_1\rangle^{\f12}
	\p_1^{\leq a} \big( \ud{\La_{+}^1 }- \ud{\hat{\La}_{+}^1} \big)
	\Big\|_{L^2(\BR)}
	\big\| \langle w^-\rangle^{2\mu} \langle \p_1\rangle^{\f12}
	\p_1^{\leq a}(\p_t-\p_1) f \big\|_{L^2(\BR)}  \\\nonumber
	&\lesssim G_{s-1}^{\f12}E^f_{s-\f12}\mathcal{P}(1+\| f'\|_{H^{s-\f32}}). 
\end{align*}

For the estimate of the other terms in \eqref{G-18}, it is similar to \eqref{G-20},
hence the details are omitted.
\end{proof}

\subsection{Weighted $H^1$ estimate of $\hat{p}$ and $p$ }
In this subsection, we treat the weighted $H^1$ estimate of the pressure.
\begin{lem}\label{lemP1}
Let $\f12<\mu\leq\f35$, for the pressure satisfying \eqref{H0}, there holds
\begin{align*}
& \| \langle w^\pm\rangle^{5\mu} \langle w^\mp\rangle^{\mu}  \nabla^{\leq 1} p \|_{L^2(\Om_f)}^2
+\| \langle w^\pm\rangle^{5\mu} \langle w^\mp\rangle^{\mu}  \nabla^{\leq 1}  \hat{p} \|_{L^2(\Om_f)}^2
\lesssim   E_2G_2^b \mathcal{P}(1+\|\p_1f\|_{H^1(\bR)}) \,.
\end{align*}
\end{lem}
\begin{proof}
The proof will be divided into nine steps in order to gradually improve the index of the weights.

\textbf{Step 1: estimate of
$\|  p\|_{H^1(\Om_f)}^2
+\|  \hat{p}\|_{H^1(\hat{\Om}_f)}^2$}

Taking the $L^2(\Om_f)$ inner product of $\eqref{H0}_1$ with $ p$ and
taking the $L^2(\hat{\Om}_f)$ inner product of $\eqref{H0}_2$ with $ \hat{p}$, respectively,
by the boundary condition, we obtain
\begin{align}\label{G10}
&\| \nabla p \|_{L^2(\Om_f)}^2
+\| \nabla \hat{p} \|_{L^2(\hat{\Om}_f)}^2
=\int_{\Ga_f} (\ud{\p_{n_f} p} -\ud{\p_{n_f} \hat{p}})\cdot \ud{p} \d\sigma (x) \\\nonumber
&\quad+\int_{\Om_f}  \big(\nabla \cdot (\La_{-} \cdot \nabla \La_{+}) \big) \cdot p\, \dx
+\int_{\hat{\Om}_f}  \big(\nabla \cdot (\hat{\La}_{-} \cdot \nabla \hat{\La}_{+}) \big) \cdot \hat{p} \,\dx  .
\end{align}
By using Lemma \ref{lemPI}, we infer
\begin{align*}
\|\ud{p} \|_{L^2(\BR)}
\lesssim \| \p_2 p \|_{L^2(\Om_f)}+\| \p_2 \hat{p} \|_{L^2(\hat{\Om}_f)}+\|\La_{-}^1\|_{L^2(\Om_f)} \|\La_{+}^1\|_{L^2(\Om_f)}+\| \hat{\La}_{-}^1\|_{L^2{(\hat{\Om}_f)}} \|\hat{\La}_{+}^1\|_{L^2{(\hat{\Om}_f)}}.
\end{align*}
Thus by using Lemma \ref{lemPFL}, there holds
\begin{align*}
&\int_{\Ga_f} (\ud{\p_{n_f} p} -\ud{\p_{n_f} \hat{p}})\cdot \ud{p} \d\sigma (x)
=\int_{\BR} (\ud{N_f\cdot\nabla  p} -\ud{N_f\cdot\nabla  \hat{p}})\cdot \ud{p}  \dx_1
 \\\nonumber
&\leq \| \ud{N_f\cdot\nabla  p} -\ud{N_f\cdot\nabla  \hat{p}} \|_{L^2(\BR)}\|\ud{p} \|_{L^2(\bR)} \\\nonumber
&\leq C (G_2^b E_2)^{\f12} \mathcal{P}(1+\|\p_1f\|_{H^1(\bR)})
\big (\| \nabla p \|_{L^2(\Om_f)}+\| \nabla \hat{p} \|_{L^2(\hat{\Om}_f)}+(G_2^b E_2)^{\f12} \big)\\\nonumber
&\leq \f{\delta}{2}\| \nabla p\|_{L^2(\Om_f)}^2
     +\f{\delta}{2}\| \nabla \hat{p}\|_{L^2(\hat{\Om}_f)}^2
     +C_{\delta} G_2^b E_2    \, .
\end{align*}
Next, we have
\begin{align*}
&\int_{\Om_f}  \big( \nabla \cdot (\La_{-} \cdot \nabla \La_{+}) \big) \cdot p\, \dx
+\int_{\hat{\Om}_f}  \big( \nabla \cdot (\hat{\La}_{-} \cdot \nabla \hat{\La}_{+}) \big) \cdot \hat{p}\, \dx\\
&=-\int_{\Om_f} (\La_{-} \cdot \nabla \La_{+})  \cdot \nabla p\, \dx
-\int_{\hat{\Om}_f} (\hat{\La}_{-} \cdot \nabla \hat{\La}_{+}) \cdot \nabla \hat{p}\, \dx\\
&\quad-\int_{\Ga_f} (\La_{-} \cdot \nabla \La_{+}) \cdot n_f p\, \d\sigma(x)
+\int_{\Ga_f}  (\hat{\La}_{-} \cdot \nabla \hat{\La}_{+}) \cdot n_f \hat{p}\, \d\sigma{x}\\
&\leq  \big\| |\La_{-}| \cdot |\nabla \La_{+}|  \big\|_{L^2(\Om_f)}
\| \nabla p \|_{L^2(\Om_f)}
+ \big\| |\hat{\La}_{-}| \cdot |\nabla \hat{\La}_{+}|  \big\|_{L^2(\hat{\Om}_f)}
\| \nabla \hat{p}\|_{L^2(\hat{\Om}_f)}\\
&\quad+ \Big( \big\| |\ud{\La_{-}}| \cdot |\ud{\nabla \La_{+}}|  \big\|_{L^2(\bR)}
\| \ud{p} \|_{L^2(\bR)}
+ \big\| |\ud{\hat{\La}_{-}}| \cdot |\ud{\nabla \hat{\La}_{+}}|  \big\|_{L^2(\bR)}
\| \ud{\hat{p}}\|_{L^2(\bR)}\Big)(1+\|\p_1f\|_{L^\infty(\bR)})\\
&\leq \f{\delta}{2}\| \nabla p\|_{L^2(\Om_f)}^2
     +\f{\delta}{2}\| \nabla \hat{p}\|_{L^2(\hat{\Om}_f)}^2
     +C_{\delta} G_2^b E_2    \, .
\end{align*}
Hence by the H\"older inequality, \eqref{G10} can be further bounded by the product of
the energy and the ghost weight energy with some remainders:
\begin{align}\label{G11}
& \| \nabla p \|_{L^2(\Om_f)}^2
+\| \nabla \hat{p} \|_{L^2(\hat{\Om}_f)}^2 \\\nonumber
& \leq C_\delta G_2^b E_2 \mathcal{P}(1+\|\p_1f\|_{H^1(\bR)})
 +\delta \big(\| \nabla p \|^2_{L^2(\Om_f)}
+ \| \nabla \hat{p} \|^2_{L^2(\hat{\Om}_f)} \big)
\end{align}
where $\delta>0$ can be any positive number, $C_\delta$ is a constant depending on
$\delta$. Taking $\delta<1$, absorbing the last term of \eqref{G11} yields
\begin{align*}
 \|  \nabla p \|_{L^2(\Om_f)}^2
+\| \hat{p} \|_{L^2(\hat{\Om}_f)}^2
 \leq C G_2^b E_2 \mathcal{P}(1+\|\p_1f\|_{H^1(\bR)}) \,.
\end{align*}
By using Lemma \ref{lemPI}, this further yields
\begin{align*}
 \|  p \|_{H^1(\Om_f)}^2
+\| \hat{p} \|_{H^1(\hat{\Om}_f)}^2
 \leq C G_2^b E_2 \mathcal{P}(1+\|\p_1f\|_{H^1(\bR)}) \,.
\end{align*}


\textbf{Step 2: estimate of
$\| \langle w^\pm\rangle \nabla^{\leq 1} p \|_{L^2(\Om_f)}^2
+\| \langle w^\pm\rangle \nabla^{\leq 1} \hat{p} \|_{L^2(\hat{\Om}_f)}^2$.}

Taking the $L^2(\Om_f)$ inner product of $\eqref{H0}_1$ with $ \langle w^\pm\rangle^2 p$ and
taking the $L^2(\hat{\Om}_f)$ inner product of $\eqref{H0}_2$ with $ \langle w^\pm\rangle^2 \hat{p}$, respectively,
we obtain
\begin{align}\label{G13}
&\| \langle w^\pm\rangle  \nabla p \|_{L^2(\Om_f)}^2
+\| \langle w^\pm\rangle  \nabla p \|_{L^2(\Om_f)}^2\\\nonumber
&=-\int_{\Om_f} \nabla\langle w^\pm\rangle^2   p \cdot\nabla p \dx
-\int_{\hat{\Om}_f} \nabla\langle w^\pm\rangle^2   \hat{p} \cdot\nabla \hat{p} \dx\\\nonumber
&\quad+\int_{\bR} \langle w^\pm\rangle^2 (\ud{N_f\cdot\nabla  p} -\ud{N_f\cdot\nabla  \hat{p}})\cdot \ud{p} \dx_1\\\nonumber
&\quad+\int_{\Om_f}  \langle w^\pm\rangle^2 \big( \nabla \cdot (\La_{-} \cdot \nabla \La_{+}) \big) \cdot p\, \dx
+\int_{\hat{\Om}_f}  \langle w^\pm\rangle^2 \big( \nabla \cdot (\hat{\La}_{-} \cdot \nabla \hat{\La}_{+}) \big) \cdot \hat{p}\, \dx  .
\end{align}
For the second line of \eqref{G13} which corresponding to
 the commutator of the pressure and the weight functions, by the H\"older inequality, one estimates
\begin{align*}
&-\int_{\Om_f} \nabla\langle w^\pm\rangle^2   p \cdot\nabla p \dx
-\int_{\hat{\Om}_f} \nabla\langle w^\pm\rangle^2   \hat{p} \cdot\nabla \hat{p} \dx
\\
&\leq 2\|\langle w^\pm\rangle    \nabla p\|_{L^2(\Om_f)} \| p\|_{L^2(\Om_f)}
+2\|\langle w^\pm\rangle    \nabla \hat{p}\|_{L^2(\hat{\Om}_f)} \| \hat{p}\|_{L^2(\hat{\Om}_f)}
 \\
&\leq \delta\|\langle w^\pm\rangle    \nabla p\|^2_{L^2(\Om_f)}+\f{1}{\delta} \| p\|^2_{L^2(\Om_f)}
+\delta\|\langle w^\pm\rangle    \nabla \hat{p}\|^2_{L^2(\hat{\Om}_f)}+\f{1}{\delta} \| \hat{p}\|^2_{L^2(\hat{\Om}_f)} .
\end{align*}
The third line of \eqref{G13} which corresponds to the boundary term can be estimated similar to the step 1.
By Lemma \ref{lemD1} and Lemma \ref{lemPFL}, it is bounded by
\begin{align*}
& \| \langle \ud{w}^\pm\rangle  (\ud{N_f\cdot\nabla  p} -\ud{N_f\cdot\nabla  \hat{p}}) \|_{L^2(\BR)}\|\langle \ud{w}^\pm\rangle\ud{p} \|_{L^2(\BR)}\\
&\leq \f{1}{4\delta}\| \langle \ud{w}^\pm\rangle  (\ud{N_f\cdot\nabla  p} -\ud{N_f\cdot\nabla  \hat{p}}) \|_{L^2(\BR)}^2+\delta\|\langle \ud{w}^\pm\rangle\ud{p} \|_{L^2(\BR)}^2 \\
&\leq  C_{\delta}G_2^b E_2 \mathcal{P}(1+\|\p_1f\|_{H^1(\bR)}) + \delta
\big( \| \langle w^\pm\rangle \nabla p \|_{L^2(\Om_f)}^2
+\| \langle w^\pm\rangle \nabla \hat{p} \|_{L^2(\hat{\Om}_f)}^2 \big)\, .
\end{align*}
The last line of \eqref{G13} which corresponds to the boundary term can be estimated similar to the step 1.
By using Lemma \ref{lemPI}, it is bounded by
\begin{align*}
&\int_{\Om_f}  \langle w^\pm\rangle^2 \big( \nabla \cdot (\La_{-} \cdot \nabla \La_{+}) \big) \cdot p\, \dx
+\int_{\hat{\Om}_f}  \langle w^\pm\rangle^2 \big( \nabla \cdot (\hat{\La}_{-} \cdot \nabla \hat{\La}_{+}) \big) \cdot \hat{p}\, \dx  \\\nonumber
&\leq \delta\| \langle w^\pm\rangle\nabla p\|_{L^2(\Om_f)}^2
     +\delta\| \langle w^\pm\rangle\nabla \hat{p}\|_{L^2(\hat{\Om}_f)}^2
     +C_{\delta} G_2^b E_2    \, .
\end{align*}
Thus \eqref{G13} reduces to
\begin{align*}
&\| \langle w^\pm\rangle    \nabla p \|_{L^2(\Om_f)}^2
+\| \langle w^\pm\rangle   \nabla \hat{p} \|_{L^2(\hat{\Om}_f)}^2\\
&\leq  C_{\delta} G_2^b E_2 \mathcal{P}(1+\|\p_1f\|_{H^1(\bR)}) + 3\delta
\big( \| \langle w^\pm\rangle \nabla p \|_{L^2(\Om_f)}^2
+\| \langle w^\pm\rangle \nabla \hat{p} \|_{L^2(\hat{\Om}_f)}^2 \big)\, .
\end{align*}
Taking $0<\delta<\f12$, absorbing the last term yields
\begin{align*}
\| \langle w^\pm\rangle    \nabla p \|_{L^2(\Om_f)}^2
+\| \langle w^\pm\rangle   \nabla \hat{p} \|_{L^2(\hat{\Om}_f)}^2
\lesssim G_2^b E_2 \mathcal{P}(1+\|\p_1f\|_{H^1(\bR)}).
\end{align*}
By using Lemma \ref{lemPI}, this further yields
\begin{align*}
\| \langle w^\pm\rangle    \nabla^{\leq 1} p \|_{L^2(\Om_f)}^2
+\| \langle w^\pm\rangle   \nabla^{\leq 1} \hat{p} \|_{L^2(\hat{\Om}_f)}^2
\lesssim G_2^b E_2 \mathcal{P}(1+\|\p_1f\|_{H^1(\bR)}).
\end{align*}

\textbf{Step 3: estimate of
$\| \langle w^\pm\rangle^2 \nabla^{\leq 1} p \|_{L^2(\Om_f)}^2
+ \| \langle w^\pm\rangle^2 \nabla^{\leq 1} \hat{p} \|_{L^2(\hat{\Om}_f)}^2$.}
The estimate is the same to step 2. We will obtain
\begin{align*}
\| \langle w^\pm\rangle^2 \nabla^{\leq 1} p \|_{L^2(\Om_f)}^2
+ \| \langle w^\pm\rangle^2 \nabla^{\leq 1} \hat{p} \|_{L^2(\hat{\Om}_f)}^2
\lesssim G_2^b E_2 \mathcal{P}(1+\|\p_1f\|_{H^1(\bR)}).
\end{align*}

\textbf{Step 4: estimate of
$\| \langle w^\pm\rangle^{5\mu} \nabla^{\leq 1} p \|_{L^2(\Om_f)}^2
+\| \langle w^\pm\rangle^{5\mu} \nabla^{\leq 1} \hat{p} \|_{L^2(\hat{\Om}_f)}^2$
with $\f12<\mu\leq \f35$.}

The estimate is similar to step 2. We only explain the differences by the index of the weight functions. Let $\f12<\mu\leq \f35$.
Taking the $L^2(\Om_f)$ inner product of $\eqref{H0}_1$ with $ \langle w^\pm\rangle^{10\mu} p$ and
taking the $L^2(\hat{\Om}_f)$ inner product of $\eqref{H0}_2$ with $ \langle w^\pm\rangle^{10\mu} \hat{p}$, respectively. After similar estimate of the boundary term, the estimate of the commutator of the pressure and the weights and the estimate of the inhomogeneous term, we obtain
\begin{align}\label{G15}
&\| \langle w^\pm\rangle^{5\mu} \nabla p \|_{L^2(\Om_f)}^2
+\| \langle w^\pm\rangle^{5\mu} \nabla \hat{p} \|_{L^2(\hat{\Om}_f)}^2\\\nonumber
&\leq  C_{\delta} G_2^b E_2\mathcal{P}(1+\|\p_1f\|_{H^1(\bR)})+ 10\mu \|\langle w^\pm\rangle^{5\mu}   \nabla p\|_{L^2(\Om_f)}
\| \langle w^\pm\rangle^{5\mu-1} p\|_{L^2(\Om_f)} \\\nonumber
&\quad+ 10\mu \|\langle w^\pm\rangle^{5\mu}   \nabla \hat{p}\|_{L^2(\hat{\Om}_f)}
\| \langle w^\pm\rangle^{5\mu-1} \hat{p}\|_{L^2(\hat{\Om}_f)}\\\nonumber
&\leq C_{\delta}G_2^b E_2\mathcal{P}(1+\|\p_1f\|_{H^1(\bR)})+\delta\|\langle w^\pm\rangle^{5\mu}   \nabla p\|^2_{L^2(\Om_f)}
+\f{25\mu^2}{\delta} \| \langle w^\pm\rangle^{5\mu-1} p\|^2_{L^2(\Om_f)}\\\nonumber
&\quad+\delta\|\langle w^\pm\rangle^{5\mu}   \nabla \hat{p}\|^2_{L^2(\hat{\Om}_f)}
+\f{25\mu^2}{\delta} \| \langle w^\pm\rangle^{5\mu-1} \hat{p}\|^2_{L^2(\hat{\Om}_f)}.
\end{align}
The parameter $2\mu>1$ is for the $L^1$ in time integrability issue in order to obtain global solutions.
On the other side, we need $5\mu-1\leq 2$ so that
$\f{25\mu^2}{\delta} \| \langle w^\pm\rangle^{5\mu-1} p\|^2_{L^2(\Om_f)}
+\f{25\mu^2}{\delta} \| \langle w^\pm\rangle^{5\mu-1} \hat{p}\|^2_{L^2(\hat{\Om}_f)}$
can be controlled by $\| \langle w^\pm\rangle^2    p \|_{L^2(\Om_f)}^2
+\| \langle w^\pm\rangle^2    \hat{p} \|_{L^2(\hat{\Om}_f)}^2$ (multiplying by a contant) in step 3. Hence $\f12<\mu\leq \f35$.
Then taking $0<\delta<\f12$, \eqref{G15} yields
\begin{align*}
\| \langle w^\pm\rangle^{5\mu} \nabla^{\leq 1} p \|_{L^2(\Om_f)}^2
+\| \langle w^\pm\rangle^{5\mu} \nabla^{\leq 1} \hat{p} \|_{L^2(\hat{\Om}_f)}^2
\lesssim G_2^b E_2 \mathcal{P}(1+\|\p_1f\|_{H^1(\bR)}).
\end{align*}

Furthermore, similar to the above steps, we have

\textbf{Step 5, estimate of
$\| \langle w^\pm\rangle^{\f12}\langle w^\mp\rangle^{\f12} \nabla^{\leq 1} p \|_{L^2(\Om_f)}^2
+\| \langle w^\pm\rangle^{\f12}\langle w^\mp\rangle^{\f12} \nabla^{\leq 1} \hat{p} \|_{L^2(\hat{\Om}_f)}^2$ }

\textbf{Step 6, estimate of
$\| \langle w^\pm\rangle^{\f32} \langle w^\mp\rangle^{\f12}  \nabla^{\leq 1} p \|_{L^2(\Om_f)}^2
+\| \langle w^\pm\rangle^{\f32} \langle w^\mp\rangle^{\f12} \nabla^{\leq 1} \hat{p} \|_{L^2(\hat{\Om}_f)}^2.$}

\textbf{Step 7, estimate of
$\| \langle w^\pm\rangle^{\f32} \langle w^\mp\rangle^{\mu} \nabla^{\leq 1} p \|_{L^2(\Om_f)}^2
+\| \langle w^\pm\rangle^{\f32} \langle w^\mp\rangle^{\mu} \nabla^{\leq 1} \hat{p} \|_{L^2(\hat{\Om}_f)}^2 $ with $\f12<\mu\leq\f35$}

\textbf{Step 8, estimate of
$\| \langle w^\pm\rangle^{\f52} \langle w^\mp\rangle^{\mu} \nabla^{\leq 1} p \|_{L^2(\Om_f)}^2
+\| \langle w^\pm\rangle^{\f52} \langle w^\mp\rangle^{\mu} \nabla^{\leq 1} \hat{p} \|_{L^2(\hat{\Om}_f)}^2$ with $\f12<\mu\leq\f35$}

\textbf{Step 9, estimate of
$\| \langle w^\pm\rangle^{5\mu} \langle w^\mp\rangle^{\mu} \nabla^{\leq 1} p \|_{L^2(\Om_f)}^2
+\| \langle w^\pm\rangle^{5\mu} \langle w^\mp\rangle^{\mu} \nabla^{\leq 1} \hat{p} \|_{L^2(\hat{\Om}_f)}^2$ with $\f12<\mu\leq\f35$}

Finally, for $\f12<\mu\leq\f35$, we would obtain
\begin{align*}
&\| \langle w^\pm\rangle^{5\mu} \langle w^\mp\rangle^{\mu}\nabla^{\leq 1} p \|_{L^2(\Om_f)}^2
+\| \langle w^\pm\rangle^{5\mu} \langle w^\mp\rangle^{\mu}\nabla^{\leq 1} \hat{p} \|_{L^2(\hat{\Om}_f)}^2
\lesssim G_2^b E_2 \mathcal{P}(1+\|\p_1f\|_{H^1(\bR)}) .
\end{align*}
This ends the proof of the lemma.
\end{proof}

\subsection{Higher derivative ($2\leq k\leq s$) estimate of pressure}

In this subsection, we present $k$-order ($2\leq k\leq s$) derivatives estimate of the pressure
in Lemma \ref{lemPre}.
We first introduce the tangential derivatives and
then present their estimates.
 The tangential derivatives are globally defined for the curved free surface.
The estimates will do no harm to the weight functions. Moreover, the structure of the derivatives is well-balanced.

To construct the tangential derivatives,
we will introduce two functions defined in $\Om_f$ and in $\hat{\Om}_f$ respectively.
Under the \textit{a priori}  assumption
$\|\p_1 f\|_{H^{s-\f12}(\BR)}+\| f\|_{L^\infty(\bR)} \lesssim 1$,
we will construct functions $\Phi(x)$ and $\hat{\Phi}(x)$ satisfying
 $\Phi\in H^s(\Om_f)$ and $\hat{\Phi}\in H^s(\hat{\Om}_f)$.
Moreover, for $0\leq k\leq  s$,
\begin{equation}\label{F-1}
\begin{cases}
\Phi\big|_{\Ga_f}=\hat{\Phi}\big|_{\Ga_f}=\p_1f ,\quad
\Phi\big|_{x_2=- 1}=0,\quad \hat{\Phi}\big|_{x_2=1}=0, \\
\| \nabla^k \Phi\|_{L^2(\Om_f)}
\lesssim   \| \langle \p_1\rangle^{-\f12} \p_1 f(t,\cdot) \|_{H^k(\BR)} ,\\
\| \nabla^k \hat{\Phi}\|_{L^2(\hat{\Om}_f)}
\lesssim  \| \langle \p_1\rangle^{-\f12} \p_1 f(t,\cdot) \|_{H^k(\BR)}.
\end{cases}
\end{equation}
Then the tangential derivative is defined by
\begin{equation}\label{tangD}
\begin{cases}
\p_\tau=\Phi\p_2+\p_1 \quad \textrm{in}\,\, \Om_f,\\
\p_\tau=\hat{\Phi}\p_2+\p_1 \quad \textrm{in}\,\, \hat{\Om}_f.
\end{cases}
\end{equation}
We remind that the definition of the tangential derivative here is different from \eqref{tangD-}.

For functions $g$ defined in $\Om_f$ or $\hat{\Om}_f$, recaling that $\ud{g}$ denotes
the trace of $g$ on $\Ga_f.$ Simple calculation yields
\begin{align*}
&\p_1 \ud{g}
= \ud{\p_1 g}+\p_1f \ud{\p_2 g}
= \ud{\p_\tau g}.
\end{align*}
Repeating the above procedure, we obtain
\begin{align*}
\p_1^k \ud{g}=\ud{\p_\tau^k g},
\end{align*}
for all $1\leq k \leq s$.
We will see later that there is a cut-off in $\Phi$ and $\hat{\Phi}$, thus if $x_2$ near the fixed boundary $\pm 1$, we have
\begin{equation*}
\p_\tau =\p_1.
\end{equation*}

\subsubsection{Estimate of Sobolev norms of the tangential derivatives}

Let
\begin{align*}
\psi(t,x)&=e^{-x_2|\p_1|} \p_1 f(t,x_1),\quad x_2\in\BR^+,\,x_1\in\BR.
\end{align*}
Simple calculation yields
\begin{align*}
\p_1\psi(t,x)&= e^{-x_2|\p_1|} \p_1^2 f(t,x_1),\\
\p_2\psi(t,x)&=-e^{-x_2|\p_1|} |\p_1|\p_1 f(t,x_1) .
\end{align*}

Now we estimate $\| \psi\|_{L^2( \BR\times[0,1] )}$.
Firstly, one has
\begin{align*}
&\int_0^1 |e^{ -x_2 |\xi_1|} |^2 \d x_2
\leq C (|\xi_1|+1)^{-1}.
\end{align*}
Consequently,  by the Parseval-Plancherel identity, there holds
\begin{align*}
\|\psi(t,\cdot)\|_{L^2(\BR\times[0,1])}^2
 =\int_{\BR} |\xi_1 \hat{f}(t,\xi_1)|^2 \d\xi_1
\int_0^1 | e^{ -x_2 |\xi_1|} |^2  \d x_2
 \leq C\| \langle \p_1\rangle^{-\f12} \p_1 f(t,\cdot)\|^2_{L^2(\BR )}.
\end{align*}
where $\xi_1\in\bR$.
Through a deduction argument, the similar estimates hold for the higher order norms.
\begin{lem}
For $0\leq k \leq s$, there holds
\begin{align*}
\| \nabla^k \psi\|_{L^2( \BR\times[0,1])}&\leq C\| \langle \p_1\rangle^{-\f12}\p_1  f(t,\cdot) \|_{\dot{H}^{k}(\BR)} .
\end{align*}
\end{lem}
Now we introduce
\begin{align*}
\Phi(t,x)&= \zeta(x_2) \psi(t,x_1, f(t,x_1)-x_2) ,\quad -1\leq x_2\leq f(t,x_1),\,x_1\in\BR,\\
\hat{\Phi}(t,x)&= \zeta(x_2) \psi(t,x_1, x_2 -f(t,x_1)) ,\quad f(t,x_1)\leq x_2\leq 1,\,x_1\in\BR,
\end{align*}
where $\zeta(x_2)\in C^\infty(\BR)$ and
\begin{equation*}
\zeta(x_2)=
\begin{cases}
1 , &\text{ if } |x_2|\leq 1-\f{c_0}{2} ,\\
0 , &\text{ if } |x_2|\geq 1-\f{c_0}{2}.
\end{cases}
\end{equation*}

\textbf{Estimate of $\Phi$ in $\Om_f$}

Now we present the estimate for $\Phi$. Firstly, there holds
\begin{align*}
\| \Phi\|^2_{L^2(\Om_f )}
\leq\| \psi\|^2_{L^2(\bR\times[0,1] )}
\leq C\| \langle \p_1\rangle^{-\f12} \p_1f(t,\cdot) \|^2_{L^2(\BR)}.
\end{align*}
For the gradient of $\Phi$ in the domain $\Om_f$, we calculate
\begin{align*}
\p_1 \Phi(t,x)&=\zeta(x_2)\p_1f \p_2\psi(t,x_1, f-x_2) +\zeta(x_2)\p_1\psi(t,x_1,f-x_2),\\
\p_2 \Phi(t,x)&=\zeta'(x_2)\psi(t,x_1, f-x_2)-\zeta(x_2)\p_2\psi(t,x_1, f-x_2).
\end{align*}
Consequently,
\begin{align*}
\| \nabla \Phi\|_{L^2(\Om_f)}
&\lesssim \big( \| \nabla^{\leq 1}  \psi\|_{L^2(\BR\times [0,1])} \big) (1+\|\p_1 f\|_{L^\infty(\BR)}) \\
&\lesssim \big(\| \langle \p_1\rangle^{-\f12} \p_1 f(t,\cdot) \|_{L^2(\BR)}
+\| \langle \p_1\rangle^{-\f12} \p_1 f(t,\cdot) \|_{\dot{H}^1(\BR)} \big)
(1+\|\p_1 f\|_{H^1(\BR)})\\
&\lesssim \| \p_1 f(t,\cdot) \|_{H^{\f12}(\BR)}
(1+\|\p_1 f\|_{H^1(\BR)}).
\end{align*}

\textbf{Estimate of $\hat{\Phi}$ in $\hat{\Om}_f$}

For $\hat{\Phi}$ in $\hat{\Om}_f$, there holds
\begin{align*}
\| \hat{\Phi}\|^2_{L^2(\hat{\Om}_f )}
 \leq \|\psi(t,\cdot)\|_{L^2(\BR\times[0,1] )}^2
 \leq C\| \langle \p_1\rangle^{-\f12}\p_1 f(t,\cdot) \|_{L^2(\BR)}.
\end{align*}
In $\hat{\Om}_f$, we calculate
\begin{align*}
\p_1 \hat{\Phi}(t,x)&=-\zeta(x_2)\p_1f \p_2\psi(t,x_1,x_2-f) +\zeta(x_2)\p_1\psi(t,x_1,x_2-f),\\
\p_2 \hat{\Phi}(t,x)&=\zeta(x_2)\p_2\psi(t,x_1,x_2-f)+\zeta'(x_2) \psi(t,x_1,x_2-f).
\end{align*}
Consequently,
\begin{align*}
\| \nabla \hat{\Phi}\|_{L^2(\hat{\Om}_f )}
&\lesssim \big( \| \nabla^{\leq 1}  \psi\|_{L^2(\BR\times [0,1])} \big) (1+\|\p_1 f\|_{L^\infty(\BR)}) \\
&\lesssim \big(\| \langle \p_1\rangle^{-\f12} \p_1 f(t,\cdot) \|_{L^2(\BR)}
+\| \langle \p_1\rangle^{-\f12} \p_1 f(t,\cdot) \|_{\dot{H}^1(\BR)} \big)
(1+\|\p_1 f\|_{H^1(\BR)})\\
&\lesssim \| \p_1 f(t,\cdot) \|_{H^{\f12}(\BR)}
(1+\|\p_1 f\|_{H^1(\BR)}).
\end{align*}
Through a deduction argument, the similar estimates hold for the higher order norms.
\begin{lem}\label{lemPsi}
There hold
\begin{align*}
\| \Phi\|_{L^2(\Om_f)}
&\lesssim \| \langle \p_1\rangle^{-\f12} \p_1 f(t,\cdot) \|_{L^2(\BR)},\\
\| \hat{\Phi}\|_{L^2(\hat{\Om}_f)}
&\lesssim \| \langle \p_1\rangle^{-\f12} \p_1 f(t,\cdot) \|_{L^2(\BR)},\\
\| \nabla \Phi\|_{L^2(\Om_f)}
&\lesssim \| \p_1 f(t,\cdot) \|_{H^{\f12}(\BR)}
 \big(1+\| \p_1 f(t,\cdot) \|_{H^1(\BR)}\big),\\
\| \nabla\hat{\Phi}\|_{L^2(\hat{\Om}_f)}
&\lesssim \| \p_1 f(t,\cdot) \|_{H^{\f12}(\BR)}
 \big(1+\| \p_1 f(t,\cdot) \|_{H^1(\BR)}\big) \, .
\end{align*}
For $2\leq k\leq s$, there hold 
\begin{align*}
\| \nabla^k \Phi\|_{L^2(\Om_f)}
&\lesssim  \| \p_1 f(t,\cdot) \|_{H^{k-\f12}(\BR)}
\mathcal{P} \big(1+\| \p_1 f(t,\cdot) \|_{H^{k-\f12}(\BR)} \big),\\
\| \nabla^k \hat{\Phi}\|_{L^2(\hat{\Om}_f)}
&\lesssim   \| \p_1 f(t,\cdot) \|_{H^{k-\f12}(\BR)}
\mathcal{P}\big(1+\| \p_1 f(t,\cdot) \|_{H^{k-\f12}(\BR)} \big),
\end{align*}
where $\mathcal{P}(\cdot)$ is a polynomial.
\end{lem}
Recalling the \textit{a priori} assumption $\|\p_1f\|_{H^{s-\f 12}}\lesssim 1$,
and note that $\text{supp}\, \Phi$ and $\text{supp}\, \hat{\Phi}$ is near the free surface,
hence \eqref{F-1} holds.

\subsubsection{Second derivative estimate of $p$ and $\hat{p}$}
Applying $\p_\tau$ to $\eqref{H0}_1$ and $\eqref{H0}_2$, one obtains
\begin{align}
\label{K10}
-\Delta\p_\tau p -\left[\p_\tau, \Delta\right] p =\p_\tau\nabla \cdot (\La_{-} \cdot \nabla \La_{+})
\quad \text{in}\,\,\Om_f,\\
\label{K11}
-\Delta\p_\tau \hat{p} -\left[\p_\tau, \Delta\right] \hat{p} =\p_\tau\nabla \cdot (\hat{\La}_{-} \cdot \nabla \hat{\La}_{+})
\quad \text{in}\,\,\hat{\Om}_f.
\end{align}
Taking the $L^2$ inner product of \eqref{K10} with
$ \langle w^+\rangle^{10\mu}\langle w^-\rangle^{2\mu} \p_\tau p$ in $\Om_f$
and taking the $L^2$ inner product of \eqref{K11} with
$\langle w^+\rangle^{10\mu}\langle w^-\rangle^{2\mu} \p_\tau \hat{p}$ in $\hat{\Om}_f$,
by integration by parts, we have
\begin{align}\label{K12}
&\int_{\Om_f} \langle w^+\rangle^{10\mu}\langle w^-\rangle^{2\mu}
|\nabla\p_\tau p|^2 \dx
+\int_{\hat{\Om}_f} \langle w^+\rangle^{10\mu}\langle w^-\rangle^{2\mu}
|\nabla\p_\tau \hat{p}|^2 \dx \\\nonumber
&=\int_{\Ga_f} \langle w^+\rangle^{10\mu}\langle w^-\rangle^{2\mu}
\big(\p_{n_f}\p_\tau p\cdot \p_\tau p-\p_{n_f}\p_\tau \hat{p}\cdot \p_\tau \hat{p}\big) \,\d\sigma(x)\\\nonumber
&\quad +\int_{\hat{\Ga}}  \langle w^+\rangle^{10\mu}\langle w^-\rangle^{2\mu}
\p_2\p_\tau \hat{p}\cdot \p_\tau \hat{p} \,\d\sigma(x)
-\int_{\Ga}  \langle w^+\rangle^{10\mu}\langle w^-\rangle^{2\mu}
\p_2\p_\tau p\cdot \p_\tau p \,\d\sigma(x)
 \\\nonumber
&\quad-\int_{\Om_f} \nabla \big(\langle w^+\rangle^{10\mu}\langle w^-\rangle^{2\mu} \big)
\cdot \nabla\p_\tau p\, \p_\tau p \dx
-\int_{\hat{\Om}_f} \nabla \big( \langle w^+\rangle^{10\mu}\langle w^-\rangle^{2\mu} \big)
\cdot \nabla\p_\tau \hat{p}\, \p_\tau \hat{p} \dx \\\nonumber
&\quad+\int_{\Om_f} \langle w^+\rangle^{10\mu}\langle w^-\rangle^{2\mu}
\left[\p_\tau, \Delta\right]p\, \p_\tau p \dx
+\int_{\hat{\Om}_f} \langle w^+\rangle^{10\mu}\langle w^-\rangle^{2\mu}
\left[\p_\tau, \Delta\right]\hat{p}\, \p_\tau \hat{p} \dx \\\nonumber
&\quad+\int_{\Om_f} \langle w^+\rangle^{10\mu}\langle w^-\rangle^{2\mu}
\p_\tau(\p_j\La_{-}^i\p_i\La_{+}^j)\, \p_\tau p \dx  \\\nonumber
&\quad +\int_{\hat{\Om}_f} \langle w^+\rangle^{10\mu}\langle w^-\rangle^{2\mu}
\p_\tau(\p_j\hat{\La}_{-}^i\p_i\hat{\La}_{+}^j)\, \p_\tau \hat{p} \dx \, .
\end{align}
For the fourth line of \eqref{K12}, one has
\begin{align*}
&\big|\int_{\Om_f} \nabla_i \big(  \langle w^+\rangle^{10\mu}\langle w^-\rangle^{2\mu}\big)\nabla_i \p_\tau p\cdot \p_\tau p \dx
+\int_{\hat{\Om}_f} \nabla_i \big(  \langle w^+\rangle^{10\mu}\langle w^-\rangle^{2\mu}\big)\nabla_i \p_\tau \hat{p}\cdot \p_\tau \hat{p} \dx \big|\\
&\leq \delta  \int_{\Om_f} \langle w^+\rangle^{10\mu}\langle w^-\rangle^{2\mu} |\nabla\p_\tau p|^2\dx
+C_\delta \int_{\Om_f} \langle w^+\rangle^{10\mu}\langle w^-\rangle^{2\mu} |\p_\tau p|^2\dx\\
&\quad+ \delta  \int_{\hat{\Om}_f} \langle w^+\rangle^{10\mu}\langle w^-\rangle^{2\mu} |\nabla\p_\tau \hat{p}|^2\dx
+C_\delta \int_{\hat{\Om}_f} \langle w^+\rangle^{10\mu}\langle w^-\rangle^{2\mu} |\p_\tau \hat{p}|^2\dx\,.
\end{align*}
Here $\delta$ is a small constant. $C_\delta$ is a constant depending on $\delta$.

For the sixth line of \eqref{K12}, we estimate
\begin{align*}
&\big|\int_{\Om_f} \langle w^+\rangle^{10\mu}\langle w^-\rangle^{2\mu}
\p_\tau(\p_j\La_{-}^i\p_i\La_{+}^j)\, \p_\tau p \dx \big| \\\nonumber
&\leq \| \langle w^+\rangle^{5\mu}\langle w^-\rangle^{\mu}
\p_\tau(\p_j\La_{-}^i\p_i\La_{+}^j)\|_{L^2(\Om_f)}
\| \langle w^+\rangle^{5\mu}\langle w^-\rangle^{\mu}
 \p_\tau p\|_{L^2(\Om_f)}\\\nonumber
&\leq E_3^{\f12} (G_2^b)^{\f12}(1+\|\Phi\|_{L^\infty(\Om_f)}) \| \langle w^+\rangle^{5\mu}\langle w^-\rangle^{\mu}
 \p_\tau p\|_{L^2(\Om_f)}\, .
\end{align*}
Similarly,
\begin{align*}
&\big| \int_{\hat{\Om}_f} \langle w^+\rangle^{10\mu}\langle w^-\rangle^{2\mu}
\p_\tau(\p_j\hat{\La}_{-}^i\p_i\hat{\La}_{+}^j)\, \p_\tau \hat{p} \dx \big| \\\nonumber
&\leq E_3^{\f12} (G_2^b)^{\f12}(1+\|\hat{\Phi}\|_{L^\infty(\hat{\Om}_f)}) \| \langle w^+\rangle^{5\mu}\langle w^-\rangle^{\mu}
 \p_\tau \hat{p}\|_{L^2(\hat{\Om}_f)}\, .
\end{align*}

\textbf{Estimate of the second line of \eqref{K12}.}

We write
\begin{align*} 
&\int_{\Ga_f} \langle w^+\rangle^{10\mu}\langle w^-\rangle^{2\mu}
\big(\p_{n_f}\p_\tau p\cdot \p_\tau p-\p_{n_f}\p_\tau \hat{p}\cdot \p_\tau \hat{p}\big) \,\d\sigma(x) \\\nonumber
&= \int_{\BR} \langle w^+\rangle^{10\mu}\langle w^-\rangle^{2\mu}
(\ud{N_f\cdot\nabla \p_\tau p}-\ud{N_f\cdot\nabla \p_\tau \hat{p}} )\cdot \ud{\p_\tau p} \,\d x_1\\\nonumber
&=\int_{\BR}\langle w^+\rangle^{10\mu}\langle w^-\rangle^{2\mu}
\p_1 \big(\ud{N_f\cdot\nabla  p}-\ud{N_f\cdot\nabla  \hat{p}} \big)\cdot \ud{\p_\tau p} \,\d x_1 \\\nonumber
&\quad- \int_{\BR} \langle w^+\rangle^{10\mu}\langle w^-\rangle^{2\mu}
  \Big( -\ud{[N_f\cdot\nabla ,\p_\tau] p}+\ud{[N_f\cdot\nabla ,\p_\tau] \hat{p}} \Big)\cdot \ud{\p_\tau p} \,\d x_1 \\\nonumber
&\leq  \big\|\langle w^+\rangle^{5\mu}\langle w^-\rangle^{\mu}
\langle\p_1\rangle^{-\f12} \p_1 \big(\ud{N_f\cdot\nabla  p}-\ud{N_f\cdot\nabla  \hat{p}}\big) \big\|_{L^2(\BR)} \\\nonumber
&\qquad\cdot \big\| \langle w^+\rangle^{-5\mu}\langle w^-\rangle^{-\mu} \langle\p_1\rangle^{\f12}  \big( \langle w^+\rangle^{10\mu}\langle w^-\rangle^{2\mu} \ud{\p_\tau p} \big)\big\|_{L^2(\BR)}\\\nonumber
&\quad + \big\| \langle w^+\rangle^{-5\mu}\langle w^-\rangle^{-\mu}
\langle \p_1\rangle^{\f 12}
\big(\langle w^+\rangle^{5\mu}\langle w^-\rangle^{\mu}
 ( -\ud{[N_f\cdot\nabla ,\p_\tau] p}+\ud{[N_f\cdot\nabla ,\p_\tau] \hat{p}})\big)\big\|_{L^2(\BR)} \\\nonumber
&\quad\quad\cdot  \|\langle w^+\rangle^{5\mu}\langle w^-\rangle^{\mu}
\langle \p_1\rangle^{-\f 12} \p_1\ud{p}  \|_{L^2(\BR)}\\\nonumber
&=I_1\cdot I_2 +I_3\cdot I_4.
\end{align*}
In the above calculation, we have used the fact that
\begin{align*}
&-\ud{N_f\cdot\nabla \p_\tau p}+\ud{N_f\cdot\nabla \p_\tau \hat{p}}\\
&=-\ud{\p_\tau N_f\cdot\nabla  p}+\ud{\p_\tau N_f\cdot\nabla  \hat{p}}-\ud{[N_f\cdot\nabla ,\p_\tau] p}+\ud{[N_f\cdot\nabla ,\p_\tau] \hat{p}}\\
&=-\p_1 \big(\ud{N_f\cdot\nabla  p}-\ud{N_f\cdot\nabla  \hat{p}} \big) -\ud{[N_f\cdot\nabla ,\p_\tau] p}+\ud{[N_f\cdot\nabla ,\p_\tau] \hat{p}} \, .
\end{align*}
For $I_1$, by Lemma \ref{lemW} and 
Lemma \ref{lemPFL}, we have
\begin{align*}
I_1&= \big\|\langle w^+\rangle^{5\mu}\langle w^-\rangle^{\mu}
\langle\p_1\rangle^{-\f12}\p_1 \big(\ud{N_f\cdot\nabla  p}-\ud{N_f\cdot\nabla  \hat{p}}\big) \big\|_{L^2(\BR)}\\
&\lesssim \big\|\langle w^+\rangle^{5\mu}\langle w^-\rangle^{\mu}
\langle\p_1\rangle^{\f12} \big(\ud{N_f\cdot\nabla  p}-\ud{N_f\cdot\nabla  \hat{p}}\big) \big\|_{L^2(\BR)}\\
&\leq E_2^{\f12} (G_2^b)^{\f12} \mathcal{P}(1+ \| \p_1f\|_{H^{\f52}}) \,.
\end{align*}
For $I_2$, 
 by  commutator Lemma \ref{lemD5} and weighted trace Lemma \ref{lemD3}, we imply
\begin{align*}
I_2=&
\big\| \langle w^+\rangle^{-5\mu}\langle w^-\rangle^{-\mu} \langle\p_1\rangle^{\f12}  \big( \langle w^+\rangle^{10\mu}\langle w^-\rangle^{2\mu} \ud{\p_\tau p} \big)\big\|_{L^2(\BR)}\\[-5mm]\\
\lesssim & \sum_{|\alpha|\leq 1}\big\| \langle w^+\rangle^{5\mu}\langle w^-\rangle^{ \mu} \p^\alpha\p_\tau  p\big\|_{L^2(\Om_f)}\lesssim \sum_{1\leq i\leq 2}\big\| \langle w^+\rangle^{5\mu}\langle w^-\rangle^{ \mu}
\nabla^i p  \big\|_{L^2(\Om_f)} (1+ \| \p_1f\|_{L^\infty}).
\end{align*}
To estimate $I_3$, we calculate the commutator in $\Om_f$
\begin{align*}
[N_f\cdot\nabla ,\p_\tau]
&=[\p_2-\p_1f \p_1,\Phi\p_2+\p_1]\\
&=\left( \p_2\p_1+\Phi \p_2^2+\p_2\Phi\p_2
-\p_1f(\p_1^2+\Phi\p_1\p_2+\p_1\Phi\p_2  ) \right) \\
&\quad-\left( \p_2\p_1+\Phi \p_2^2-\Phi\p_1f\p_2\p_1
-\p_1f\p_1^2 -\p_1^2f\p_1 \right)\\
&=\p_2\Phi\p_2-\p_1f\p_1\Phi\p_2+\p_1^2f\p_1,
\end{align*}
which turns out to be one order operator. Similar calculation for $[N_f\cdot\nabla ,\p_\tau]$ also holds in $\hat{\Om}_f$.
Consequently, by Lemma \ref{lemW}, weighted trace Lemma \ref{lemD3} and Lemma \ref{lemPsi},
 we infer that
\begin{align*}
I_3\cdot I_4
\lesssim&
\Big(\| \langle w^+\rangle^{2\mu}\langle w^-\rangle^{\mu} \nabla^{\leq 1}
[N_f\cdot\nabla ,\p_\tau] p \|_{L^2(\Om_f)}
+\| \langle w^+\rangle^{2\mu}\langle w^-\rangle^{\mu} \nabla^{\leq 1}
[N_f\cdot\nabla ,\p_\tau] \hat{p}\|_{L^2(\hat{\Om}_f)}\Big)\\
&\cdot \|\langle w^+\rangle^{5\mu}\langle w^-\rangle^{\mu} \nabla^{\leq 1}p \|_{L^2(\Om_f)} \\[-5mm]\\
\lesssim&
\sum_{1\leq i\leq 2}\big( \big\| \langle w^+\rangle^{5\mu}\langle w^-\rangle^{ \mu}
\nabla^i  p  \big\|_{L^2(\Om_f)}
+\big\| \langle w^+\rangle^{5\mu}\langle w^-\rangle^{ \mu}
\nabla^i  \hat{p}  \big\|_{L^2(\hat{\Om}_f)} \big) \\
&\,\, \cdot  \big\| \langle w^+\rangle^{5\mu}\langle w^-\rangle^{ \mu}
\nabla^{\leq 1}  p  \big\|_{L^2(\Om_f)}
 \| \p_1f\|_{H^{\f52}(\bR)}  ( 1+\| \p_1f\|_{H^{\f52}(\bR)}.
\end{align*}

\textbf{Estimate of the third line of \eqref{K12}}
One has
\begin{align*}
&\int_{\hat{\Ga}}  \langle w^+\rangle^{10\mu}\langle w^-\rangle^{2\mu}
\p_2\p_1 \hat{p}\cdot \p_1 \hat{p}\, \d x_1
-\int_{\Ga}  \langle w^+\rangle^{10\mu}\langle w^-\rangle^{2\mu}
\p_2\p_1 p\cdot \p_1 p\, \d x_1=0.
\end{align*}

\textbf{Estimate of the fifth line of \eqref{K12}}

We have
\begin{align*} 
&\big| \int_{\Om_f} \langle w^+\rangle^{10\mu}\langle w^-\rangle^{2\mu}
\left[\p_\tau, \Delta\right]p\, \p_\tau p \dx
+\int_{\hat{\Om}_f} \langle w^+\rangle^{10\mu}\langle w^-\rangle^{2\mu}
\left[\p_\tau, \Delta\right]\hat{p}\, \p_\tau \hat{p} \dx  \big| \\\nonumber
&\leq \| \langle w^+\rangle^{5\mu}\langle w^-\rangle^{\mu} \left[\p_\tau, \Delta\right]p \|_{L^2(\Om_f)}
\| \langle w^+\rangle^{5\mu}\langle w^-\rangle^{\mu}  \p_\tau p \|_{L^2(\Om_f)}
\\[-4mm]\nonumber\\\nonumber
&\quad+ \| \langle w^+\rangle^{5\mu}\langle w^-\rangle^{\mu} \left[\p_\tau, \Delta\right]\hat{p} \|_{L^2(\hat{\Om}_f)}
\| \langle w^+\rangle^{5\mu}\langle w^-\rangle^{\mu}  \p_\tau \hat{p} \|_{L^2(\hat{\Om}_f)}\,.
\end{align*}
Now we calculate the commutator of Laplacian and $\p_\tau$ in $\Om_f$.
Firstly,
\begin{align*}
[\p,\p_\tau]&=\p \Phi \p_2 ,\\
[\p_1^2,\p_\tau]&=\p_1 [\p_1,\p_\tau]+ [\p_1,\p_\tau]\p_1
=\p_1^2 \Phi \p_2 +2\p_1 \Phi \p_1\p_2 ,\\
[\p_2^2,\p_\tau]&=\p_2 [\p_2,\p_\tau]+ [\p_2,\p_\tau]\p_2
=\p_2^2 \Phi \p_2 +2\p_2 \Phi \p^2_2.
\end{align*}
Consequently,
\begin{align*}
[\Delta,\p_\tau]
= \Delta  \Phi  \p_2
 +2 \nabla\Phi\cdot \p_2\nabla.
\end{align*}
Thus
\begin{align*} 
&\| \langle w^+\rangle^{5\mu}\langle w^-\rangle^{\mu} \left[\p_\tau, \Delta\right]p \|_{L^2(\Om_f)}\\\nonumber
&\leq  \|\nabla^{\leq 2} \Phi\|_{L^4(\Om_f)}
 \| \langle w^+\rangle^{5\mu}\langle w^-\rangle^{\mu} \nabla p\|_{L^4(\Om_f)}
 \\\nonumber
&\quad+  \| \nabla \Phi\|_{L^\infty(\Om_f)}
 \| \langle w^+\rangle^{5\mu}\langle w^-\rangle^{\mu} \nabla^2 p\|_{L^2(\Om_f)}
 \\\nonumber
&\lesssim \| \nabla^{\leq 3}\Phi\|_{L^2(\Om_f)}
 \sum_{1\leq j \leq 2} \| \langle w^+\rangle^{2\mu}\langle w^-\rangle^{\mu}
 \nabla^j p\|_{L^2(\Om_f)} .
\end{align*}
Consequently, 
\begin{align*} 
&\big| \int_{\Om_f}\langle w^+\rangle^{10\mu}\langle w^-\rangle^{2\mu} \left[\p_\tau, \Delta\right]p\, \p_\tau p \dx \big| \\\nonumber
&\lesssim  \| \nabla^{\leq 3}\Phi\|_{L^2(\Om_f)}
 \sum_{1\leq j \leq 2} \| \langle w^+\rangle^{2\mu}\langle w^-\rangle^{\mu}
 \nabla^j p\|_{L^2(\Om_f)}
\| \langle w^+\rangle^{2\mu}\langle w^-\rangle^{\mu} \p_\tau p\|_{L^2(\Om_f)} .
\end{align*}
Similarly,
\begin{align*}
&\big| \int_{\hat{\Om}_f}\langle w^+\rangle^{10\mu}\langle w^-\rangle^{2\mu} \left[\p_\tau, \Delta\right]\hat{p}\, \p_\tau \hat{p} \dx \big| \\\nonumber
&\lesssim
\| \nabla^{\leq 3}\Phi\|_{L^2(\hat{\Om}_f)}
 \sum_{1\leq j \leq 2} \| \langle w^+\rangle^{2\mu}\langle w^-\rangle^{\mu}
 \nabla^j \hat{p}\|_{L^2(\hat{\Om}_f)}
\| \langle w^+\rangle^{2\mu}\langle w^-\rangle^{\mu} \p_\tau \hat{p}\|_{L^2(\hat{\Om}_f)} .
\end{align*}
Gathering the above estimates, by Lemma \ref{lemPsi} and Lemma \ref{lemP1}, \eqref{K12} reduces to
\begin{align}\label{K25}
&\|\langle w^+\rangle^{5\mu}\langle w^-\rangle^{\mu}\nabla\p_\tau p\|^2_{L^2(\Om_f)}
+\|\langle w^+\rangle^{5\mu}\langle w^-\rangle^{\mu}\nabla\p_\tau \hat{p}\|^2_{L^2(\hat{\Om}_f)}\\\nonumber
&\lesssim \big(\| \langle w^+\rangle^{5\mu}\langle w^-\rangle^{\mu}  \nabla^2 p\|_{L^2(\Om_f)}+\| \langle w^+\rangle^{5\mu}\langle w^-\rangle^{\mu}  \nabla^2 \hat{p}\|_{L^2(\hat{\Om}_f)} \big)
E_3^{\f12}(G_2^b)^{\f12} \mathcal{P}(1+\|\p_1 f\|_{H^3} )\\\nonumber
&\quad+ E_3 G_2^b \mathcal{P}(1+ \|\p_1 f\|_{H^{\f52}} ).
\end{align}
\textbf{Second normal derivative estimate of the pressure}

We first calculate some commutators in $\Om_f$.
It is easy to see that
\begin{align*} 
[\nabla,\p_\tau]
&=\nabla \Phi\p_2 ,\\
|\p_1 \nabla p|
&\leq   |\p_\tau \nabla p|+ |\Phi \p_2 \nabla p| \\\nonumber
&\leq |\nabla \p_\tau  p|+ |[\p_\tau, \nabla] p| + |\Phi \p_2 \nabla p| \\\nonumber
&\leq |\nabla \p_\tau  p|+ |\nabla \Phi \p_2 p| + |\Phi \p_2 \nabla p| \, .
\end{align*}
Consequently, we have
\begin{align}\label{K27}
&\|\langle w^+\rangle^{5\mu}\langle w^-\rangle^{\mu} \nabla\p_1 p \|^2_{L^2(\Om_f)}\\\nonumber
&\lesssim \|\langle w^+\rangle^{5\mu}\langle w^-\rangle^{\mu} \nabla \p_\tau p \|^2_{L^2(\Om_f)}
+\| \langle w^+\rangle^{5\mu}\langle w^-\rangle^{\mu} \p_2 p \|^2_{L^4(\Om_f)}
\| \nabla \Phi \|^2_{L^4(\Om_f)}\\\nonumber
&\quad+\| \langle w^+\rangle^{5\mu}\langle w^-\rangle^{\mu}  \p_2 \nabla p \|^2_{L^2(\Om_f)}
\| \Phi \|^2_{L^\infty(\Om_f)}.
\end{align}

Now we treat the second normal derivative estimate of the pressure.
Rewriting $\eqref{H0}_1$ as follows
\begin{align}\label{K28}
\p_2^2 p=-\p_1^2 p -\nabla_j\La_{-}^i \nabla_i\La_{+}^j.
\end{align}
Multiplying \eqref{K28} by $ \langle w^+\rangle^{5\mu}\langle w^-\rangle^{\mu}$,
then taking the $L^2(\Om_f)$ norm of the consequent equation yields
\begin{align}\label{K29}
\big\| \langle w^+\rangle^{5\mu}\langle w^-\rangle^{\mu}
\p_2^2 p \big\|_{L^2(\Om_f)}
\leq \big\| \langle w^+\rangle^{5\mu}\langle w^-\rangle^{\mu}
 \p_1^2 p \big\|_{L^2(\Om_f)}
+C (G_2^b)^{\f12} E_2^{\f12}.
\end{align}
Consequently, by \eqref{K25}, \eqref{K27}, \eqref{K29}, Lemma \ref{lemPsi} and Lemma \ref{lemP1}, we derive
\begin{align}\label{K31}
&\| \langle w^+\rangle^{5\mu}\langle w^-\rangle^{\mu} \nabla^2 p \|^2_{L^2(\Om_f)}
\\\nonumber
&=\|\langle w^+\rangle^{5\mu}\langle w^-\rangle^{\mu} \p_2^2 p \|^2_{L^2(\Om_f)}
+\| \langle w^+\rangle^{5\mu}\langle w^-\rangle^{\mu} \nabla\p_1 p \|^2_{L^2(\Om_f)} \\\nonumber
&\leq C\big\| \langle w^+\rangle^{5\mu}\langle w^-\rangle^{\mu}
\nabla\p_1 p \big\|^2_{L^2(\Om_f)} +C G_2^b E_2  \\\nonumber
&\leq C\big(\| \langle w^+\rangle^{5\mu}\langle w^-\rangle^{\mu}  \nabla^2 p\|_{L^2(\Om_f)}
           +\| \langle w^+\rangle^{5\mu}\langle w^-\rangle^{\mu}  \nabla^2 \hat{p}\|_{L^2(\hat{\Om}_f)}
\big) E_3^{\f12} (G_2^b)^{\f12} \mathcal{P}(1+ \|\p_1 f\|_{H^{\f52}} )  \\\nonumber
&\quad+\|  \langle w^+\rangle^{5\mu}\langle w^-\rangle^{\mu}  \p_2\nabla p \|^2_{L^2(\Om_f)}\|\p_1 f \|^2_{H^{\f32}}
+ CE_3 G_2^b \mathcal{P}(1+ \|\p_1 f\|_{H^{\f52}} ).
 \end{align}
Similarly,
\begin{align}\label{K32}
&\| \langle w^+\rangle^{5\mu}\langle w^-\rangle^{\mu} \nabla^2 \hat{p} \|^2_{L^2(\hat{\Om}_f)}\\\nonumber
&\leq C\big(\| \langle w^+\rangle^{5\mu}\langle w^-\rangle^{\mu}  \nabla^2 p\|_{L^2(\Om_f)}+\| \langle w^+\rangle^{5\mu}\langle w^-\rangle^{\mu}  \nabla^2 \hat{p}\|_{L^2(\hat{\Om}_f)}\big)
E_3^{\f12} G_2^{\f12} \mathcal{P}(1+ \|\p_1 f\|_{H^{\f52}}) \\\nonumber
&\quad+\| \langle w^+\rangle^{5\mu}\langle w^-\rangle^{\mu}  \p_2 \nabla \hat{p} \|^2_{L^2(\hat{\Om}_f)}\|\p_1 f \|^2_{H^{\f32}}
+ CE_3 G_2 \mathcal{P}(1+ \|\p_1 f\|_{H^{\f52}}  ).
 \end{align}
Combining \eqref{K31} and \eqref{K32}, by the H\"older inequality,  note the assumption $\|\p_1 f \|_{H^{\f32}}\leq\delta\ll 1$, we obtain
\begin{align*}
\|  \langle w^+\rangle^{5\mu}\langle w^-\rangle^{\mu} \nabla^2 p \|^2_{L^2(\Om_f)}
+\|  \langle w^+\rangle^{5\mu}\langle w^-\rangle^{\mu} \nabla^2 \hat{p} \|^2_{L^2(\hat{\Om}_f)}
\leq CE_3 G_2^b \mathcal{P}(1+ \|\p_1 f\|_{H^{\f52}} ).
\end{align*}
Thus we obtain $2$-order derivatives estimate of the pressure
in Lemma \ref{lemPre}.

\subsubsection{Three order estimate of the pressure}
Applying $\p_\tau^2$ to $\eqref{H0}_1$ and $\eqref{H0}_2$, one obtains
\begin{align}
\label{K50}
-\Delta\p_\tau^2 p -\left[\p_\tau^2, \Delta\right] p =\p_\tau^2\nabla \cdot (\La_{-} \cdot \nabla \La_{+})
\quad \text{in}\,\,\Om_f,\\
\label{K51}
-\Delta\p_\tau^2 \hat{p} -\left[\p_\tau^2, \Delta\right] \hat{p} =\p_\tau^2\nabla \cdot (\hat{\La}_{-} \cdot \nabla \hat{\La}_{+})
\quad \text{in}\,\,\hat{\Om}_f.
\end{align}
Taking the $L^2$ inner product of \eqref{K50} with
$ \langle w^+\rangle^{10\mu}\langle w^-\rangle^{2\mu} \p_\tau^2 p$ in $\Om_f$
and taking the $L^2$ inner product of \eqref{K51} with
$ \langle w^+\rangle^{10\mu}\langle w^-\rangle^{2\mu} \p_\tau^2 \hat{p}$ in $\hat{\Om}_f$,
by integration by parts, we have
\begin{align}\label{K52}
&\int_{\Om_f} \langle w^+\rangle^{10\mu}\langle w^-\rangle^{2\mu}
|\nabla\p_\tau^2 p|^2 \dx
+\int_{\hat{\Om}_f} \langle w^+\rangle^{10\mu}\langle w^-\rangle^{2\mu}
|\nabla\p_\tau^2 \hat{p}|^2 \dx \\\nonumber
&=\int_{\Ga_f} \langle w^+\rangle^{10\mu}\langle w^-\rangle^{2\mu}
\big(\p_{n_f}\p_\tau^2 p\cdot \p_\tau^2 p-\p_{n_f}\p_\tau^2 \hat{p}\cdot \p_\tau^2 \hat{p}\big) \,\d\sigma(x)\\\nonumber
&\quad +\int_{\hat{\Ga}} \langle w^+\rangle^{10\mu}\langle w^-\rangle^{2\mu}
\p_2\p_\tau^2 \hat{p}\cdot \p_\tau^2 \hat{p} \,\d\sigma(x)
-\int_{\Ga} \langle w^+\rangle^{10\mu}\langle w^-\rangle^{2\mu}
\p_2\p_\tau^2 p\cdot \p_\tau^2 p \,\d\sigma(x) \\\nonumber
&\quad-\int_{\Om_f} \nabla \big(\langle w^+\rangle^{10\mu}\langle w^-\rangle^{2\mu} \big)
\cdot \nabla\p_\tau^2 p\, \p_\tau^2 p \dx
-\int_{\hat{\Om}_f} \nabla \big(\langle w^+\rangle^{10\mu}\langle w^-\rangle^{2\mu} \big)
\cdot \nabla\p_\tau^2 \hat{p}\, \p_\tau^2 \hat{p} \dx \\\nonumber
&\quad+\int_{\Om_f}\langle w^+\rangle^{10\mu}\langle w^-\rangle^{2\mu}
\left[\p_\tau^2, \Delta\right]p\, \p_\tau^2 p \dx
+\int_{\hat{\Om}_f}\langle w^+\rangle^{10\mu}\langle w^-\rangle^{2\mu}
\left[\p_\tau^2, \Delta\right]\hat{p}\, \p_\tau^2 \hat{p} \dx  \\\nonumber
&\quad+\int_{\Om_f}\langle w^+\rangle^{10\mu}\langle w^-\rangle^{2\mu}
\p_\tau^2(\p_j\La_{-}^i\p_i\La_{+}^j)\, \p_\tau^2 p \dx  \\\nonumber
&\quad +\int_{\hat{\Om}_f}\langle w^+\rangle^{10\mu}\langle w^-\rangle^{2\mu}
\p_\tau^2(\p_j\hat{\La}_{-}^i\p_i\hat{\La}_{+}^j)\, \p_\tau^2 \hat{p} \dx \, .
\end{align}

For the fourth line of \eqref{K52}, one has
\begin{align*}
&\big| \int_{\Om_f} \nabla_i \big( \langle w^+\rangle^{10\mu}\langle w^-\rangle^{2\mu}\big)\nabla_i \p_\tau^2 p\cdot \p_\tau^2 p \dx
+\int_{\hat{\Om}_f} \nabla_i \big(\langle w^+\rangle^{10\mu}\langle w^-\rangle^{2\mu}\big)\nabla_i \p_\tau^2 \hat{p}\cdot \p_\tau^2 \hat{p} \dx \big| \\
&\leq \delta  \|\langle w^+\rangle^{5\mu}\langle w^-\rangle^{\mu} \nabla\p_\tau^2 p\|^2_{L^2(\Om_f)}
+C_\delta \| \langle w^+\rangle^{5\mu}\langle w^-\rangle^{\mu} \p_\tau^2 p\|^2_{L^2(\Om_f)}\\[-5mm]\\
&\quad+ \delta  \|\langle w^+\rangle^{5\mu}\langle w^-\rangle^{\mu} \nabla\p_\tau^2 \hat{p} \|^2_{L^2(\hat{\Om}_f)}
+C_\delta \| \langle w^+\rangle^{5\mu}\langle w^-\rangle^{\mu} \p_\tau^2 \hat{p}\|_{L^2(\hat{\Om}_f)}\,.
\end{align*}
Here $\delta$ is a small constant. $C_\delta$ is a constant depending on $\delta$.

For the sixth and the seventh line of \eqref{K52}, we estimate
\begin{align*}
&\big|\int_{\Om_f}\langle w^+\rangle^{10\mu}\langle w^-\rangle^{2\mu}
\p_\tau^2(\p_j\La_{-}^i\p_i\La_{+}^j)\, \p_\tau^2 p \dx \big| \\\nonumber
&\leq E_3^{\f12} G_3^{\f12}(1+\|\nabla^{\leq 3}\Phi\|_{L^2(\Om_f)}) \| \langle w^+\rangle^{5\mu}\langle w^-\rangle^{\mu}
 \p_\tau^2 p\|_{L^2(\Om_f)}, \\\nonumber
&\big|\int_{\hat{\Om}_f} \langle w^+\rangle^{10\mu}\langle w^-\rangle^{2\mu}
\p_\tau^2(\p_j\hat{\La}_{-}^i\p_i\hat{\La}_{+}^j)\, \p_\tau^2 \hat{p} \dx \big| \\\nonumber
&\leq E_3^{\f12} G_3^{\f12}(1+\|\nabla^{\leq 3}\hat{\Phi}\|_{L^2(\hat{\Om}_f)}) \| \langle w^+\rangle^{5\mu}\langle w^-\rangle^{\mu}
  \p_\tau^2 \hat{p}\|_{L^2(\hat{\Om}_f)}\, .
\end{align*}

For the second line of \eqref{K52}, in order to use the boundary condition of the pressure, we write
\begin{align}\label{K54}
& \int_{\Ga_f} \langle w^+\rangle^{10\mu}\langle w^-\rangle^{2\mu}
\big(\p_{n_f}\p_\tau^2 p\cdot \p_\tau^2 p-\p_{n_f}\p_\tau^2 \hat{p}\cdot \p_\tau^2 \hat{p}\big) \,\d\sigma(x)  \\\nonumber
&= \int_{\BR}\langle w^+\rangle^{10\mu}\langle w^-\rangle^{2\mu}
\p_1^2 \big(\ud{N_f\cdot\nabla  p}-\ud{N_f\cdot\nabla  \hat{p}} \big)\cdot \ud{\p_\tau^2 p} \,\d x_1 \\\nonumber
&\quad- \int_{\BR} \langle w^+\rangle^{10\mu}\langle w^-\rangle^{2\mu}
  \Big( -\ud{[N_f\cdot\nabla ,\p_\tau^2] p}+\ud{[N_f\cdot\nabla ,\p_\tau^2] \hat{p}} \Big)\cdot \ud{\p_\tau^2 p} \,\d x_1\,.
\end{align}
For the second line of \eqref{K54}, we will use the boundary condition.
By Lemma \ref{lemW}, 
Lemma \ref{lemPFH},
the commutator Lemma \ref{lemD5} and weighted trace Lemma \ref{lemD3}, it is bounded by
\begin{align*}
& \big\|\langle w^+\rangle^{5\mu}\langle w^-\rangle^{\mu}
\langle\p_1\rangle^{-\f12}\p_1^2 \big(\ud{N_f\cdot\nabla  p}-\ud{N_f\cdot\nabla  \hat{p}}\big) \big\|_{L^2(\BR)}\\
&\cdot\big\| \langle w^+\rangle^{-5\mu}\langle w^-\rangle^{-\mu} \langle\p_1\rangle^{\f12}  \big( \langle w^+\rangle^{10\mu}\langle w^-\rangle^{2\mu} \ud{\p_\tau^2 p} \big)\big\|_{L^2(\BR)}\\
&\lesssim \big\|\langle w^+\rangle^{5\mu}\langle w^-\rangle^{\mu}
\langle\p_1\rangle^{\f12} \p_1 \big(\ud{N_f\cdot\nabla  p}-\ud{N_f\cdot\nabla  \hat{p}}\big) \big\|_{L^2(\BR)}
\big\| \langle w^+\rangle^{5\mu}\langle w^-\rangle^{ \mu} \p^{\leq 1}
\p_\tau^2  p\big\|_{L^2(\Om_f)}\\
&\lesssim E_3^{\f12} G_3^{\f12}\cdot \big\| \langle w^+\rangle^{5\mu}\langle w^-\rangle^{ \mu}
\nabla^{1\leq a\leq 3}  p  \big\|_{L^2(\Om_f)} (1+ \| \p_1 f\|_{H^2} ) \,.
\end{align*}

To estimate the last line of \eqref{K54}, we calculate the commutator in $\Om_f$:
\begin{align*}
[N_f\cdot\nabla ,\p_\tau]
&=\p_2\Phi\p_2-\p_1f\p_1\Phi\p_2+\p_1^2f\p_1,\\
[N_f\cdot\nabla ,\p_\tau^2]&=[N_f\cdot\nabla ,\p_\tau]\p_\tau+\p_\tau[N_f\cdot\nabla ,\p_\tau]\\
&=\left(\p_2\Phi\p_2-\p_1f\p_1\Phi\p_2+\p_1^2f\p_1 \right)\p_\tau
+\p_\tau \left( \p_2\Phi\p_2-\p_1f\p_1\Phi\p_2+\p_1^2f\p_1 \right)\\
&= \left(\p_2\Phi\p_2-\p_1f\p_1\Phi\p_2+\p_1^2f\p_1 \right)\p_\tau
+\p_\tau \left( \p_2\Phi\p_2-\p_1f\p_1\Phi\p_2 \right)+\p_1^2f\p_\tau \p_1
+\p_1^3f \p_1\\
&=D_{\Phi,f}^3+\p_1^3f \p_1.
\end{align*}
Here $D_{\Phi,f}^3$ is second order operator involving $\Phi$ and $f$.
Similar calculation for $[N_f\cdot\nabla ,\p_\tau^2]$ also holds in $\hat{\Om}_f$.
Consequently, by Lemma \ref{lemW} and weighted trace Lemma \ref{lemD3},
we infer that
\begin{align*}
&\big| \int_{\BR} \langle w^+\rangle^{10\mu}\langle w^-\rangle^{2\mu}\cdot \ud{[N_f\cdot\nabla ,\p_\tau^2] p}\cdot \ud{\p_\tau^2 p} \,\d x_1 \big|\\
&=\big| \int_{\BR} \langle w^+\rangle^{10\mu}\langle w^-\rangle^{2\mu}\cdot \ud{D_{\Phi,f}^3p+\p_1^3f \p_1 p}\cdot \ud{\p_\tau^2 p} \,\d x_1 \big|\\
&\leq \big\| \langle w^+\rangle^{-5\mu}\langle w^-\rangle^{-\mu}
\langle \p_1\rangle^{\f12}\big( \langle w^+\rangle^{10\mu}\langle w^-\rangle^{2\mu}
 \ud{D_{\Phi,f}^3p} \big)\big\|_{L^2(\BR)}\cdot \big\|\langle w^+\rangle^{5\mu}\langle w^-\rangle^{\mu}
 \langle \p_1\rangle^{-\f12}\p_1\ud{\p_\tau p} \big\|_{L^2(\BR)}\\\nonumber
&\quad +\| \p_1^3 f\|_{L^2(\bR)}
\sum_{1\leq j\leq 2} \big\| \langle w^+\rangle^{5\mu}\langle w^-\rangle^{ \mu}\nabla^j  p  \big\|_{L^2(\Om_f)}
\| \langle w^+\rangle^{5\mu}\langle w^-\rangle^{\mu}\nabla^{\leq 1}\p_\tau^2 p  \|_{L^2(\Om_f)} \\\nonumber
&\lesssim
\big\| \langle w^+\rangle^{5\mu}\langle w^-\rangle^{ \mu}
\nabla^{\leq 3}  p  \big\|_{L^2(\Om_f)}
 \| \langle w^+\rangle^{5\mu}\langle w^-\rangle^{\mu}\nabla^{\leq 2}  p  \|_{L^2(\Om_f)}
\| \p_1f\|_{H^{\f52}(\bR)} \mathcal{P}(1+\| \p_1f\|_{H^{\f52}(\bR)}).
\end{align*}
Similarly,
\begin{align*}
&\big| \int_{\BR} \langle w^+\rangle^{10\mu}\langle w^-\rangle^{2\mu}\cdot \ud{[N_f\cdot\nabla ,\p_\tau^2] \hat{p}}\cdot \ud{\p_\tau^2 p} \,\d x_1 \big|\\
&\lesssim
\big\| \langle w^+\rangle^{5\mu}\langle w^-\rangle^{ \mu}
\nabla^{\leq 3}  \hat{p}  \big\|_{L^2(\hat{\Om}_f)}
 \| \langle w^+\rangle^{5\mu}\langle w^-\rangle^{\mu}\nabla^{\leq 2}  \hat{p} \|_{L^2(\hat{\Om}_f)}
\| \p_1f\|_{H^{\f52}(\bR)} \mathcal{P}(1+\| \p_1f\|_{H^{\f52}(\bR)}).
\end{align*}


For the third line of \eqref{K52}, due to the boundary condition of the pressure, it vanishes:
\begin{align*}
&\int_{\hat{\Ga}}  \langle w^+\rangle^{10\mu}\langle w^-\rangle^{2\mu}
\p_2\p_1^2 \hat{p}\cdot \p_1^2 \hat{p}\, \d x_1
-\int_{\Ga}  \langle w^+\rangle^{10\mu}\langle w^-\rangle^{2\mu}
\p_2\p_1^2 p\cdot \p_1^2 p\, \d x_1=0\,.
\end{align*}


\textbf{Estimate of the fifth line of \eqref{K52}}

We first calculate the commutator of Laplacian and $\p_\tau^2$ in $\Om_f$.
Firstly,
\begin{align*}
[\nabla,\p_\tau]&=\nabla \Phi \p_2 ,\\
[\p_2^2,\p_\tau]&=\p_2 [\p_2,\p_\tau]+ [\p_2,\p_\tau]\p_2
=\p_2^2 \Phi \p_2 +2\p_2 \Phi \p^2_2,\\
[\p_1^2,\p_\tau]&=\p_1 [\p_1,\p_\tau]+ [\p_1,\p_\tau]\p_1
=\p_1^2 \Phi \p_2 +2\p_1 \Phi \p_2\p_1 . 
\end{align*}
Then
\begin{align*}
[\Delta,\p_\tau]
&=\Delta \Phi \p_2
 +2 \nabla \Phi\cdot \nabla\p_2 , \\
[\Delta,\p_\tau^2]
&=[\Delta,\p_\tau]\p_\tau+ \p_\tau[\Delta,\p_\tau] \, \\
&=\big( \Delta \Phi \p_2
 +2 \nabla \Phi\cdot \nabla\p_2 \big)\p_\tau
 +\p_\tau \big( \Delta \Phi \p_2
 +2 \nabla \Phi\cdot \nabla\p_2 \big),
\end{align*}
which is a third order operator. Hence we obtain
\begin{align*}
&|[\Delta,\p_\tau]p|\lesssim
|\Delta\Phi| \cdot|\p_2p|
 +|\nabla\Phi \nabla^2 p|\, \\
&|[\Delta,\p_\tau^2]p|\lesssim(1+|\Phi|)(1+|\nabla^{\leq 1}\Phi|)
\big( |\nabla\Phi \nabla^3p|+ |\nabla^{\leq 2}\Phi \nabla^2p|
+ |\nabla^{\leq 3}\Phi \nabla p| \big) \, .
\end{align*}
Thus
\begin{align*}
&\|  \langle w^+\rangle^{5\mu}\langle w^-\rangle^{\mu} \left[\p_\tau^2, \Delta\right]p \|_{L^2(\Om_f)}\\\nonumber
&\leq \big(1+\|\nabla^{\leq 1} \Phi\|_{L^\infty(\Om_f)}\big)^2
\Big(\|\nabla^{\leq 3} \Phi\|_{L^2(\Om_f)}
 \|  \langle w^+\rangle^{5\mu}\langle w^-\rangle^{\mu} \nabla p\|_{L^\infty(\Om_f)} \\\nonumber
&\quad+  \|\nabla^{\leq 2} \Phi\|_{L^4(\Om_f)}
 \|  \langle w^+\rangle^{5\mu}\langle w^-\rangle^{\mu} \nabla^2 p\|_{L^4(\Om_f)}
 + \|\nabla\Phi\|_{L^\infty(\Om_f)}
 \|  \langle w^+\rangle^{5\mu}\langle w^-\rangle^{\mu} \nabla^3 p\|_{L^2(\Om_f)} \Big) \\\nonumber
&\lesssim \big(1+ \|\nabla^{\leq 3} \Phi\|_{L^2(\Om_f)}\big)^2
 \| \nabla^{\leq 3}\Phi\|_{L^2(\Om_f)}
 \sum_{1\leq j \leq 3} \|   \langle w^+\rangle^{5\mu}\langle w^-\rangle^{\mu}
 \nabla^j p\|_{L^2(\Om_f)} .
\end{align*}
By the second order derivative estimate of the pressure, we have
\begin{align*}
&\big| \int_{\Om_f} \langle w^+\rangle^{10\mu}\langle w^-\rangle^{2\mu}
\left[\p_\tau^2, \Delta\right]p\, \p_\tau^2 p \dx \\\nonumber
&\leq \|  \langle w^+\rangle^{5\mu}\langle w^-\rangle^{\mu} \left[\p_\tau^2, \Delta\right]p \|_{L^2(\Om_f)}
\| \langle w^+\rangle^{5\mu}\langle w^-\rangle^{\mu}  \p_\tau^2 p \|_{L^2(\Om_f)}
\\[-4mm]\nonumber\\\nonumber
&\lesssim \sum_{1\leq j \leq 3} \|   \langle w^+\rangle^{5\mu}\langle w^-\rangle^{\mu}
 \nabla^j p\|_{L^2(\Om_f)}  E_3^{\f12} G_2^{\f12} \|\nabla^{\leq 3} \Phi\|_{L^2(\Om_f)} F (1+ \|\nabla^{\leq 3} \Phi\|_{L^2(\Om_f)} ).
\end{align*}
Note the similar third weighted derivative estimate also holds for $\hat{p}$.
\begin{align*}
&\big| \int_{\hat{\Om}_f} \langle w^+\rangle^{10\mu}\langle w^-\rangle^{2\mu}
\left[\p_\tau^2, \Delta\right]\hat{p}\, \p_\tau^2 \hat{p} \dx \\\nonumber
&\lesssim
 \sum_{1\leq j \leq 3} \|   \langle w^+\rangle^{5\mu}\langle w^-\rangle^{\mu}
 \nabla^j \hat{p}\|_{L^2(\hat{\Om}_f)}  E_3^{\f12} G_2^{\f12} \|\nabla^{\leq 3} \hat{\Phi}\|_{L^2(\hat{\Om}_f)} F (1+ \|\nabla^{\leq 3} \hat{\Phi}\|_{L^2(\hat{\Om}_f)} ).
\end{align*}

Gathering the above estimates, by Lemma \ref{lemPsi}, the first and the second order weighted estimate of the pressure,
 \eqref{K52} reduces to
\begin{align}\label{K59}
&\|\langle w^+\rangle^{5\mu}\langle w^-\rangle^{\mu} \nabla\p_\tau^2 p\|^2_{L^2(\Om_f)}
+\|\langle w^+\rangle^{5\mu}\langle w^-\rangle^{\mu} \nabla\p_\tau^2 \hat{p}\|^2_{L^2(\hat{\Om}_f)} \\\nonumber
&\lesssim
\big(\|  \langle w^+\rangle^{5\mu}\langle w^-\rangle^{\mu}  \nabla^3 p\|_{L^2(\Om_f)}
+   \|  \langle w^+\rangle^{5\mu}\langle w^-\rangle^{\mu}  \nabla^3 \hat{p}\|_{L^2(\hat{\Om}_f)}\big)
E_3^{\f12} G_3^{\f12} \mathcal{P}(1+ \| \p_1 f\|_{H^{\f52}}) \\\nonumber
&\quad+ E_3 G_3 \mathcal{P}(1+ \| \p_1 f\|_{H^{\f52}} )\, .
\end{align}

\textbf{Third order normal derivative estimate}
Now we treat the third normal derivative estimate of the pressure.
Applying $\nabla$ on $\eqref{H0}_1$, one has
\begin{align}\label{K60}
\nabla \p_2^2 p=-\nabla\p_1^2 p-\nabla(\nabla_j\La_{-}^i \nabla_i\La_{+}^j).
\end{align}
Multiplying \eqref{K60} by $ \langle w^+\rangle^{5\mu}\langle w^-\rangle^{\mu}$,
then taking the $L^2(\Om_f)$ norm of the consequent equation, we obtain
\begin{align*}
\big\| \langle w^+\rangle^{5\mu}\langle w^-\rangle^{\mu}
\nabla\p_2^2 p \big\|_{L^2(\Om_f)}
\leq \big\|\langle w^+\rangle^{5\mu}\langle w^-\rangle^{\mu}
\nabla\p_1^2 p \big\|_{L^2 }
+C G_3^{\f12} E_3^{\f12} \,.
\end{align*}
Consequently,
\begin{align}\label{K63}
&\| \langle w^+\rangle^{5\mu}\langle w^-\rangle^{\mu} \nabla^3 p \|^2_{L^2(\Om_f)}
\\\nonumber
&=\| \langle w^+\rangle^{5\mu}\langle w^-\rangle^{\mu} \nabla  \p_2^2 p \|^2_{L^2(\Om_f)}
+ \| \langle w^+\rangle^{5\mu}\langle w^-\rangle^{\mu} \nabla  \p_1^2 p \|^2_{L^2(\Om_f)}\\\nonumber
&\lesssim  \| \langle w^+\rangle^{5\mu}\langle w^-\rangle^{\mu} \nabla  \p_1^2 p \|^2_{L^2(\Om_f)}
+C G_3 E_3 .
\end{align}
Now we connect $\nabla \p_1^2 p$ with $\nabla  \p_\tau^2 p$ in the last line in \eqref{K63}.
In $\Om_f$, we calculate
\begin{align*}
[\nabla,\p_\tau^2]
&=[\nabla,\p_\tau]\p_\tau+\p_\tau[\nabla,\p_\tau]
=\nabla \Phi \p_2 \p_\tau+\p_\tau(\nabla \Phi \p_2),\\
\p_\tau^2
&=\p_1^2+2\Phi\p_1\p_2+\p_1\Phi\p_2
+|\Phi|^2 \p_2^2 +\Phi \p_2\Phi \p_2 .
\end{align*}
Thus
\begin{align}\label{K66}
&\| \langle w^+\rangle^{5\mu}\langle w^-\rangle^{\mu} [\p_\tau^2, \nabla] p \|^2_{L^2(\Om_f)}\\[-4mm]\nonumber\\\nonumber
&\leq \sum_{1\leq j\leq 2}\| \langle w^+\rangle^{5\mu}\langle w^-\rangle^{\mu} \nabla^j p\|^2_{L^2(\Om_f)} \| \nabla^{\leq 3} \Phi\|^2_{L^2(\Om_f)}
\mathcal{P}\big(1+ \| \nabla^{\leq 3} \Phi\|_{L^2(\Om_f)} \big)  \\\nonumber
&\leq E_3 G_3 \mathcal{P}\big(1+  \| \nabla^{\leq 3} \Phi\|_{L^2(\Om_f)} \big),
\end{align}
\begin{align}\label{K67}
&\| \langle w^+\rangle^{5\mu}\langle w^-\rangle^{\mu} \nabla \p_1^2 p \|^2_{L^2(\Om_f)}
\leq \|\langle w^+\rangle^{5\mu}\langle w^-\rangle^{\mu} \p_\tau^2 \nabla p \|^2_{L^2(\Om_f)} \\[-4mm]\nonumber\\\nonumber
&\quad + \sum_{2\leq j\leq 3}\| \langle w^+\rangle^{5\mu}\langle w^-\rangle^{\mu} \nabla^j p \|^2_{L^2(\Om_f)}
 \| \nabla^{\leq 3} \Phi \|_{L^2(\Om_f)}^2
\mathcal{P}\big(1+  \| \nabla^{\leq 3} \Phi \|_{L^2(\Om_f)} \big)\,.
\end{align}
By Lemma \ref{lemPsi}, \eqref{K59}, \eqref{K66}, the weighted first and second order estimate of the pressure,
 we thus obtain
\begin{align}\label{K70}
&\|\langle w^+\rangle^{5\mu}\langle w^-\rangle^{\mu} \p_\tau^2 \nabla p \|^2_{L^2(\Om_f)}\\\nonumber
&\leq \|\langle w^+\rangle^{5\mu}\langle w^-\rangle^{\mu}  \nabla \p_\tau^2 p \|^2_{L^2(\Om_f)}
+\|\langle w^+\rangle^{5\mu}\langle w^-\rangle^{\mu} [\p_\tau^2, \nabla] p \|^2_{L^2(\Om_f)}\\\nonumber
&\leq \big(\| \langle w^+\rangle^{5\mu}\langle w^-\rangle^{\mu}  \nabla^3 p\|_{L^2(\Om_f)}
+ \| \langle w^+\rangle^{5\mu}\langle w^-\rangle^{\mu}  \nabla^3 \hat{p}\|_{L^2(\hat{\Om}_f)} \big)
 E_3^{\f12} G_3^{\f12} \mathcal{P}(1+ \| \p_1 f\|_{H^{\f52}}) \\\nonumber
&\quad +E_3 G_3 \mathcal{P}(1+ \| \p_1 f\|_{H^{\f52}} )\, .
\end{align}
Note similar estimate also holds for $\hat{p}$.
Combining \eqref{K63}, \eqref{K67}, \eqref{K70},
by applying Lemma \ref{lemPsi}, the weighted first and second order estimate of the pressure
and the H\"older inequality,
 note $\|\p_1 f \|_{H^{\f52}}\leq\delta\ll 1$, we obtain
\begin{align*} 
\|  \langle w^+\rangle^{5\mu}\langle w^-\rangle^{\mu} \nabla^3 p \|^2_{L^2(\Om_f)}
+\|  \langle w^+\rangle^{5\mu}\langle w^-\rangle^{\mu} \nabla^3 \hat{p} \|^2_{L^2(\hat{\Om}_f)}
\leq C E_3G_3\,
\mathcal{P}\big(1+ \| \p_1f\|_{H^{\f52}(\bR)} \big)   .
\end{align*}
Thus we obtain $3$-order derivatives estimate of the pressure
in Lemma \ref{lemPre}.

\subsubsection{Higher order estimate of the pressure ($4\leq m\leq s$)}
The highest order $4\leq m\leq s$ weighted estimate of the pressure is similar to the lower order case, we only sketch the main steps.
Applying $\p_\tau^{m-1}$ to $\eqref{H0}_1$ and $\eqref{H0}_2$, one obtains
\begin{align}
\label{K80}
-\Delta\p_\tau^{m-1} p -\left[\p_\tau^{m-1}, \Delta\right] p =\p_\tau^{m-1}\nabla \cdot (\La_{-} \cdot \nabla \La_{+})
\quad \text{in}\,\,\Om_f,\\
\label{K81}
-\Delta\p_\tau^{m-1} \hat{p} -\left[\p_\tau^{m-1}, \Delta\right] \hat{p} =\p_\tau^{m-1}\nabla \cdot (\hat{\La}_{-} \cdot \nabla \hat{\La}_{+})
\quad \text{in}\,\,\hat{\Om}_f.
\end{align}
Taking the $L^2$ inner product of \eqref{K80} with
$ \langle w^+\rangle^{4\mu}\langle w^-\rangle^{2\mu} \p_\tau^{m-1} p$ in $\Om_f$
and taking the $L^2$ inner product of \eqref{K81} with
$\langle w^+\rangle^{4\mu}\langle w^-\rangle^{2\mu} \p_\tau^{m-1} \hat{p}$ in $\hat{\Om}_f$,
by integration by parts, we have
\begin{align}\label{K82}
&\int_{\Om_f} \langle w^+\rangle^{4\mu}\langle w^-\rangle^{2\mu}
|\nabla\p_\tau^{m-1} p|^2 \dx
+\int_{\hat{\Om}_f} \langle w^+\rangle^{4\mu}\langle w^-\rangle^{2\mu}
|\nabla\p_\tau^{m-1} \hat{p}|^2 \dx \\\nonumber
&=\int_{\Ga_f} \langle w^+\rangle^{4\mu}\langle w^-\rangle^{2\mu}
\big(\p_{n_f}\p_\tau^{m-1} p\cdot \p_\tau^{m-1} p-\p_{n_f}\p_\tau^{m-1} \hat{p}\cdot \p_\tau^{m-1} \hat{p}\big) \,\d\sigma(x)\\\nonumber
&\quad +\int_{\hat{\Ga}}  \langle w^+\rangle^{4\mu}\langle w^-\rangle^{2\mu}
\p_2\p_\tau^{m-1} \hat{p}\cdot \p_\tau^{m-1} \hat{p} \,\d\sigma(x)
-\int_{\Ga}  \langle w^+\rangle^{4\mu}\langle w^-\rangle^{2\mu}
\p_2\p_\tau^{m-1} p\cdot \p_\tau^{m-1} p \,\d\sigma(x)
 \\\nonumber
&\quad-\int_{\Om_f} \nabla \big(\langle w^+\rangle^{4\mu}\langle w^-\rangle^{2\mu} \big)
\cdot \nabla\p_\tau^{m-1} p\, \p_\tau^{m-1} p \dx
-\int_{\hat{\Om}_f} \nabla \big( \langle w^+\rangle^{4\mu}\langle w^-\rangle^{2\mu} \big)
\cdot \nabla\p_\tau^{m-1} \hat{p}\, \p_\tau^{m-1} \hat{p} \dx \\\nonumber
&\quad+\int_{\Om_f} \langle w^+\rangle^{4\mu}\langle w^-\rangle^{2\mu}
\left[\p_\tau^{m-1}, \Delta\right]p\, \p_\tau^{m-1} p \dx
+\int_{\hat{\Om}_f} \langle w^+\rangle^{4\mu}\langle w^-\rangle^{2\mu}
\left[\p_\tau^{m-1}, \Delta\right]\hat{p}\, \p_\tau^{m-1} \hat{p} \dx \\\nonumber
&\quad+\int_{\Om_f} \langle w^+\rangle^{4\mu}\langle w^-\rangle^{2\mu}
\p_\tau^{m-1}(\p_j\La_{-}^i\p_i\La_{+}^j)\, \p_\tau^{m-1} p \dx  \\\nonumber
&\quad +\int_{\hat{\Om}_f} \langle w^+\rangle^{4\mu}\langle w^-\rangle^{2\mu}
\p_\tau^{m-1}(\p_j\hat{\La}_{-}^i\p_i\hat{\La}_{+}^j)\, \p_\tau^{m-1} \hat{p} \dx \, .
\end{align}

For the fourth line of \eqref{K82}, one has
\begin{align*}
&\big| \int_{\Om_f} \nabla \big( \langle w^+\rangle^{4\mu}\langle w^-\rangle^{2\mu}\big)\cdot\nabla \p_\tau^{m-1} p\cdot \p_\tau^{m-1} p \dx
+\int_{\hat{\Om}_f} \nabla \big(\langle w^+\rangle^{4\mu}\langle w^-\rangle^{2\mu}\big)\cdot\nabla \p_\tau^{m-1} \hat{p}\cdot \p_\tau^{m-1} \hat{p} \dx\big| \\
&\leq \delta
\| \langle w^+\rangle^{2\mu}\langle w^-\rangle^{\mu} \nabla\p_\tau^{m-1} p \|^2_{L^2(\Om_f)}
+C_\delta \| \langle w^+\rangle^{2\mu}\langle w^-\rangle^{\mu} \p_\tau^{m-1} p \|^2_{L^2(\Om_f)}\\[-5mm]\\
&\quad+ \delta \| \langle w^+\rangle^{2\mu}\langle w^-\rangle^{\mu} \nabla\p_\tau^{m-1} \hat{p}\|^2_{L^2(\hat{\Om}_f)}
+C_\delta \| \langle w^+\rangle^{2\mu}\langle w^-\rangle^{\mu} \nabla\p_\tau^{m-1} \hat{p}\|^2_{L^2(\hat{\Om}_f)}  \,.
\end{align*}
Here $\delta$ is a small constant. $C_\delta$ is a constant depending on $\delta$.
For the sixth and the seventh line of \eqref{K82}, we estimate
\begin{align*}
&\big|\int_{\Om_f}\langle w^+\rangle^{4\mu}\langle w^-\rangle^{2\mu}
\p_\tau^{m-1}(\p_j\La_{-}^i\p_i\La_{+}^j)\, \p_\tau^{m-1} p \dx
+\int_{\hat{\Om}_f} \langle w^+\rangle^{4\mu}\langle w^-\rangle^{2\mu}
\p_\tau^{m-1}(\p_j\hat{\La}_{-}^i\p_i\hat{\La}_{+}^j)\, \p_\tau^{m-1} \hat{p} \dx \big| \\\nonumber
&\leq E_s^{\f12} G_s^{\f12} \mathcal{P} (1+\|\nabla^{\leq s}\Phi\|_{L^2(\Om_f)}+\|\nabla^{\leq s}\hat{\Phi}\|_{L^2(\hat{\Om}_f)}) \\\nonumber
&\quad\cdot\big( \| \langle w^+\rangle^{2\mu}\langle w^-\rangle^{\mu}
 \p_\tau^{m-1} p\|_{L^2(\Om_f)}
+ \| \langle w^+\rangle^{2\mu}\langle w^-\rangle^{\mu}
  \p_\tau^{m-1} \hat{p}\|_{L^2(\hat{\Om}_f)} \big)\, .
\end{align*}

For the second line of \eqref{K82}, in order to use the boundary of the pressure, we write
\begin{align}\label{K85}
& \int_{\Ga_f} \langle w^+\rangle^{4\mu}\langle w^-\rangle^{2\mu}
\big(\p_{n_f}\p_\tau^{m-1} p\cdot \p_\tau^{m-1} p-\p_{n_f}\p_\tau^{m-1} \hat{p}\cdot \p_\tau^{m-1} \hat{p}\big) \,\d\sigma(x)  \\\nonumber
&= \int_{\BR}\langle w^+\rangle^{4\mu}\langle w^-\rangle^{2\mu}
\p_1^{m-1} \big(\ud{N_f\cdot\nabla  p}-\ud{N_f\cdot\nabla  \hat{p}} \big)\cdot \ud{\p_\tau^{m-1} p} \,\d x_1 \\\nonumber
&\quad- \int_{\BR} \langle w^+\rangle^{4\mu}\langle w^-\rangle^{2\mu}
  \Big( -\ud{[N_f\cdot\nabla ,\p_\tau^{m-1}] p}+\ud{[N_f\cdot\nabla ,\p_\tau^{m-1}] \hat{p}} \Big)\cdot \ud{\p_\tau^{m-1} p} \,\d x_1\,.
\end{align}
For the second line of \eqref{K85}, by Lemma \ref{lemW}, 
Lemma \ref{lemPFH},
weighted trace Lemma \ref{lemD3} and Lemma \ref{lemPsi},
it is bounded by
\begin{align*}
& \big\|\langle w^+\rangle^{2\mu}\langle w^-\rangle^{\mu}
\langle\p_1\rangle^{-\f12}\p_1^{m-1} \big(\ud{N_f\cdot\nabla  p}-\ud{N_f\cdot\nabla  \hat{p}}\big) \big\|_{L^2(\BR)}\\
&\cdot \big\| \langle w^+\rangle^{-2\mu}\langle w^-\rangle^{-\mu} \langle\p_1\rangle^{\f12}  \big( \langle w^+\rangle^{4\mu}\langle w^-\rangle^{2\mu} \ud{\p_\tau^{m-1} p} \big)\big\|_{L^2(\BR)}\\
&\lesssim \big\|\langle w^+\rangle^{2\mu}\langle w^-\rangle^{\mu}
\langle\p_1\rangle^{\f12} \p_1^{m-2} \big(\ud{N_f\cdot\nabla  p}-\ud{N_f\cdot\nabla  \hat{p}}\big) \big\|_{L^2(\BR)}
\big\| \langle w^+\rangle^{2\mu}\langle w^-\rangle^{ \mu} \p^{\leq 1}
\p_\tau^{m-1}  p\big\|_{L^2(\Om_f)}\\
&\lesssim E_s^{\f12} G_s^{\f12}
\sum_{1\leq i\leq m}\big\| \langle w^+\rangle^{2\mu}\langle w^-\rangle^{ \mu}
\nabla^i  p  \big\|_{L^2(\Om_f)} \mathcal{P}(1+\| \p_1 f\|_{H^{s-\f12}}) \,.
\end{align*}
To estimate the last line of \eqref{K85}, similar to the commutator of $[N_f\cdot\nabla ,\p_\tau]$ and $[N_f\cdot\nabla ,\p_\tau^2]$ in $\Om_f$. The commutator of
$[N_f\cdot\nabla ,\p_\tau^{m-1}]$ is a  ${(m-1)}-$order  operator $D_{\Psi,f}^{m-1}$ with weights $\nabla^\alpha\Phi$ and $\p_1^{j+1}f$ with
$1\leq l\leq m-1$, $|\alpha|\leq m-1$, $j\leq m-2$, $l+|\alpha|\leq m+1$
and a $1-$order operator $\p_1^m f\p_1$:
\begin{align*}
&[N_f\cdot\nabla ,\p_\tau^{m-1}] p=D_{\Psi,f}^{m-1}p+\p_1^m f\p_1p \\
& = \Big(\sum_{1\leq |\alpha|, |\beta|\leq m-1, 1\leq |\alpha|+|\beta|\leq m+1}  C_{\alpha,\beta}\nabla^\alpha\Phi \p^\beta p
+\sum_{1\leq j+1, |\beta|\leq m-1, 0\leq j+|\beta|\leq m} C_{j,\beta} \p^{j+1}_1f \p^\beta p \Big)\\
&\quad\cdot \big(1+ \textrm {the multiple of the lower order terms of}\,\,  \nabla^{\leq m-2}\Phi, \p^{1+(\leq m-2)}_1f\big)
+\p_1^m f\p_1p.
\end{align*}
By applying Lemma \ref{lemPsi},
we infer that
\begin{align*}
&\big| \int_{\BR} \langle w^+\rangle^{10\mu}\langle w^-\rangle^{2\mu}\cdot \ud{[N_f\cdot\nabla ,\p_\tau^{m-1}] p}\cdot \ud{\p_\tau^{m-1} p} \,\d x_1 \big|\\
&=\big| \int_{\BR} \langle w^+\rangle^{10\mu}\langle w^-\rangle^{2\mu}\cdot \ud{D_{\Phi,f}^{m-1}p+\p_1^mf \p_1 p}\cdot \ud{\p_\tau^{m-1} p} \,\d x_1 \big|\\
&\leq \big\| \langle w^+\rangle^{-5\mu}\langle w^-\rangle^{-\mu}
\langle \p_1\rangle^{\f12}\big( \langle w^+\rangle^{10\mu}\langle w^-\rangle^{2\mu}
 \ud{D_{\Phi,f}^{m-1}p} \big)\big\|_{L^2(\BR)}\cdot \big\|\langle w^+\rangle^{5\mu}\langle w^-\rangle^{\mu}
 \langle \p_1\rangle^{-\f12}\p_1\ud{\p_\tau^{m-2} p} \big\|_{L^2(\BR)}\\\nonumber
&\quad +\| \p_1^m f\|_{L^2(\bR)}
\sum_{1\leq j\leq 2} \big\| \langle w^+\rangle^{5\mu}\langle w^-\rangle^{ \mu}\nabla^j  p  \big\|_{L^2(\Om_f)}
\| \langle w^+\rangle^{5\mu}\langle w^-\rangle^{\mu}\nabla^{\leq 1}\p_\tau^{m-1} p  \|_{L^2(\Om_f)} \\\nonumber
&\lesssim
\big\| \langle w^+\rangle^{5\mu}\langle w^-\rangle^{ \mu}
\nabla^{\leq m}  p  \big\|_{L^2(\Om_f)}
 \| \langle w^+\rangle^{5\mu}\langle w^-\rangle^{\mu}\nabla^{\leq m-1}  p  \|_{L^2(\Om_f)}
\| \p_1f\|_{H^{s-\f12}(\bR)} \mathcal{P}(1+\| \p_1f\|_{H^{s-\f12}(\bR)}).
\end{align*}
Similarly
\begin{align*}
&\big| \int_{\BR} \langle w^+\rangle^{10\mu}\langle w^-\rangle^{2\mu}\cdot \ud{[N_f\cdot\nabla ,\p_\tau^{m-1}] \hat{p}}\cdot \ud{\p_\tau^{m-1} p} \,\d x_1 \big|\\
&\lesssim
\big\| \langle w^+\rangle^{5\mu}\langle w^-\rangle^{ \mu}
\nabla^{\leq m}  \hat{p}  \big\|_{L^2(\hat{\Om}_f)}
 \| \langle w^+\rangle^{5\mu}\langle w^-\rangle^{\mu}\nabla^{\leq m-1}  \hat{p}  \|_{L^2(\hat{\Om}_f)}
\| \p_1f\|_{H^{s-\f12}(\bR)} \mathcal{P}(1+\| \p_1f\|_{H^{s-\f12}(\bR)}).
\end{align*}

For the third line of \eqref{K82},
due to the boundary condition of the pressure on $\Ga$ and $\hat{\Ga}$, there holds
\begin{align*}
&\int_{\hat{\Ga}}  \langle w^+\rangle^{4\mu}\langle w^-\rangle^{2\mu}
\p_2\p_1^{m-1} \hat{p}\cdot \p_1^{m-1} \hat{p}\, \d x_1
-\int_{\Ga}  \langle w^+\rangle^{4\mu}\langle w^-\rangle^{2\mu}
\p_2\p_1^{m-1} p\cdot \p_1^{m-1} p\, \d x_1=0\,.
\end{align*}


\textbf{Estimate of the fifth of \eqref{K82}}
We have
\begin{align*} 
&\big| \int_{\Om_f} \langle w^+\rangle^{4\mu}\langle w^-\rangle^{2\mu}
\left[\p_\tau^{m-1}, \Delta\right]p\, \p_\tau^{m-1} p \dx
+\int_{\hat{\Om}_f} \langle w^+\rangle^{4\mu}\langle w^-\rangle^{2\mu}
\left[\p_\tau^{m-1}, \Delta\right]\hat{p}\, \p_\tau^{m-1} \hat{p} \dx \big| \\\nonumber
&\leq \|  \langle w^+\rangle^{2\mu}\langle w^-\rangle^{\mu} \left[\p_\tau^{m-1}, \Delta\right]p \|_{L^2(\Om_f)}
\| \langle w^+\rangle^{2\mu}\langle w^-\rangle^{\mu}  \p_\tau^{m-1} p \|_{L^2(\Om_f)}
\\[-4mm]\nonumber\\\nonumber
&\quad+ \| \langle w^+\rangle^{2\mu}\langle w^-\rangle^{\mu} \left[\p_\tau^{m-1}, \Delta\right]\hat{p} \|_{L^2(\hat{\Om}_f)}
\| \langle w^+\rangle^{2\mu}\langle w^-\rangle^{\mu}  \p_\tau^{m-1} \hat{p} \|_{L^2(\hat{\Om}_f)}\,.
\end{align*}
%
Similar to the commutator calculation of $\big[\p_\tau^{\leq 2}, \Delta\big]p$, by an induction argument, there holds
\begin{align*} 
&\| \langle w^+\rangle^{2\mu}\langle w^-\rangle^{\mu}
\big[\p_\tau^{m-1}, \Delta\big]p \|_{L^2(\Om_f)} \\\nonumber
&\lesssim
 \sum_{1\leq i \leq m} \|  \langle w^+\rangle^{2\mu}\langle w^-\rangle^{\mu} \nabla^i p\|_{L^2(\Om_f)}
 \sum_{ i\leq m}\| \nabla^i\Phi\|_{L^2(\Om_f)}
 \mathcal{P}\big(1+\sum_{ j\leq m}\| \nabla^j\Phi\|_{L^2(\Om_f)} \big).
\end{align*}
Similarly,
\begin{align*}
&\| \langle w^+\rangle^{2\mu}\langle w^-\rangle^{\mu}
\big[\p_\tau^{m-1}, \Delta\big]\hat{p} \|_{L^2(\hat{\Om}_f)} \\\nonumber
&\lesssim
 \sum_{1\leq i \leq m} \|  \langle w^+\rangle^{2\mu}\langle w^-\rangle^{\mu} \nabla^i \hat{p}\|_{L^2(\hat{\Om}_f)}
 \sum_{ i\leq m}\| \nabla^i\hat{\Phi}\|_{L^2(\hat{\Om}_f)}
 \mathcal{P}\big(1+\sum_{ j\leq m}\| \nabla^j\hat{\Phi}\|_{L^2(\hat{\Om}_f)} \big).
\end{align*}
Gathering the above estimates, by Lemma \ref{lemPsi},
the induction result of the lower order ($k\leq m-1$) weighted estimate of the pressure,
 \eqref{K82} reduces to
\begin{align*} 
&\| \langle w^+\rangle^{2\mu}\langle w^-\rangle^{\mu}
\nabla\p_\tau^{m-1} p \|^2_{L^2(\Om_f)}
+\|  \langle w^+\rangle^{2\mu}\langle w^-\rangle^{\mu}
\nabla\p_\tau^{m-1} \hat{p} \|^2_{L^2(\hat{\Om}_f)} \\\nonumber
&\lesssim
\big(\|  \langle w^+\rangle^{2\mu}\langle w^-\rangle^{\mu}  \nabla^m p\|_{L^2(\Om_f)}
+   \|  \langle w^+\rangle^{2\mu}\langle w^-\rangle^{\mu}  \nabla^m \hat{p}\|_{L^2(\hat{\Om}_f)}\big)
E_s^{\f12} G_s^{\f12} \mathcal{P}(1+ \| \p_1 f\|_{H^{s-\f12}}) \\\nonumber
&\quad+ E_s G_s \mathcal{P}(1+ \| \p_1 f\|_{H^{s-\f12}} )\, .
\end{align*}

\textbf{$m-$order normal derivative estimate}
Now we treat the third normal derivative estimate of the pressure.
For $|\alpha|=m-2$, $m\leq s$, applying $\p^\alpha$ on $\eqref{H0}_1$, one has
\begin{align}\label{K90}
\p^\alpha \p_2^2 p=-\p^\alpha\p_1^2 p-\p^\alpha(\nabla_j\La_{-}^i \nabla_i\La_{+}^j).
\end{align}
Multiplying \eqref{K90} by $ \langle w^+\rangle^{2\mu}\langle w^-\rangle^{\mu}$,
then taking the $L^2(\Om_f)$ norm of the consequent equation, we obtain
\begin{align}\label{K91}
\big\| \langle w^+\rangle^{2\mu}\langle w^-\rangle^{\mu}
\p^\alpha \p_2^2 p \big\|_{L^2(\Om_f)}
\leq \big\|\langle w^+\rangle^{2\mu}\langle w^-\rangle^{\mu}
\p^\alpha \p_1^2 p \big\|_{L^2 }
+C G_{s-1}^{\f12} E_{s-1}^{\f12} \,.
\end{align}
We one-by-one take  $\p^\alpha=\p_1^{m-2}$, $\p^\alpha=\p_2 \p_1^{m-3}$, ..., $\p^\alpha=\p_2^{m-2}$. By an induction, \eqref{K91} yields
\begin{align*}
\sum_{|\alpha|=m-2}\big\|  \langle w^+\rangle^{2\mu}\langle w^-\rangle^{\mu}
\p^{\alpha}\p_2^2 p \big\|_{L^2(\Om_f)}
\lesssim \big\|   \langle w^+\rangle^{2\mu}\langle w^-\rangle^{\mu}
 \nabla  \p_1^{m-1} p \big\|_{L^2(\Om_f)}
+C G_{s-1}^{\f12} E_{s-1}^{\f12} .
\end{align*}
Consequently,
\begin{align}\label{K93}
&\| \langle w^+\rangle^{2\mu}\langle w^-\rangle^{\mu} \nabla^m p \|^2_{L^2(\Om_f)}
\\\nonumber
&=\sum_{|\alpha|=m-2} \| \langle w^+\rangle^{2\mu}\langle w^-\rangle^{\mu} \p^\alpha  \p_2^2 p \|^2_{L^2(\Om_f)}
                +   \| \langle w^+\rangle^{2\mu}\langle w^-\rangle^{\mu} \nabla  \p_1^{m-1} p \|^2_{L^2(\Om_f)}\\\nonumber
&\lesssim \| \langle w^+\rangle^{2\mu}\langle w^-\rangle^{\mu} \nabla  \p_1^{m-1} p \|^2_{L^2(\Om_f)}
+C G_{s-1} E_{s-1} .
\end{align}
Now we establish the relation between $\nabla \p_1^{m-1} p$ and $\nabla  \p_\tau^{m-1} p$.
Similar to \eqref{K67}, by an induction argument, we would have
\begin{align}\label{K94}
&\| \langle w^+\rangle^{2\mu}\langle w^-\rangle^{\mu} \nabla \p_1^{m-1} p \|^2_{L^2(\Om_f)}
\leq \|\langle w^+\rangle^{2\mu}\langle w^-\rangle^{\mu} \p_\tau^{m-1} \nabla p \|^2_{L^2(\Om_f)}  \\\nonumber
&\quad + \sum_{1\leq j\leq m}\| \langle w^+\rangle^{2\mu}
\langle w^-\rangle^{\mu} \nabla^j p \|^2_{L^2(\Om_f)}
 \| \nabla^{\leq m} \Phi \|_{L^2(\Om_f)}^2
\mathcal{P}\big(1+  \| \nabla^{\leq m} \Phi \|_{L^2(\Om_f)} \big)\,.
\end{align}
Similarly, by an induction, we would have
\begin{align}\label{K96}
&\| \langle w^+\rangle^{2\mu}\langle w^-\rangle^{\mu} [\p_\tau^{m-1}, \nabla] p \|^2_{L^2(\Om_f)}\\[-5mm]\nonumber\\\nonumber
&\leq \sum_{1\leq j\leq m}\| \langle w^+\rangle^{2\mu}\langle w^-\rangle^{\mu} \nabla^j p\|^2_{L^2(\Om_f)} \| \nabla^{\leq s} \Phi\|^2_{L^2(\Om_f)}
\mathcal{P}\big(1+ \| \nabla^{\leq s} \Phi\|_{L^2(\Om_f)} \big)  \\\nonumber
&\leq E_s G_s \mathcal{P}\big(1+  \| \nabla^{\leq s} \Phi\|_{L^2(\Om_f)} \big).
\end{align}
By applying \eqref{K96} and Lemma \ref{lemPsi}, \eqref{K94} is further bounded by
\begin{align}\label{K97}
&\| \langle w^+\rangle^{2\mu}\langle w^-\rangle^{\mu} \nabla \p_1^{m-1} p \|^2_{L^2(\Om_f)} \\[-5mm]\nonumber\\\nonumber
&\leq \|\langle w^+\rangle^{2\mu}\langle w^-\rangle^{\mu} \nabla \p_\tau^{m-1} p \|^2_{L^2(\Om_f)}
+ \|\langle w^+\rangle^{2\mu}\langle w^-\rangle^{\mu} [\p_\tau^{m-1},\nabla] p \|^2_{L^2(\Om_f)} \\\nonumber
&\quad + \sum_{1\leq j\leq m}\| \langle w^+\rangle^{2\mu}
\langle w^-\rangle^{\mu} \nabla^j p \|^2_{L^2(\Om_f)}
\|\p_1f\|^2_{H^{s-\f12}} \mathcal{P}\big(1+ \|\p_1f\|_{H^{s-\f12}} \big) \\\nonumber
&\lesssim  \|\langle w^+\rangle^{2\mu}\langle w^-\rangle^{\mu} \nabla \p_\tau^{m-1} p \|^2_{L^2(\Om_f)}
+ E_s G_s \mathcal{P}\big(1+ \|\p_1f\|_{H^{s-\f12}} \big)\\\nonumber
&\quad+\sum_{1\leq j\leq m}\| \langle w^+\rangle^{2\mu}
\langle w^-\rangle^{\mu} \nabla^j p \|^2_{L^2(\Om_f)}
\|\p_1f\|^2_{H^{s-\f12}} \mathcal{P}\big(1+ \|\p_1f\|_{H^{s-\f12}} \big).
\end{align}
Note that similar estimate holds for $\hat{p}$ in $\hat{\Om}_f$.
Combining \eqref{K93}, \eqref{K97},
the induction result of the lower order ($k\leq m-1$) weighted estimate of the pressure,
and by the H\"older inequality,
 note $\|\p_1 f \|_{H^{\f52}}\leq\delta\ll 1$, we obtain
\begin{align*} 
\|  \langle w^+\rangle^{2\mu}\langle w^-\rangle^{\mu} \nabla^m p \|^2_{L^2(\Om_f)}
+\|  \langle w^+\rangle^{2\mu}\langle w^-\rangle^{\mu} \nabla^m \hat{p} \|^2_{L^2(\hat{\Om}_f)}
\leq C E_sG_s\,
\mathcal{P}\big(1+ \| \p_1f\|_{H^{s-\f12}(\bR)} \big)   .
\end{align*}
Combining Lemma \ref{lemP1} and the weighted higher derivative ($2\leq k\leq s$) estimate
of pressure in this section yields  Lemma \ref{lemPre}.

\section{Weighted highest order estimate of the pressure }

This section is devoted to the $s-\f12$ order weighted estimate of $N_f\cdot(\ud{\nabla p} +\ud{\nabla \hat{p} })$,
which is required in the weighted energy estimate of the free surface.
Due to the dependence of regularity of the free surface,
we have $s-$order weighted energy estimate of the pressure inside the domain, as have been proved in Lemma \ref{lemPre}.
In the weighted estimate of the free surface,
this is not enough to obtain the control of the $s-\f12$ order weighted estimate of $N_f\cdot(\ud{\nabla p} +\ud{\nabla \hat{p} })$.
To solve this problem, we use the classical decomposition of the pressure: $p=p^1+p^2$, $\hat{p}=\hat{p}^1+\hat{p}^2$ where $p^1$ and $\hat{p}^1$ are harmonic functions with inhomogeneous boundary conditions, $p^2$ and $\hat{p}^2$ satisfy the Poisson equations with inhomogeneous boundary conditions (see \eqref{U4} and \eqref{U5}). 
Thus by using the known results for the Dirichlet-Neumann (DN) operator, some commutator estimate and the inherent cancellation structure,
 we could save one order derivative and obtain the weighted estimate of $N_f\cdot(\ud{\nabla p} +\ud{\nabla \hat{p} })$.

Recalling that the elliptic equations and the boundary conditions for the pressure  \eqref{B22} as follows
\begin{equation}\label{U2}
\left\{
\begin{array}{l}
 \Delta p = -\nabla \cdot (\La_{-} \cdot \nabla \La_{+})
\quad \text{in}\quad\Om_f,\\
\Delta \hat{p} = -\nabla \cdot (\La_{-} \cdot \nabla \La_{+})
\quad \text{in}\quad\hat{\Om}_f,\\
p=\hat{p}\quad\text{on}\quad\Ga_f ,\\
N_f\cdot\nabla p-N_f\cdot\nabla \hat{p}= -\big(\ud{\La_{+}^1} \ud{\La_{-}^1}- \ud{\hat{\La}_{+}^1} \ud{\hat{\La}_{-}^1} \big) \p_1^2 f\\
\qquad-\big(\ud{\La_{+}^1}-\ud{\hat{\La}_{+}^1} \big) \p_1 (\p_t-  \p_1) f
  -\big(\ud{\La_{-}^1}-\ud{\hat{\La}_{-}^1}\big) \p_1 (\p_t+  \p_1) f
\quad\text{on}\quad\Ga_f ,\\
\p_2 p=0\,\,   \textrm{on} \,\,  \Ga,
\quad \p_2 \hat{p}=0\,\,   \textrm{on} \,\,  \hat{\Ga}.
\end{array}\right.
\end{equation}

In this section, we are going to show
\begin{lem}\label{lemPP0}
Let $s\geq 4$ be an integer, $1/2< \mu \leq 3/5$,
$\delta>0$ be a sufficiently small number, $\| \p_1f \|_{H^{s-\f12}(\BR)} \leq \delta.$
For the pressure satisfying \eqref{H0},
there holds
\begin{align*}
\| \langle w^\pm\rangle^{2\mu} \langle w^\mp\rangle^{\mu}
\langle \p_1 \rangle^{\f12} \p_1^{s-1}\big(\ud{N_f\cdot\nabla p} + \ud{N_f\cdot\nabla \hat{p}} \big) \|_{L^2 (\BR)}
\lesssim  E_s^{\f12} G_s^{\f12}.
\end{align*}
\end{lem}

In the sequel, we only prove the above lemma when the weight function takes $\langle w^+\rangle^{2\mu} \langle w^-\rangle^{\mu}$. The other case can be estimated similarly.

We decompose $p=p^1 +p^2$ and $\hat{p}=\hat{p}^1 +\hat{p}^2$.
They are  the solution of the following elliptic equation
\begin{equation}\label{U4}
\left\{
\begin{array}{l}
\Delta p^1= 0
\quad \text{in}\quad\Om_f,\\
\Delta \hat{p}^1= 0
\quad \text{in}\quad\hat{\Om}_f,\\
p^1=\hat{p}^1=\ud{p}\quad\text{on}\quad\Gamma_f ,\\
\p_2 p^1=0\,\,   \textrm{on} \,\,  \Ga,
\quad \p_2 \hat{p}^1=0\,\,   \textrm{on} \,\,  \hat{\Ga},
\end{array}\right.
\end{equation}
and
\begin{equation}\label{U3}
\left\{
\begin{array}{l}
\Delta p^2= -\nabla \cdot (\La_{-} \cdot \nabla \La_{+})
\quad \text{in}\quad\Om_f,\\
\Delta \hat{p}^2= -\nabla \cdot (\hat{\La}_{-} \cdot \nabla \hat{\La}_{+})
\quad \text{in}\quad\hat{\Om}_f,\\
p^2=\hat{p}^2=0\quad\text{on}\quad\Ga_f ,\\
\p_2 p^2=0\,\,   \textrm{on} \,\,  \Ga,
\quad \p_2 \hat{p}^2=0\,\,   \textrm{on} \,\,  \hat{\Ga}.
\end{array}\right.
\end{equation}
Similar to the weighted energy estimate of the pressure in the above section, there holds similar weighted estimate for
$p^1,\, \hat{p}^1,\, p^2,\, \hat{p}^2.$
\begin{lem}\label{lemPP1}
Let $s\geq 4$ be an integer, $1/2< \mu \leq 3/5$,
$\delta>0$ be a sufficiently small number, $\| \p_1f \|_{H^{s-\f12}(\BR)} \leq \delta.$
For the pressure satisfying \eqref{U4},
there holds
\begin{align}\label{U5}
&\sum_{|\alpha|\leq s }\| \langle w^\pm\rangle^{2\mu}
\langle w^\mp\rangle^{\mu}  \p^\alpha p^1\|_{L^2(\Om_f)}^2
+\sum_{|\alpha|\leq s }\| \langle w^\pm\rangle^{2\mu}
\langle w^\mp\rangle^{\mu}  \p^\alpha \hat{p}^1\|_{L^2(\hat{\Om}_f)}^2\\\nonumber
&\lesssim
E_s G_s
\cdot \mathcal{P}(1+ \| \p_1f \|_{H^{s-\f12}(\BR)}),
\end{align}
For the pressure satisfying \eqref{U3},
there holds
\begin{align}\label{U6}
&\sum_{|\alpha|\leq s+1 }\| \langle w^\pm\rangle^{2\mu}
\langle w^\mp\rangle^{\mu}  \p^\alpha p^2\|_{L^2(\Om_f)}^2
+\sum_{|\alpha|\leq s+1 }\| \langle w^\pm\rangle^{2\mu}
\langle w^\mp\rangle^{\mu}  \p^\alpha \hat{p}^2\|_{L^2(\hat{\Om}_f)}^2\\\nonumber
&\lesssim
E_s G_s
\cdot \mathcal{P}(1+ \| \p_1f \|_{H^{s-\f12}(\BR)}),
\end{align}
where $\mathcal{P}(\cdot)$ is a polynomial.
\end{lem}
We remind that due to the vanishing of all the boundary conditions for $p^2$ and $\hat{p}^2$, there holds $(s+1)$-order estimate for $p^2$ and $\hat{p}^2$.
\begin{proof}
We will first show the weighted energy estimate for $p^2$ and $\hat{p}^2$.
The $s-$order weighted estimates of  \eqref{U6} is similar the weighted energy estimate for the pressure in the above section. Thus the details are omitted. The proof of the $(s+1)-$order weighted estimates in \eqref{U6}
share the similar steps with minor modifications. Note that due the vanishing boundary condition of $p^2, \hat{p}^2$. The  weighted energy estimate for $p^2$ and $\hat{p}^2$ is simpler since the all the boundary terms, for instance, the second line of \eqref{UU3}, vanish.

Once \eqref{U6} is obtained, note
\begin{align*}
N_f\cdot\nabla  p^1-N_f\cdot\nabla  \hat{p}^1=\big( N_f\cdot\nabla  p-N_f\cdot\nabla  \hat{p} \big)-\big(N_f\cdot\nabla  p^2-N_f\cdot\nabla  \hat{p}^2\big)
\quad\text{on}\quad\Ga_f .
\end{align*}
Then by \eqref{U6}, Lemma \ref{lemPFL} and Lemma \ref{lemPFH},
for $0\leq a\leq s-2$, there holds
\begin{align*}
\| \langle w^\pm\rangle^{2\mu}\langle w^\mp\rangle^{\mu} \langle\p_1\rangle^{\f 12}\p_1^a(N_f\cdot\ud{\nabla p^1}-N_f\cdot\ud{\nabla \hat{p}^1})\|_{L^2(\BR)}
\leq E_s^{\f 12}G_s^{\f 12}  \cdot \mathcal{P}(1 +\|\p_1f\|_{H^{s-\f 12}(\bR)}),
\end{align*}
In another words, the weighted $(s-\f32)-$order derivative of $N_f\cdot\ud{\nabla p^1}-N_f\cdot\ud{\nabla \hat{p}^1}$ has control.
On the other hand, by Lemma \ref{lemPP1} and weighted trace theorem \ref{lemD3}, we have weighted $(s-\f12)-$order derivative estimate of $\ud{p}$.
Then we can copy the method in the above section to show \eqref{U5}. 
The proof is exactly the same, thus the details are omitted.

 In the sequel, we only sketch the main steps of the $(s+1)-$order weighted estimates for \eqref{U6}.

Applying $\p_\tau^{s}$ to $\eqref{U3}_1$ and $\eqref{U3}_2$, one obtains
\begin{align}
\label{UU1}
-\Delta\p_\tau^{s} p^2 -\left[\p_\tau^{s}, \Delta\right] p^2 =\p_\tau^{s}\nabla \cdot (\La_{-} \cdot \nabla \La_{+})
\quad \text{in}\,\,\Om_f,\\
\label{UU2}
-\Delta\p_\tau^{s} \hat{p}^2 -\left[\p_\tau^{s}, \Delta\right] \hat{p}^2 =\p_\tau^{s}\nabla \cdot (\hat{\La}_{-} \cdot \nabla \hat{\La}_{+})
\quad \text{in}\,\,\hat{\Om}_f.
\end{align}
Taking the $L^2$ inner product of \eqref{UU1} with
$ \langle w^+\rangle^{4\mu}\langle w^-\rangle^{2\mu} \p_\tau^{m-1} p^2$ in $\Om_f$
and taking the $L^2$ inner product of \eqref{UU2} with
$\langle w^+\rangle^{4\mu}\langle w^-\rangle^{2\mu} \p_\tau^{m-1} \hat{p}^2$ in $\hat{\Om}_f$,
by integration by parts, we have
\begin{align}\label{UU3}
&\int_{\Om_f} \langle w^+\rangle^{4\mu}\langle w^-\rangle^{2\mu}
|\nabla\p_\tau^{s} p^2|^2 \dx
+\int_{\hat{\Om}_f} \langle w^+\rangle^{4\mu}\langle w^-\rangle^{2\mu}
|\nabla\p_\tau^{s} \hat{p}^2|^2 \dx \\\nonumber
&=\int_{\Ga_f} \langle w^+\rangle^{4\mu}\langle w^-\rangle^{2\mu}
\big(\p_{n_f}\p_\tau^{s} p^2\cdot \p_\tau^{s} p^2-\p_{n_f}\p_\tau^{s} \hat{p}^2\cdot \p_\tau^{s} \hat{p}^2\big) \,\d\sigma(x)\\\nonumber
&\quad-\int_{\Om_f} \nabla \big(\langle w^+\rangle^{4\mu}\langle w^-\rangle^{2\mu} \big)
\cdot \nabla\p_\tau^{s} p^2\, \p_\tau^{s} p^2 \dx
-\int_{\hat{\Om}_f} \nabla \big( \langle w^+\rangle^{4\mu}\langle w^-\rangle^{2\mu} \big)
\cdot \nabla\p_\tau^{s} \hat{p}^2\, \p_\tau^{s} \hat{p}^2 \dx \\\nonumber
&\quad+\int_{\Om_f} \langle w^+\rangle^{4\mu}\langle w^-\rangle^{2\mu}
\left[\p_\tau^{s}, \Delta\right]p^2\, \p_\tau^{s} p^2 \dx
+\int_{\hat{\Om}_f} \langle w^+\rangle^{4\mu}\langle w^-\rangle^{2\mu}
\left[\p_\tau^{s}, \Delta\right]\hat{p}^2\, \p_\tau^{s} \hat{p}^2 \dx \\\nonumber
&\quad+\int_{\Om_f} \langle w^+\rangle^{4\mu}\langle w^-\rangle^{2\mu}
\p_\tau^{s}(\p_j\La_{-}^i\p_i\La_{+}^j)\, \p_\tau^{s} p^2 \dx  \\\nonumber
&\quad +\int_{\hat{\Om}_f} \langle w^+\rangle^{4\mu}\langle w^-\rangle^{2\mu}
\p_\tau^{s}(\p_j\hat{\La}_{-}^i\p_i\hat{\La}_{+}^j)\, \p_\tau^{s} \hat{p}^2 \dx \, .
\end{align}

For the third line of \eqref{UU3}, one has
\begin{align*}
&\big| \int_{\Om_f} \nabla \big( \langle w^+\rangle^{4\mu}\langle w^-\rangle^{2\mu}\big)\cdot\nabla \p_\tau^{s} p^2\cdot \p_\tau^{s} p^2 \dx
+\int_{\hat{\Om}_f} \nabla \big(\langle w^+\rangle^{4\mu}\langle w^-\rangle^{2\mu}\big)\cdot\nabla \p_\tau^{s} \hat{p}^2\cdot \p_\tau^{s} \hat{p}^2 \dx\big| \\
&\leq \delta
\| \langle w^+\rangle^{2\mu}\langle w^-\rangle^{\mu} \nabla\p_\tau^{s} p^2 \|^2_{L^2(\Om_f)}
+C_\delta \| \langle w^+\rangle^{2\mu}\langle w^-\rangle^{\mu} \p_\tau^{s} p^2 \|^2_{L^2(\Om_f)}\\[-5mm]\\
&\quad+ \delta \| \langle w^+\rangle^{2\mu}\langle w^-\rangle^{\mu} \nabla\p_\tau^{s} \hat{p}^2\|^2_{L^2(\hat{\Om}_f)}
+C_\delta \| \langle w^+\rangle^{2\mu}\langle w^-\rangle^{\mu} \nabla\p_\tau^{s} \hat{p}^2\|^2_{L^2(\hat{\Om}_f)}  \,.
\end{align*}
Here $\delta$ is a small constant. $C_\delta$ is a constant depending on $\delta$.

For the fifth and the sixth line of \eqref{UU3}, by using integration by parts, we write
\begin{align*}
&\big|\int_{\Om_f}\langle w^+\rangle^{4\mu}\langle w^-\rangle^{2\mu}
\p_\tau^{s}(\p_j\La_{-}^i\p_i\La_{+}^j)\, \p_\tau^{s} p^2 \dx
+\int_{\hat{\Om}_f}\langle w^+\rangle^{4\mu}\langle w^-\rangle^{2\mu}
\p_\tau^{s}(\p_j\hat{\La}_{-}^i\p_i\hat{\La}_{+}^j)\, \p_\tau^{s} \hat{p}^2 \dx\big|
 \\\nonumber
&=\big| \int_{\Om_f}\langle w^+\rangle^{4\mu}\langle w^-\rangle^{2\mu}
\p_\tau^{s-1}(\p_j\La_{-}^i\p_i\La_{+}^j)\, \p_\tau^{s+1} p
+\p_1\big( \langle w^+\rangle^{4\mu}\langle w^-\rangle^{2\mu} \big)
\p_\tau^{s-1}(\p_j\La_{-}^i\p_i\La_{+}^j)\, \p_\tau^{s} p \dx \\\nonumber
&\quad  +\int_{\hat{\Om}_f}\langle w^+\rangle^{4\mu}\langle w^-\rangle^{2\mu}
\p_\tau^{s-1}(\p_j\hat{\La}_{-}^i\p_i\hat{\La}_{+}^j)\, \p_\tau^{s+1} \hat{p}^2
+\p_1\big( \langle w^+\rangle^{4\mu}\langle w^-\rangle^{2\mu} \big)
\p_\tau^{s-1}(\p_j\hat{\La}_{-}^i\p_i\hat{\La}_{+}^j)\, \p_\tau^{s} \hat{p}^2 \dx \\\nonumber
&\quad+\int_{\Om_f}\langle w^+\rangle^{4\mu}\langle w^-\rangle^{2\mu}(1+\p_2\Phi)
\p_\tau^{s-1}(\p_j\La_{-}^i\p_i\La_{+}^j)\, \p_\tau^{s} p \dx \\\nonumber
&\qquad  +\int_{\hat{\Om}_f}\langle w^+\rangle^{4\mu}\langle w^-\rangle^{2\mu} (1+\p_2\hat{\Phi})
\p_\tau^{s-1}(\p_j\hat{\La}_{-}^i\p_i\hat{\La}_{+}^j)\, \p_\tau^{s} \hat{p}^2 \dx \big| \\\nonumber
&\leq E_s^{\f12} G_s^{\f12} \mathcal{P} (1+\|\nabla^{\leq s}\Phi\|_{L^2(\Om_f)}+\|\nabla^{\leq s}\hat{\Phi}\|_{L^2(\hat{\Om}_f)}) \\\nonumber
&\quad\cdot\big( \| \langle w^+\rangle^{2\mu}\langle w^-\rangle^{\mu}
 \p_\tau^{s} p^2\|_{L^2(\Om_f)}
+ \| \langle w^+\rangle^{2\mu}\langle w^-\rangle^{\mu}
  \p_\tau^{s} \hat{p}^2\|_{L^2(\hat{\Om}_f)} \\\nonumber
&\qquad +\| \langle w^+\rangle^{2\mu}\langle w^-\rangle^{\mu}
 \p_\tau^{s+1} p^2\|_{L^2(\Om_f)}
+ \| \langle w^+\rangle^{2\mu}\langle w^-\rangle^{\mu}
  \p_\tau^{s+1} \hat{p}^2\|_{L^2(\hat{\Om}_f)} \big)\, .
\end{align*}
For the second line of \eqref{UU3}, due to $p^2=\hat{p}^2=0$ on $\Ga_f$, thus
\begin{align*}
& \int_{\Ga_f} \langle w^+\rangle^{4\mu}\langle w^-\rangle^{2\mu}
\big(\p_{n_f}\p_\tau^{s} p^2\cdot \p_\tau^{s} p^2-\p_{n_f}\p_\tau^{s} \hat{p}^2\cdot \p_\tau^{s} \hat{p}^2\big) \,\d\sigma(x)=0\,.
\end{align*}

For the fourth line of \eqref{UU3},
\begin{align*}
&\int_{\Om_f} \langle w^+\rangle^{4\mu}\langle w^-\rangle^{2\mu}
\left[\p_\tau^{s}, \Delta\right]p^2\, \p_\tau^{s} p^2 \dx
+\int_{\hat{\Om}_f} \langle w^+\rangle^{4\mu}\langle w^-\rangle^{2\mu}
\left[\p_\tau^{s}, \Delta\right]\hat{p}^2\, \p_\tau^{s} \hat{p}^2 \dx .
\end{align*}
Note that
\begin{align*}
\left[\p_\tau^{s}, \Delta\right]
=\p_\tau \left[\p_\tau^{s-1}, \Delta\right] +  \left[\p_\tau, \Delta\right]\p_\tau^{s-1}\,.
\end{align*}
Thus
\begin{align*}
&\int_{\Om_f} \langle w^+\rangle^{4\mu}\langle w^-\rangle^{2\mu}
\left[\p_\tau^{s}, \Delta\right]p^2\, \p_\tau^{s} p^2 \dx \\\nonumber
&=\int_{\Om_f} \langle w^+\rangle^{4\mu}\langle w^-\rangle^{2\mu}
\left( \p_\tau \left[\p_\tau^{s-1}, \Delta\right]p^2 +  \left[\p_\tau, \Delta\right]\p_\tau^{s-1}p^2 \right)\, \p_\tau^{s} p^2 \dx \\\nonumber
&=-\int_{\Om_f} \langle w^+\rangle^{4\mu}\langle w^-\rangle^{2\mu}
\Big(  \left[\p_\tau^{s-1}, \Delta\right]p^2 \p_\tau^{s+1} p^2+\p_2 \Phi\left[\p_\tau^{s-1}, \Delta\right]p^2 \p_\tau^{s} p^2 \Big) \dx \\\nonumber
&\quad-\int_{\Om_f} \p_1\big( \langle w^+\rangle^{4\mu}\langle w^-\rangle^{2\mu} \big)
\left[\p_\tau^{s-1}, \Delta\right]p^2  \, \p_\tau^{s} p^2 \dx \\\nonumber
&\quad -\int_{\Om_f} \langle w^+\rangle^{4\mu}\langle w^-\rangle^{2\mu}
\left(  \Delta  \Phi  \p_2 +2 \nabla\Phi\cdot \p_2\nabla \right)\p_\tau^{s-1}p^2 \, \p_\tau^{s} p^2 \dx .
\end{align*}
The absolute value of the above expression can be bounded by
\begin{align*}
&\sum_{1\leq i \leq s+1} \|  \langle w^+\rangle^{2\mu}\langle w^-\rangle^{\mu} \p_\tau^i p^2\|_{L^2(\Om_f)}
 \sum_{1\leq i \leq s} \|  \langle w^+\rangle^{2\mu}\langle w^-\rangle^{\mu} \nabla^i p^2\|_{L^2(\Om_f)} \\\nonumber
&\quad\cdot \| \p_1 f\|_{H^{s-\f12}}
 \mathcal{P}\big(1+\| \p_1 f\|_{H^{s-\f12}} \big).
\end{align*}
Similarly
\begin{align*}
&\int_{\hat{\Om}_f} \langle w^+\rangle^{4\mu}\langle w^-\rangle^{2\mu}
\left[\p_\tau^{s}, \Delta\right]\hat{p}^2\, \p_\tau^{s} \hat{p}^2 \dx
\lesssim \sum_{1\leq i \leq s+1} \|  \langle w^+\rangle^{2\mu}\langle w^-\rangle^{\mu} \p_\tau^i \hat{p}^2\|_{L^2(\Om_f)}\\\nonumber
&\quad\cdot \sum_{1\leq i \leq s} \|  \langle w^+\rangle^{2\mu}\langle w^-\rangle^{\mu} \nabla^i \hat{p}^2\|_{L^2(\Om_f)} \| \p_1 f\|_{H^{s-\f12}}
 \mathcal{P}\big(1+\| \p_1 f\|_{H^{s-\f12}} \big).
\end{align*}

Combining the above estimates, \eqref{UU3} reduces to
\begin{align*} 
&\| \langle w^+\rangle^{2\mu}\langle w^-\rangle^{\mu}
\nabla\p_\tau^{s} p^2 \|^2_{L^2(\Om_f)}
+\|  \langle w^+\rangle^{2\mu}\langle w^-\rangle^{\mu}
\nabla\p_\tau^{s} \hat{p}^2 \|^2_{L^2(\hat{\Om}_f)} \\\nonumber
&\lesssim
\big(\|  \langle w^+\rangle^{2\mu}\langle w^-\rangle^{\mu}  \nabla^{s+1} p^2\|_{L^2(\Om_f)}
+   \|  \langle w^+\rangle^{2\mu}\langle w^-\rangle^{\mu}  \nabla^{s+1} \hat{p}^2\|_{L^2(\hat{\Om}_f)}\big)
E_s^{\f12} G_s^{\f12} \mathcal{P}(1+ \| \p_1 f\|_{H^{s-\f12}}) \\\nonumber
&\quad+ E_s G_s \mathcal{P}(1+ \| \p_1 f\|_{H^{s-\f12}} )\, .
\end{align*}
For the $(s+1)-$order normal and mixed derivative estimates,
it is the same as that in the above section. Thus repeating the argument will yield the lemma.

\end{proof}

For $\ud{N_f\cdot\nabla p^1} + \ud{N_f\cdot\nabla \hat{p}^1}$, we collect a useful lemma.
\begin{lem}\label{lemPP2}
Let $s\geq 3$ be an integer, $\| \p_1f \|_{H^{s-\f12}(\BR)} \lesssim 1.$
For the pressure satisfying \eqref{U4}, there holds
\begin{align*}
\| \ud{N_f\cdot\nabla p^1} + \ud{N_f\cdot\nabla \hat{p}^1} \|_{H^{s-\f12} (\bR)}
\lesssim \| \ud{p} \|_{H^{s-\f12}(\bR)}
\end{align*}
\end{lem}
\begin{proof}
See \eqref{appd-9} in the appendix.
\end{proof}
Combining Lemma \ref{lemPP1} and Lemma \ref{lemPP2}, we have the following estimate for $\ud{N_f\cdot\nabla p^1} + \ud{N_f\cdot\nabla \hat{p}^1}$.
\begin{lem}\label{lemPP3}
Let $s\geq 4$ be an integer, $\| \p_1f \|_{H^{s-\f12}(\BR)} \lesssim 1.$
For the pressure satisfying \eqref{U2}, there holds
\begin{align*}
\| \ud{N_f\cdot\nabla p} + \ud{N_f\cdot\nabla \hat{p}} \|_{H^{s-\f12} (\bR)}
\lesssim \| \ud{p} \|_{H^{s-\f12}(\bR)}+E_s^{\f12}G_s^{\f12}.
\end{align*}
\end{lem}
\begin{proof}
Note
\begin{align*}
\ud{N_f\cdot\nabla p} + \ud{N_f\cdot\nabla \hat{p}}=\big(\ud{N_f\cdot\nabla p^1} + \ud{N_f\cdot\nabla \hat{p}^1}\big)
+\big(\ud{N_f\cdot\nabla p^2} + \ud{N_f\cdot\nabla \hat{p}^2}\big)\,.
\end{align*}
The lemma is an easy consequence of Lemma \ref{lemPP2},  Lemma \ref{lemPP1} and the weighted trace theorem \ref{lemD3}.
\end{proof}

Now we present the 
estimate for
\begin{align*}
&\| \ud{N_f\cdot\nabla (\langle w^\pm\rangle^{2\mu} \langle w^\mp\rangle^{\mu} p^1)} + \ud{N_f\cdot\nabla (\langle w^\pm\rangle^{2\mu} \langle w^\mp\rangle^{\mu} \hat{p}^1)} \|_{H^{s-\f12} (\bR)} \,.
\end{align*}
Multiplying \eqref{U4} by $\langle w^+\rangle^{2\mu}\langle w^-\rangle^{\mu} $, we obtain
\begin{equation}\label{U7}
\left\{
\begin{array}{l}
\Delta (\langle w^+\rangle^{2\mu}\langle w^-\rangle^{\mu}p^1)  =
2\nabla  (\langle w^+\rangle^{2\mu}\langle w^-\rangle^{\mu}  )\cdot \nabla p^1+\Delta (\langle w^+\rangle^{2\mu}\langle w^-\rangle^{\mu} )p^1
\quad \text{in}\quad\Om_f,\\
\Delta (\langle w^+\rangle^{2\mu}\langle w^-\rangle^{\mu}\hat{p}^1)= 2\nabla (\langle w^+\rangle^{2\mu}\langle w^-\rangle^{\mu}) \cdot \nabla \hat{p}^1+\Delta (\langle w^+\rangle^{2\mu}\langle w^-\rangle^{\mu})\hat{p}^1
\quad \text{in}\quad\hat{\Om}_f,\\
\langle w^+\rangle^{2\mu}\langle w^-\rangle^{\mu} p^1=\langle w^+\rangle^{2\mu}\langle w^-\rangle^{\mu} \hat{p}^1=\langle w^+\rangle^{2\mu}\langle w^-\rangle^{\mu} p\quad\text{on}\quad\Gamma_f ,\\
N_f\cdot\nabla  (\langle w^+\rangle^{2\mu}\langle w^-\rangle^{\mu} p^1)
-N_f\cdot\nabla  (\langle w^+\rangle^{2\mu}\langle w^-\rangle^{\mu}\hat{p}^1)\\
\quad= \langle w^+\rangle^{2\mu}\langle w^-\rangle^{\mu} \big(N_f\cdot\nabla  p^1-N_f\cdot\nabla  \hat{p}^1\big) \qquad\qquad\text{on}\quad\Ga_f ,\\
\p_2 (\langle w^+\rangle^{2\mu}\langle w^-\rangle^{\mu}  p^1)=0\,\,   \textrm{on} \,\,  \Ga,
\quad \p_2 (\langle w^+\rangle^{2\mu}\langle w^-\rangle^{\mu} \hat{p}^1)=0\,\,   \textrm{on} \,\,  \hat{\Ga}.
\end{array}\right.
\end{equation}
Then we treat $\langle w^+\rangle^{2\mu} \langle w^-\rangle^{\mu} p^1$, $\langle w^+\rangle^{2\mu} \langle w^-\rangle^{\mu} \hat{p}^1$ similar to the estimate of the $p$, $\hat{p}$ where we have just conducted:
decomposing $\langle w^+\rangle^{2\mu} \langle w^-\rangle^{\mu} p^1$ and $\langle w^+\rangle^{2\mu} \langle w^-\rangle^{\mu} \hat{p}^1$ into two functions, respectively. One is harmonic function with inhomogeneous boundary conditions, another one satisfies the Poisson equation with homogeneous boundary conditions:
\begin{align*}
\langle w^+\rangle^{2\mu} \langle w^-\rangle^{\mu} p^1=p^{11}_w+p^{12}_w,\,\,
\langle w^+\rangle^{2\mu} \langle w^-\rangle^{\mu} \hat{p}^1=\hat{p}^{11}_w+\hat{p}^{12}_w.
\end{align*}
They satisfy
 \begin{equation*}
\left\{
\begin{array}{l}
\Delta p^{11}_w =0 \quad \text{in}\quad\Om_f,\\
\Delta \hat{p}^{11}_w= 0 \quad \text{in}\quad\hat{\Om}_f,\\
p^{11}_w=\hat{p}^{11}_w =\langle w^+\rangle^{2\mu}\langle w^-\rangle^{\mu} p\quad\text{on}\quad\Gamma_f ,\\
\p_2 p^{11}_w =0\,\,   \textrm{on} \,\,  \Ga,
\quad \p_2 \hat{p}^{11}_w =0\,\,   \textrm{on} \,\,  \hat{\Ga},
\end{array}\right.
\end{equation*}
and
\begin{equation*}
\left\{
\begin{array}{l}
\Delta p^{12}_w  =
2\nabla  (\langle w^+\rangle^{2\mu}\langle w^-\rangle^{\mu}  )\cdot \nabla p^1+\Delta (\langle w^+\rangle^{2\mu}\langle w^-\rangle^{\mu} )p^1
\quad \text{in}\quad\Om_f,\\
\Delta \hat{p}^{12}_w= 2\nabla (\langle w^+\rangle^{2\mu}\langle w^-\rangle^{\mu}) \cdot \nabla \hat{p}^1+\Delta (\langle w^+\rangle^{2\mu}\langle w^-\rangle^{\mu})\hat{p}^1
\quad \text{in}\quad\hat{\Om}_f,\\
p^{12}_w=\hat{p}^{12}_w=0 \quad\text{on}\quad\Gamma_f ,\\
\p_2 p^{12}_w=0\,\,   \textrm{on} \,\,  \Ga,
\quad \p_2 \hat{p}^{12}_w=0\,\,   \textrm{on} \,\,  \hat{\Ga}.
\end{array}\right.
\end{equation*}
%
%
Now we estimate $p^{11}_w$ ($\hat{p}^{11}_w$) and $p^{12}_w$ ($\hat{p}^{12}_w$) similar to the estimate of $p^1$ ($\hat{p}^1$) and $p^2$ ($\hat{p}^2$),
By Lemma \ref{lemD3},
Lemma \ref{lemPP1}, Lemma \ref{lemPP2}, Lemma \ref{lemPP3} and the method of their proof, we would obtain similar conclusion to Lemma \ref{lemPP3} with weight $\langle w^+\rangle^{2\mu}\langle w^-\rangle^{\mu} $. Note that when the weight takes $\langle w^-\rangle^{2\mu}\langle w^+\rangle^{\mu} $, similar conclusion still holds. Thus we would derive the following bound.
\begin{lem}\label{lemPP4}
Let $s\geq 4$ be an integer, $\| \p_1f \|_{H^{s-\f12}(\BR)} \lesssim 1.$
For the pressure satisfying \eqref{U7}, there holds
\begin{align*}
\| \ud{N_f\cdot\nabla (\langle w^\pm\rangle^{2\mu} \langle w^\mp\rangle^{\mu} p^1)} + \ud{N_f\cdot\nabla (\langle w^\pm\rangle^{2\mu} \langle w^\mp\rangle^{\mu} \hat{p}^1)} \|_{H^{s-\f12} (\bR)}
\lesssim 
E_s^{\f12} G_s^{\f12}.
\end{align*}
\end{lem}

Now we prove Lemma \ref{lemPP0}.
\begin{proof}[Proof of Lemma \ref{lemPP0}]
By Lemma \ref{lemD9}, we infer that
\begin{align}\label{U10}
&\| \langle w^+\rangle^{2\mu} \langle w^-\rangle^{\mu} \langle \p_1 \rangle^{\f12} \p_1^{s-1}\big(\ud{N_f\cdot\nabla p^1} + \ud{N_f\cdot\nabla \hat{p}^1} \big) \|_{L^2 (\BR)} \\\nonumber
&\lesssim \| \langle \p_1 \rangle^{\f12}\p_1\big( \langle w^+\rangle^{2\mu} \langle w^-\rangle^{\mu} \p_1^{s-2}\big(\ud{N_f\cdot\nabla p^1} + \ud{N_f\cdot\nabla \hat{p}^1} \big) \big) \|_{L^2 (\BR)}\\\nonumber
&\quad +\|  \langle w^+\rangle^{2\mu} \langle w^-\rangle^{\mu} \langle \p_1 \rangle^{\f12} \p_1^{s-2}\big(\ud{N_f\cdot\nabla p^1} + \ud{N_f\cdot\nabla \hat{p}^1} \big) \|_{L^2 (\BR)}.
\end{align}
Note that
\begin{align*}
\langle w^+\rangle^{2\mu} \langle w^-\rangle^{\mu} \p_1^{s-2} F
=\p_1^{s-2}(\langle w^+\rangle^{2\mu} \langle w^-\rangle^{\mu}F)
-\sum_{a+b=s-2,a\neq 0} C_{s-2}^a \p_1^a \big(\langle w^+\rangle^{2\mu} \langle w^-\rangle^{\mu} \big) \p_1^b F.
\end{align*}
On the other hand, note that Lemma \ref{lemD5} and Lemma \ref{lemD9} still hold
if $\langle w^+\rangle^{\mu}\langle w^-\rangle^{\nu}$ on the left hand side of the inequalities is replaced by $\p_1^a (\langle w^+\rangle^{2\mu} \langle w^-\rangle^{\mu})$ ($a\geq 0$) and
$\langle w^+\rangle^{\mu}\langle w^-\rangle^{\nu}$ on the right hand side takes $\langle w^+\rangle^{2\mu} \langle w^-\rangle^{\mu}$. Repeatedly using this variant version of Lemma \ref{lemD5} and Lemma \ref{lemD9}, \eqref{U10} reduces to
\begin{align}\label{U12}
&\|  \langle w^+\rangle^{2\mu} \langle w^-\rangle^{\mu} \langle \p_1 \rangle^{\f12} \p_1^{s-1}\big(\ud{N_f\cdot\nabla p^1} + \ud{N_f\cdot\nabla \hat{p}^1} \big) \|_{L^2 (\BR)}\\\nonumber
&\lesssim \| \langle \p_1\rangle^{\f12}\p_1^{s-1}\big( \langle w^+\rangle^{2\mu} \langle w^-\rangle^{\mu} \big(\ud{N_f\cdot\nabla p^1} + \ud{N_f\cdot\nabla \hat{p}^1} \big) \big) \|_{L^2 (\BR)}\\\nonumber
&\quad +\|  \langle w^+\rangle^{2\mu} \langle w^-\rangle^{\mu} \langle \p_1 \rangle^{\f12} \p_1^{\leq s-2}\big(\ud{N_f\cdot\nabla p^1} + \ud{N_f\cdot\nabla \hat{p}^1} \big) \|_{L^2 (\BR)}\\\nonumber
&\lesssim \| \langle \p_1\rangle^{\f12}\p_1^{s-1}\big( \ud{N_f\cdot\nabla  (\langle w^+\rangle^{2\mu} \langle w^-\rangle^{\mu} p^1)} + \ud{N_f\cdot\nabla (\langle w^+\rangle^{2\mu} \langle w^-\rangle^{\mu}\hat{p}^1)} \big) \|_{L^2 (\BR)}\\\nonumber
&\quad+\| \langle \p_1\rangle^{\f12}\p_1^{s-1} \big(\ud{p^1N_f\cdot\nabla (\langle w^+\rangle^{2\mu} \langle w^-\rangle^{\mu}) } + \ud{\hat{p}^1N_f\cdot\nabla } (\langle w^+\rangle^{2\mu} \langle w^-\rangle^{\mu}) \big) \|_{L^2 (\BR)}\\\nonumber
&\quad +\|  \langle w^+\rangle^{2\mu} \langle w^-\rangle^{\mu} \langle \p_1 \rangle^{\f12} \p_1^{\leq s-2}\big(\ud{N_f\cdot\nabla p^1} + \ud{N_f\cdot\nabla \hat{p}^1} \big) \|_{L^2 (\BR)}.
\end{align}
For the first line of on the right hand side of \eqref{U12}, by Lemma \ref{lemPP4}, Lemma \ref{lemPFH} and Lemma \ref{lemD5}, we obtain
\begin{align*}
\| \langle \p_1\rangle^{\f12}\p_1^{s-1}\big( \ud{N_f\cdot\nabla  (\langle w^+\rangle^{2\mu} \langle w^-\rangle^{\mu}p^1)} + \ud{N_f\cdot\nabla (\langle w^+\rangle^{2\mu} \langle w^-\rangle^{\mu}\hat{p}^1)} \big)  \|_{L^2 (\BR)} \lesssim E_s^{\f12} G_s^{\f12}.
\end{align*}
Note that for  Lemma \ref{lemD11}, the weight $\langle w^+\rangle^{\mu}\langle w^-\rangle^{\nu}$ on the left hand side of Lemma \ref{lemD11} is replaced by $\p_1W$, the lemma still holds.
Thus for the second line of on the right hand side of \eqref{U12}, by weighted trace Lemma \ref{lemD3},
weighted chain rule Lemma \ref{lemG3}, the estimate of pressure Lemma \ref{lemPre}, variant of weighted commutator estimate Lemma \ref{lemD11}, and by classical commutator estimate,
it is bounded by
\begin{align*}
&\| \langle \p_1\rangle^{\f12}\p_1^{s-1} \big(\ud{p^1N_f\cdot\nabla (\langle w^+\rangle^{2\mu} \langle w^-\rangle^{\mu}) } + \ud{\hat{p}^1N_f\cdot\nabla } (\langle w^+\rangle^{2\mu} \langle w^-\rangle^{\mu}) \big) \|_{L^2 (\BR)}\\\nonumber
&= 2 \| \langle \p_1\rangle^{\f12}\p_1^{s-1} \big(\ud{p} \p_1f\p_1(\langle w^+\rangle^{2\mu} \langle w^-\rangle^{\mu}) \big)   \|_{L^2 (\BR)}\\\nonumber
&\lesssim \| \langle \p_1\rangle^{\f12}\p_1^{s-1} \big(\ud{p} \p_1(\langle w^+\rangle^{2\mu} \langle w^-\rangle^{\mu}) \big)   \|_{L^2 (\BR)} \|\p_1f\|_{H^{s-\f12}(\bR)}\\\nonumber
&\lesssim E_s^{\f12} G_s^{\f12} \|\p_1f\|_{H^{s-\f12}(\bR)}.
\end{align*}
For the last  line of \eqref{U12},
by using weighted commutator estimate Lemma \ref{lemD11}, weighted chain rule Lemma \ref{lemG3}, weighted trace Lemma \ref{lemD3},
the estimate of $p^1$ Lemma \ref{lemPP1}, we derive that
\begin{align*}
&\| \langle w^+\rangle^{2\mu} \langle w^-\rangle^{\mu} \langle \p_1 \rangle^{\f12} \p_1^{\leq s-2}\ud{N_f\cdot\nabla p^1} \|_{L^2 (\BR)}\\\nonumber
&\leq\| \langle w^+\rangle^{2\mu} \langle w^-\rangle^{\mu}   \langle \p_1 \rangle^{\f12} \p_1^{\leq s-2}(\p_1f\ud{\p_1 p^1})   \|_{L^2(\BR)}
+\| \langle w^+\rangle^{2\mu} \langle w^-\rangle^{\mu}   \langle \p_1 \rangle^{\f12} \p_1^{\leq s-2}\ud{\p_2 p^1} \|_{L^2(\BR)} \\\nonumber
&\lesssim \| \langle w^+\rangle^{2\mu} \langle w^-\rangle^{\mu}   \langle \p_1 \rangle^{\f12} \p_1^{\leq s-2} \ud{\p_1 p^1}  \|_{L^2 (\bR)}  \| \p_1f\|_{H^{s-\f32}(\bR)}
+\| \langle w^+\rangle^{2\mu} \langle w^-\rangle^{\mu}   \langle \p_1 \rangle^{\f12} \p_1^{\leq s-2} \ud{\p_2 p^1}  \|_{L^2 (\bR)}\\\nonumber
&\lesssim \| \langle w^+\rangle^{2\mu} \langle w^-\rangle^{\mu}   \p^{\leq s-1}\nabla p^1  \|_{L^2 (\Om_f)} (1+\| \p_1f\|_{H^{s-\f32}(\bR)})\\\nonumber
&\lesssim  E_s^{\f12} G_s^{\f12} (1+\| \p_1f\|_{H^{s-\f12}(\bR)}) .
\end{align*}
Similarly,
\begin{align*}
\| \langle w^+\rangle^{2\mu} \langle w^-\rangle^{\mu} \langle \p_1 \rangle^{\f12} \p_1^{\leq s-2}\ud{N_f\cdot\nabla \hat{p}^1} \|_{L^2 (\BR)}
\lesssim  E_s^{\f12} G_s^{\f12} (1+\| \p_1f\|_{H^{s-\f12}(\bR)}) .
\end{align*}
On the other hand, similarly, by using Lemma \ref{lemD11}, Lemma \ref{lemG3} and Lemma \ref{lemPP1}, there holds
\begin{align*}
&\|  \langle w^+\rangle^{2\mu} \langle w^-\rangle^{\mu} \langle \p_1 \rangle^{\f12} \p_1^{s-1}\big(\ud{N_f\cdot\nabla p^2} + \ud{N_f\cdot\nabla \hat{p}^2} \big) \|_{L^2 (\BR)}\\\nonumber
 &\lesssim \big(\| \langle w^+\rangle^{2\mu} \langle w^-\rangle^{\mu}   \p^{\leq s}\nabla p^2  \|_{L^2 (\Om_f)}+\| \langle w^+\rangle^{2\mu} \langle w^-\rangle^{\mu}   \p^{\leq s}\nabla \hat{p}^2  \|_{L^2 (\hat{\Om}_f)}\big) (1+\| \p_1f\|_{H^{s-1}(\bR)})\\\nonumber
&\lesssim  E_s^{\f12} G_s^{\f12} (1+\| \p_1f\|_{H^{s-\f12}(\bR)}) .
\end{align*}
Combining all the above estimates, we derive
\begin{align*}
&\|  \langle w^+\rangle^{2\mu} \langle w^-\rangle^{\mu} \langle \p_1 \rangle^{\f12} \p_1^{s-1}\big(\ud{N_f\cdot\nabla p} + \ud{N_f\cdot\nabla \hat{p}} \big) \|_{L^2 (\BR)}\\\nonumber
&\lesssim \|  \langle w^+\rangle^{2\mu} \langle w^-\rangle^{\mu} \langle \p_1 \rangle^{\f12} \p_1^{s-1}\big(\ud{N_f\cdot\nabla p^1} + \ud{N_f\cdot\nabla \hat{p}^1} \big) \|_{L^2 (\BR)}\\\nonumber
&\quad+\|  \langle w^+\rangle^{2\mu} \langle w^-\rangle^{\mu} \langle \p_1 \rangle^{\f12} \p_1^{s-1}\big(\ud{N_f\cdot\nabla p^2} + \ud{N_f\cdot\nabla \hat{p}^2} \big) \|_{L^2 (\BR)} \\\nonumber
&\lesssim  E_s^{\f12} G_s^{\f12} (1+\| \p_1f\|_{H^{s-\f12}(\bR)}) .
\end{align*}
This ends the proof of the lemma.
\end{proof}

\section{Weighted energy estimate of the free surface}\label{Ef}
In this section, we are going to estimate the amplitude of the free surface and
deal with the weighted energy estimate of the free surface.
The estimation the amplitude of the free surface is carried out using the characteristic integral method.
Here, we only  require the estimation of the amplitude of the free surface without weights.
The weighted energy estimate of the free surface is performed by applying the weights $\langle w^\pm\rangle^{2\mu}$
and by utilizing the ghost weight method. Here, we recall $w^\pm=x_1\pm t$ where $\f12 < \mu\leq \f35$.
The nonlinearities in the weighted energy estimate of the free surface consists of three part.
The first part involves the highest order nonlinear terms that may lead to derivative loss.
The second part comprises the nonlinear commutator terms that do not incur derivative loss.
The third part includes the pressure terms arising from $N_f \cdot \nabla p + N_f \cdot \nabla \hat{p}$.
The commutators can be estimated by employing the intrinsic null structure and  the weighted commutator estimates for fractional derivatives in Section 3.
The weighted estimate of $N_f \cdot \nabla p + N_f \cdot \nabla \hat{p}$ has been demonstrated in Section 5. 
Regarding the highest order nonlinear terms that may lead to derivative loss issues, a detailed and lengthy analysis of the nonlinearities reveals that the system does exhibit the necessary symmetry to save one order derivative. However, for the term \eqref{rm}, it appears that the null structure is obscured.
To address this issue, we  apply increased power of weights
for $\Lambda_\pm$ and $\hat{\Lambda}_\pm$ in the lower-order terms.

\subsection{Estimate of the amplitude of the free surface}

In this subsection, we are going to show
\begin{align*}
\| f(t,\cdot)\|_{L^\infty(\BR)}
&\leq \| f(0,\cdot)\|_{L^\infty(\BR)}
+\f12 M\sup_{0\leq \tau\leq t} (1+\|\p_1f(\tau,\cdot)\|_{L^\infty})
\| \langle w^\pm \rangle^{2\mu}   \La_{\pm}(\tau,\cdot) \|_{L^\infty}
 \, , 
\end{align*}
for $\mu>\f12$. Here $M=\int_{\bR} \langle x_1\rangle^{-2\mu} \d x_1$.

Note that the kinetic boundary equation 
can be organized as follows:
\begin{align*}
(\p_t \pm \p_1)f=\ud{\La_{\pm}}\cdot N_f= \ud{\La^2_{\pm}}-\ud{\La^1_{\pm}}\p_1f \,.
\end{align*}
Introducing the characteristics $\phi_\pm(s)=(x_1\mp t)\pm s$.
Along the characteristics $\phi_\pm(s)$, $f(t,x_1)$ satisfies
\begin{align*}
\f{\d }{\ds} f(s,\phi_\pm(s) )
=(\p_t \pm \p_1)f (s, \phi_\pm(s) )=(\La^2_{\pm}-\La^1_{\pm}\p_1f)(s,\phi_\pm (s),x_2)\big|_{x_2=f(s,\phi_\pm(s))}\, .
\end{align*}
Taking integrals along the characteristics over $[0,t]$, one has
\begin{align*}
f(t,x_1)=f(0,\phi_\pm(0))+\int_0^t (\La^2_{\pm}-\La^1_{\pm}\p_1f)(\tau,\phi_\pm (\tau),x_2)\big|_{x_2=f(\tau,\phi_\pm(\tau))}  \d\tau\,.
\end{align*}
We calculate
\begin{align*}
&\int_0^t (\La^2_{\pm}-\La^1_{\pm}\p_1f)(\tau,\phi_\pm (\tau),x_2)\big|_{x_2=f(\tau,\phi_\pm(\tau))} \d\tau \\
&\leq \sup_{0\leq \tau\leq t}
(1+\|\p_1f(\tau,\cdot)\|_{L^\infty})\| \langle w^\pm \rangle^{2\mu} \La_\pm (\tau,\cdot)  \|_{L^\infty}
\int_0^t  \langle w^\pm (\tau,x_1) \rangle^{-2\mu} |_{x_1= \phi_\pm (\tau)} \d\tau\\
&\leq \sup_{0\leq \tau\leq t} (1+\|\p_1f(\tau,\cdot)\|_{L^\infty})\| \langle w^\pm \rangle^{2\mu} \La_\pm (\tau,\cdot) \|_{L^\infty}
\int_{-\infty}^{\infty} \langle w^\pm  \rangle^{-2\mu} \d w^\pm
\cdot\sup_{\tau}\f{1}{\big|\f{\d w^\pm(\tau,\phi_\pm (\tau)) }{\d\tau}\big|}\\
&= \f12M\sup_{0\leq \tau\leq t} (1+\|\p_1f(\tau,\cdot)\|_{L^\infty})\| \langle w^\pm \rangle^{2\mu} \La_\pm (\tau,\cdot) \|_{L^\infty} \, .
\end{align*}
Then the conclusion is obtained.

\subsection{Weighted energy estimate of the free surface}
In this subsection, we are going to deal with the weighted energy estimate of the free surface.
Recalling that for positive integer $s\geq 1$, the weighted energy of the free surface is defined as follows:
\begin{align*}
E_{s+\f12}^f(t)=\sum_{+,-}\sum_{|a|\leq s-1}
\|\langle w^\pm\rangle^{2\mu}(\p_t\pm\p_1)\langle \p_1\rangle^{ \f12} \p_1^a f(t,\cdot)\|_{L^2(\bR)}^2 \,,
\end{align*}
where $w^\pm=x_1\pm t$, $\f12 < \mu\leq \f35$. The corresponding ghost weight energy is defined as follows
\begin{align*}
G_{s+\f12}^f(t)=\sum_{+,-} \sum_{|a|\leq s-1}
\big\|\f{\langle w^\pm\rangle^{2\mu}(\p_t\pm\p_1)\langle \p_1\rangle^{ \f12} \p_1^a f}{\langle w^\mp\rangle^{\mu}} (t,\cdot)
\big\|_{L^2(\bR)}^2 \, .
\end{align*}
In the sequel, we only present the energy estimate for 
$\sum_{|a|\leq s-1}\|\langle w^-\rangle^{2\mu}(\p_t-\p_1)\langle \p_1\rangle^{ \f12} \p_1^a f\|_{L^2}^2 .$
The estimate for the right Alfv\'en waves $\sum_{|a|\leq s-1}\|\langle w^+\rangle^{2\mu}(\p_t+\p_1)\langle \p_1\rangle^{ \f12} \p_1^a f\|_{L^2}^2$
is the same and the details are omitted.

We recall that the equation \eqref{BB10} for the free surface as follows
\begin{align}\label{E9}
\p_t^2 f-\p_1^2 f
=&-\f12\big(\ud{\La_{+}^1}+\ud{\hat{\La}_{+}^1} \big) \p_1 (\p_t-  \p_1) f
 -\f12\big(\ud{\La_{-}^1}+\ud{\hat{\La}_{-}^1} \big) \p_1 (\p_t+  \p_1) f\\\nonumber
&-\f12\big(\ud{\La_{+}^1} \ud{\La_{-}^1} +\ud{\hat{\La}_{+}^1} \ud{\hat{\La}_{-}^1} \big)\p_1^2 f
-\f12(N_f\cdot\ud{\nabla p} +N_f\cdot\ud{\nabla \hat{p}}) \, .
\end{align}

Now we study the weighted energy estimate for the free surface with the aid of the ghost weight method.
The ghost weight method was introduced by Alinhac in the study of the two dimensional wave equations \cite{Alinhac00}.
Here we use a variant version of the ghost weight method of Alinhac for the one dimensional wave equations, similar to the modifications made by the authors \cite{CL}.

Let $q(\theta)=\int_0^\theta \langle\tau\rangle^{-2\mu} \d\tau$, $\f12<\mu\leq \f35,$
thus $|q(\theta)|\lesssim 1$.
To simply the notations, we abbreviate $e^{q^+}=e^{q(w^-)}$ and $e^{q^-}=e^{q(-w^+)}$.
Let $s\geq 4$ be an integer, $0\leq a\leq s-1$.
Applying $\langle \p_1\rangle^{\f12}\p_1^a$ to \eqref{E9}.
Then taking the $L^2$ inner product of the consequent equations with
 $2\langle w^-\rangle^{4\mu}(\p_t-\p_1)\langle \p_1\rangle^{\f12}\p_1^af\cdot e^{q^-}$
(to study the weighted energy estimate for $(\p_t+\p_1)f$, similarly we take the $L^2$ inner product of the consequent equations with $2\langle w^+\rangle^{4\mu}(\p_t+\p_1)\langle \p_1\rangle^{\f12}\p_1^af \cdot e^{q^+}$),
one has
\begin{align}\label{E10}
& \f{\d}{\dt}
\int_{\BR}\big| \langle w^-\rangle^{2\mu}(\p_t-\p_1) \langle \p_1\rangle^{\f12}\p_1^a f
\big|^2 e^{q^-}\d x_1
+2\int_{\BR}\f{\big| \langle w^-\rangle^{2\mu}(\p_t-\p_1) \langle \p_1\rangle^{\f12}\p_1^a f
\big|^2}{\langle w^+\rangle^{2\mu}} e^{q^-}\d x_1\\\nonumber
&=2\int_{\BR}(\p_t+\p_1) \langle w^-\rangle^{2\mu} (\p_t-\p_1) \langle \p_1\rangle^{\f12}\p_1^a f \cdot \langle w^-\rangle^{2\mu}(\p_t-\p_1) \langle \p_1\rangle^{\f12}\p_1^a f e^{q^-}\d x_1\\\nonumber
&=-\int_{\BR} \langle w^-\rangle^{4\mu} (\p_t-\p_1) \langle \p_1\rangle^{\f12}\p_1^a f\cdot
(\ud{\La^1_{+}}\ud{\La^1_{-}}
+\ud{\hat{\La}^1_{+}}\ud{\hat{\La}^1_{-}})\p_1^2 \langle \p_1\rangle^{\f12}\p_1^a f e^{q^-}\d x_1\\\nonumber
&\quad-\int_{\BR} \langle w^-\rangle^{4\mu} (\p_t-\p_1)\langle \p_1\rangle^{\f12}\p_1^a f\cdot
  (\ud{\La^1_{+}}+\ud{\hat{\La}^1_{+}}) \p_1(\p_t-\p_1)
 \langle \p_1\rangle^{\f12}\p_1^a f e^{q^-} \d x_1 \\
&\quad-\int_{\BR} \langle w^-\rangle^{4\mu} (\p_t-\p_1)\langle \p_1\rangle^{\f12}\p_1^a f\cdot
  (\ud{\La^1_{-}}+\ud{\hat{\La}^1_{-}})\p_1 (\p_t+\p_1)
 \langle \p_1\rangle^{\f12}\p_1^a f e^{q^-} \d x_1 \nonumber\\
&\quad-\int_{\BR} \langle w^-\rangle^{4\mu} (\p_t-\p_1)\langle \p_1\rangle^{\f12}\p_1^a f\cdot
  [\langle \p_1\rangle^{\f12}\p_1^a,\ud{\La^1_{+}}\ud{\La^1_{-}}
+\ud{\hat{\La}^1_{+}}\ud{\hat{\La}^1_{-}}] \p_1^2f e^{q^-} \d x_1 \nonumber\\
&\quad-\int_{\BR} \langle w^-\rangle^{4\mu} (\p_t-\p_1)\langle \p_1\rangle^{\f12}\p_1^a f\cdot
  [\langle \p_1\rangle^{\f12}\p_1^a,\ud{\La^1_{+}}+\ud{\hat{\La}^1_{+}}]
  \p_1(\p_t-\p_1) f e^{q^-} \d x_1 \nonumber\\
&\quad-\int_{\BR} \langle w^-\rangle^{4\mu} (\p_t-\p_1)\langle \p_1\rangle^{\f12}\p_1^a f\cdot
  [\langle \p_1\rangle^{\f12}\p_1^a,\ud{\La^1_{-}}+\ud{\hat{\La}^1_{-}}]
  \p_1(\p_t+\p_1) f e^{q^-} \d x_1 \nonumber\\
&\quad-\int_{\BR} \langle w^-\rangle^{4\mu} (\p_t-\p_1)\langle \p_1\rangle^{\f12}\p_1^a f\cdot
\langle \p_1\rangle^{\f12}\p_1^a\Big( N_f\cdot(\ud{\nabla p} +\ud{\nabla \hat{p} })\Big) e^{q^-} \d x_1 \nonumber\\\nonumber
&=A_1+A_2+A_3+A_4+A_5+A_6+A_7.
\end{align}
Here $A_1$, $A_2$, $A_3$ correspond to the terms containing the highest-order derivative, which may cause derivative loss.
$A_4$, $A_5$, $A_6$ correspond to the commutator terms which do not have derivative loss.
$A_7$ is the pressure term. In the sequel, we will divide the estimate into several steps.

\textbf{Estimate of the highest-order terms $A_1$, $A_2$, $A_3$}


\textbf{ Estimate of $A_1$}
\begin{align*}
A_1&=-\int_{\BR} \langle w^-\rangle^{4\mu} (\p_t-\p_1) \langle \p_1\rangle^{\f12}\p_1^a f\cdot
(\ud{\La^1_{+}} \ud{\La^1_{-}}
+\ud{\hat{\La}^1_{+}} \ud{\hat{\La}^1_{-}})  \langle \p_1\rangle^{\f12}\p_1^{a+2} f e^{q^-}\d x_1\nonumber\\
&= \int_{\BR} (\p_t-\p_1) \langle \p_1\rangle^{\f12}\p_1^af
\p_1\langle \p_1\rangle^{\f12}\p_1^a f\p_1  \Big(e^{q^-}\langle w^-\rangle^{4\mu}
(\ud{\La^1_{+}} \ud{\La^1_{-}}
+\ud{\hat{\La}^1_{+}} \ud{\hat{\La}^1_{-}})\Big) \d x_1 \\
&\quad -\f12 \int_{\BR} | \langle \p_1\rangle^{\f12}\p_1^{a+1} f |^2 \cdot
(\p_t-\p_1)\Big( e^{q^-}\langle w^-\rangle^{4\mu} (\ud{\La^1_{+}} \ud{\La^1_{-}}
+\ud{\hat{\La}^1_{+}} \ud{\hat{\La}^1_{-}}) \Big) \d x_1 \\
&\quad +\f12 \f{\d}{\dt}\int_{\BR}
| \langle \p_1\rangle^{\f12}\p_1^{a+1} f|^2 \cdot
e^{q^-}\langle w^-\rangle^{4\mu} (\ud{\La^1_{+}} \ud{\La^1_{-}}
+\ud{\hat{\La}^1_{+}} \ud{\hat{\La}^1_{-}})  \d x_1\\
&=A_{11}+A_{12}+A_{13}.
\end{align*}
$A_{11}$ has good control. By weighted Sobolev inequalities Lemma \ref{Sobo2} and weighted trace Lemma \ref{lemD2}, one has
\begin{align*}
A_{11}
&\lesssim \Big\| \f{\langle w^-\rangle^{2\mu}(\p_t-\p_1) \langle \p_1\rangle^{\f12}\p_1^af}{\langle w^+\rangle^{\mu}}\Big\|_{L^2(\bR)}
\|\p_1\langle \p_1\rangle^{\f12}\p_1^a f\|_{L^2(\bR)} \\
&\qquad \cdot \|\langle w^+\rangle^{-\mu} \p_1  \Big(e^{q^-}\langle w^-\rangle^{4\mu}
(\ud{\La^1_{+}} \ud{\La^1_{-}}
+\ud{\hat{\La}^1_{+}} \ud{\hat{\La}^1_{-}})\Big) \|_{L^\infty(\bR)}\\
&\lesssim E_s^{\f12}  G_s^{\f12} E_3^{\f12}  G_3^{\f12} (1+\| \p_1f\|_{L^\infty}).
\end{align*}
 $A_{12}$ and $A_{13}$ need further concerns.


\textbf{Estimate of $A_2$}

By weighted Sobolev inequalities Lemma \ref{Sobo2} and weighted trace Lemma \ref{lemD2}, one derives that
\begin{align*}
A_2=&-\int_{\BR} \langle w^-\rangle^{4\mu} (\p_t-\p_1)\langle \p_1\rangle^{\f12}\p_1^a f\cdot
  (\ud{\La^1_{+}}+\ud{\hat{\La}^1_{+}}) \p_1 (\p_t-\p_1)
 \langle \p_1\rangle^{\f12}\p_1^a f e^{q^-} \d x_1 \nonumber\\
=&\f12\int_{\BR} |(\p_t-\p_1)\langle \p_1\rangle^{\f12}\p_1^a f|^2\cdot \p_1 \big(
   \langle w^-\rangle^{4\mu}(\ud{\La^1_{+}}+\ud{\hat{\La}^1_{+}})  e^{q^-}\big) \d x_1 \nonumber\\
\lesssim & E_3^{\f12}G_s (1+\| \p_1f\|_{L^\infty}) .
\end{align*}

\textbf{Estimate of $A_3$}
\begin{align*}
A_3&=-\int_{\BR} \langle w^-\rangle^{4\mu} (\p_t-\p_1)\langle \p_1\rangle^{\f12}\p_1^a f\cdot
  (\ud{\La^1_{-}}+\ud{\hat{\La}^1_{-}})\cdot \p_1(\p_t+\p_1)
 \langle \p_1\rangle^{\f12}\p_1^a f e^{q^-} \d x_1 \nonumber\\
&=-\f{\d}{\dt}\int_{\BR} \langle w^-\rangle^{4\mu}  (\p_t-\p_1)\langle \p_1\rangle^{\f12}\p_1^a  f \big(\ud{\La_{-}^1}
+\ud{\hat{\La}_{-}^1}) \cdot \p_1  \langle \p_1 \rangle^{\f12}\p_1^a  f e^{q^-} \d x_1  \nonumber\\
&\quad+\int_{\BR} \langle w^-\rangle^{4\mu}  (\p_t-\p_1)\langle \p_1\rangle^{\f12}\p_1^a  f\cdot (\p_t+ \p_1)\big(\ud{\La_{-}^1}+\ud{\hat{\La}_{-}^1})
 \p_1  \langle \p_1\rangle^{\f12}\p_1^a  f e^{q^-} \d x_1  \nonumber\\
&\quad+\int_{\BR} \langle w^-\rangle^{4\mu} (\p_t-  \p_1)\langle \p_1\rangle^{\f12}\p_1^a  f \big(\ud{\La_{-}^1}
+\ud{\hat{\La}_{-}^1}) \cdot   \p_1  \langle \p_1\rangle^{\f12}\p_1^af(\p_t+\p_1) e^{q^-}\d x_1 \nonumber\\
&\quad+\int_{\BR} \langle w^-\rangle^{4\mu} (\p_t-  \p_1)(\p_t+\p_1)\langle \p_1\rangle^{\f12}\p_1^a  f \big(\ud{\La_{-}^1}
+\ud{\hat{\La}_{-}^1}) \cdot   \p_1  \langle \p_1\rangle^{\f12}\p_1^af e^{q^-} \d x_1 \\
&=A_{31}+A_{32}+A_{33}+A_{34}.
\end{align*}
Here $A_{31}$ will absorbed into the energy. 
For$A_{32}$,
by \eqref{B20}, there holds
\begin{align*}
	(\p_t +\p_1) \big( \ud{\La_{-}^1}+\ud{\hat{\La}_{-}^1} \big)
	&= -\ud{\La_{+}^1}\p_1\ud{\La_{-}^1}
	-\ud{\hat{\La}_{+}^1}\p_1\ud{\hat{\La}_{-}^1}
	-\ud{\p_1 p}-\ud{\p_1 \hat{p}}\,.
\end{align*}
Thus by weighted Sobolev inequalities Lemma \ref{Sobo2}, weighted trace Lemma \ref{lemD2} and Lemma \ref{lemPre}, one has
\begin{align*}
A_{32}= &
\int_{\BR} \langle w^-\rangle^{4\mu}  (\p_t-\p_1)\langle \p_1\rangle^{\f12}\p_1^a  f\cdot (\p_t+ \p_1)\big(\ud{\La_{-}^1}+\ud{\hat{\La}_{-}^1})
 \p_1  \langle \p_1\rangle^{\f12}\p_1^a  f e^{q^-} \d x_1  \nonumber\\
 \leq& \Big\| \f{\langle w^-\rangle^{2\mu}  (\p_t-\p_1)\langle \p_1\rangle^{\f12}\p_1^a  f}{ \langle w^+\rangle^{\mu} } \Big\|_{L^2}
 \big\| \langle w^-\rangle^{2\mu}\langle w^+\rangle^{\mu} (\p_t+ \p_1)\big(\ud{\La_{-}^1}+\ud{\hat{\La}_{-}^1})\big\|_{L^\infty}
 \|\p_1  \langle \p_1\rangle^{\f12}\p_1^a  f \|_{L^2}\\\nonumber
 \lesssim& E_s^{\f12}  G_s^{\f12} E_3^{\f12}  G_3^{\f12} (1+\| \p_1f\|_{L^\infty}) \,.
\end{align*}
For $A_{33}$, by weighted Sobolev inequalities Lemma \ref{Sobo2} and weighted trace Lemma \ref{lemD2}, $A_{33}$ can be controlled by
\begin{align*}
A_{33}\lesssim E_s^{\f12} G_s^{\f12} G_2^{\f12}.
\end{align*}
$A_{34}$ need further calculation.
Actually, by the equation of the free surface \eqref{E9}, plugging  $(\p_t^2-\p_1^2)\langle \p_1\rangle^{\f12}\p_1^a  f$
into $A_{34}$, we obtain
\begin{align*}
A_{34}=&\int_{\BR} \langle w^-\rangle^{4\mu} (\p_t^2-\p_1^2)\langle \p_1\rangle^{\f12}\p_1^a  f \big(\ud{\La_{-}^1}
+\ud{\hat{\La}_{-}^1}) \cdot   \p_1  \langle \p_1\rangle^{\f12}\p_1^af e^{q^-} \d x_1 \\
=&-\f12\int_{\BR} \langle w^-\rangle^{4\mu} \langle \p_1\rangle^{\f12}\p_1^a
\Big( \big(\ud{\La_{+}^1}+\ud{\hat{\La}_{+}^1} \big) \p_1 (\p_t-  \p_1) f  \Big)  \big(\ud{\La_{-}^1}
+\ud{\hat{\La}_{-}^1}) \cdot   \p_1  \langle \p_1\rangle^{\f12}\p_1^af e^{q^-} \d x_1\\
&-\f12\int_{\BR} \langle w^-\rangle^{4\mu} \langle \p_1\rangle^{\f12}\p_1^a
\Big( \big(\ud{\La_{-}^1}+\ud{\hat{\La}_{-}^1} \big) \p_1 (\p_t+  \p_1) f  \Big)  \big(\ud{\La_{-}^1}
+\ud{\hat{\La}_{-}^1}) \cdot   \p_1  \langle \p_1\rangle^{\f12}\p_1^af e^{q^-} \d x_1\\
&-\f12\int_{\BR} \langle w^-\rangle^{4\mu} \langle \p_1\rangle^{\f12}\p_1^a
\Big( \big(\ud{\La_{+}^1} \ud{\La_{-}^1} +\ud{\hat{\La}_{+}^1} \ud{\hat{\La}_{-}^1} \big)\p_1^2 f  \Big)  \big(\ud{\La_{-}^1}
+\ud{\hat{\La}_{-}^1}) \cdot   \p_1  \langle \p_1\rangle^{\f12}\p_1^af e^{q^-} \d x_1\\
&-\f12\int_{\BR} \langle w^-\rangle^{4\mu} \langle \p_1\rangle^{\f12}\p_1^a
\Big( N_f\cdot\ud{\nabla p} +N_f\cdot\ud{\nabla \hat{p}}  \Big)  \big(\ud{\La_{-}^1}
+\ud{\hat{\La}_{-}^1}) \cdot   \p_1  \langle \p_1\rangle^{\f12}\p_1^af e^{q^-} \d x_1\\
=&-\f12\int_{\BR} \langle w^-\rangle^{4\mu}
 \big(\ud{\La_{+}^1}+\ud{\hat{\La}_{+}^1} \big) \p_1 \langle \p_1\rangle^{\f12}\p_1^a (\p_t-  \p_1) f  \big(\ud{\La_{-}^1}
+\ud{\hat{\La}_{-}^1}) \cdot   \p_1  \langle \p_1\rangle^{\f12}\p_1^af e^{q^-} \d x_1\\
& -\f12\int_{\BR} \langle w^-\rangle^{4\mu}
 \big(\ud{\La_{-}^1}+\ud{\hat{\La}_{-}^1} \big) \p_1 \langle \p_1\rangle^{\f12}\p_1^a(\p_t+  \p_1) f    \big(\ud{\La_{-}^1}
+\ud{\hat{\La}_{-}^1}) \cdot   \p_1  \langle \p_1\rangle^{\f12}\p_1^af e^{q^-} \d x_1\\
& -\f12\int_{\BR} \langle w^-\rangle^{4\mu}
\big(\ud{\La_{+}^1} \ud{\La_{-}^1} +\ud{\hat{\La}_{+}^1} \ud{\hat{\La}_{-}^1} \big)\langle \p_1\rangle^{\f12}\p_1^a \p_1^2 f    \big(\ud{\La_{-}^1}
+\ud{\hat{\La}_{-}^1}) \cdot   \p_1  \langle \p_1\rangle^{\f12}\p_1^af e^{q^-} \d x_1\\
&-\f12\int_{\BR} \langle w^-\rangle^{4\mu}
 \big[\langle \p_1\rangle^{\f12}\p_1^a,\big(\ud{\La_{+}^1}+\ud{\hat{\La}_{+}^1} \big) \big] \p_1(\p_t-  \p_1) f  \cdot \big(\ud{\La_{-}^1}
+\ud{\hat{\La}_{-}^1}) \cdot   \p_1  \langle \p_1\rangle^{\f12}\p_1^af e^{q^-} \d x_1\\
& -\f12\int_{\BR} \langle w^-\rangle^{4\mu}
\big[\langle \p_1\rangle^{\f12}\p_1^a, \big(\ud{\La_{-}^1}+\ud{\hat{\La}_{-}^1} \big)\big]\p_1(\p_t+  \p_1) f  \cdot  \big(\ud{\La_{-}^1}
+\ud{\hat{\La}_{-}^1}) \cdot   \p_1  \langle \p_1\rangle^{\f12}\p_1^af e^{q^-} \d x_1\\
& -\f12\int_{\BR} \langle w^-\rangle^{4\mu}\big[\langle \p_1\rangle^{\f12}\p_1^a,
\big(\ud{\La_{+}^1} \ud{\La_{-}^1} +\ud{\hat{\La}_{+}^1} \ud{\hat{\La}_{-}^1} \big)\big] \p_1^2 f  \cdot  \big(\ud{\La_{-}^1}
+\ud{\hat{\La}_{-}^1}) \cdot   \p_1  \langle \p_1\rangle^{\f12}\p_1^af e^{q^-} \d x_1\\
&-\f12\int_{\BR} \langle w^-\rangle^{4\mu} \langle \p_1\rangle^{\f12}\p_1^a
\Big( N_f\cdot\ud{\nabla p} +N_f\cdot\ud{\nabla \hat{p}}  \Big)  \big(\ud{\La_{-}^1}
+\ud{\hat{\La}_{-}^1}) \cdot   \p_1  \langle \p_1\rangle^{\f12}\p_1^af e^{q^-} \d x_1\\
=&B_1+B_2+B_3+B_4+B_5+B_6+B_7.
\end{align*}
Here $B_1$, $B_2$, $B_3$ correspond to the terms containing the highest-order derivative, which may cause derivative loss. $B_4$, $B_5$, $B_6$ correspond to the commutator terms which do not have derivative loss.
$B_7$ is the pressure term.
In the sequel, we will estimate $B_1$-$B_7$.

We first calculate $B_1$, $B_2$, $B_3$.
Simple calculation yields
\begin{align*}
B_1=&-\f12\int_{\BR} \langle w^-\rangle^{4\mu}
 \big(\ud{\La_{+}^1}+\ud{\hat{\La}_{+}^1} \big)\big(\ud{\La_{-}^1}+\ud{\hat{\La}_{-}^1})
  \p_1 \langle \p_1\rangle^{\f12}\p_1^a (\p_t-  \p_1) f   \cdot   \p_1  \langle \p_1\rangle^{\f12}\p_1^af e^{q^-} \d x_1\\
=&-\f14\f{\d}{\dt}\int_{\BR} \langle w^-\rangle^{4\mu}
 \big(\ud{\La_{+}^1}+\ud{\hat{\La}_{+}^1} \big)\big(\ud{\La_{-}^1}+\ud{\hat{\La}_{-}^1})
  \big|\p_1 \langle \p_1\rangle^{\f12}\p_1^a  f\big|^2 e^{q^-} \d x_1\\
&  +\f14\int_{\BR} (\p_t-  \p_1)  \Big(\langle w^-\rangle^{4\mu}
 \big(\ud{\La_{+}^1}+\ud{\hat{\La}_{+}^1} \big)\big(\ud{\La_{-}^1}+\ud{\hat{\La}_{-}^1}) e^{q^-} \Big)
 \big| \p_1 \langle \p_1\rangle^{\f12}\p_1^a f  \big|^2  \d x_1\\
&=B_{11}+B_{12}.
\end{align*}
They need further concerns.

For $B_2$,
\begin{align*}
B_2&=-\f12\int_{\BR} \langle w^-\rangle^{4\mu}
 \big(\ud{\La_{-}^1}+\ud{\hat{\La}_{-}^1} \big)^2 \p_1 \langle \p_1\rangle^{\f12}\p_1^a(\p_t+  \p_1) f
 \cdot   \p_1  \langle \p_1\rangle^{\f12}\p_1^af e^{q^-} \d x_1\\
&=-\f14\f{\d}{\dt}\int_{\BR} \langle w^-\rangle^{4\mu}
 \big(\ud{\La_{-}^1}+\ud{\hat{\La}_{-}^1} \big)^2
  \big|\p_1 \langle \p_1\rangle^{\f12}\p_1^a  f\big|^2 e^{q^-} \d x_1\\
&\quad  +\f14\int_{\BR} (\p_t+  \p_1)  \Big(\langle w^-\rangle^{4\mu}
 \big(\ud{\La_{-}^1}+\ud{\hat{\La}_{-}^1} \big)^2 e^{q^-} \Big)
 \big| \p_1 \langle \p_1\rangle^{\f12}\p_1^a f  \big|^2  \d x_1\\
&=B_{21}+B_{22}.
\end{align*}
$B_{21}+B_{22}$ has good control. Actually, $B_{21}$ will be absorbed into the energy. For $B_{22}$,
we plug in the equation of $ (\p_t+  \p_1) \big(\ud{\La_{-}^1}+\ud{\hat{\La}_{-}^1} \big)$ by \eqref{B20}.
Then by weighted Sobolev inequalities Lemma \ref{Sobo2}, weighted trace Lemma \ref{lemD2} and Lemma \ref{lemPre}, we imply that
\begin{align*}
B_{22}=& \f14\int_{\BR} (\p_t+  \p_1)  \Big(\langle w^-\rangle^{4\mu}
 \big(\ud{\La_{-}^1}+\ud{\hat{\La}_{-}^1} \big)^2 e^{q^-} \Big)
 \big| \p_1 \langle \p_1\rangle^{\f12}\p_1^a f  \big|^2  \d x_1\\
=& -\f12\int_{\BR}   \f{\langle w^-\rangle^{4\mu}
 \big(\ud{\La_{-}^1}+\ud{\hat{\La}_{-}^1} \big)^2 e^{q^-} }{\langle w^+\rangle^{2\mu}}
 \big| \p_1 \langle \p_1\rangle^{\f12}\p_1^a f  \big|^2  \d x_1\\
 &+ \f12\int_{\BR}  \langle w^-\rangle^{4\mu} \big(\ud{\La_{-}^1}+\ud{\hat{\La}_{-}^1} \big)\cdot
 (\p_t+  \p_1)\big(\ud{\La_{-}^1}+\ud{\hat{\La}_{-}^1} \big)  e^{q^-}
 \big| \p_1 \langle \p_1\rangle^{\f12}\p_1^a f  \big|^2  \d x_1\\
 \lesssim& E_s  G_2 +E_s E_3^{\f12} G_3 .
\end{align*}
For $B_3$, similarly by Lemma \ref{Sobo2} and Lemma \ref{lemD2}, we deduce that
\begin{align*}
B_3=&-\f12\int_{\BR} \langle w^-\rangle^{4\mu}
\big(\ud{\La_{+}^1} \ud{\La_{-}^1} +\ud{\hat{\La}_{+}^1} \ud{\hat{\La}_{-}^1} \big)\big(\ud{\La_{-}^1}
+\ud{\hat{\La}_{-}^1})  \langle \p_1\rangle^{\f12}\p_1^a \p_1^2 f    \cdot   \p_1  \langle \p_1\rangle^{\f12}\p_1^af e^{q^-} \d x_1\\
 =&\f14\int_{\BR} \p_1  \Big(\langle w^-\rangle^{4\mu}
 \big(\ud{\La_{+}^1} \ud{\La_{-}^1} +\ud{\hat{\La}_{+}^1} \ud{\hat{\La}_{-}^1} \big)\big(\ud{\La_{-}^1}+\ud{\hat{\La}_{-}^1}) e^{q^-} \Big)
 \big| \p_1 \langle \p_1\rangle^{\f12}\p_1^a f  \big|^2  \d x_1\\
\lesssim& E_s E_3^{\f12} G_3 .
\end{align*}

\textbf{Estimate of the highest order nonlinear terms $A_{12},\,A_{13},\,B_{11},\,B_{12}$}

For the remaining highest order nonlinear terms $A_{12},\,A_{13},\,B_{11},\,B_{12}$,
we combine them and write
\begin{align}\label{rm}
&A_{13}+B_{11}+A_{12}+B_{12} \\\nonumber
&= \f14 \f{\d}{\dt}\int_{\BR}
|\p_1 \langle \p_1\rangle^{\f12}\p_1^a f|^2
e^{q^-}\cdot\langle w^-\rangle^{4\mu} (\ud{\La^1_{+}}-\ud{\hat{\La}^1_{+}}) (\ud{\La^1_{-}}-\ud{\hat{\La}^1_{-}}) \d x_1\\\nonumber
&\quad-\f14\int_{\BR} \big| \p_1 \langle \p_1\rangle^{\f12}\p_1^a f  \big|^2 e^{q^-}\cdot (\p_t-  \p_1)
\Big(\langle w^-\rangle^{4\mu} \big(\ud{\La_{+}^1}-\ud{\hat{\La}_{+}^1} \big)\big(\ud{\La_{-}^1}-\ud{\hat{\La}_{-}^1}\big) \Big)
  \d x_1\\\nonumber
&=C_1+C_2.
\end{align}
$C_1$ will be absorbed into the energy. For $C_2$, we write
\begin{align*}
C_2=&-\f{1}{4}\int_{\BR} \big| \p_1 \langle \p_1\rangle^{\f12}\p_1^a f  \big|^2 e^{q^-}\cdot (\p_t-  \p_1)
\Big(\langle w^-\rangle^{4\mu} \big(\ud{\La_{+}^1}-\ud{\hat{\La}_{+}^1} \big)\big(\ud{\La_{-}^1}-\ud{\hat{\La}_{-}^1}) \Big)
  \d x_1  \\
=&-\f{1}{4}\int_{\BR} \big|\p_1 \langle \p_1\rangle^{\f12}\p_1^a f  \big|^2 e^{q^-}\cdot
\Big( (\p_t-  \p_1) \langle w^-\rangle^{4\mu} \big(\ud{\La_{+}^1}-\ud{\hat{\La}_{+}^1} \big)\big(\ud{\La_{-}^1}-\ud{\hat{\La}_{-}^1}) \Big)
  \d x_1  \\
&-\f{1}{4}\int_{\BR} \big|\p_1 \langle \p_1\rangle^{\f12}\p_1^a f  \big|^2 e^{q^-}\cdot
\Big(\langle w^-\rangle^{4\mu} (\p_t-  \p_1) \big(\ud{\La_{+}^1}-\ud{\hat{\La}_{+}^1} \big)\big(\ud{\La_{-}^1}-\ud{\hat{\La}_{-}^1}) \Big)
  \d x_1  \\
&-\f{1}{4}\int_{\BR} \big|\p_1 \langle \p_1\rangle^{\f12}\p_1^a f  \big|^2 e^{q^-}\cdot
\Big(\langle w^-\rangle^{4\mu} \big(\ud{\La_{+}^1}-\ud{\hat{\La}_{+}^1} \big)(\p_t-  \p_1)\big(\ud{\La_{-}^1}-\ud{\hat{\La}_{-}^1}) \Big)
  \d x_1  \\
=&C_{21}+ C_{22}+ C_{23} \, .
\end{align*}
For $C_{21}$,  by weighted Sobolev inequalities Lemma \ref{Sobo2}, weighted trace Lemma \ref{lemD2} and Lemma \ref{lemPre},
we estimate
\begin{align*}
C_{21}=&-\f{1}{4}\int_{\BR} \big|\p_1 \langle \p_1\rangle^{\f12}\p_1^a f  \big|^2 e^{q^-}\cdot
\Big( (\p_t-  \p_1) \langle w^-\rangle^{4\mu} \big(\ud{\La_{+}^1}-\ud{\hat{\La}_{+}^1} \big)\big(\ud{\La_{-}^1}-\ud{\hat{\La}_{-}^1}) \Big)
  \d x_1  \\
\lesssim &  \|\p_1 \langle \p_1\rangle^{\f12}\p_1^a f \|^2_{L^2(\BR)}
\Big\| \f{\langle w^-\rangle^{5\mu-1} \big(\ud{\La_{-}^1}-\ud{\hat{\La}_{-}^1})}{\langle w^+\rangle^{\mu}} \Big\|_{L^\infty(\BR)}
\Big\|\f{  \langle w^+\rangle^{\mu} (\ud{\La_{+}^1}-\ud{\hat{\La}_{+}^1})}{\langle w^-\rangle^{\mu}}
 \Big\|_{L^\infty(\BR)}\\
\lesssim &   E_s  G_2 \,.
\end{align*}
For $C_{22}$, similarly we
plug in the equation of $ (\p_t-\p_1) \big(\ud{\La_{+}^1}-\ud{\hat{\La}_{+}^1} \big)$ by \eqref{B20},
then by weighted Sobolev inequalities Lemma \ref{Sobo2} and weighted trace Lemma \ref{lemD2} and Lemma \ref{lemPre}, we get
\begin{align*}
C_{22}=
&-\f{1}{4}\int_{\BR} \big|\p_1 \langle \p_1\rangle^{\f12}\p_1^a f  \big|^2 e^{q^-}\cdot
\Big(\langle w^-\rangle^{4\mu} (\p_t-  \p_1) \big(\ud{\La_{+}^1}-\ud{\hat{\La}_{+}^1} \big)\big(\ud{\La_{-}^1}-\ud{\hat{\La}_{-}^1}) \Big)
  \d x_1  \\
\lesssim & \| \p_1 \langle \p_1\rangle^{\f12}\p_1^a f \|^2_{L^2(\BR)}
\Big\| \f{\langle w^-\rangle^{2\mu } \big(\ud{\La_{-}^1}-\ud{\hat{\La}_{-}^1})}{\langle w^+\rangle^{\mu}} \Big\|_{L^\infty(\BR)}
\big\| \langle w^-\rangle^{2\mu}\langle w^+\rangle^{\mu } (\p_t-\p_1)\big(\ud{\La_{+}^1}-\ud{\hat{\La}_{+}^1}\big) \big\|_{L^\infty(\BR)}\\
\lesssim &   E_s E_3^{\f12} G_3   \,.
\end{align*}
For $C_{23}$, similarly by \eqref{B20}, we
plug in the equation of $ \p_t\big(\ud{\La_{-}^1}-\ud{\hat{\La}_{-}^1} \big)$,
then by Lemma \ref{Sobo2}, Lemma \ref{lemD2} and  Lemma \ref{lemPre},
we estimate
\begin{align*}
C_{23}=&-\f{1}{4}\int_{\BR} \big|\p_1 \langle \p_1\rangle^{\f12}\p_1^a f  \big|^2 e^{q^-}\cdot
\Big(\langle w^-\rangle^{4\mu} \big(\ud{\La_{+}^1}-\ud{\hat{\La}_{+}^1} \big)(\p_t-  \p_1)\big(\ud{\La_{-}^1}-\ud{\hat{\La}_{-}^1}) \Big)
  \d x_1  \\
\lesssim &  \| \p_1 \langle \p_1\rangle^{\f12}\p_1^a f \|^2_{L^2(\BR)}
\Big\| \f{\langle w^-\rangle^{5\mu} (\p_t-\p_1)\big(\ud{\La_{-}^1}-\ud{\hat{\La}_{-}^1})}{\langle w^+\rangle^{\mu}} \Big\|_{L^\infty(\BR)}
\Big\| \f{\langle w^+\rangle^{\mu} (\ud{\La_{+}^1}-\ud{\hat{\La}_{+}^1})}{\langle w^-\rangle^{\mu}}
\Big\|_{L^\infty(\BR)}\\
\lesssim &    E_s  G_s +E_s^{\f32} G_s  \,.
\end{align*}
\begin{remark}
 To take care of the above $C_1$ and $C_2$,
we need  estimate
\begin{align*}
\sum_{+,-}\sum_{|a|\leq 1}
\| \langle w^\pm\rangle^{5\mu} \big(|\p_1^a \ud{\La_{\pm}^1}|+|\p_1^a\ud{\hat{\La}_{\pm}^1}|\big) \|_{L^\infty}.
\end{align*}
Hence we introduce the lower order weighted energy with increased power of weights:
\begin{align*}
E_3^b (t)=
 \sum_{+,-}\sum_{|\alpha|\leq 3}
\big(\| \langle w^\pm \rangle^{5\mu}\nabla^\alpha \La_{\pm} \|^2_{L^2(\Om_f)}
+\| \langle w^\pm \rangle^{5\mu}\nabla^\alpha \hat{\La}_{\pm}\|^2_{L^2(\hat{\Om}_f)}\big).
\end{align*}
\end{remark}

\textbf{Estimate of the commutator terms}

In order to estimate the commutator terms $A_4$, $A_5$, $A_6$ and $B_4$, $B_5$, $B_6$, we need the weighted fractional commutator lemma.
In the sequel, we only present the estimate for $A_4$, $A_5$, $A_6$. The estimate of
$B_4$, $B_5$, $B_6$ is the same and the details are omitted.
\begin{align*}
&A_4+A_5+A_6\\
&=-\int_{\BR} \langle w^-\rangle^{4\mu} (\p_t-\p_1)\langle \p_1\rangle^{\f12}\p_1^a f\cdot
  [\langle \p_1\rangle^{\f12}\p_1^a,\ud{\La^1_{+}}\ud{\La^1_{-}}
+\ud{\hat{\La}^1_{+}}\ud{\hat{\La}^1_{-}}] \p_1^2f e^{q^-} \d x_1 \nonumber\\
&\quad-\int_{\BR} \langle w^-\rangle^{4\mu} (\p_t-\p_1)\langle \p_1\rangle^{\f12}\p_1^a f\cdot
  [\langle \p_1\rangle^{\f12}\p_1^a,\ud{\La^1_{+}}+\ud{\hat{\La}^1_{+}}]
  \p_1(\p_t-\p_1) f e^{q^-} \d x_1 \nonumber\\
&\quad-\int_{\BR} \langle w^-\rangle^{4\mu} (\p_t-\p_1)\langle \p_1\rangle^{\f12}\p_1^a f\cdot
  [\langle \p_1\rangle^{\f12}\p_1^a,\ud{\La^1_{-}}+\ud{\hat{\La}^1_{-}}]
  \p_1(\p_t+\p_1) f e^{q^-} \d x_1 \nonumber\\
&\leq \big\| \f{\langle w^-\rangle^{2\mu} (\p_t-\p_1)\langle \p_1\rangle^{\f12}\p_1^a f}{\langle w^+\rangle^{\mu}} \big\|_{L^2(\BR)}\cdot
\Big(
\big\| \langle w^-\rangle^{2\mu}\langle w^+\rangle^{\mu}
[\langle \p_1\rangle^{\f12}\p_1^a,\ud{\La^1_{+}}\ud{\La^1_{-}}
+\ud{\hat{\La}^1_{+}}\ud{\hat{\La}^1_{-}}] \p_1^2f\big\|_{L^2}\\
&\qquad+\big\| \langle w^-\rangle^{2\mu}\langle w^+\rangle^{\mu}
[\langle \p_1\rangle^{\f12}\p_1^a,\ud{\La^1_{+}}+\ud{\hat{\La}^1_{+}}]
  \p_1(\p_t-\p_1) f \big\|_{L^2}\\
&\qquad+\big\| \langle w^-\rangle^{2\mu}\langle w^+\rangle^{\mu}
[\langle \p_1\rangle^{\f12}\p_1^a,\ud{\La^1_{-}}+\ud{\hat{\La}^1_{-}}]
  \p_1(\p_t+\p_1) f \big\|_{L^2} \Big).
\end{align*}

\textbf{Estimate of
$\| \langle w^+\rangle^{\mu} \langle w^-\rangle^{2\mu} \big[\langle\p_1\rangle^{\f12}\p_1^a,\ud{\La_{+}^1 }+\ud{\hat{\La}_{+}^1 }\ \big] \p_1(\p_t-\p_1) f\|_{L^2(\BR)}$ and \\
$\| \langle w^+\rangle^{\mu} \langle w^-\rangle^{2\mu} \big[\langle\p_1\rangle^{\f12}\p_1^a,\ud{\La_{-}^1}+\ud{\hat{\La}_{-}^1 }\ \big] \p_1(\p_t+\p_1) f\|_{L^2(\BR)}$}

For these two terms, they have similar structure.
Hence we only present the estimate for the first term.
The second term can be estimated similarly and the details are omitted.

If $a=0$, by the weighted commutator estimate Lemma \ref{lemD8},
the weighted Sobolev inequalities Lemma \ref{Sobo2} and the weighted trace Lemma \ref{lemD3}, we infer that
\begin{align*}
&\| \langle w^+\rangle^{\mu} \langle w^-\rangle^{2\mu} \big[\langle\p_1\rangle^{\f12},\ud{\La_{+}^1 }+\ud{\hat{\La}_{+}^1 }\ \big] \p_1(\p_t-\p_1) f\|_{L^2(\BR)} \\
&\lesssim \big\| \langle w^+\rangle^{2\mu} \langle \p_1\rangle^{\f12}
\big( \ud{\La_{+}^1 }+ \ud{\hat{\La}_{+}^1} \big)
\big\|_{L^2(\BR)}
\Big\|\f{\langle w^-\rangle^{2\mu}}{\langle w^+\rangle^{\mu}}
 \p_1 (\p_t-\p_1) f \Big\|_{L^\infty(\BR)}  \\\nonumber
&\lesssim \big( \| \langle w^+\rangle^{2\mu}\p^{\leq 1}  \La_{+}^1\|_{L^2(\Om_f)}
 + \| \langle w^+\rangle^{2\mu} \p^{\leq 1}  \hat{\La}_{+}^1\|_{L^2(\hat{\Om}_f)} \big)  G_3^{\f12} \\\nonumber
&\lesssim E_1^{\f12} G_3^{\f12} .
\end{align*}
If $a=1$, by the weighted commutator estimate Lemma \ref{lemD10},
the weighted Sobolev inequalities Lemma \ref{Sobo2},
the weighted chain rules Lemma \ref{lemG3} and the weighted trace Lemma \ref{lemD3}, we deduce that
\begin{align*}
&\| \langle w^+\rangle^{\mu} \langle w^-\rangle^{2\mu} \big[\langle\p_1\rangle^{\f12}\p_1,\ud{\La_{+}^1 }+\ud{\hat{\La}_{+}^1 }\ \big] \p_1(\p_t-\p_1) f\|_{L^2(\BR)} \\
&\lesssim \big\| \langle w^+\rangle^{2\mu} \langle \p_1\rangle^{\f12}\p_1
\big( \ud{\La_{+}^1 }+ \ud{\hat{\La}_{+}^1} \big)
\big\|_{L^2(\BR)}
\Big\|\f{\langle w^-\rangle^{2\mu}}{\langle w^+\rangle^{\mu}}
 \p_1 (\p_t-\p_1) f \Big\|_{L^\infty(\BR)}  \\\nonumber
&\quad+ \big\| \langle w^+\rangle^{2\mu} \langle \p_1\rangle^{\f12}
\big( \ud{\La_{+}^1 }+ \ud{\hat{\La}_{+}^1} \big)
\big\|_{L^\infty(\BR)}
\Big\|\f{\langle w^-\rangle^{2\mu}}{\langle w^+\rangle^{\mu}}
 \p_1^2 (\p_t-\p_1) f \Big\|_{L^2(\BR)}  \\\nonumber
&\lesssim \big( \| \langle w^+\rangle^{2\mu}\p^{\leq 2}  \La_{+}^1\|_{L^2(\Om_f)}
 + \| \langle w^+\rangle^{2\mu} \p^{\leq 2}  \hat{\La}_{+}^1\|_{L^2(\hat{\Om}_f)} \big)  G_3^{\f12} \\\nonumber
&\lesssim E_2^{\f12} G_3^{\f12} .
\end{align*}
If $2\leq a\leq s-1$, by weighted commutator estimate Lemma \ref{lemD11}, weighted chain rule Lemma \ref{lemG3} and the weighted trace Lemma \ref{lemD3}, there holds
\begin{align*}
&\| \langle w^+\rangle^{\mu} \langle w^-\rangle^{2\mu} \big[\langle\p_1\rangle^{\f12}\p_1^a,\ud{\La_{+}^1 }+\ud{\hat{\La}_{+}^1 }\ \big] \p_1(\p_t-\p_1) f\|_{L^2(\BR)} \\
&\lesssim \big\| \langle w^+\rangle^{2\mu} \langle \p_1\rangle^{\f12}
 \p_1^{\leq a} \big( \ud{\La_{+}^1 }+ \ud{\hat{\La}_{+}^1} \big)
\big\|_{L^2(\BR)}
\Big\|\f{\langle w^-\rangle^{2\mu}}{\langle w^+\rangle^{\mu}} \langle \p_1\rangle^{\f12}
 \p_1^{\leq a}(\p_t-\p_1) f \Big\|_{L^2(\BR)}  \\\nonumber
&\lesssim \big( \| \langle w^+\rangle^{2\mu}\p^{\leq s}  \La_{+}^1\|_{L^2(\Om_f)}
 + \| \langle w^+\rangle^{2\mu} \p^{\leq s}  \hat{\La}_{+}^1\|_{L^2(\hat{\Om}_f)} \big)
 \mathcal{P}(1+\| f'\|_{H^{s-\f32}})
   G_s^{\f12} \\\nonumber
&\lesssim E_s^{\f12} G_s^{\f12} \mathcal{P}(1+\| f'\|_{H^{s-\f32}}).
\end{align*}

\textbf{Estimate of $\| \langle w^+\rangle^{\mu} \langle w^-\rangle^{2\mu} \big[\langle\p_1\rangle^{\f12}\p_1^a, \ud{\La_{-}^1} \ud{\La_{+}^1}\ \big] \p_1^2 f\|_{L^2(\BR)}$
and \\
$\| \langle w^+\rangle^{\mu} \langle w^-\rangle^{2\mu} \big[\langle\p_1\rangle^{\f12}\p_1^a, \ud{\hat{\La}_{-}^1} \ud{\hat{\La}_{+}^1}\ \big] \p_1^2 f\|_{L^2(\BR)}$}

The estimates of these two terms are the same. Hence we only present the estimate for
the first term.
Similar to the above estimate, by the weighted Sobolev inequalities Lemma \ref{Sobo2},
the weighted commutator estimate Lemma \ref{lemD8}, Lemma \ref{lemD10}, Lemma \ref{lemD11},
the weighted chain rules Lemma \ref{lemG3} and the weighted trace Lemma \ref{lemD3},
we imply
\begin{align*}
&\| \langle w^+\rangle^{\mu} \langle w^-\rangle^{2\mu} \big[\langle\p_1\rangle^{\f12}\p_1^a, \ud{\La_{-}^1} \ud{\La_{+}^1}\ \big] \p_1^2 f\|_{L^2(\BR)} \\
&\lesssim \big\| \langle w^+\rangle^{\mu} \langle w^-\rangle^{2\mu} \langle \p_1\rangle^{\f12}
 \p_1^{\leq s-1} \big( \ud{\La_{-}^1} \ud{\La_{+}^1} \big)
\big\|_{L^2(\BR)}
\|\langle \p_1\rangle^{\f12} \p_1^{\leq s} f \|_{L^2(\BR)}\\\nonumber
&\lesssim E_s G_s^{\f12} \mathcal{P}(1+\| f'\|_{H^{s-\f32}}).
\end{align*}

\textbf{Estimate of the pressure terms $A_7$ and $B_7$}

In order to estimate the pressure terms, we need to employ the inherent structure to
gain one order derivative. By Lemma \ref{lemPre} and Lemma \ref{lemPP0}, we finally derive that
\begin{align*}
A_7=&-\int_{\BR} \langle w^-\rangle^{4\mu} (\p_t-\p_1)\langle \p_1\rangle^{\f12}\p_1^a f\cdot
\langle \p_1\rangle^{\f12}\p_1^a\Big( N_f\cdot(\ud{\nabla p} +\ud{\nabla \hat{p} })\Big) e^{q^-} \d x_1 \\\nonumber
\lesssim
&\big\| \f{\langle w^-\rangle^{2\mu} (\p_t-\p_1)\langle \p_1\rangle^{\f12}\p_1^a f}{\langle w^+\rangle^{\mu}} \big\|_{L^2(\bR)}
\big\| \langle w^-\rangle^{2\mu} \langle w^+\rangle^{\mu} \langle \p_1\rangle^{\f12}\p_1^a\big( N_f\cdot(\ud{\nabla p} +\ud{\nabla \hat{p} })\big) \big\|_{L^2(\bR)}\\\nonumber
\lesssim
& E_s^{\f12} G_s \, .
\end{align*}

Combining all the estimates in this subsection, \eqref{E10} reduces to
\begin{align*}
& \f{\d}{\dt}
\int_{\BR}\big| \langle w^-\rangle^{2\mu}(\p_t-\p_1) \langle \p_1\rangle^{\f12}\p_1^a f
\big|^2 e^{q^-}\d x_1
+\int_{\BR}\f{\big| \langle w^-\rangle^{2\mu}(\p_t-\p_1) \langle \p_1\rangle^{\f12}\p_1^a f
\big|^2}{\langle w^+\rangle^{2\mu}} e^{q^-}\d x_1\\\nonumber
& +\f{\d}{\dt}\int_{\BR} \langle w^-\rangle^{4\mu}  (\p_t-\p_1)\langle \p_1\rangle^{\f12}\p_1^a  f \big(\ud{\La_{-}^1}
+\ud{\hat{\La}_{-}^1}) \cdot \p_1  \langle \p_1 \rangle^{s-\f12}  f e^{q^-} \d x_1
\qquad (\textrm{from}\,\,A_{31}) \\\nonumber
&+\f14\f{\d}{\dt}\int_{\BR} \langle w^-\rangle^{4\mu}
 \big(\ud{\La_{-}^1}+\ud{\hat{\La}_{-}^1} \big)^2
  \big|\p_1 \langle \p_1\rangle^{\f12}\p_1^a  f\big|^2 e^{q^-}\d x_1 \qquad (\textrm{from}\,\,B_{21})\\\nonumber
&- \f14 \f{\d}{\dt}\int_{\BR}
|\p_1 \langle \p_1\rangle^{\f12}\p_1^a f|^2 \cdot
e^{q^-}\langle w^-\rangle^{4\mu} (\ud{\La^1_{+}}-\ud{\hat{\La}^1_{+}}) (\ud{\La^1_{-}}-\ud{\hat{\La}^1_{-}}) \d x_1 \qquad (\textrm{from}\,\, A_{13}+B_{11}) \\\nonumber
&\lesssim (1+E_s^{\f12}+E_s )E_s^{\f12} G_s.
\end{align*}
Summing over all $|a|\leq s-1$ and then taking the integral over $[0,t]$,
 note that the same estimate also holds for the right Alfv\'en waves
 and $e^{q^\pm}\sim 1$, by using the a priori estimates that $E_s \ll 1$, we obtain
\begin{align*}
\sup_{0\leq \tau\leq t}E_{s+{\f12}}^f(\tau) + \int_0^t G_{s+{\f12}}^f(\tau)\d\tau
\lesssim E_{s+{\f12}}^f(0)+\int_0^t (1+E_s^{\f12}+E_s )E_s^{\f12} G_s (\tau)\d\tau.
\end{align*}
This yields \eqref{AA8}.

\section{Weighted lower order tangential energy estimate}\label{Et}

This section is devoted to the lower-order $(3-\textrm{order})$ tangential weighted energy estimate of Alfv\'en waves inside the domain
$\Om_f$ and $\hat{\Om}_f$.
The estimates are conducted by applying increased power of weights: $\langle w^\pm\rangle^{5\mu}$.
Here we recall the weight functions $ w^\pm=x_1\pm t$ where $\f12 < \mu\leq \f35$.
At the same time, we also need to use the ghost weight method.
Due to the non-vanishing boundary conditions, some of the estimates will reduce to the free boundary.
Especially, for the last line of \eqref{F70} involving the pressure term:
\begin{align*}
&-\int_{\Ga_f} e^{q^+}\langle w^+\rangle^{10\mu} \p_\tau^m \hat{p} n_f\cdot \p_\tau^m \hat{\Lambda}_{+}\d\sigma(x)
+\int_{\Ga_f} e^{q^+}\langle w^+\rangle^{10\mu} \p_\tau^m p n_f\cdot \p_\tau^m \La_{+}\d\sigma(x),
\end{align*}
where $m\leq 3$.
By weighted trace theorem, we need the information for
$ \| \langle w^\pm \rangle^{5\mu}\p_\tau^{\leq 3.5} \La_{\pm}\|^2_{L^2(\Om_f)}$ and
$\| \langle w^\pm \rangle^{5\mu}\p_\tau^{\leq 3.5} \hat{\La}_{\pm}\|^2_{L^2(\hat{\Om}_f)}
$. This is not true. To save one derivative on $\La_{\pm}$ and $\hat{\La}_{\pm}$,  we use the boundary condition and apply integration by parts to move one $\p_\tau$ away.
However, $\p_1^{a}f$ $(1\leq a\leq 4)$ are present. Fortunately, no weight is needed to apply on $\p_1^{a}f$. Thus its $L^2$ norm can be bounded by $E_4^f$. Note that for the weighted tangential derivative estimates, we always have this serious derivative loss problem. Hence
the weighted tangential estimates can not be conducted in the highest order derivatives.
\begin{remark}
In \cite{CMST},	Coulombel, Morando,  Secchi and Trebeschi  proved an {\it a priori estimate} without loss of derivatives for the nonlinear current-vortex sheet problem under a strong stability condition.
They introduced extra $\lambda^\pm$ in the tangential derivatives to cancellate the above highest order terms of the pressure on the free boundary. Their method does not apply in our case since they estimate the velocity field and the magnetic field separately, and we need to work on $\La_\pm$ variables.
\end{remark}


Recalling that the tangential derivative $\p_\tau$ is defined by \eqref{tangD-}:
\begin{equation*}
\begin{cases}
\p_\tau=\Psi\p_2+\p_1 \quad \textrm{in}\,\, \Om_f,\\
\p_\tau=\hat{\Psi}\p_2+\p_1 \quad \textrm{in}\,\, \hat{\Om}_f.
\end{cases}
\end{equation*}
where
\begin{equation*}
\begin{cases}
\Psi(t,x)= \zeta(x_2)\p_1f \quad \textrm{in}\,\, \Om_f,\\
\hat{\Psi}(t,x)= \zeta(x_2)\p_1f \quad \textrm{in}\,\, \hat{\Om}_f.
\end{cases}
\end{equation*}
Here $\zeta(x_2)\in C^\infty(\BR)$ and
\begin{equation*}
\zeta(x_2)=
\begin{cases}
1 , &\text{ if } |x_2|\leq 1-\f{c_0}{2} ,\\
0 , &\text{ if } |x_2|\geq 1-\f{c_0}{2},
\end{cases}
\end{equation*}
where  $c_0$ is the constant in Theorem \ref{thm}.


For function $g$ defined in $\Om_f$ or $\hat{\Om}_f$, recalling that $\ud{g}$ denotes
the trace of $g$ on $\Ga_f.$ There holds
\begin{align*}
&\p_1 \ud{g}
= \ud{\p_1 g}+\p_1f \ud{\p_2 g}
= \ud{\p_\tau g}.
\end{align*}

We define the lower order weighted tangential energy inside the bulk of the domain as follows:
\begin{align} \label{F01}
E_3^{b,\tau} (t)=
& \sum_{+,-}\sum_{0\leq m\leq 3} \big(
\| \langle w^\pm \rangle^{5\mu}\p_\tau^m \La_{\pm}(t,\cdot)\|^2_{L^2(\Om_f)}
+\| \langle w^\pm \rangle^{5\mu}\p_\tau^m \hat{\La}_{\pm}(t,\cdot)\|^2_{L^2(\hat{\Om}_f)}
\big).
\end{align}
Note that the weight is $\langle w^\pm \rangle^{5\mu}$. The index is different from the higher-order energy.
The corresponding weighted tangential ghost weight energy is defined as follows:
\begin{align}\label{F02}
G_3^{b,\tau} (t)
=&\sum_{+,-}\sum_{0\leq m\leq 3}\Big(
\big\|\f{ \langle w^\pm \rangle^{5\mu}\p_\tau^m  \La_{\pm}  }{ \langle w^\mp \rangle^{\mu} }(t,\cdot)\big\|^2_{L^2(\Om_f)}
+\big\|\f{ \langle w^\pm \rangle^{5\mu}\p_\tau^m  \hat{\La}_{\pm}  }{ \langle w^\mp \rangle^{\mu} }(t,\cdot)\big\|^2_{L^2(\hat{\Om}_f)}
\Big)\,.
\end{align}
Due to the similarity between the left Alfv\'en waves and the right Alfv\'en waves,
we only present the estimate for $\La_{+}$ and $\hat{\La}_+$.

We recall the evolution equations for $\La_+$ and $\hat{\La}_+$ from $\eqref{A18}$, $\eqref{A19}$
as follows
\begin{equation} \label{F2}
\begin{cases}
(\p_t +Z_{-}\cdot \nabla) \La_{+} + \nabla  p =0,\quad\textrm{in} \,\, \Om_f,\\
(\p_t +\hat{Z}_{-}\cdot \nabla) \hat{\La}_{+} + \nabla  \hat{p} =0,\quad\textrm{in} \,\, \hat{\Om}_f,\\
\nabla\cdot\La_{\pm}=\nabla\cdot\hat{\La}_{\pm}=0.
\end{cases}
\end{equation}
For $0\leq m\leq 3$, applying derivative $\p_\tau^m$ to \eqref{F2}, we organize that
\begin{equation} \label{F3}
\begin{cases}
(\p_t +Z_{-}\cdot\nabla) \p_\tau^m \La_{+} 
+[\p_\tau^m,\p_t +Z_{-}\cdot\nabla ] \La_{+} + \nabla \p_\tau^m p
 +[\p_\tau^m,\nabla] p=0, \,\textrm{in} \  \Om_f,
\\[-4mm]\\
(\p_t +\hat{Z}_{-}\cdot\nabla) \p_\tau^m \hat{\La}_{+}
+[\p_\tau^m,\p_t +\hat{Z}_{-}\cdot\nabla] \hat{\La}_{+} + \nabla \p_\tau^m \hat{p}
 +[\p_\tau^m,\nabla] \hat{p}=0,\, \textrm{in} \  \hat{\Om}_f.
\end{cases}
\end{equation}
where we have used the commutator notation
$$
[\nabla^\alpha,f\cdot\nabla] g
=\nabla^\alpha(f\cdot\nabla g)
-f\cdot\nabla \nabla^\alpha g.
$$
Multiplying $\eqref{F3}$ by $\langle w^+ \rangle^{5\mu} $,
we writes
\begin{equation} \label{F4}
\begin{cases}
(\p_t+Z_{-}\cdot\nabla )(\langle w^+ \rangle^{5\mu}\p_\tau^m \La_{+})
-\big(\La_{-}\cdot\nabla \langle w^+\rangle^{5\mu}\big)\p_\tau^m \La_{+}
\\[-4mm]\\
\quad+\langle w^+\rangle^{5\mu}
\big([\p_\tau^m,\p_t+Z_{-}\cdot\nabla] \La_{+} + \nabla \p_\tau^m p
+[\p_\tau^m, \nabla]  p \big) =0, \quad \textrm{in} \ \Om_f,
\\[-4mm]\\
(\p_t+\hat{Z}_{-}\cdot\nabla )(\langle w^+ \rangle^{5\mu}\p_\tau^m \hat{\La}_{+})
- \big(\hat{\La}_{-}\cdot\nabla\langle w^+\rangle^{5\mu}\big) \p_\tau^m \hat{\La}_{+}
\\[-4mm]\\
\quad+\langle w^+\rangle^{5\mu}\big([\p_\tau^m,\p_t+\hat{Z}_{-}\cdot\nabla] \hat{\La}_{+}
+ \nabla \p_\tau^m \hat{p} +[\p_\tau^m, \nabla]\hat{p} \big) =0,
\quad \textrm{in} \ \hat{\Om}_f.
\end{cases}
\end{equation}
Let $1/2<\mu \leq 3/5$, $q(\theta)=\int_0^\theta \langle\tau\rangle^{-2\mu} d\tau$,
thus $|q(\theta)|\lesssim 1$.
Denote $e^{q^+}=e^{q(w^-)}$ and $e^{q^-}=e^{q(-w^+)}$ for simplicity.
Due to the boundary condition, we derive that
\begin{align} \label{F5}
&\f12\f{\d}{\dt} \Big( \int_{\Om_f}
| \langle w^+ \rangle^{5\mu} \p_\tau^m \La_{+}|^2 e^{q^+} \dx
+\int_{\hat{\Om}_f}
| \langle w^+ \rangle^{5\mu} \p_\tau^m \hat{\La}_{+}|^2 e^{q^+} \dx \Big)\\\nonumber
&=\f12 \int_{\Om_f} (\p_t+Z_{-} \cdot\nabla)
\big(| \langle w^+\rangle^{5\mu} \p_\tau^m \La_{+}|^2 e^{q^+} \big)\dx \\\nonumber
&\quad +
\f12 \int_{\hat{\Om}_f} (\p_t+\hat{Z}_{-} \cdot\nabla)
\big(| \langle w^+\rangle^{5\mu} \p_\tau^m \hat{\La}_{+}|^2 e^{q^+} \big)\dx.
\end{align}
Plugging \eqref{F4} into \eqref{F5}, we deduce that
\begin{align}\label{F6}
&\f12\f{\d}{\dt} \Big( \int_{\Om_f}
| \langle w^+ \rangle^{5\mu} \p_\tau^m \La_{+}|^2 e^{q^+} \dx
+\int_{\hat{\Om}_f}
| \langle w^+ \rangle^{5\mu} \p_\tau^m \hat{\La}_{+}|^2 e^{q^+} \dx \Big)\\\nonumber
&-\f12\int_{\Om_f}
|\langle w^+ \rangle^{5\mu} \p_\tau^m \La_{+} |^2
(\p_t - \p_1)e^{q^+}  \dx
-\f12\int_{\hat{\Om}_f}
|\langle w^+ \rangle^{5\mu} \p_\tau^m \hat{\La}_{+} |^2
(\p_t - \p_1)e^{q^+}  \dx \\\nonumber
&=\f12\int_{\Om_f}
|\p_\tau^m \La_{+} |^2 \La_{-}\cdot\nabla(\langle w^+ \rangle^{10\mu} e^{q^+})  \dx
+\f12\int_{\hat{\Om}_f}
|\p_\tau^m \hat{\La}_{+} |^2 \hat{\La}_{-}\cdot\nabla(\langle w^+ \rangle^{10\mu} e^{q^+})  \dx \\\nonumber
&\quad-\int_{\Om_f}
\langle w^+ \rangle^{10\mu}\p_\tau^m \La_{+} e^{q^+} \cdot
 [\p_\tau^m,\p_t+Z_{-}\cdot\nabla] \La_{+}\dx \\\nonumber
&\quad-\int_{\hat{\Om}_f}
\langle w^+ \rangle^{10\mu}\p_\tau^m \hat{\La}_{+} e^{q^+} \cdot
 [\p_\tau^m,\p_t+\hat{Z}_{-}\cdot\nabla] \hat{\La}_{+}\dx \\\nonumber
&\quad-\int_{\Om_f}
\langle w^+ \rangle^{10\mu}\p_\tau^m \La_{+} e^{q^+} \cdot
 [\p_\tau^m,\nabla]  p\dx
 -\int_{\hat{\Om}_f}
\langle w^+ \rangle^{10\mu}\p_\tau^m \hat{\La}_{+} e^{q^+} \cdot
 [\p_\tau^m,\nabla] \hat{p}\dx \\\nonumber
&\quad-\int_{\Om_f}
\langle w^+ \rangle^{10\mu}\p_\tau^m \La_{+} e^{q^+} \cdot
 \nabla \p_\tau^m p\dx
 -\int_{\hat{\Om}_f}
\langle w^+ \rangle^{10\mu}\p_\tau^m \hat{\La}_{+} e^{q^+} \cdot
 \nabla \p_\tau^m \hat{p}\dx .
\end{align}
Now we estimate \eqref{F6}.

The second line of \eqref{F6} corresponds to the ghost weight energy of Alinhac and it is equal to
\begin{align*}
& \int_{\Om_f}
 \f{|\langle w^+ \rangle^{5\mu}\p_\tau^m \La_{+} |^2}{ \langle w^- \rangle^{2\mu} }
 e^{q^+}\dx
+\int_{\hat{\Om}_f} \f{|\langle w^+ \rangle^{5\mu}\p_\tau^m \hat{\La}_{+} |^2}{ \langle w^- \rangle^{2\mu} }
 e^{q^+}  \dx.
\end{align*}
For the third line of \eqref{F6}, by weighted Sobolev inequalities Lemma \ref{Sobo1}, we infer that
\begin{align*}
&\f12\big| \int_{\Om_f}
|\p_\tau^m \La_{+} |^2
\La_{-}\cdot\nabla (\langle w^+ \rangle^{10\mu} e^{q^+} ) \dx \big| \\\nonumber
&=\big| \int_{\Om_f}
|\p_\tau^m \La_{+} |^2
\La_{-}^1\cdot \Big(5\mu\langle w^+ \rangle^{10\mu-2}w^+ \p_1w^+
+\f12\f{\langle w^+ \rangle^{10\mu} }{\langle w^- \rangle^{2\mu}}\p_1 w^- \Big) e^{q^+}  \dx \big| \\\nonumber
&\lesssim E_3^{\f12} G_2^b.
\end{align*}
The second term in the third line of \eqref{F6} can be estimated similarly and has the same the bound.

\textbf{Estimate of the fourth and fifth line of \eqref{F6}}

The estimate of the fourth and the fifth line are the same. We only present the details for the fourth line of \eqref{F6}.
For this fourth line, it equals
\begin{align}\label{F10}
&-\int_{\Om_f}
\langle w^+ \rangle^{10\mu}\p_\tau^m \La_{+} e^{q^+} \cdot
 [\p_\tau^m,\La_{-}\cdot\nabla] \La_{+} \dx\\\nonumber
&-\int_{\Om_f}
\langle w^+ \rangle^{10\mu}\p_\tau^m \La_{+} e^{q^+} \cdot
 [\p_\tau^m,\p_t-\p_1] \La_{+} \dx .
\end{align}
For the first line of \eqref{F10},
by weighted Sobolev inequalities Lemma \ref{Sobo1} and tangential derivative estimate Lemma \ref{lemTang}, we defer that
\begin{align*}
&\big| \int_{\Om_f}
\langle w^+ \rangle^{10\mu}\p_\tau^m \La_{+} e^{q^+} \cdot
 [\p_\tau^m,\La_{-}\cdot\nabla] \La_{+} \dx \big| \nonumber\\
&=\big| \sum_{a+b=m,\, b\neq m}C_m^a\int_{\Om_f}
\langle w^+ \rangle^{10\mu}\p_\tau^m \La_{+} e^{q^+} \cdot
 (\p_\tau^a \La_{-}\cdot\p_\tau^b\nabla \La_{+}) \dx \big| \nonumber\\
&\lesssim
\Big\| \f{\langle w^+ \rangle^{5\mu}\p_\tau^m \La_{+}}{\langle w^-\rangle^{\mu}}\Big\|_{L^2(\Om_f)}
\Big(\|\langle w^- \rangle^{2\mu}\p_\tau^{\leq 1} \La_{-}\|_{L^\infty(\Om_f)}
\Big\| \f{\langle w^+ \rangle^{5\mu}\p_\tau^{\leq 2} \nabla  \La_{+}}{\langle w^- \rangle^{\mu}}\Big\|_{L^2(\Om_f)}
\nonumber\\
&\qquad+
\|\langle w^- \rangle^{2\mu}\p_\tau^{\leq 2} \La_{-}\|_{L^4(\Om_f)}
\Big\| \f{\langle w^+ \rangle^{5\mu} \p_\tau^{\leq 1} \nabla  \La_{+}}{\langle w^- \rangle^{\mu}}\Big\|_{L^4(\Om_f)} \nonumber\\
&\qquad+\|\langle w^- \rangle^{2\mu}\p_\tau^{\leq 3} \La_{-}\|_{L^2(\Om_f)}
\Big\| \f{\langle w^+ \rangle^{5\mu}  \nabla  \La_{+}}{\langle w^- \rangle^{\mu}}\Big\|_{L^\infty(\Om_f)}
\Big)\nonumber\\
&\lesssim E_3^{\f12} G_3^b \mathcal{P}(1+ \| \p_1f\|_{H^2(\bR)} ).
\end{align*}
%

\textbf{Estimate of the commutator of the tangential derivatives and $\p_t-\p_1$: the second line of \eqref{F10}.}
%
For the second line of \eqref{F10}, we write
\begin{align*}
&\big| \int_{\Om_f}
\langle w^+ \rangle^{10\mu}\p_\tau^m \La_{+} e^{q^+} \cdot
 [\p_\tau^m,\p_t-\p_1] \La_{+} \dx \big| \nonumber\\
&\lesssim \big\| \f{\langle w^+ \rangle^{5\mu}\p_\tau^m \La_{+}}{\langle w^- \rangle^{\mu}}\big\|_{L^2(\Om_f)}
\big\|  \langle w^+ \rangle^{5\mu}\langle w^- \rangle^{\mu}  [\p_\tau^m,\p_t-\p_1] \La_{+} \big\|_{L^2(\Om_f)}.
\end{align*}
To estimate this term, we need calculate the commutator of the tangential derivative
$\p_\tau$ and $\p_t-\p_1$.

Let us recall that the tangential derivative is defined as follows:
\begin{align*}
\p_\tau=\Psi\p_2+\p_1 ,
\end{align*}
Thus we calculate
\begin{align*}
[\p_\tau,\p_t-\p_1] \La_{+}
&=-(\p_t-\p_1) \Psi\cdot \p_2 \La_{+} ,\\
[\p_\tau^2,\p_t-\p_1] \La_{+}
&=\p_\tau [\p_\tau,\p_t-\p_1] \La_{+}+[\p_\tau,\p_t-\p_1] \p_\tau\La_{+}  \\
&=- (\p_t-\p_1) \Psi\cdot
\big( 2\p_1\p_2 \La_{+} +2\Psi\p_2^2\La_{+}+\p_2\Psi\p_2\La_{+}  \big)
-( \Psi\p_2+\p_1)(\p_t-\p_1)\Psi \p_2\La_{+},\\
[\p_\tau^3,\p_t-\p_1] \La_{+}&=
(\p_1 + \Psi\p_2)(\p_1^2 +\p_1\Psi\p_2+2\Psi\p_{12}+\Psi\p_2\Psi\p_2
+\Psi^2\p_2^2\big)(\p_t-\p_1) \La_{+} \\
&\quad- (\p_t-\p_1)(\p_1 + \Psi\p_2)(\p_1^2 +\p_1\Psi\p_2+2\Psi\p_{12}+\Psi\p_2\Psi\p_2
+\Psi^2\p_2^2\big) \La_{+}.
\end{align*}
Thus
\begin{align*}
|[\p_\tau,\p_t-\p_1] \La_{+}|
& \leq |(\p_t-\p_1) \Psi\p\La_{+}|,\\
|[\p_\tau^2,\p_t-\p_1] \La_{+}|
&\lesssim
(1+|\Psi|)
\Big(|(\p_t-\p_1)\Psi|( |\p^2 \La_{+}|+ |\p\Psi \p\La_{+}| )
+ |(\p_t-\p_1)\p\Psi \p \La_{+}|  \Big),\\
\big| [\p_\tau^3,\p_t-\p_1] \La_{+} \big|
&\lesssim \big( 1+|\Psi| \big)^2 \Big( |(\p_t-\p_1)\p^2\Psi \p\La_{+}|+
|(\p_t-\p_1)\p\Psi| (|\p^2\La_{+}|+|\p\La_{+}\p\Psi|) \\
&\quad+|(\p_t-\p_1)\Psi| (|\p^3\La_{+}|+|\p^2\La_{+}\p\Psi|
+|\p\La_{+}\p^2\Psi|)
\Big)  \,.
\end{align*}
Consequently, for $m=1$, by weighted Sobolev inequalities Lemma \ref{Sobo1}, Lemma \ref{Sobo2}
and Lemma \ref{lemPsi-s},
for $\f 12<\mu<\f 35$, we obtain
\begin{align*}
&\big\|  \langle w^+ \rangle^{5\mu}\langle w^- \rangle^{\mu}
 [\p_\tau,\p_t-\p_1] \La_{+} \big\|_{L^2(\Om_f)}\\
&\lesssim \big\| \f{\langle w^+ \rangle^{5\mu}  \p_2\La_{+}}{\langle w^- \rangle^{\mu}}\big\|_{L^4(\Om_f)}
\big\|   \langle w^- \rangle^{2\mu} (\p_t-\p_1) \Psi \big\|_{L^4(\Om_f)}
\leq  (G_2^b E_3)^{\f12} \,.
\end{align*}
For $m=2$, similarly by Lemma \ref{Sobo1}, Lemma \ref{Sobo2} and Lemma \ref{lemPsi-s},
we derive that
\begin{align*}
&\big\|  \langle w^+ \rangle^{5\mu}\langle w^- \rangle^{\mu}
 [\p_\tau^2,\p_t-\p_1] \La_{+} \big\|_{L^2(\Om_f)}\\
&\lesssim
\big\| \langle w^+ \rangle^{5\mu}\langle w^- \rangle^{\mu}
\big( |(\p_t-\p_1)\p\Psi \p\La_{+}|+
|(\p_t-\p_1)\Psi| (|\p^2\La_{+}|+|\p\La_{+}\p\Psi|)\big)\big\|_{L^2(\Om_f)} \\
&\quad\cdot \big( 1+\big\|\Psi \big\|_{L^\infty(\Om_f)}  \big)\\
&\lesssim \big\| \f{\langle w^+ \rangle^{5\mu}  \p^{\leq 3}\La_{+}}{\langle w^- \rangle^{\mu}}\big\|_{L^2(\Om_f)}
 \big\|   \langle w^- \rangle^{2\mu} (\p_t-\p_1)\p^{\leq 1} \Psi\|_{L^2(\Om_f)}
(1+\| \p^{\leq 2} \Psi\|_{L^2(\Om_f)} )^2\\
&\lesssim (E_3 G_3^b)^{\f12} (1+\| \p_1f\|_{H^2(\bR)} )^2 \,.
\end{align*}
For $m=3$, by Lemma \ref{Sobo1}, Lemma \ref{Sobo2} and Lemma \ref{lemPsi-s}, we have
\begin{align*}
&\big\|  \langle w^+ \rangle^{5\mu}\langle w^- \rangle^{\mu}
 [\p_\tau^3,\p_t-\p_1] \La_{+} \big\|_{L^2(\Om_f)}\\
&\lesssim
\big\| \langle w^+ \rangle^{5\mu}\langle w^- \rangle^{\mu}
\big( |(\p_t-\p_1)\p^2\Psi \p\La_{+}|+
|(\p_t-\p_1)\p\Psi| (|\p^2\La_{+}|+|\p\La_{+}\p\Psi|) \\
&\quad+|(\p_t-\p_1)\Psi| (|\p^3\La_{+}|+|\p^2\La_{+}\p\Psi|
+|\p\La_{+}\p^2\Psi|
\big)\big\|_{L^2(\Om_f)}
\cdot \big( 1+\big\|\Psi \big\|_{L^\infty(\Om_f)}  \big)^2\\
&\lesssim
\big\| \f{\langle w^+ \rangle^{5\mu} \p\La_{+} }{\langle w^- \rangle^{\mu}}\big\|_{L^\infty(\Om_f)}
\big\|   \langle w^- \rangle^{2\mu} (\p_t-\p_1) \p^2 \Psi \big\|_{L^2 (\Om_f)}\\
&\quad+\big\| \f{\langle w^+ \rangle^{5\mu} (|\p^{\leq 2}\La_{+}|
+|\p \La_{+}\p\Psi| )}{\langle w^- \rangle^{\mu}}\big\|_{L^4(\Om_f)}
\big\|   \langle w^- \rangle^{2\mu} (\p_t-\p_1) \p \Psi \big\|_{L^4 (\Om_f)}\\
&\quad+\big\| \f{\langle w^+ \rangle^{5\mu}  (|\p^{\leq 3}\La_{+}|+|\p^{\leq 2}\La_{+}\p\Psi|+|\p\La_{+}\p^2\Psi|)}{\langle w^- \rangle^{\mu}}\big\|_{L^2(\Om_f)}
\big\|   \langle w^- \rangle^{2\mu} (\p_t-\p_1) \Psi \big\|_{L^\infty(\Om_f)}\\
&\lesssim \sum_{|\alpha|\leq 3}\big\| \f{\langle w^+ \rangle^{5\mu}  \p^\al\La_{+}}{\langle w^- \rangle^{\mu}}\big\|_{L^2(\Om_f)}
\sum_{0\leq a\leq 2}\big\|   \langle w^- \rangle^{2\mu} (\p_t-\p_1)\p_1^{1+a} f\|_{L^2}
\big(1+\| \p_1f\|_{H^2} \big)^3 \\
&\leq E_4^{\f12} G_3^{\f12} \big(1+\| \p_1f\|_{H^2} \big)^3  \,.
\end{align*}

\textbf{Estimate of the commutator term involving tangential derivatives: the sixth line of \eqref{F6}}
\begin{align*}
&\quad-\int_{\Om_f}
\langle w^+ \rangle^{10\mu}\p_\tau^m \La_{+} e^{q^+} \cdot
 [\p_\tau^m,\nabla]  p\dx
 -\int_{\hat{\Om}_f}
\langle w^+ \rangle^{10\mu}\p_\tau^m \hat{\La}_{+} e^{q^+} \cdot
 [\p_\tau^m,\nabla] \hat{p}\dx .
\end{align*}
Let us first calculate the commutators in $\Om_f$:
\begin{align*}
[\nabla,\p_\tau]p
&=[\nabla,\Psi\p_2 ]p  =\nabla \Psi \p_2p ,\\
[\nabla,\p_\tau^2]p
&=[\nabla,\p_\tau]\p_\tau p+\p_\tau[\nabla,\p_\tau] p
=\nabla \Psi \p_2\p_\tau p
+(\Psi\p_2+\p_1)(\nabla\Psi \p_2p ),\\
[\nabla,\p_\tau^3]p
&=[\nabla,\p_\tau^2]\p_\tau p+\p_\tau^2[\nabla,\p_\tau] p \\
&=\nabla \Psi \p_2\p_\tau^2 p
+  \Psi \nabla\p_2\Psi \p_2\p_\tau p+\nabla\p_1\Psi \p_2\p_\tau p
+  \Psi \nabla\Psi \p_2^2\p_\tau p+\nabla\Psi \p_2\p_1\p_\tau p\\
&\quad+\p_\tau^2\big(\nabla\Psi \p_2p\big).
\end{align*}
Thus
\begin{align*}
|[\nabla,\p_\tau]p|&\leq | \nabla\Psi \p_2p| ,\\
|[\nabla,\p_\tau^2]p|
&\lesssim (1+|\Psi|)\Big( | \nabla^2\Psi \nabla p|+ |\nabla\Psi \nabla^2 p|+| \nabla\Psi|^2| \nabla p|\Big),\\
|[\nabla,\p_\tau^3]p|
&\lesssim(1+|\Psi|)^2\Big( |\nabla^3 p \nabla\Psi|+|\nabla^2 p|( |\nabla^2\Psi|+|\nabla\Psi|^2) \\
&\qquad\qquad +|\nabla p|(|\nabla^3\Psi|+|\nabla^2\Psi \nabla\Psi|+|\nabla\Psi|^3) \Big)\, .
\end{align*}
Consequently, by Lemma \ref{lemTang} and  Lemma \ref{lemPre}, there holds
\begin{align*}
&\big| \int_{\Om_f}
\langle w^+ \rangle^{10\mu}\p_\tau^m \La_{+} e^{q^+} \cdot
 [\p_\tau^m,\nabla]  p\dx \big| \\
&\lesssim  \big\| \f{\langle w^+ \rangle^{5\mu}\p_\tau^m \La_{+}}{\langle w^- \rangle^{\mu}} \big\|_{L^2(\Om_f)}
\big\|  \langle w^+ \rangle^{5\mu}\langle w^- \rangle^{\mu}
[\p_\tau^m,\nabla]  p \big\|_{L^2(\Om_f)}\\
&\quad +\big\| \f{\langle w^+ \rangle^{5\mu}\p_\tau^m \hat{\La}_{+}}{\langle w^- \rangle^{\mu}} \big\|_{L^2(\Om_f)}
\big\|  \langle w^+ \rangle^{5\mu}\langle w^- \rangle^{\mu}
[\p_\tau^m,\nabla]  \hat{p} \big\|_{L^2(\hat{\Om}_f)}\\
&\lesssim (G_3^b)^{\f12} \sum_{1\leq a\leq m}
\big( \big\|  \langle w^+ \rangle^{5\mu}\langle w^- \rangle^{\mu} \nabla^a p \big\|_{L^2(\Om_f)}
+\big\|  \langle w^+ \rangle^{5\mu}\langle w^- \rangle^{\mu} \nabla^a \hat{p} \big\|_{L^2(\hat{\Om}_f)}
\big)\\
&\quad \cdot \mathcal{P}( \| \p_1f \|_{H^3}+1)\\
&\lesssim G_3^b E_3^\f12 \cdot \mathcal{P}( \| \p_1f \|_{H^3}+1) \,.
\end{align*}
On the other hand, by Lemma \ref{lemTang} and  Lemma \ref{lemPre}, there similarly holds
\begin{align*}
\big|\int_{\hat{\Om}_f}
\langle w^+ \rangle^{10\mu}\p_\tau^m \hat{\La}_{+} e^{q^+} \cdot
 [\p_\tau^m,\nabla] \hat{p}\dx \big|
\lesssim G_3^b E_3^\f12 \cdot \mathcal{P}( \| \p_1f \|_{H^3}+1) \,.
\end{align*}

\textbf{Estimate of the pressure terms: the seventh line of \eqref{F6}}

By integration by parts, one obtains
\begin{align}\label{F70}
&\int_{\hat{\Om}_f} e^{q^+}\langle w^+\rangle^{10\mu} \nabla \p_\tau^m \hat{p}\cdot \p_\tau^m \hat{\Lambda}_{+} \dx
+\int_{\Om_f} e^{q^+}\langle w^+\rangle^{10\mu} \nabla \p_\tau^m p\cdot \p_\tau^m \Lambda_{+} \dx\\\nonumber
&=-\int_{\hat{\Om}_f}  e^{q^+}\langle w^+\rangle^{10\mu} \p_\tau^m \hat{p}\cdot \div \p_\tau^m \hat{\Lambda}_{+} \dx
-\int_{\Om_f}    e^{q^+}\langle w^+\rangle^{10\mu} \p_\tau^m p\cdot \div \p_\tau^m \Lambda_{+} \dx\\\nonumber
&\quad-\int_{\hat{\Om}_f} \nabla( e^{q^+}\langle w^+\rangle^{10\mu}) \p_\tau^m \hat{p}\cdot   \p_\tau^m \hat{\Lambda}_{+} \dx
  -\int_{\Om_f} \nabla( e^{q^+}\langle w^+\rangle^{10\mu}) \p_\tau^m p\cdot   \p_\tau^m \Lambda_{+} \dx\\\nonumber
&\quad +\int_{\hat{\Ga}} e^{q^+}\langle w^+\rangle^{10\mu} \p_\tau^m \hat{p}\, n \cdot \p_\tau^m \hat{\Lambda}_{+}\dx_1
-\int_{\Ga} e^{q^+}\langle w^+\rangle^{10\mu} \p_\tau^m p\, n \cdot \p_\tau^m \Lambda_{+}\dx_1\\\nonumber
&\quad -\int_{\Ga_f} e^{q^+}\langle w^+\rangle^{10\mu} \p_\tau^m \hat{p} n_f\cdot \p_\tau^m \hat{\Lambda}_{+}\d\sigma(x)
+\int_{\Ga_f} e^{q^+}\langle w^+\rangle^{10\mu} \p_\tau^m p n_f\cdot \p_\tau^m \La_{+}\d\sigma(x).
\end{align}

\textbf{Estimate of the second line of \eqref{F70}.} We need to calculate the commutator $[\div, \p_\tau^k] \La_{+}$.
We conduct the calculation in $\Om_f$ since the estimate in $\hat{\Om}_f$ is the same.
\begin{align*}
[\div, \p_\tau ] \La_{+}&= \div  \p_\tau   \La_{+}=\nabla \Psi\cdot \p_2\La_{+}, \\
[\div, \p_\tau^2 ] \La_{+}&= [\div, \p_\tau ]\p_\tau \La_{+}+\p_\tau [\div, \p_\tau ] \La_{+}
=\nabla \Psi\cdot \p_2\p_\tau\La_{+} +\p_\tau \big( \nabla \Psi\cdot \p_2\La_{+} \big),\\
[\div, \p_\tau^3 ] \La_{+}&=[\div, \p_\tau^2 ]\p_\tau \La_{+}
+\p_\tau^2 [\div, \p_\tau ] \La_{+}\\
&= \nabla\Psi \cdot\p_2 \p_\tau^2 \La_{+}
+\p_\tau \big( \nabla\Psi \p_2 \p_\tau\La_{+} \big)
 +\p_\tau^2\big(  \nabla\Psi\cdot \p_2 \La_{+}\big).
\end{align*}
Thus
\begin{align*}
|[\div, \p_\tau ] \La_{+}|&\leq  |\p \Psi \p\La_{+}|,\\
\big| [\div, \p_\tau^2 ] \La_{+}\big| &\lesssim
(1+|\Psi|) \Big( | \p^2 \La_{+}\p\Psi| +|\p \La_{+}| (|\p^2\Psi|+ |\p\Psi|^2) \Big) ,\\
\big| [\div, \p_\tau^3 ] \La_{+}\big|&\lesssim
(1+|\Psi|)^2\Big( |\p\Psi \p^3 \La_{+}| +
 \big(|\p^2\Psi|+|\p\Psi|^2 \big)|\p^2 \La_{+}| \\
&\qquad+ (|\p^3\Psi|+|\p^2\Psi\p\Psi|+|\p \Psi|^3 )|\p \La_{+}| \Big) \, .
\end{align*}
Thus by  Lemma \ref{Sobo1}, Lemma \ref{lemPsi-s} and Lemma \ref{lemPre}, one infers that
\begin{align*}
&\big| \int_{\Om_f}    e^{q^+}\langle w^+\rangle^{10\mu} \p_\tau^m p\cdot \div \p_\tau^m \Lambda_{+} \dx\big|\\
&=\big| \int_{\Om_f}    e^{q^+}\langle w^+\rangle^{10\mu} \p_\tau^m p\cdot [\div, \p_\tau^m] \Lambda_{+} \dx\big| \\
&\lesssim \big\|  \langle w^+\rangle^{5\mu} \langle w^-\rangle^{\mu}\p_\tau^m p \big\|_{L^2(\hat{\Om}_f)} \cdot
\Big\| \f{\langle w^+\rangle^{5\mu}  [\div, \p_\tau^m] \La_{+}}{\langle w^-\rangle^{\mu}} \Big\|_{L^2(\Om_f)} \\\nonumber
&\lesssim G_3^b E_3^\f12 \cdot \mathcal{P}( \| \p_1f \|_{H^3}+1) \,.
\end{align*}
Similarly, there holds
\begin{align*}
\big| \int_{\hat{\Om}_f}  e^{q^+}\langle w^+\rangle^{10\mu} \p_\tau^m \hat{p}\cdot \div \p_\tau^m \hat{\Lambda}_{+} \dx\big|
\lesssim G_3^b E_3^\f12 \cdot \mathcal{P}( \| \p_1f \|_{H^3}+1) \,.
\end{align*}

\textbf{Estimate of the third and the fourth line of \eqref{F70}.}
For the third line, by Lemma \ref{lemTang} and  Lemma \ref{lemPre}, we estimate
\begin{align*}
&\big| \int_{\hat{\Om}_f} \nabla( e^{q^+}\langle w^+\rangle^{10\mu}) \p_\tau^m \hat{p}\cdot   \p_\tau^m \hat{\Lambda}_{+} \dx
  +\int_{\Om_f} \nabla( e^{q^+}\langle w^+\rangle^{10\mu}) \p_\tau^m p\cdot   \p_\tau^m \Lambda_{+} \dx\big| \\\nonumber
&\lesssim  \big\|  \langle w^+\rangle^{5\mu} \langle w^-\rangle^{\mu} \p_\tau^m \hat{p} \|_{L^2 (\hat{\Om}_f)}
\Big\| \f{\langle w^+\rangle^{5\mu} \p_\tau^m \hat{\Lambda}_{+}}{\langle w^-\rangle^{\mu}} \Big\|_{L^2(\hat{\Om}_f)} \\\nonumber
&\quad+  \big\|  \langle w^+\rangle^{5\mu} \langle w^-\rangle^{\mu} \p_\tau^m p \|_{L^2(\Om_f)}
\Big\| \f{\langle w^+\rangle^{5\mu} \p_\tau^m \La_{+}}{\langle w^-\rangle^{\mu}} \Big\|_{L^2(\Om_f)} \\\nonumber
&\lesssim G_3^b E_3^{\f12} \,.
\end{align*}
The forth line of \eqref{F70} is zero due to the vanishing of $e_2 \cdot \hat{\Lambda}_{\pm}$
 and $e_2 \cdot\Lambda_{\pm}$ on $\hat{\Ga}$ and on $\Ga$ respectively.

\textbf{Estimate the last line of \eqref{F70}.}
 Note that on $\Ga_f$, there holds $\ud{\p_\tau^m \hat{p}}=\p_1^m\ud{\hat{p}}=\p_1^m\ud{p}=\ud{\p_\tau^m p}$. Thus the last line equals
\begin{align*}
&-\int_{\Ga_f} e^{q^+}\langle w^+\rangle^{10\mu} \p_\tau^m \hat{p} n_f\cdot \p_\tau^m \hat{\Lambda}_{+}\d\sigma(x)
+\int_{\Ga_f} e^{q^+}\langle w^+\rangle^{10\mu} \p_\tau^m p n_f\cdot \p_\tau^m \La_{+}\d\sigma(x)\\
&=-\int_{\BR} e^{q^+}\langle w^+\rangle^{10\mu} \p_1^m \ud{\hat{p}} N_f \cdot
\big(\ud{\p_\tau^m \hat{\Lambda}_{+}}-\ud{\p_\tau^m \La_{+}}\big) \d x_1 .
\end{align*}
Here, we need reduce one derivative on
$\ud{\p_\tau^m (\hat{\Lambda}_{+}-\La_{+})}$ by
using the boundary condition on $\Ga_f$.
Recalling $\eqref{BB5}_1$ that
\begin{align*}
&(\p_t\pm\p_1)f=N_f\cdot \ud{\La_{\pm}}=N_f \cdot \ud{\hat{\La}_{\pm}}.
\end{align*}
Consequently,
\begin{align*}
&N_f \cdot
\big(\ud{\p_\tau^m \hat{\Lambda}_{+}}-\ud{\p_\tau^m \La_{+}}\big)
=N_f \cdot
\p_1^m \big(\ud{ \hat{\Lambda}_{+}}- \ud{\La_{+}}\big)\\
&=\p_1^m \Big( N_f \cdot
 \big(\ud{ \hat{\Lambda}_{+}}- \ud{\La_{+}}\big) \Big)
 -\sum_{0<a\leq m}C_m^a \p_1^a N_f \cdot \p_1^{m-a}
 \big(\ud{ \hat{\Lambda}_{+}}- \ud{\La_{+}}\big)\\
&=\sum_{0<a\leq m}C_m^a \p_1^{a+1} f \cdot \big(\ud{ \p_\tau^{m-a} \hat{\La}_{+}^1}- \ud{\p_\tau^{m-a} \La_{+}^1}\big).
\end{align*}
Note that 
the term $\p_1^{a+1} f$ contains higher order derivative
when $a=m=3$.

For $0\leq m \leq 2$, by weighted Sobolev inequalities Lemma \ref{Sobo2}, weighted trace Lemma \ref{lemD2} and  Lemma \ref{lemPre}, one deduces
\begin{align*}
&\big| \int_{\BR} e^{q^+}\langle w^+\rangle^{10\mu} \p_1^m \ud{\hat{p}} N_f \cdot
\big(\ud{\p_\tau^m \hat{\Lambda}_{+}}-\ud{\p_\tau^m \La_{+}}\big) \d x_1 \big| \\
&=\big| \int_{\BR} e^{q^+}\langle w^+\rangle^{10\mu}  \p_1^m \ud{\hat{p}} \sum_{0<a\leq m}C_m^a\p_1^{a+1} f \cdot \p_1^{m-a}
 \big(\ud{ \hat{\La}_{+}^1}- \ud{\La_{+}^1}\big) \d x_1\big| \\
&\lesssim \big\| \langle w^+\rangle^{5\mu}\langle w^-\rangle^{\mu} \p_1^m \ud{\hat{p}} \big\|_{L^2(\BR)}\sum_{0<a\leq m}
\Big\| \f{\langle w^+\rangle^{5\mu}\p_1^{a+1} f \cdot \p_1^{m-a}
 \big(\ud{ \hat{\La}_{+}^1}- \ud{\La_{+}^1}\big) }{\langle w^-\rangle^{\mu} } \Big\|_{L^2(\BR)}\\
 &\lesssim G_3^b E_3 .
\end{align*}
When $m=3$, we need move half derivative on the pressure away.
Denote
\begin{align*}
\Upsilon=\Big(e^{q^+}\langle w^+\rangle^{10\mu}
\sum_{0<a\leq m}C_m^a \p_1^{a+1} f \cdot
 \big(\ud{ \p_\tau^{m-a}\hat{\La}_{+}^1}- \ud{\p_\tau^{m-a}\La_{+}^1}\big) \Big).
\end{align*}
Then we write
\begin{align*}
& \big| \int_{\BR} e^{q^+}\langle w^+\rangle^{10\mu} \p_1^m \ud{\hat{p}}
 \sum_{0<a\leq m}C_m^a\p_1^{a+1} f \cdot \p_1^{m-a}
 \big(\ud{ \hat{\La}_{+}^1}- \ud{\La_{+}^1}\big) \d x_1\big| \\
&=\big| \int_{\BR}   \langle w^+\rangle^{5\mu}
 \langle w^-\rangle^{\mu}\langle \p_1\rangle^{-\f12}\p_1^m \ud{\hat{p}}\\
&\qquad\cdot
 \langle w^+\rangle^{-5\mu}    \langle w^-\rangle^{-\mu}
\langle \p_1\rangle^{\f12} \Big(e^{q^+}\langle w^+\rangle^{10\mu}
\sum_{0<a\leq m}C_m^a \p_1^{a+1} f \cdot \p_1^{m-a}
 \big(\ud{ \hat{\La}_{+}^1}- \ud{\La_{+}^1}\big) \Big)\d x_1\big| \\
&\leq \big\|  \langle w^+\rangle^{5\mu}\langle w^-\rangle^{\mu} \langle \p_1\rangle^{-\f12}\p_1^m \ud{\hat{p}} \big\|_{L^2(\BR)}
\| \langle w^+\rangle^{-5\mu}    \langle w^-\rangle^{-\mu}
\langle \p_1\rangle^{\f12} \Upsilon\|_{L^2(\BR)}.
\end{align*}
By Lemma \ref{lemW}, weighted chain rule Lemma \ref{lemG3}, weighted trace Lemma \ref{lemD3} and  Lemma \ref{lemPre},
 we deduce that
\begin{align*}
&\big\|  \langle w^+\rangle^{5\mu}\langle w^-\rangle^{\mu} \langle \p_1\rangle^{-\f12}\p_1^m \ud{\hat{p}} \big\|_{L^2(\BR)}
\lesssim \big\|  \langle w^+\rangle^{5\mu}\langle w^-\rangle^{\mu} \langle \p_1\rangle^{\f12}\p_1^{m-1} \ud{\hat{p}} \big\|_{L^2(\BR)}\\
&\quad\lesssim \| \langle w^+\rangle^{5\mu}\langle w^-\rangle^{\mu} \nabla^{\leq m} \hat{p}\|_{L^2(\hat{\Om}_f)}\lesssim (E_3 G_3^b)^{\f12} .
\end{align*}
For the other term, by weighted Sobolev inequalities Lemma \ref{Sobo1}, Lemma \ref{Sobo2} and weighted trace Lemma \ref{lemD3}, we estimate
\begin{align*}
&\| \langle w^+\rangle^{-5\mu}    \langle w^-\rangle^{-\mu}
\langle \p_1\rangle^{\f12} \Upsilon\|_{L^2(\BR)} \\
&\lesssim \sum_{0<a\leq m}\|    \langle w^+\rangle^{5\mu}    \langle w^-\rangle^{-\mu}
\p_1^{a+1} f \cdot \p_\tau^{m-a}\La_{+}^1 \|_{H^1(\Om_f)}\\
&\quad+ \sum_{0<a\leq m}\|    \langle w^+\rangle^{5\mu}    \langle w^-\rangle^{-\mu}
\p_1^{a+1} f \cdot \p_\tau^{m-a}\hat{\La}_{+}^1 \|_{H^1(\hat{\Om}_f)}\\
&\lesssim (E_4 G_3^b)^{\f12}  .
\end{align*}

Combining all the above estimates, \eqref{F6} reduces to
\begin{align*}
&\f12\f{\d}{\dt} \Big( \int_{\Om_f}
| \langle w^+ \rangle^{5\mu} \p_\tau^m \La_{+}|^2 e^{q^+} \dx
+\int_{\hat{\Om}_f}
| \langle w^+ \rangle^{5\mu} \p_\tau^m \hat{\La}_{+}|^2 e^{q^+} \dx \Big)\\\nonumber
&+\int_{\Om_f}
 \f{|\langle w^+ \rangle^{5\mu}\p_\tau^m \La_{+} |^2}{ \langle w^- \rangle^{2\mu} }
 e^{q^+}\dx
+\int_{\hat{\Om}_f} \f{|\langle w^+ \rangle^{5\mu}\p_\tau^m \hat{\La}_{+} |^2}{ \langle w^- \rangle^{2\mu} }
 e^{q^+}  \dx
\\\nonumber
&\lesssim E_4^{\f12}G_3^b \mathcal{P}(1+E_4^{\f12})\,.
\end{align*}
Taking the sum for $0\leq m\leq 3$, then taking the integral in time over $[0,t]$, note similar estimates also hold for $\La_-$ and $\hat{\La}_-$,
we obtain
\begin{align*}
&\sup_{0\leq \bar{\tau}\leq t} E_3^{b,\tau}(\bar{\tau}) +\int_0^t G_3^{b,\tau} (\bar{\tau}) \d \bar{\tau}
\leq C E_3^{b,\tau}(0) +C\int_0^t E_4^{\f 12}(\bar{\tau}) G_3^b(\bar{\tau})\d \bar{\tau}.
\end{align*}
This yields \eqref{AA9}.

\section{Weighted energy estimate of the vorticity and current}\label{Ev}


%
%
In this section, we derive the weighted energy estimate for the vorticity and current:
\begin{align*}
\nabla\times \La_{\pm},\quad \nabla\times \hat{\La}_{\pm}\,.
\end{align*}
Note that we are in two dimensional dimension. For two dimensional vector value function $F$, the curl operator on $F$ denotes
$\nabla\times F=(-\p_2,\p_1) \cdot F$. In addition, for function $g$, there holds
$\nabla g \times F=(-\p_2,\p_1) g\cdot F$.
Thus $\nabla\times (gF)=g \nabla\times F + \nabla g \times F=g (-\p_2,\p_1) \cdot F+(-\p_2,\p_1) g\cdot F$.

By taking curl of \eqref{BB3} and \eqref{BB4}, we derive the following system of equations for
 $\nabla\times \La_{\pm},\,\nabla\times \hat{\La}_{\pm}$:
\begin{equation}\label{F111}
\begin{cases}
\p_t  \nabla\times \La_{+}+Z_{-} \cdot \nabla \nabla\times \La_{+}= \p_j \La_{+} \times \nabla \La_{-}^j, \\
\p_t  \nabla\times \La_{-}+Z_{+} \cdot \nabla \nabla\times \La_{-}= \p_j \La_{-} \times \nabla \La_{+}^j,
\end{cases}
\quad\text{in}\,\, \Om_f
\end{equation}
and
\begin{equation}\label{F112}
\begin{cases}
\p_t  \nabla\times \hat{\La}_{+}+\hat{Z}_{-} \cdot \nabla \nabla\times \hat{\La}_{+}= \p_j \hat{\La}_{+} \times \nabla \hat{\La}_{-}^j, \\
\p_t  \nabla\times \hat{\La}_{-}+\hat{Z}_{+} \cdot \nabla \nabla\times \hat{\La}_{-}= \p_j \hat{\La}_{-} \times \nabla \hat{\La}_{+}^j,
\end{cases}
\quad\text{in}\,\, \hat{\Om}_f.
\end{equation}

We define the weighted energy inside the bulk of the domain for the vorticity and current as follows:
\begin{align}\label{F113}
&E_3^{b,\omega}(t)=\sum_{+,-}\sum_{|\alpha|\leq 2} \Big( \int_{\Om_f}
| \langle w^\pm \rangle^{5\mu} \p^\alpha \nabla\times \La_{\pm}|^2 \dx
+\int_{\hat{\Om}_f}
| \langle w^\pm \rangle^{5\mu} \p^\alpha \nabla\times \hat{\La}_{\pm}|^2 \dx \Big).
\end{align}
The corresponding weighted ghost energy estimate is defined by
\begin{align}\label{F114}
&G_3^{b,\omega}(t)=\sum_{+,-} \sum_{|\alpha|\leq 2} \Big( \int_{\Om_f}
 \f{|\langle w^\pm \rangle^{5\mu}\p^\alpha \nabla\times \La_{\pm} |^2}{ \langle w^\mp \rangle^{2\mu} } \dx
+\int_{\hat{\Om}_f} \f{|\langle w^\pm \rangle^{5\mu}\p^\alpha \nabla\times \hat{\La}_{\pm} |^2}{ \langle w^\mp \rangle^{2\mu} }
 \dx\Big) \,.
\end{align}
In the sequel, we deal with the weighted energy estimate for right Alfv\'en waves.
The lower order energy for the left Alfv\'en waves can be estimated similarly.

The estimate is similar to the weighted tangential derivative estimate.
For $|\alpha|\leq 2$, by employing \eqref{F111} and \eqref{F112}, we have
\begin{align}\label{F115}
&\f12\f{\d}{\dt} \Big( \int_{\Om_f}
| \langle w^+ \rangle^{5\mu} \p^\alpha \nabla\times \La_{+}|^2 e^{q^+} \dx
+\int_{\hat{\Om}_f}
| \langle w^+ \rangle^{5\mu} \p^\alpha \nabla\times \hat{\La}_{+}|^2 e^{q^+} \dx \Big)\\\nonumber
&+ \int_{\Om_f}
 \f{|\langle w^+ \rangle^{5\mu}\p^\alpha \nabla\times \La_{+}|^2}{ \langle w^- \rangle^{2\mu} }
 e^{q^+}\dx
+\int_{\hat{\Om}_f} \f{|\langle w^+ \rangle^{5\mu}\p^\alpha \nabla\times \hat{\La}_{+}|^2}{ \langle w^- \rangle^{2\mu} }
 e^{q^+}  \dx \\\nonumber
&=\f12\int_{\Om_f}
|\p^\alpha \nabla\times \La_{+}|^2 \La_{-}\cdot\nabla(\langle w^+ \rangle^{10\mu} e^{q(\sigma^{+})})  \dx
+\f12\int_{\hat{\Om}_f}
|\p^\alpha \nabla\times \hat{\La}_{+}|^2 \hat{\La}_{-}\cdot\nabla(\langle w^+ \rangle^{10\mu} e^{q(\sigma^{+})})  \dx \\\nonumber
&\quad-\int_{\Om_f}
\langle w^+ \rangle^{10\mu}\p^\alpha \nabla\times \La_{+}e^{q^+} \cdot
 [\p^\alpha,\La_{-}\cdot\nabla] \nabla\times \La_{+}\dx \\\nonumber
&\quad-\int_{\hat{\Om}_f}
\langle w^+ \rangle^{10\mu}\p^\alpha \nabla\times \hat{\La}_{+}e^{q^+} \cdot
 [\p^\alpha,\hat{\La}_{-}\cdot\nabla] \nabla\times \hat{\La}_{+} \dx \\\nonumber
&\quad-\int_{\Om_f}
\langle w^+ \rangle^{10\mu}\p^\alpha \nabla\times \La_{+}e^{q^+} \cdot
 \p^\alpha \big( \p_j \La_{+} \times \nabla \La_{-}^j \big) \dx \\\nonumber
&\quad -\int_{\hat{\Om}_f}
\langle w^+ \rangle^{10\mu}\p^\alpha \nabla\times \hat{\La}_{+}e^{q^+} \cdot
 \p^\alpha \big( \p_j \hat{\La}_{+} \times \nabla \hat{\La}_{-}^j \big) \dx .
\end{align}
For the first term on first line on the right hand side of \eqref{F115}, by weighted Sobolev inequalities Lemma \ref{Sobo1}, we obtain
\begin{align}\label{F116}
&\f12\big|\int_{\Om_f}
|\p^\alpha \nabla\times \La_{+}|^2
\La_{-}\cdot\nabla (\langle w^+ \rangle^{10\mu} e^{q^+} ) \dx \big| \\\nonumber
&=\big| \int_{\Om_f}
|\p^\alpha \nabla\times \La_{+}|^2
\La_{-}^1\cdot \Big(5\mu\langle w^+ \rangle^{10\mu-2}w^+ \p_1w^+
+\f12\f{\langle w^+ \rangle^{10\mu} }{\langle w^- \rangle^{2\mu}}\p_1 w^- \Big) e^{q^+}  \dx \big| \\\nonumber
&\lesssim E_2^{\f12} G_3^b.
\end{align}
Next, for the  second line on the right hand side of \eqref{F115}, similarly, one deduces that
\begin{align}\label{F117}
&\big| \int_{\Om_f}
\langle w^+ \rangle^{10\mu}\p^\alpha \nabla\times \La_{+}e^{q^+} \cdot
 [\p^\alpha,\La_{-}\cdot\nabla] \nabla\times \La_{+} \dx \big|\\
&=\big| \sum_{\beta+\gamma=\alpha,\, \beta\neq 0} C_\alpha^\beta \int_{\Omega_f}
\langle w^+\rangle^{10\mu}\p^\alpha \nabla\times \La_{+}e^{q^+} \cdot
 (\p^\beta \La_{-}\cdot\nabla \p^\gamma\nabla\times \La_{+}) \dx \big| \nonumber\\
&=\big| \sum_{\beta+\gamma=\alpha,\, \beta\neq 0 } C_\alpha^\beta \int_{\Omega_f}
\f{\langle w^+\rangle^{5\mu}\p^\alpha \nabla\times \La_{+}}{\langle w^-\rangle^{\mu}} e^{q^+} \cdot
 \Big(\langle w^-\rangle^{2\mu} \p^\beta \La_{-} \cdot
 \f{\langle w^+\rangle^{5\mu} \nabla \p^\gamma \nabla\times \La_{+}}{\langle w^-\rangle^{\mu} } \Big) \dx\big| \nonumber\\ \nonumber
&\lesssim E_3^{\f12} G_3^b.
\end{align}
For the forth line on the right hand side of \eqref{F115}, we estimate
\begin{align}\label{F118}
&\big| \int_{\Om_f}
\langle w^+ \rangle^{10\mu}\p^\alpha \nabla\times \La_{+}e^{q^+} \cdot
 \p^\alpha \big( \p_j \La_{+} \times \nabla \La_{-}^j \big) \dx \big|\\
&=\big|\sum_{\beta+\gamma=\alpha } C_\alpha^\beta \int_{\Omega_f}
\langle w^+\rangle^{10\mu}\p^\alpha \nabla\times \La_{+}e^{q^+} \cdot
 \big(\p_j\p^\beta \La_{+} \times \nabla \p^\gamma \La_{-}^j\big) \dx \big| \nonumber\\
&=\big| \sum_{\beta+\gamma=\alpha } C_\alpha^\beta \int_{\Omega_f}
\f{\langle w^+\rangle^{5\mu}\p^\alpha \nabla\times \La_{+}}{\langle w^-\rangle^{\mu}} e^{q^+} \cdot
 \Big(\f{\langle w^+\rangle^{5\mu} \p_j\p^\beta \La_{+}}{\langle w^-\rangle^{\mu} } \times
 \langle w^-\rangle^{2\mu}\nabla \p^\gamma \La_{-}^j\Big) \dx\big| \nonumber\\ \nonumber
&\lesssim E_3^{\f12} G_3^b .
\end{align}
For the second term on the first line on the right hand side of \eqref{F115},
 the third line on the right hand side of \eqref{F115}, the fifth line on the right hand side of \eqref{F115}, they can be estimated similiar to \eqref{F116}, \eqref{F117}, \eqref{F118}
and can be bounded by $E_s^{\f12} G_2^b$. Note the estimate of \eqref{F115} is for all $|\alpha|\leq 1$. Hence we imply
\begin{align*}
&\sum_{|\alpha|\leq 2}\f12\f{\d}{\dt} \Big( \int_{\Om_f}
| \langle w^+ \rangle^{5\mu} \p^\alpha \nabla\times \La_{+}|^2 e^{q^+} \dx
+\int_{\hat{\Om}_f}
| \langle w^+ \rangle^{5\mu} \p^\alpha \nabla\times \hat{\La}_{+}|^2 e^{q^+} \dx \Big)\\\nonumber
&+ \sum_{|\alpha|\leq 2} \int_{\Om_f}
 \f{|\langle w^+ \rangle^{5\mu}\p^\alpha \nabla\times \La_{+}|^2}{ \langle w^- \rangle^{2\mu} }
 e^{q^+}\dx
+\int_{\hat{\Om}_f} \f{|\langle w^+ \rangle^{5\mu}\p^\alpha \nabla\times \hat{\La}_{+}|^2}{ \langle w^- \rangle^{2\mu} }
 e^{q^+}  \dx \\\nonumber
&\lesssim  E_3^{\f12} G_3^b.
\end{align*}
Taking the integral in time over $[0,t]$, note similar estimates also hold for $\nabla\times \La_{-}$
and $\nabla\times \hat{\La}_{-}$,
we obtain
\begin{align}\label{F119}
&\sup_{0\leq \bar{\tau}\leq t} E_3^{b,\omega}(\bar{\tau}) +\int_0^t G_3^{b,\omega} (\bar{\tau}) \d \bar{\tau}
\leq C E_3^{b,\omega}(0) +C\int_0^t E_3^{\f 12}(\bar{\tau}) G_3^b(\bar{\tau})\d \bar{\tau}.
\end{align}
This yields \eqref{AA10}.

For $s\geq 4$, we define the weighted energy inside the bulk of the domain for the vorticity and current as follows:
\begin{align}\label{F121}
&E_s^{b,\omega}(t)=\sum_{+,-}\sum_{|\alpha|\leq s-1} \Big( \int_{\Om_f}
| \langle w^\pm \rangle^{2\mu} \p^\alpha \nabla\times \La_{\pm}|^2\dx
+\int_{\hat{\Om}_f}
| \langle w^\pm \rangle^{2\mu} \p^\alpha \nabla\times \hat{\La}_{\pm}|^2\dx \Big).
\end{align}
The corresponding weighted ghost energy estimate is defined by
\begin{align}\label{F122}
&G_s^{b,\omega}(t)=\sum_{+,-} \sum_{|\alpha|\leq s-1} \Big( \int_{\Om_f}
 \f{|\langle w^\pm \rangle^{2\mu}\p^\alpha \nabla\times \La_{\pm}|^2}{ \langle w^\mp \rangle^{2\mu} }\dx
+\int_{\hat{\Om}_f} \f{|\langle w^\pm \rangle^{2\mu}\p^\alpha \nabla\times \hat{\La}_{\pm}|^2}{ \langle w^\mp \rangle^{2\mu}}\dx\Big) \,.
\end{align}
Similar to the lower order weighted estimate, we can obtain the higher order energy estimate for the vorticity and the current:
\begin{align*}
&\sup_{0\leq \bar{\tau}\leq t} E_s^{b,\omega}(\bar{\tau}) +\int_0^t G_s^{b,\omega} (\bar{\tau}) \d \bar{\tau}
\leq C E_s^{b,\omega}(0) +C\int_0^t E_s^{\f 12}(\bar{\tau}) G_s^b(\bar{\tau})\d \bar{\tau}.
\end{align*}
 Since the proof is the same to \eqref{F119}, the details are omitted. This yields \eqref{AA11}.

\section{Local solutions for the 2D ideal
incompressible MHD current-vortex sheet}
In this last section, we show that the local solutions of current-vortex sheet free boundary problem
for the two dimensional ideal incompressible magnetohydrodynamics under strong magnetic field.
The  local solution  to  the problem  will be obtained by the method of successive approximations.
Precisely, we show that a sequence of solutions
will finally converge to the solutions of the current-vortex sheet free boundary problem.
We remind that the notions in this section is independent of the notations in the previous argument.

Let $s\geq 4$ be an integer.
For $f$:\,$[0,T)\times \bR \rightarrow \bR$,
$\La_\pm,\omega_\pm$:\,$[0,T)\times \Om_f \rightarrow \bR$,
$\hat{\La}_\pm,\hat{\omega}_\pm$:\,$[0,T)\times \hat{\Om}_f \rightarrow \bR$,
we denote  the energy functional
\begin{align*} 
E_{s+\f12}^f(\p_tf,f)&=\| \p_t f(t,\cdot)\|_{H^{s-\f12}(\bR)}^2+\| \p_1 f(t,\cdot)\|_{H^{s-\f12}(\bR)}^2,\\
E_{s}^b(\La_\pm,\hat{\La}_\pm)&=\sum_{+,-}\big(
\| \La_\pm(t,\cdot )\|_{H^s(\Om_f)}^2
+\| \hat{\La}_\pm(t,\cdot )\|_{H^s(\hat{\Om}_f)}^2\big),\\
E_{s}^w(\omega_\pm,\hat{\omega}_\pm)&=\sum_{+,-}\big(
\| \omega_\pm(t,\cdot )\|_{H^{s-1}(\Om_f)}^2
+\| \hat{\omega}_\pm(t,\cdot )\|_{H^{s-1}(\hat{\Om}_f)}^2\big),\\
E_s(\La_\pm,\hat{\La}_\pm,f)&=E_{s}^b(\La_\pm,\hat{\La}_\pm)+E_{s+\f12}^f(\p_tf,f)+\| f(t,\cdot) \|^2_{L^\infty}.
\end{align*}
We remind that the energy notations are simpler compared with the definition given in Section 2.4
since no weight functions are applied.

\textbf{Linearized system of the free surface and the voricity and the current}

For the free surface \eqref{BB9}, we introduce the following  approximate solutions
\begin{align} \label{Li-BB9}
&(\p_t^2-\p_1^2) f^{(n+1)}
+ \f12({\ud{\La_+^1}}^{(n)} +{\ud{\La_-^1}}^{(n)}
+{\ud{\hat{\La}_+^1}}^{(n)} +{\ud{\hat{\La}_-^1}}^{(n)})
\p_t\p_1 f^{(n+1)} \\\nonumber
&+\f12\big({\ud{\La_-^1}}^{(n)}  - {\ud{\La_+^1}}^{(n)}
+ {\ud{\hat{\La}_-^1}}^{(n)}  - {\ud{\hat{\La}_+^1}}^{(n)}
+ {\ud{\La_-^1}}^{(n)}  {\ud{\La_+^1}}^{(n)}
+ {\ud{\hat{\La}_-^1}}^{(n)}  {\ud{\hat{\La}_+^1}}^{(n)} \big)\p_1^2 f^{(n+1)} \\\nonumber
&=-\f12\big( N_{f^{(n)}}\cdot\ud{\nabla p^{(n)}}+ N_{f^{(n)}}\cdot\ud{\nabla \hat{p}^{(n)}} \big),
\end{align}
where the pressure is defined by
\begin{equation}\label{li-B22}
\begin{cases}
\Delta p^{(n)} =-\nabla \cdot (\La_{-}^{(n)} \cdot \nabla \La_{+}^{(n)}) \quad \textrm{in}\,\, \Om_{f^{(n)}},\\
\Delta \hat{p}^{(n)} =-\nabla \cdot (\hat{\La}_{-}^{(n)} \cdot \nabla \hat{\La}_{+}^{(n)})\quad \textrm{in}\,\, \hat{\Om}_{f^{(n)}},\\
p^{(n)} =\hat{p}^{(n)}\quad \textrm{on} \,\, \Ga_{f^{(n)}},\\
\p_{N_{f^{(n)}}}p^{(n)}-\p_{N_{f^{(n)}}}\hat{p}^{(n)}
= -\big({\ud{\La_{+}^1}}^{(n)} {\ud{\La_{-}^1}}^{(n)}
- {\ud{\hat{\La}_{+}^1}}^{(n)} {\ud{\hat{\La}_{-}^1}}^{(n)} \big) \p_1^2 f^{(n)}\\
\quad-\big({\ud{\La_{+}^1}}^{(n)}-{\ud{\hat{\La}_{+}^1}}^{(n)} \big) \p_1 (\p_t-  \p_1) f^{(n)}
  -\big({\ud{\La_{-}^1}}^{(n)}-{\ud{\hat{\La}_{-}^1}}^{(n)}\big) \p_1 (\p_t+  \p_1) f^{(n)}
\quad\text{on}\,\,\Ga_{f^{(n)}} ,\\
\p_2 p^{(n)}=0 \,\, \textrm{on} \,\,  \Ga,\quad \p_2 \hat{p}^{(n)}=0 \,\, \textrm{on} \,\,  \hat{\Ga}.
\end{cases}
\end{equation}

For the vorticity system \eqref{F111}-\eqref{F112},
we introduce the following  approximate solutions:
\begin{equation}\label{li-F111}
\begin{cases}
\p_t \omega_{+}^{(n+1)} +(\La_{-}^{(n)}-e_1) \cdot \nabla \omega_{+}^{(n+1)}
= \p_j \La_{+}^{(n)} \times \nabla {\La_{-}^j}^{(n)}, \\
\p_t \omega_{-}^{(n+1)} +(\La_{+}^{(n)}+e_1) \cdot \nabla \omega_{-}^{(n+1)}
= \p_j \La_{-}^{(n)} \times \nabla {\La_{+}^j}^{(n)},
\end{cases}
\quad\text{in}\,\, \Om_{f^{(n)}}
\end{equation}
and
\begin{equation}\label{li-F112}
\begin{cases}
\p_t \hat{\omega}_+^{(n+1)} +(\hat{\La}_-^{(n)}-e_1) \cdot \nabla \hat{\omega}_+^{(n+1)}
 = \p_j \hat{\La}_+^{(n)} \times \nabla
 {\hat{\La}_-}^{j\,\,(n)}
 , \\
\p_t \hat{\omega}_-^{(n+1)} +(\hat{\La}_+^{(n)}+e_1) \cdot \nabla\hat{\omega}_-^{(n+1)}
= \p_j \hat{\La}_-^{(n)} \times \nabla
{\hat{\La}_+}^{j\,\,(n)} ,
\end{cases}
\quad\text{in}\,\, \hat{\Om}_{f^{(n)}}.
\end{equation}

Let $s\geq 4$ be an integer, $A>0$, $0<\delta<\f12 $ be given constants. Assume that
\begin{equation}\label{R-10}
\begin{cases}
E_s (\La_\pm^{(n)},\hat{\La}_\pm^{(n)}, f^{(n)})\leq C(A,\delta)A,\\
\| f^{(n)} \|_{L^\infty}\leq 1-\delta,
\end{cases}
\end{equation}
where $C(A,\delta)$ is a constant depending on $A$ and $\delta$.
Similar to the energy estimate of the free surface without weights in Section \ref{Ef}, we can derive the energy estimate for \eqref{Li-BB9}.
\begin{lem}\label{Lem-fL}
Given the initial data  $(\p_t  f(0)=v^0, f(0)= f^0)$ with $(v^0,  f^0)\in H^{s-\f12}\times H^{s+\f12}(\bR)$,
there exists $T>0$ depending on $A$ and $\delta$, and there exits a unique solution $(f^{(n+1)},\p_tf^{(n+1)})\in C\big([0,T];H^{s+\f12}\times H^{s-\f12}(\bR)\big)$ to the system \eqref{Li-BB9} such that
\begin{align*}
&\sup_{t\in[0,T]}E_{s+\f12}^f(\p_tf^{(n+1)},f^{(n+1)})
 \leq \| v^0\|_{H^{s-\f12}}^2+\| \p_1f^0\|_{H^{s-\f12}}^2\\
&\quad + C(A,\delta)\Big(\int_0^T
E_{s+\f12}^f(\p_tf^{(n+1)},f^{(n+1)}) E_s^{\f12} (\La_\pm^{(n)},\hat{\La}_\pm^{(n)},f^{(n)})\\
&\qquad+\big( E_{s+\f12}^f(\p_tf^{(n+1)},f^{(n+1)}) \big)^{\f12}
\|\big( \p_{N_{f^{(n)}}}p^{(n)}+\p_{N_{f^{(n)}}}\hat{p}^{(n)} \big)(\tau)\|_{H^{s-\f12}(\bR)}\d\tau\Big).
 \end{align*}
Also, we have
\begin{align*}
&\|f^{(n+1)}(t)-f^{(n+1)}(0)\|_{L^\infty}\le \|f^{(n+1)}(t)-f^{(n+1)}(0)\|_{H^1}
\le  t\sup_{0\leq \tau\leq t}\|\p_\tau f^{(n+1)}(\tau)\|_{H^{s-\f12}} .
\end{align*}
\end{lem}


Similar to the energy estimate of voricity and current estimate in Section \ref{Ev},
we can derive the energy estimate for \eqref{li-F111} and \eqref{li-F112}.
\begin{lem}\label{Lem-vL}
Given the initial data $\omega_\pm(0)=\omega^0_\pm\in H^{s-1}(\Om_{f_0}),\, \hat{\omega}_\pm(0)=\hat{\omega}^0_\pm\in H^{s-1}(\hat{\Om}_{f_0})$,
 there exists $T>0$ depending on $A$ and $\delta$, and there
  exist  unique solution $\omega_\pm^{(n+1)}\in C\big([0,T];H^{s-1}(\Om_{f^{(n)}})\big)$ to the system \eqref{li-F111}
and solution $\hat{\omega}_\pm^{(n+1)}\in C\big([0,T];H^{s-1}(\hat{\Om}_{f^{(n)}})\big)$ to the system \eqref{li-F112}
such that
 \begin{align*}
\sup_{t\in[0,T]}&  E_{s}^w(\omega_\pm^{(n+1)},\hat{\omega}_\pm^{(n+1)})\leq E_{s}^w(\omega_\pm^0,\hat{\omega}_\pm^0) \\
 &\qquad  +C(A,\delta)\Big(\int_0^T
E_{s}^w(\omega_\pm^{(n+1)},\hat{\omega}_\pm^{(n+1)}) E_s^{\f12} (\La_\pm^{(n)},\hat{\La}_\pm^{(n)},f^{(n)}) \\
&\qquad +\big(E_{s}^w(\omega_\pm^{(n+1)},\hat{\omega}_\pm^{(n+1)})\big)^{\f12} E_s  (\La_\pm^{(n)},\hat{\La}_\pm^{(n)},f^{(n)})
\d\tau.
 \end{align*}
\end{lem}


For the pressure $ p^{(n)},\,\hat{p}^{(n)}$ of the system \eqref{li-B22},
similar to the estimate of pressure in Lemma \ref{lemPP3},
we would obtain the following estimate.
\begin{lem}\label{lem-PL}
There exist solution $p^{(n)}\in H^{s}(\Om_{f^{(n)}})$ and $\hat{p}^{(n)}\in H^{s}(\hat{\Om}_{f^{(n)}})$ to the system \eqref{li-F112}
such that
\begin{align*} 
\| \p_{N_{f^{(n)}}}p^{(n)}+\p_{N_{f^{(n)}}}\hat{p}^{(n)} \|_{H^{s-\f12}(\bR)}
&\leq C(A,\delta) E_s (\La_\pm^{(n)},\hat{\La}_-^{(n)}, f^{(n)}).
\end{align*}
\end{lem}

~\\~
\textbf{The initial data}

Assume that
\begin{align*}
f^0\in \dot{H}^1(\BR)\cap \dot{H}^{s+\f12}(\BR) \cap L^\infty(\BR),\quad
v^0\in H^{s-\f12}(\BR) ,\quad
\La_{\pm}^0 \in H^{s}(\Omega_{f_0}),\, \hat{\La}_{\pm}^0 \in H^{s}(\hat{\Om}_{f_0}).
\end{align*}
Furthermore, assume that there exists a small constant that $c_0>0$ such that
\begin{align}\nonumber
\| \p_1 f^0\|_{H^{s-\f12}(\bR)}^2
+\| v^0\|_{H^{s-\f12}(\bR)}^2+
\| \La_\pm^0\|_{H^s(\Om_{f_0})}^2
+\| \hat{\La}_\pm^0\|_{H^s(\hat{\Om}_{f^0})}^2\leq c_0
\end{align}
and $$\| f^0\|_{L^\infty} \leq 1-2\delta$$
for $0<\delta<\f12$.

The iteration step need to be constructed in a fixed domain.
We choose $f_*=f_0$ and take $\Omega_*=\Omega_{f_0},\, \hat{\Om}_*=\hat{\Om}_{f_0}$ as the reference region.
The initial data $(f^{(n)}(0),\p_tf^{(n)}(0),\omega_{\pm,*}^{(n)}(0),\hat{\omega}_{\pm,*}^{(n)}(0)$ for the equivalent system is defined as follows
\begin{align*}
&f^{(n)}(0)=f^0,\quad \p_tf^{(n)}(0)=v^0=u^0_\pm(x_1,f_0(x_1))\cdot(-\p_1f_0,1),\\
&\omega_{\pm,*}^{(n)}(0)=\curl\La_0^\pm,\quad \hat{\omega}_{\pm,*}^{(n)}(0)=\curl\hat{\La}_0^\pm.
\end{align*}
In addition,
\begin{align*}
\|\p_1 f^0\|^2_{H^{s-\f12}}+\|v^0\|^2_{H^{s-\f12}}+\sum_{+,-}\|\omega_{\pm,*}^{(n)}(0)\|^2_{H^{s-1}(\Om_*)}
+\sum_{+,-} \|\hat{\omega}_{\pm,*}^{(n)}(0)\|^2_{H^{s-1}(\hat{\Om}_*)}
 \leq \f12 A.
\end{align*}

\textbf{Recover of the velocity and magnetic field in the bulk region from the voticity and the current}
Recall that
\begin{align*}
\Om_f=\big\{ x \in \Om| x_2 < f (t,x_1)\big\}, \quad \hat{\Om}_f=\big\{ x \in \Om| x_2> f (t,x_1)\big\}.
\end{align*}
To consider the solution in a fixed domain, we introduce the
harmonic coordinate map $\Phi_f :\Omega_* \to\Omega_f$ and $\hat{\Phi}_f:\hat{\Om}_*\to\hat{\Om}_f$.

Let  $\Gamma_*$ be a fixed graph given by
\begin{align*}
\Gamma_*=\{(y_1,y_2):y_2=f_*(y_1)\}.
\end{align*}
The reference domains $\Om_*$ and $\hat{\Om}_*$ are given by
\begin{align*}
\Om_*=\{ y \in \bR\times [-1,1]| y_2 \leq f_*(y_1)\},
\quad \hat{\Om}_*=\{ y \in \bR\times [-1,1]| y_2 \geq f_*(y_1)\}.
\end{align*}
Then $\Phi_f$ is defined by the following equation.
\begin{equation*}
\begin{cases}
\Delta_y \Phi_f=0,\qquad \text{for } y\in \Omega_*,\\
\Phi_f(y_1,f_*(y_1))=(y_1,f(y_1)), \quad y_1\in\bR,\\
\Phi_f(y_1,-1)=(y',-1), \quad y_1\in\bR.
\end{cases}
\end{equation*}
$\hat{\Phi}_f$ can be defined similarly.
$\Phi_f$ is a bijection depending on $\|f_*\|_{W^{1,\infty}}$.
For the properties of the harmonic coordinate map, we refer to Lemma 2.1 in \cite{SWZ18}.

\textbf{Define the iteration process}
Note the domains of $\La_\pm^{(n)}$ ($\hat{\La}_\pm^{(n)}$) are $\Om_{f^{(n)}}$ ($\hat{\Om}_{f^{(n)}}$) which are different regions, we can not estimate their difference directly.
Thus, we introduce for $n\in\bN$,
$$
\La_{\pm,*}^{(n)}=\La_\pm^{(n)}\circ\Phi_{f^{(n)}}^{\pm},\quad
\hat{\La}_{\pm,*}^{(n)}=\hat{\La}_\pm^{(n)}\circ\hat{\Phi}_{f^{(n)}}^{\pm}.
$$
Note that the domains of $\omega_{\pm}^{(n+1)}, \hat{\omega}_{\pm}^{(n+1)}$ are
$\Om_{f^{(n)}}$, $\hat{\Om}_{f^{(n)}}$.
Thus we define
$$
\omega_{\pm,*}^{(n+1)}=\omega_\pm^{(n+1)}\circ\Phi_{f^{(n)}}^{\pm},\quad
\hat{\omega}_{\pm,*}^{(n+1)}=\hat{\omega}_\pm^{(n+1)}\circ\hat{\Phi}_{f^{(n)}}^{\pm},
$$
for $n\in\bN$.
By the properties of the harmonic coordinate map, there holds
\begin{align*}
&\| \La_{\pm,*}^{(n)} \|_{H^s(\Om_*)}\leq C(A,\delta) \| \La_{\pm}^{(n)} \|_{H^s(\Om_{f^{(n)}})},
\quad \| \hat{\La}_{\pm,*}^{(n)} \|_{H^s(\hat{\Om}_*)}\leq C(A,\delta) \| \hat{\La}_{\pm}^{(n)} \|_{H^s(\hat{\Om}_{f^{(n)}})},\\
&\| \omega_{\pm,*}^{(n+1)} \|_{H^s(\Om_*)}\leq C(A,\delta) \| \omega_{\pm}^{(n+1)} \|_{H^s(\Om_{f^{(n)}})},
\quad \| \hat{\omega}_{\pm,*}^{(n+1)} \|_{H^s(\hat{\Om}_*)}\leq C(A,\delta) \| \hat{\omega}_{\pm}^{(n+1)} \|_{H^s(\hat{\Om}_{f^{(n)}})}.
\end{align*}

Once we have obtained $(f^{(n+1)},\omega_{\pm}^{(n+1)}, \hat{\omega}_{\pm}^{(n+1)})$ from the linearized system, we need solve
$(\La_{\pm}^{(n+1)},\hat{\La}_{\pm}^{(n+1)})$ from the div-curl system
\begin{equation}\label{DC1}
\begin{cases}
\curl \La^\pm=\omega^\pm,\quad\div \La^\pm =0\quad  \text{in}\quad\Om_f,\\
\La^\pm\cdot N_f=(\p_t\pm\p_1)f\quad\text{on}\quad\Gamma_{f}, \\
\La^\pm\cdot e_2 = 0\quad\text{on}\quad\Gamma \,,
\end{cases}
\end{equation}
and
\begin{equation}\label{DC2}
\begin{cases}
\curl \hat{\La}^\pm=\hat{\omega}^\pm,\quad\div \La^\pm =0\quad  \text{in}\quad\hat{\Om}_f,\\
\hat{\La}^\pm\cdot N_f=(\p_t\pm\p_1)f\quad\text{on}\quad\Gamma_{f}, \\
\hat{\La}^\pm\cdot e_2 = 0\quad\text{on}\quad\hat{\Ga} .
\end{cases}
\end{equation}
In order to solve $(\La_{\pm}^{(n+1)},\hat{\La}_{\pm}^{(n+1)})$ in $(\Om_{f^{(n+1)}},\hat{\Om}_{f^{(n+1)}})$,
by the harmonic coordinate map, it suffices
to solve $(\La_{\pm,*}^{(n+1)},\hat{\La}_{\pm,*}^{(n+1)})$ from $\omega_{\pm,*}^{(n+1)}, \omega_{\pm,*}^{(n+1)}$.
To keep the div-curl structure of  \eqref{DC1} and \eqref{DC2} under $\Om_f$, $\hat{\Om}_f$,
the corresponding div-curl system in $\Om_*$, $\hat{\Om}_*$ need to be changed.

For a vector field $v_*^{(n)}$ ($\hat{v}_*^{(n)}$) defined in $\Om_*$ ($\Om_*^{(n)}$), we define
\begin{align*}
\curl_n v_*^{(n)}=\big(\curl (v_*^{(n)}\circ(\Phi_{f^{(n)}})^{-1})\big) \circ \Phi_{f^{(n)}},\quad
\div_n  v_*^{(n)}=\big(\div(v_*^{(n)}  \circ(\Phi_{f^{(n)}})^{-1})\big) \circ \Phi_{f^{(n)}},
\end{align*}
for $n\in\bN$. Then we find  $\La_{\pm,*}^{(n)}$ from the following equations:
\begin{equation}\label{DC3}
\begin{cases}
\curl_n \La_{\pm,*}^{(n)}=\omega_{\pm,*}^{(n)}  \quad\text{in}\quad\Om_*,\\
\div_n \La_{\pm,*}^{(n)}=0\quad\text{in}\quad\Om_*,\\
\La_{\pm,*}^{(n)}\cdot N_{f^{(n)}}=(\p_t\pm \p_1) f^{(n)} \quad\text{on}\quad\Ga_{*},\\
\La_{\pm,*}^{(n)}\cdot e_2 = 0, \quad\text{on}\quad\Ga .
\end{cases}
\end{equation}
Similar definition applies to $\hat{\La}_{\pm,*}^{(n)}$.
By solving the div-curl system \eqref{DC3}, there holds
\begin{align}\label{PD-10}
&\|\La_{\pm,*}^{(n)}\|_{H^{s}(\Om_*^\pm)}
\le C(A,\delta)\big(\| \omega_{\pm,*}^{(n)}\|_{H^{s-1}(\Omega_*^\pm)}
+\|(\p_t\pm\p_1)f^{(n)}\|_{H^{s-\f12}} \big).
\end{align}
Similar estimate also holds for $\hat{\La}_{\pm,*}^{(n)}$:
\begin{align}\label{PD-11}
\|\hat{\La}_{\pm,*}^{(n)}\|_{H^{s}(\Om_*^\pm)}
\le C(A,\delta)\big(\| \hat{\omega}_{\pm,*}^{(n)}\|_{H^{s-1}(\Omega_*^\pm)}
+\|(\p_t\pm\p_1)f^{(n)}\|_{H^{s-\f12}} \big).
\end{align}

\textbf{Define the iteration map}

The iteration sequence is defined as follows
\begin{equation*}
 \big(f^{(n)},\omega_{\pm,*}^{(n)}, \hat{\omega}_{\pm,*}^{(n)}\big).
\end{equation*}
Now we show the uniform bound of $\omega_\pm^{(n)},\hat{\omega}_\pm^{(n)}, f^{(n)}$ for all $n\geq 0$
for small time interval.
\begin{lem}\label{Lem-Abound}
Let $A<\f14, 0<\delta<\f12$.
Assume that
\begin{align*}
&\|\p_1 f^0\|^2_{H^{s-\f12}}+\|v^0\|^2_{H^{s-\f12}}+\sum_{+,-}\|\omega_{\pm,*}^{(n)}(0)\|^2_{H^{s-1}(\Om_*)}
+\sum_{+,-} \|\hat{\omega}_{\pm,*}^{(n)}(0)\|^2_{H^{s-1}(\hat{\Om}_*)}
 \leq \f12 A,\\
&\quad \| f^0\|_{L^\infty} \leq 1-2\delta.
\end{align*}
There exits $T_*>0$ depending on $A$ and $\delta$  such that
\begin{align*}
&\sup_{t\in[0,T_*]} E_{s}^w(\omega_\pm^{(n)},\hat{\omega}_\pm^{(n)})
+\sup_{t\in[0,T_*]}E_{s+\f12}^f(\p_tf^{(n)},f^{(n)}) \leq A, \\
&\qquad \sup_{t\in[0,T_*]}\| f^{(n)}(t) \|_{L^\infty}\leq 1-\delta.
\end{align*}
\end{lem}
\begin{remark}
The condition $0<A<\f14$ and $\|f^0\|_{L^\infty}<1-\delta$ are to guarantee the hyperbolicity condition and
 that the free surface will not touch the above the bottom fixed boundary.
The main purpose in this paper is the global solution for small initial data.
Hence the initial condition $0<A<\f14$ and $\|f^0\|_{L^\infty}<1-\delta$ are enough for our purpose.
\end{remark}

\begin{proof}
Assume that for some $T_*$ depending on $A$, $\delta$, there holds for $k\leq n$:
 \begin{align*} 
&\sup_{t\in[0,T_*]} E_{s}^w(\omega_\pm^{(k)},\hat{\omega}_\pm^{(k)})
+\sup_{t\in[0,T_*]}E_{s+\f12}^f(\p_tf^{(k)},f^{(k)}) \leq A, \\
&\qquad \sup_{t\in[0,T_*]}\| f^{(k)}(t) \|_{L^\infty}\leq 1-\delta.
\end{align*}
Then by using \eqref{PD-10} and \eqref{PD-11}, the bound \eqref{R-10} holds.
By using Lemma \ref{Lem-fL}, \ref{Lem-vL}, \ref{lem-PL}, we would derive that
\begin{align*}
&\sup_{0\leq t\leq T_*}  \Big( E_{s}^w(\omega_\pm^{(n+1)},\hat{\omega}_\pm^{(n+1)})
+ E_{s+\f12}^f(\p_tf^{(n+1)},f^{(n+1)}) \Big) \\
&\leq (\| v^0\|_{H^{s-\f12}}^2+\| \p_1f^0\|_{H^{s-\f12}}^2
+\|\omega_\pm^0\|^2_{H^{s-1}(\Om_{f^0})}+\|\hat{\omega}_\pm^0\|^2_{H^{s-1}(\hat{\Om}_{f^0})}
+C_1(A,\delta) T_*) e^{C_1(A,\delta) T_*},
\end{align*}
where $C_1(A,\delta)$ is a positive number depending on $A$ and $\delta$.
By choosing positive $T_*\leq 
\min \{\f14C_1^{-1}, \ln\f43 C_1^{-1}A^{-1}\}$. Then there holds
\begin{align*}
&\sup_{t\in[0,T_*]} E_{s}^w(\omega_\pm^{(n+1)},\hat{\omega}_\pm^{(n+1)})
+\sup_{t\in[0,T_*]}E_{s+\f12}^f(\p_tf^{(n+1)},f^{(n+1)}) \leq A.
\end{align*}

On the other hand, for $0\leq t\leq T_*$,
\begin{align*}
&\|f^{(n+1)}(t)-f^{(n+1)}(0)\|_{L^\infty}\le \|f^{(n+1)}(t)-f^{(n+1)}(0)\|_{H^1} \\
&\le  t\sup_{0\leq \tau\leq t}\|\p_\tau f^{(n+1)}(\tau)\|_{H^{s-\f12}} \leq T_* A^{\f12} .
\end{align*}
Set
$$\|f^{(n+1)}(0)\|_{L^\infty}\leq 1-2\delta. $$
By choosing $T_*=\min\{ \delta A^{-\f12}, \f14C_1^{-1}, \ln\f43 C_1^{-1}A^{-1}\}>0$, then
\begin{align*}
&\|f^{(n+1)}(t) \|_{L^\infty}
\leq \|f^{(n+1)}(0) \|_{L^\infty}+T_* A^{\f12} \leq 1-\delta.
\end{align*}
The lemma is proved.
\end{proof}

\textbf{Contraction of the iteration map}

In the sequel, we will show that the sequence
\begin{equation*}
 \big(f^{(n)},\omega_{\pm,*}^{(n)}, \hat{\omega}_{\pm,*}^{(n)}\big)
\end{equation*}
 is a Cauchy sequence in some Sobolev spaces.
\begin{lem}\label{Lem-contr}
Let $s'=s-1$. There exists $T_*'>0$ depending on $s, A, \delta$ such that
\begin{align*}
E^{(n+1)}\triangleq
&\sup_{0\leq t\leq T_*'}
\Big(\sum_{+,-} E_{s'}^w(\omega_\pm^{(n+1)}-\omega_\pm^{(n)},\hat{\omega}_\pm^{(n+1)}-\hat{\omega}_\pm^{(n)}) \\
&\quad + E_{s'+\f12}^f(\p_tf^{(n+1)}-\p_tf^{(n)},f^{(n+1)}-f^{(n)})
+\| f^{(n+1)}-f^{(n)} \|^2_{L^\infty(\bR)} \Big) \\
\leq& \f12\sup_{0\leq t\leq T_*'}  \Big(\sum_{+,-} E_{s'}^w(\omega_\pm^{(n)}-\omega_\pm^{(n-1)},\hat{\omega}_\pm^{(n)}-\hat{\omega}_\pm^{(n-1)}) \\
&\qquad+ E_{s'+\f12}^f(\p_tf^{(n)}-\p_tf^{(n-1)},f^{(n)}-f^{(n-1)})+\| f^{(n)}-f^{(n-1)} \|^2_{L^\infty(\bR)} \Big)=E^{(n)}.
\end{align*}
\end{lem}
\begin{proof}
Let us first show that
\begin{align*} 
\|\La_{\pm,*}^{(n)}-\La_{\pm,*}^{(n-1)}\|^2_{H^{s-1}(\Om_*)}
+\|\hat{\La}_{\pm,*}^{(n)}-\hat{\La}_{\pm,*}^{(n-1)}\|^2_{H^{s-1}(\Om_*)}\le CE^{(n)}.
\end{align*}
By the div-curl system \eqref{DC3}, we have
\begin{equation}\nonumber
\begin{cases}
\curl_{n}(\La_{\pm,*}^{(n)}-\La_{\pm,*}^{(n-1)})
=(\omega_{\pm,*}^{(n)}-\omega_{\pm,*}^{(n-1)})-(\curl_n-\curl_{n-1})\La_{\pm,*}^{(n-1)}\quad\text{in}\quad \Om_*,\\
 \div_n (\La_{\pm,*}^{(n)}-\La_{\pm,*}^{(n-1)})
=(\div_{n-1}-\div_{n})\La_{\pm,*}^{(n-1)} \quad\text{in}\quad\Om_*,\\
(\La_{\pm,*}^{(n)}-\La_{\pm,*}^{(n-1)}) \cdot N_{f^{(n)}}
=(\p_t\pm \p_1) (f^{(n)}-f^{(n-1)})+\La_{\pm,*}^{(n-1)}\cdot(N_{f^{(n-1)}}-N_{f^{(n)}})\quad\text{on}\quad\Ga_{*}, \\
(\La_{\pm,*}^{(n)}-\La_{\pm,*}^{(n-1)}) \cdot e_2=0\quad\text{on}\quad\Ga.
\end{cases}.
\end{equation}
By the elliptic estimate, we have
\begin{align*}
&\|\Phi_{f^{(n)}}-\Phi_{f^{(n-1)}}\|^2_{H^{s-1}(\Om_*^\pm)}
+\|\hat{\Phi}_{f^{(n)}}-\hat{\Phi}_{f^{(n-1)}}\|^2_{H^{s-1}(\hat{\Om}_*^\pm)} \\
&\le C(A,\delta)\big( \|\p_1f^{(n)}-\p_1f^{(n-1)}\|^2_{H^{s-\f32}}+\|f^{(n)}-f^{(n-1)}\|^2_{L^\infty}\big)
\le CE^{(n)}.
\end{align*}
On the other hand, there holds
\begin{align*}
\|(\curl_{n}-\curl_{n-1})\La_{\pm,*}^{(n-1)}\|^2_{H^{s-2}(\Omega_*^\pm)}
\le C\|\Phi_{f^{(n)}}-\Phi_{f^{(n-1)}}\|^2_{H^{s-1}(\Om_*^\pm)}
\le CE^{(n)}.
\end{align*}
Similarly, we have
\begin{align*}
&\|\div_{n}(\La_{\pm,*}^{(n)}-\La_{\pm,*}^{(n-1)})\|^2_{H^{s-2}(\Omega_*^\pm)}\le CE^{(n)},\\
&\|\ud{\La_{\pm,*}^{(n-1)}}\cdot(N_{f^{(n)}}-N_{f^{(n-1)}})\|^2_{H^{s-\f32}(\bR)}\le CE^{(n)}.
\end{align*}

Then we infer from the elliptic estimate that
\begin{align*} 
&\|\La_{\pm,*}^{(n)}-\La_{\pm,*}^{(n-1)}\|^2_{H^{s-1}(\Om_*^\pm)} \\
&\le C\left(\| \omega_{\pm,*}^{(n)}-\omega_{\pm,*}^{(n-1)}\|^2_{H^{s-2}(\Omega_*^\pm)}
+\|(\p_t\pm\p_1)(f^{(n)}-f^{(n-1)})\|^2_{H^{s-\f32}}+E^{(n)}\right)\le CE^{(n)}.
\end{align*}
Similarly, we have
\begin{align*}
\|\hat{\La}_{\pm,*}^{(n)}-\hat{\La}_{\pm,*}^{(n-1)}\|^2_{H^{s-1}(\hat{\Om}_*^\pm)}\le& CE^{(n)}.
\end{align*}

It follows from \eqref{Li-BB9} that
\begin{align} \label{Li-BB9-D}
&(\p_t^2-\p_1^2) (f^{(n+1)}-f^{(n)})
+ \f12({\ud{\La_+^1}}^{(n)} +{\ud{\La_-^1}}^{(n)}
+{\ud{\hat{\La}_+^1}}^{(n)} +{\ud{\hat{\La}_-^1}}^{(n)})
\p_t\p_1 (f^{(n+1)}-f^{(n)}) \\\nonumber
&+ \f12({\ud{\La_+^1}}^{(n)} +{\ud{\La_-^1}}^{(n)}
+{\ud{\hat{\La}_+^1}}^{(n)} +{\ud{\hat{\La}_-^1}}^{(n)}
-{\ud{\La_+^1}}^{(n-1)}-{\ud{\La_-^1}}^{(n-1)}
-{\ud{\hat{\La}_+^1}}^{(n-1)}-{\ud{\hat{\La}_-^1}}^{(n-1)})
\p_t\p_1 f^{(n)} \\\nonumber
&+\f12\big({\ud{\La_-^1}}^{(n)}  - {\ud{\La_+^1}}^{(n)}
+ {\ud{\hat{\La}_-^1}}^{(n)}  - {\ud{\hat{\La}_+^1}}^{(n)}
 \big)\p_1^2 (f^{(n+1)}-f^{(n)})  \\\nonumber
&+\f12\big({\ud{\La_-^1}}^{(n)}  - {\ud{\La_+^1}}^{(n)}
+ {\ud{\hat{\La}_-^1}}^{(n)}  - {\ud{\hat{\La}_+^1}}^{(n)}
-{\ud{\La_-^1}}^{(n-1)}+{\ud{\La_+^1}}^{(n-1)}
-{\ud{\hat{\La}_-^1}}^{(n-1)}+{\ud{\hat{\La}_+^1}}^{(n-1)}
 \big)\p_1^2 f^{(n)}\\\nonumber
&+\f12\big( {\ud{\La_-^1}}^{(n)}  {\ud{\La_+^1}}^{(n)}
+ {\ud{\hat{\La}_-^1}}^{(n)}  {\ud{\hat{\La}_+^1}}^{(n)} \big)\p_1^2 (f^{(n+1)}-f^{(n)})  \\\nonumber
&+\f12\big( {\ud{\La_-^1}}^{(n)}  {\ud{\La_+^1}}^{(n)}
+ {\ud{\hat{\La}_-^1}}^{(n)}  {\ud{\hat{\La}_+^1}}^{(n)}
-{\ud{\La_-^1}}^{(n-1)}  {\ud{\La_+^1}}^{(n-1)}
-{\ud{\hat{\La}_-^1}}^{(n-1)} {\ud{\hat{\La}_+^1}}^{(n-1)}  \big)\p_1^2 f^{(n)} \\\nonumber
&=-\f12\big( N_{f^{(n)}}\cdot\ud{\nabla p^{(n)}}+ N_{f^{(n)}}\cdot\ud{\nabla \hat{p}^{(n)}}
-N_{f^{(n-1)}}\cdot\ud{\nabla p^{(n-1)}}+ N_{f^{(n-1)}}\cdot\ud{\nabla \hat{p}^{(n-1)}} \big).
\end{align}
Then similar to the estimate for \eqref{Li-BB9}, for $s'=s-1$,
we treat the
$H^{s'+\f12}(\BR)$ energy estimate and $L^\infty(\bR)$ estimate
for \eqref{Li-BB9-D} with slight modifications. By using Lemma \ref{Lem-Abound}, for $0<T'_*\leq T_*$,
 we conclude  that
\begin{align}\label{R-1}
&\sup_{0\leq t\leq T'_*}
\Big( E_{s'+\f12}^f(\p_tf^{(n+1)}-\p_tf^{(n)},f^{(n+1)}-f^{(n)})
+\| f^{(n+1)}-f^{(n)} \|^2_{L^\infty(\bR)} \Big) \\\nonumber
&\leq C(A,\delta)T'_* \big(E^{(n+1)}+E^{(n)}\big).
\end{align}
Similarly, we can show that
\begin{align}\label{R-2}
&\sup_{0\leq t\leq T'_*}
\sum_{+,-} E_{s'}^w(\omega_{\pm}^{(n+1)}-\omega_{\pm}^{(n)},\hat{\omega}_{\pm}^{(n+1)}-\hat{\omega}_{\pm}^{(n)})  \\\nonumber
&\leq C(A,\delta)T'_* \big(E^{(n+1)}+E^{(n)}\big).
\end{align}
Combining the above \eqref{R-1} and \eqref{R-2}, by taking $T'_*=\min\{\f{1}{3C(A,\delta)},T_*\}$, then the conclusion follows.
\end{proof}
Lemma \ref{Lem-Abound} and Lemma \ref{Lem-contr} ensure that the map $\mathcal{F}$ has a unique fixed point
$(f,\omega^\pm, \hat{\omega}^\pm)$. From the construction of $\mathcal{F}$, we know that $(f,\omega^\pm, \hat{\omega}^\pm)$ satisfies
the equation of the free surface  \eqref{BB9}, the vorticity equation \eqref{F111}-\eqref{F112} and the div-curl system \eqref{DC1}-\eqref{DC2}. And the pressure satisfies \eqref{B22}.
The limit system is equivalent to the current-vortex sheet system \eqref{A18}-\eqref{A20}.
For its proof, we refer to section 9 in \cite{SWZ18}.

\section{Appendix}
In this appendix, we summarize some estimates of the Dirichlet-Neumann opertor.

\subsection{The elliptic equation in $\Om_f$ and $\hat{\Om}_f$}
We first recall some classical results for the elliptic boundary value problem. The reader may refer to, for instance, \cite{ABZ}.
\begin{equation}\label{EE1}
\begin{cases}
\Delta \Psi =0 \quad \textrm{in}\,\, \Om_f,\\
\Psi =\phi \quad \textrm{on} \,\, \Gamma_f,\\
\p_2 \Psi=0\,\,   \textrm{on} \,\,  \Ga .
\end{cases}
\end{equation}
Let $\phi \in H^{\sigma+\f12}( {\bR})$, $\sigma\in\bN$. We define an~$H^{\sigma+1}$ lifting of~$\phi$ in~$\Omega_f.$
To do so, let
\begin{align*}
\psi(x)&=e^{-x_2|\p_1|} \phi(x_1),\quad x_2\in\BR^+,\,x_1\in\BR.
\end{align*}
Now we introduce
\begin{align*}
\Phi(x)&= \zeta(x_2) \psi(x_1, f-x_2) ,\quad -1\leq x_2\leq f(t,x_1),\,x_1\in\BR\,.
\end{align*}
where $\zeta(x_2)\in C^\infty(\BR)$ and
\begin{equation*}
\zeta(x_2)=
\begin{cases}
1 , &\text{ if } |x_2|\leq 1-\f{c_0}{2} ,\\
0 , &\text{ if } |x_2|\geq 1-\f{c_0}{2},
\end{cases}
\end{equation*}
By the property of the Poisson kernel, we have
\begin{align*}
\| \Phi\|_{L^2(\Om_f)}
&\lesssim \| \langle \p_1\rangle^{-\f12} \phi \|_{L^2(\BR)},\\
\| \nabla \Phi\|_{L^2(\Om_f)}
&\lesssim \|  \phi \|_{H^{\f12}(\BR)}
 \big(1+\| \phi \|_{H^1(\BR)}\big) \, .
\end{align*}
For $2\leq k\leq \sigma+1$, there hold
\begin{align*}
\| \nabla^k \Phi\|_{L^2(\Om_f)}
&\lesssim  \|  \phi \|_{H^{k-\f12}(\BR)}
\mathcal{P} \big(1+\|  \phi \|_{H^{k-\f12}(\BR)} \big).
\end{align*}
where $\mathcal{P}(\cdot)$ is a polynomial.

Then for the problem
\begin{equation*}
\begin{cases}
-\Delta \Phi_1  =\Delta \Phi \quad \textrm{in}\,\, \Om_f,\\
\Phi_1 =0 \quad \textrm{on} \,\, \Gamma_f,\\
\p_2 \Psi_1=0\,\,   \textrm{on} \,\,  \Ga .
\end{cases}
\end{equation*}
It is easy to construct a variational solution $\Phi_1$ satisfying the boundary conditions and
$$ \| \nabla \Phi_1 \|_{L^2(\Om_f)}\lesssim \| \nabla \Phi  \|_{L^2(\Om_f)} \lesssim \| \phi \|_{H^{\f12}(\bR)}.$$
By the regularity for elliptic equations, there holds
\begin{align*}
\| \nabla \Phi_1\|_{H^{\sigma}(\Om_f)}
& \lesssim \| \nabla \Phi  \|_{H^{\sigma}(\Om_f)}\lesssim \| \phi \|_{H^{\sigma+\f12}(\bR)}.
\end{align*}
Due to the vanishing condition of $\Phi_1$ on $\Ga_f$, by using Poincar\'e inequality, there holds
$$ \| \Phi_1 \|_{L^2(\Om_f)} \lesssim \| \nabla \Phi_1 \|_{L^2(\Om_f)}  \lesssim \| \phi \|_{H^{\f12}(\bR)}.$$

Thus $\Psi=\Phi+\Phi_1$ is the solution to \eqref{EE1} satisfying
\begin{align*}
\| \Phi \|_{H^{\sigma+1}(\Om_f)} \lesssim \| \phi \|_{H^{\sigma+\f12}(\bR)}.
\end{align*}
On the other hand, $\phi$ is the restriction of $\Phi$ on $\Ga_f$ in the trace sense.
Hence
\begin{align*}
 \| \phi \|_{H^{\sigma+\f12}(\bR)}\lesssim \| \Phi \|_{H^{\sigma+1}(\Om_f)}.
\end{align*}

Similarly, for the elliptic boundary value problem
\begin{equation*}
\begin{cases}
\Delta \hat{\Psi} =0 \quad \textrm{in}\,\, \hat{\Om}_f,\\
\hat{\Psi} =\phi \quad \textrm{on} \,\, \Gamma_f,\\
\p_2\hat{\Psi}=0\,\,   \textrm{on} \,\,  \hat{\Ga} .
\end{cases}
\end{equation*}
We can construct the solution such that
\begin{align*}
\| \hat{\Phi} \|_{H^{\sigma+1}(\hat{\Om}_f)}\lesssim \| \phi \|_{H^{\sigma+\f12}(\bR)},\quad
 \| \phi \|_{H^{\sigma+\f12}(\bR)}\lesssim \| \hat{\Phi} \|_{H^{\sigma+1}(\hat{\Om}_f)}.
\end{align*}

\subsection{Dirichlet-Neumann operator}
In this subsection, we recall some estimates for the Dirichlet-Neumann operator.
For the details of the proof, we refer to section 3 in \cite{Lannes}, section 3.2 in \cite{ABZ}, section 3 and Appendix in \cite{SWZ18}.
The idea is straightening the free boundary and then
 transform the elliptic estimate into
the parabolic estimate by decoupling the elliptic equation into a forward and a
backward parabolic evolution equation.
Denote $K_{s+\f12,f}$ a constant depending on $\| \p_1f\|_{H^{s+\f12}}+\|f\|_{L^\infty}$
if $\| \p_1f\|_{H^{s+\f12}}+\|f\|_{L^\infty}<\infty $ for $s>\f52$.

The Dirichlet-Neumann operator is defined as follows:
\begin{align*}
\mathcal{N}_f \phi = N_f\cdot\ud{\nabla \Psi},\,\,
\mathcal{\hat{N}}_f \phi = - N_f\cdot\ud{\nabla \hat{\Psi}}\,.
\end{align*}
By the paradifferential operator and the well-known Bony's decomposition, the Dirichlet-Neumann operator can decomposed as follows:
(see for instance \cite{Lannes, ABZ, SWZ18})
\begin{align*}
\mathcal{N}_f \phi=\mathcal{T}\phi+ \mathcal{R}_1\phi,
 \quad \mathcal{\hat{N}}_f \phi=\mathcal{T}\phi+ \mathcal{R}_2\phi\,.
\end{align*}
Here $\mathcal{T}\phi$ is the leading term where the symbol $\lambda(x,\xi)$ is given by
\begin{align*}
\lambda(x,\xi)=
\sqrt{(1+|\nabla f|^2)|\xi|^2-(\nabla f\cdot\xi)^2}.
\end{align*}
$\mathcal{R}_1\phi$ and $\mathcal{R}_2\phi$ are the remainder terms.
For any $\sigma\in \big[\f12,s-\f12\big]$, there holds (see for instance \cite{Lannes, ABZ, SWZ18})
\begin{align*}
\|\mathcal{T}\phi\|_{H^{\sigma}(\bR)}\leq  K_{s+\f12,f}\|\phi\|_{H^{\sigma+1}(\bR)},\\
\|\mathcal{R}_1\phi\|_{H^{\sigma}(\bR)}+\|\mathcal{R}_2\phi\|_{H^{\sigma}(\bR)}\leq  K_{s+\f12,f}\|\phi\|_{H^{\sigma+1}(\bR)}.
\end{align*}
Thus for any $\sigma\in \big[\f12,s-\f12\big]$,
\begin{align}
&\| \mathcal{N}_f \phi\|_{H^{\sigma}(\bR)}+\| \mathcal{\hat{N}}_f \phi\|_{H^{\sigma}(\bR)}\leq  K_{s+\f12,f}\|\phi\|_{H^{\sigma+1}(\bR)},\nonumber\\
\label{appd-9}
&\| \mathcal{N}_f \phi-\mathcal{\hat{N}}_f \phi\|_{H^{\sigma}(\bR)}\leq  K_{s+\f12,f}\|\phi\|_{H^{\sigma}(\bR)}.
\end{align}

\section*{Acknowledgement.}

Cai is supported by NSFC grants (No. 12201122).
 Lei is in part supported by NSFC (No. 12431007), Sino-German Center (No. M-0548)
  and New Cornerstone Science Foundation through the XPLORER PRIZE.

\end{document}